\documentclass[10pt,onecolumn]{article}
\usepackage[left=1in,top=1in,right=1in,bottom=1in]{geometry}
\usepackage{graphicx,caption,subcaption,hyperref,authblk,enumitem,parskip}
\usepackage{amssymb,amsmath,amsthm,xcolor,mathrsfs}
\usepackage{diagbox}
\numberwithin{equation}{section}
\graphicspath{{./Figures/}}	
\DeclareMathOperator{\eps}{\varepsilon}

\theoremstyle{plain}

\newtheorem{proposition}{Proposition}[section]

\newtheorem{lemma}{Lemma}[section]
\newtheorem{remark}{Remark}[section]

\title{\vspace{-4em}Les Canards de Turing}
\author[1]{Theodore Vo}
\author[2]{Arjen Doelman}
\author[3]{Tasso J. Kaper}

\affil[1]{\footnotesize School of Mathematics, Monash University, Clayton, Victoria 3800, Australia}
\affil[2]{\footnotesize Mathematisch Instituut, Universiteit Leiden, 2300 RA Leiden, the Netherlands}
\affil[3]{\footnotesize Department of Mathematics and Statistics, Boston University, Boston, MA 02215, USA}

\begin{document}
\bibliographystyle{abbrv}
\maketitle
%=========================================================================================

%---------------------------------------------------------------------------------------------------
\begin{abstract}
In this article, we study a prototypical system of reaction-diffusion equations in which the diffusivities are widely separated. We report on the discovery of families of spatially periodic canard solutions that emerge from {\em singular Turing bifurcations}.  
We show that the small-amplitude, spatially periodic solutions that emerge from the Turing bifurcations form families of spatially periodic canards that oscillate about the homogeneous equilibrium, with wavenumbers near the critical value obtained from the Turing analysis. 
The emergence of these spatially periodic canards asymptotically close to the Turing bifurcations, which are reversible 1:1~resonant Hopf bifurcations in the spatial ODE system, is an analog in spatial dynamics of the emergence of limit cycle canards in the canard explosions that occur asymptotically close to Hopf bifurcations in time-dependent ODEs. 
We also find families of large-amplitude, spatially periodic canards, including some with $\mathcal{O}(1)$ wavenumber and some with small wavenumbers. 
These lie further from the homogeneous state and have a ``fast-slow" spatial structure, with segments of steep gradients and segments of gradual variation.
% (can only state this once we know the stability types) Overall, numerical continuation shows that the Turing bifurcation creates a Busse balloon filled with spatially periodic canard solutions.
In the full PDE system, we show that for most parameter values under study the Turing bifurcation is sub-critical, and we present the results of some direct numerical simulations showing that several of the different types of spatial canard patterns are attractors in the prototypical PDE.
    
To support the main numerical discoveries, we use the method of geometric desingularization and geometric singular perturbation theory on the spatial ODE system to demonstrate the existence of these families of spatially periodic canards. 
Crucially, in the singular limit, 
we study a novel class of {\em reversible folded singularities} of the spatial ODE system. 
In particular, there are two reversible folded saddle-node bifurcations of type II (RFSN-II), each occurring asymptotically close to a Turing bifurcation. We derive analytical formulas for these singularities and show that their canards play key roles in the observed families of small-amplitude and large-amplitude spatially periodic canard solutions. 
Then, for an interval of values of the bifurcation parameter further below the Turing bifurcation and RFSN-II point, the spatial ODE system also has spatially periodic canard patterns, however these are created by a reversible folded saddle (instead of the RFSN-II).   
It also turns out that there is an interesting scale invariance, so that some components of some spatial canards exhibit nearly self-similar dynamics. 
\end{abstract}
%---------------------------------------------------------------------------------------------------



%---------------------------------------------------------------------------------------------------
\section{Introduction}		\label{sec:intro}
%--------------------------------------------------------------------------------------------------- 

In mathematical models of pattern formation in biology, chemistry, ecology, engineering, material science, physics, and many other fields, the Turing bifurcation \cite{T1952} is one of the key mechanisms that generates spatially periodic patterns. 
It was a pioneering discovery of Alan Turing in 1952 that diffusion can destabilize spatially homogeneous steady states, which are stable states of the associated reaction kinetics, and that this instability to plane wave perturbations results in the formation of spatially periodic patterns as the attractors. See for example \cite{CH1993,E1965,EK2005,EP1998,HI2011,M1982,MG2000,M1993,SU2017,W1997}.

In this article, we report on the discovery and analysis of reaction-diffusion systems in which the spatially periodic solutions created in Turing bifurcations are spatially periodic canards. 
Like the known spatially periodic patterns that emerge from Turing bifurcations, 
these spatially periodic canards oscillate about the homogeneous equilibrium state. 
However, unlike the known periodic patterns, they consist of canard segments generated by folded singularities.

In the systems of spatial ordinary differential equations (ODEs) that govern the time-independent states of reaction-diffusion models,  Turing bifurcations correspond to reversible 1:1~resonant Hopf bifurcation points, see for example  \cite{HI2011,SU2017}.
We show that, asymptotically close to these reversible 1:1~Hopf bifurcations, the spatial ODE systems can have reversible folded saddle-node singularities of type II (RFSN-II) and that, together with the true and faux canards attached to them, these singularities can serve as the mechanisms responsible for the creation of the observed spatially periodic canard patterns.
We focus first on the van der Pol partial differential equation (PDE),
\begin{equation} \label{eq:vdp}
  \begin{split}
    u_t &= v - f(u) + d u_{xx}, \\ 
    v_t &= \eps (a-u) + v_{xx},
  \end{split}
\end{equation}
and later generalise to a class of activator-inhibitor systems.
The van der Pol PDE is a prototypical reaction diffusion system of activator-inhibitor type.
Here, $t \ge 0$, $x \in \mathbb R$, $u$ represents a voltage (or more genrally an activator), $v$ represents a recovery variable (or more generally an inhibitor), the nonlinear reaction function is $f(u) = \tfrac{1}{3}u^3-u$, $a$ is a threshold parameter, $\eps$ measures the timescale separation for the underlying kinetics, and $0< d \ll 1$ is the ratio of the diffusivities. 
With this small parameter, the activator diffuses more slowly than the inhibitor.
This spatial scale separation, i.e., the difference between the diffusivities of the interacting species, is an important --though not necessary-- feature of the emergence of periodic patterns in spatially extended systems.

For the van der Pol PDE \eqref{eq:vdp}, we consider all positive $\mathcal{O}(1)$ values of the  parameter $\eps$. 
This will allow us to consider the general kinetics of the van der Pol system. 
We recall that, in the regime of small $\eps$, the van der Pol ordinary differential equation (ODE), which corresponds to the ODE satisfied by spatially-homogeneous solutions of \eqref{eq:vdp}, is in the strongly nonlinear limit, and the limit cycles of the temporal ODE are relaxation oscillations, created in an explosion of temporal limit cycle canards. 
By contrast, in the regime of large $\eps$, the associated oscillations are in the weakly nonlinear limit, and the limit cycles of the kinetics ODE are small perturbations of circular orbits. 
(We refer the reader to \cite{H1980,P1921,P1926} for analysis of the classical van der Pol ODE, {\it i.e.} of the ODE satisfied by the spatially-independent solutions of \eqref{eq:vdp}.)

The main outcomes of this article are the numerical and analytical demonstrations that the van der Pol  PDE  \eqref{eq:vdp} possesses several main types of spatially periodic canards.
We find small-amplitude canards with $\mathcal{O}(1)$ wavenumbers (and hence $\mathcal{O}(1)$ spatial periods), and small-amplitude canards with small wavenumbers (and hence large spatial periods). 
These small-amplitude solutions emerge along branches emanating from the Turing point, and they oscillate spatially near the homogeneous state (see Fig.~\ref{fig:introspatialperiodic}). 
In addition, we find large-amplitude canards with $\mathcal{O}(1)$ wavenumbers, as well as large-amplitude canards with small wavenumbers (and hence large periods, which are also referred to as ``near-homoclinic" periodic solutions). 
These lie further from the homogeneous state in norm and have a distinct ``fast-slow" spatial structure, with segments of steep gradients and segments of gradual variation. 
The small-amplitude canards transition continuously into those with large-amplitudes. 
Moreover, at each value of $a$ in the main interval studied, one can find spatial canards of different types.

\begin{figure}[h!]
  \centering
  \includegraphics[width=5in]{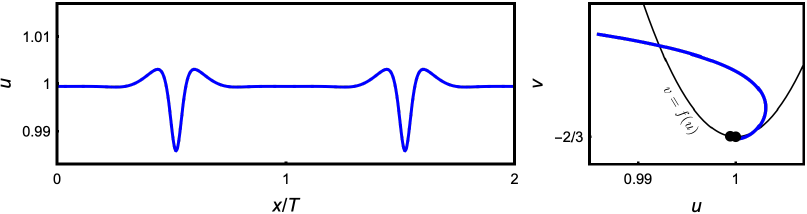}
  \put(-364,86){(a)}
  \put(-124,88){(b)}
  \caption{A small-amplitude spatially periodic canard solution for $a=0.999433, \eps = 0.1$,  and $\delta = 0.01$. (a) Spatial profile of the $u$-component of the solution over two periods $T$, with $k\approx 1.257$. (b) Projection of the solution onto the $(u,v)$ plane, showing that it is concentrated around a small neighbourhood of the fold of the $v=f(u)$ nullcline.
  The left dot on the nullcline is the equilibrium,   and the right dot is at the local minimum, where we will see that there is a key folded singularity.}
\label{fig:introspatialperiodic}
\end{figure}

We use the method of geometric desingularization (also known as the blow up method) and analysis of folded singularities,  especially of the new RFSN-II singularities, to establish the existence of the different types of spatially periodic canards.
The time-independent solutions of \eqref{eq:vdp} satisfy the following fourth-order system of ODEs, in which the spatial variable $x$ is the independent variable: 
\begin{equation} \label{eq:spatialODE-x-intro}
  \begin{split}
    \delta u_x &= p, \\ 
    \delta p_x &= f(u) - v, \\ 
    v_x &= q, \\
    q_x &= \eps(u-a),
  \end{split}
\end{equation}
where $\delta = \sqrt{d}$.
%and the auxiliary variables are defined by $p = \delta u_x$ and $q = v_x$.
We show that this spatial ODE system has a two-dimensional critical manifold in the four-dimensional $(u,p,v,q)$ phase space.
This manifold is induced by the cubic function $v=f(u)$, and hence it has three branches. 
We will see that two branches consist of saddle equilibria and one of center equilibria.
Crucially, there are fold sets that separate these  branches, and the RFSN-II singularities that are responsible for creating the spatial canards for values of $a$ near the Turing value 
\begin{equation}
    a_T = \sqrt{1 -2\delta \sqrt{\eps}} 
\end{equation}
lie on these fold sets.

Additionally, we study the desingularized reduced vector field on the critical manifold.
We will show that the desingularized reduced system has a cusp point precisely at $a=1$, where the RFSN-II singularity occurs in the limit $\delta=0$. 
Then, for all sufficiently small values of $\delta>0$, we show that key center-unstable and center-stable manifolds coincide
at locally unique critical values of $a$.
The first of these is 
\begin{equation}
\label{eq:ac}
    a_c (\delta) = 1 - \frac{5\eps}{48}\delta^2 +\mathcal{O}(\delta^3).
\end{equation} 
It corresponds to the locally unique parameter value at which key center-unstable and center-stable manifolds coincide.
Of all of the canards that pass through the neighborhood of the cusp point, the maximal canard has the longest segments near the stable and unstable manifolds of the cusp point (Fig.~\ref{fig:loops}(a)). 
There are also maximal canards that make a small loop about the equilibrium. as seen in the projection on to the $(u,q)$ coordinates  in phase space (Fig.~\ref{fig:loops}(b)). 
Then, under further variations in $a$, additional small loops develop around the equilibrium state, and the solutions exhibit nearly self-similar dynamics (Fig.~\ref{fig:loops}(c)). 
The analysis near the RFSN-II singularity will be valid for all values of $a$ and $\delta$ such that $a = 1 + \mathcal{O}(\delta^{3/2})$. 

\begin{figure}[h!]
  \centering
  \includegraphics[width=5in]{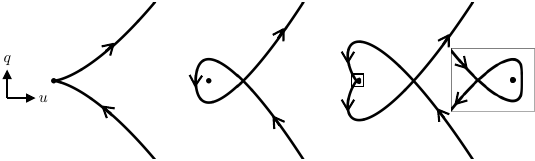}
  \put(-330,106){(a)}
  \put(-232,106){(b)}
  \put(-136,106){(c)}
  \caption{Development of small loops in the canard solutions under variation of $a$, shown in the $(u,q)$ projection. (a) The canard solutions meet at a cusp point. (b) Single loop about the equilibrium state. (c) Double loop around the equilibrium state. Inset: zoom on a neighbourhood of the equilibrium.}
  \label{fig:loops}
\end{figure}

In the spatial ODE, we also study the spatially periodic solutions for values of $a$ further from the Turing value.
We will show that there are reversible folded saddles (RFS) on the fold sets for all $\mathcal{O}(1)$ values of $a \in [0,1)$. 
Together with their true and faux canards, these RFS points will be shown numerically to be responsible for creating the spatial canards in this interval away from the Turing bifurcation $a_T$. 

To complement the analysis of the periodic canard solutions of the spatial ODE system, we also present the results of some direct numerical simulations of the PDE \eqref{eq:vdp}. 
These reveal that several of the large-amplitude spatial canards are stationary attractors. 
Moreover, we will find that there are also parameter regimes in which the attractors are small-amplitude solutions that are periodic in both space and time, with profiles given by spatial canards of the form discovered here. 
Overall, it turns out that the main Turing bifurcations we study (in the parameter regime in which $\eps \delta^2 < \frac{64}{625}$) are sub-critical bifurcations (also referred to as the focusing case), in which the coefficient on the cubic term in the Ginzburg-Landau equation is positive.
We will discuss how the spatial canard patterns here constitute a new class of nonlinear attractors in the sub-critical case, and how the canards provide a new selection mechanism.

Having outlined the main ODE and PDE results in this article, we briefly highlight the motivation for this study, and we provide a short comparison of the new spatial canards to classical temporal canards in fast-slow systems of ODEs.
To highlight the motivation, we recall that Turing bifurcations in two-component reaction-diffusion systems with widely-separated spatial scales, such as \eqref{eq:vdp}, are reversible, 1:1~resonant Hopf bifurcations in the associated spatial ODE systems 
\cite{D2019,HI2011,IMD1989,IP1993,MDK2001}.
In particular, the critical points that undergo Turing/1:1 Hopf bifurcations cannot lie on normally hyperbolic slow manifolds, but must instead be located on the boundaries of such manifolds. 
This is precisely the case with the RFSN-II singularities that lie on the fold sets between the two-dimensional, saddle and center slow manifolds here in the spatial ODE. 
Therefore, in some respect, from the point of view of spatial dynamics, there should be a natural connection between Turing bifurcations and the formation of canards. 

As to the comparison, the spatial canards introduced in this article may be viewed as analogs in spatial dynamics of the temporal limit cycle canards first discovered in the van der Pol ODE \cite{BCDD1981,D1984}, as well as of the canards of folded saddles and folded saddle-nodes in many other (temporal) systems, see for example \cite{B2013,DR1996,E1983,KS2001,MKKR1984,Moehlis,RKZE2003}. 
As just recalled, the parameter values at which Turing bifurcations occur are reversible 1:1~resonant Hopf bifurcations in the spatial ODE system \cite{D2019,HI2011,IMD1989,IP1993,MDK2001}. 
These are the analogs in spatial dynamics of the (singular) Hopf bifurcations that occur in fast-slow ODEs.
Then, just after the Turing bifurcations, families of spatially periodic canards are created, with small-amplitudes for $a$ close to $a_T$,  wavenumbers close to the critical wavenumber $k_T$ (determined by the point of marginal stability), and spatial profiles close to the plane wave $e^{ik_T x}$. 
Hence, these are analogs in spatial dynamics of the small-amplitude, temporally oscillating solutions that exist close to the Hopf bifurcation in fast-slow systems of ODEs and that oscillate essentially as $e^{i\omega t}$, where $\omega$ is the imaginary part of the eigenvalues at the Hopf bifurcation. 
Moreover, for parameter values further from $a_T$,  the spatial canards appear to be created by folded saddles, and they appear numerically to exist over a broad range of parameter values, as is also observed for folded saddles and their temporal canards in fast-slow systems of ODEs 
(see for example \cite{DKO2010,Mitry2017}).
Also, in the spatial dynamics, the maximal spatial canards act as separatrices that locally partition the phase and paramater spaces into regions of distinct spatial behavior, which is also analagous to the roles played by maximal canards in the phase spaces and parameter spaces of temporal fast-slow ODEs. Finally, in comparing, we will see that there are also key new features of the spatial canards that arise due to the dimension and geometry of the critical manifolds.

Zooming out more broadly, we suggest that the results presented here for spatially periodic canards also contribute to the growing literature about canards and bifurcation delay in spatially-extended systems.
Canards in the kinetics of a two-component model for the Belousov-Zhabotinsky reaction were demonstrated to play a crucial role in the nucleation and annihilation of trigger waves in a 1-D medium of phase waves \cite{Buchholtz1995}. 
Canards arise in the traveling wave ODEs of some reaction-diffusion systems \cite{Buric2006}. 
Slow passage through a saddle-node bifurcation in linear and semi-linear heat equations in 1-D can lead robustly to solutions that spend long times near unstable states as shown analytically in \cite{DPK2009}. 
In astrophysics, folded saddles and their canards play a central role in a model of solar wind  when there is a steady, spherically-symmetric outflow from the surface of a star \cite{CKW2017}. 
In an Amari-type neural field integral model \cite{ADK2017}, temporal canards were observed in the spatial patterns of coherent structures, as were  some more complex spatio-temporal patterns containing canard segments.

More recently for PDEs, canard solutions have been studied in an ODE model derived from a sub-critical, infinite-dimensional, pattern-forming system with nonlinear advection on a bounded domain
\cite{ADKK2017}, where they play a role in the nonlinear transitions between two primary states of the system describing the locations of stationary fronts.
Spatio-temporal canards serve as boundaries in multi-mode attractors of reaction-diffusion systems, in which different regions of the domains exhibit different modes of stable oscillation and the canards mediate the transition intervals, keeping  the regions separated, see 
\cite{KV2021,VBK2020}.
Delayed loss of stability occurs in nonlinear PDEs that undergo slow passage through Hopf bifurcations, with examples including the CGL equation, the Brusselator model, the FitzHugh-Nagumo PDE, and the Hodgkin-Huxley PDE (see \cite{GKV2022,KV2018}).
The solutions remain for long times near unstable states in a rich manner governed by space-time buffer curves. 
A rigorous framework for the local analysis of canard solutions and other forms of bifurcation delay was developed in \cite{ADVW2020} for systems in which the fast variables are governed by a PDE (i.e., infinite-dimensional dynamical system) and the slow variables are governed by ODEs (i.e., by a finite dimensional dynamical system). 
Slow passage through fold bifurcations has been studied in fast-slow systems of reaction-diffusion equations, using a Galerkin approach \cite{EHKPPZ2022}. 
Slow passage through Turing bifurcations has been studied in reaction-diffusion systems (see  \cite{AACS2024,ADVW2020,HJK2022,JK2024}), as has slow passage through pitchfork bifurcations in Allen-Cahn type equations, in the presence of quenching fronts with small spatial gradients (see \cite{GKSV2023,GKS2024}).

%There is also possible numerical evidence for spatially periodic canard solutions in a zebrafish model, see \cite{KLSBE2021}.

The article is organized as follows.
Section~\ref{sec:turingbifn} contains the application of classical analysis to identify the Turing bifurcations and Turing-Hopf bifurcations in \eqref{eq:vdp} and the application of the classical normal form analysis to show that the Turing bifurcation  to spatially periodic canards is sub-critical in the main parameter regimes we study, and also to identify where it is super-critical. 
In Section~\ref{sec:nearTuring}, we introduce the four main types of spatial canards, and present the results obtained from numerical continuation to identify regimes in the $(a,k)$ parameter plane in which spatial canards exist. 
In Section~\ref{sec:fast}, we begin the geometric singular perturbation analysis by analyzing the fast system, also known as the layer problem. We identify the jump conditions for the fast homoclinics that generate spikes and for the sharp-interfaces (Proposition~\ref{prop:jumpcondition}). Also, we identify the cusp of the fast system singularity that will be the organizing center for the spatial canard dynamics. 
We continue the geometric singular perturbation analysis in Section~\ref{sec:slow} by deriving the desingularized reduced vector field on the critical manifold and by studying the slow flow.
We show that there are folded saddle singularities with reversibility symmetry. Moreover, these undergo reversible folded saddle-node bifurcations of type II (RFSN-II) under variation of the system parameters. 
Then, in Section~\ref{sec:desingFSNII}, we rigorously analyze the dynamics around the RFSN-II using the blow-up technique. We determine the key parameter values for which there is an explosion of spatial canards with reversibility symmetry. 
Next in Section~\ref{sec:Turingspatialcanards}, we analyze in detail the geometry of the small-amplitude and large-amplitude spatially periodic canard solutions where the analysis is informed by the rigorous results about the fast system, the desingularized reduced vector field, and the folded singularity with its canards 
in Sections~\ref{sec:fast}--\ref{sec:desingFSNII}.  
In Section~\ref{sec:isolas}, we use that same information to deconstruct the bifurcation sequences along isolas of spatially periodic canards.
Then, in Section~\ref{sec:selfsimilar}, we analyze the aspects of the spatial canard dynamics that are nearly self-similar. 
Section~\ref{sec:spatialdynamicsanalog} describes how the spatial canards are analogs in spatial dynamics of the temporal limit cycle canards found in many fast-slow ODEs, such as the van der Pol ODE, FitzHugh-Nagumo ODE, the Lengyel-Epstein model, the Kaldor model, among others, and we identify important differences.
The final new results are in Section~\ref{sec:pdedyn}, where we present results from some direct PDE simulations that complement the analysis. Several of the different types of spatial canards discovered here are    observed to be attractors in the PDE \eqref{eq:vdp}.
We conclude the article in Section~\ref{sec:conclusions} with a summary of the main results, as well as a generalization from the prototype \eqref{eq:vdp} to a class of activator-inhibitor systems, and discussion of future work and open problems.
The appendices contain further information about the main numerical methods we employed, as well as the proofs of some of the propositions and lemmas.

%---------------------------------------------------------
\section{Turing Bifurcation to Spatially Periodic Solutions}		\label{sec:turingbifn}
%---------------------------------------------------------
In this section, we apply the classical Turing analysis \cite{T1952} (see also \cite{E1965,EK2005}) to derive the neutral stability curve of the homogeneous steady state $(u,v)=(a,f(a))$ of \eqref{eq:vdp}.
There are Turing and Hopf bifurcations that occur asymptotically close to each other in the parameter $a$. 
We will obtain the critical wavenumbers and system parameters for the onset of spatially periodic patterns. 
In addition, we will apply the classical normal form theory \cite{HI2011} for reversible 1:1~resonant Hopf points to show that, in the main parameter regimes studied here with $0<\delta \ll 1$,  the bifurcation to spatially periodic solutions is sub-critical; and, we also identify the conditions under which  it is super-critical.

We linearize \eqref{eq:vdp} about the homogeneous state  and then Fourier transform the system in space, so that the governing equations become
\begin{equation}
\label{eq:eigenvaluecurves}
  \begin{split}
    \begin{bmatrix} \dot U \\ \dot V \end{bmatrix} = DF \begin{bmatrix} U \\ V  \end{bmatrix} = \begin{bmatrix} -f^\prime(a)-dk^2 & 1 \\ -\eps & -k^2 \end{bmatrix} \begin{bmatrix} U \\ V  \end{bmatrix}.
  \end{split}
\end{equation}
The overdot denotes the time derivative, $k$ is the wavenumber, and $U(k,t)$ and $V(k,t)$ denote the Fourier transforms of $u(x,t)$ and $v(x,t)$. 
For each $k \in \mathbb{R}$, the Jacobian has two eigenvalues $\lambda_\pm(k)$, and the state $(a,f(a))$ is spectrally stable as a solution of the PDE (\ref{eq:vdp}) if Re$(\lambda_\pm(k)) < 0$ for all $k\in \mathbb{R}$.  
The trace of the Jacobian, $\operatorname{tr} DF= -f^\prime(a)-(d+1)k^2$, is negative whenever 
$f'(a) = a^2 - 1 > 0$, and positive for $-1<a<1$ (where $f'(a)<0$). 
Hence, we recover the classical (non-spatial) Hopf bifurcation of the van der Pol equation at $a=1$. 
See the left panel in Fig. \ref{fig:EVcurves}.
\begin{figure}[ht]
\centering
        \begin{subfigure}[t]{0.475\textwidth}
		\centering
		\includegraphics[width=\linewidth]{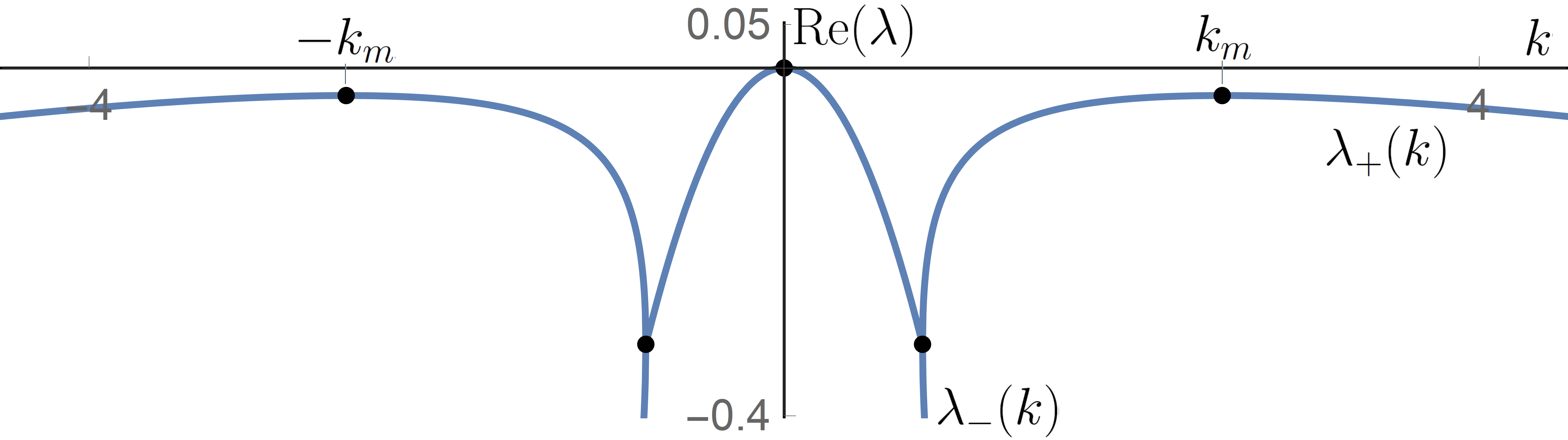}
%		\caption{}
	\end{subfigure}
        \hspace{0.25cm}
        \begin{subfigure}[t]{0.475\textwidth}
		\centering
		\includegraphics[width=\linewidth]{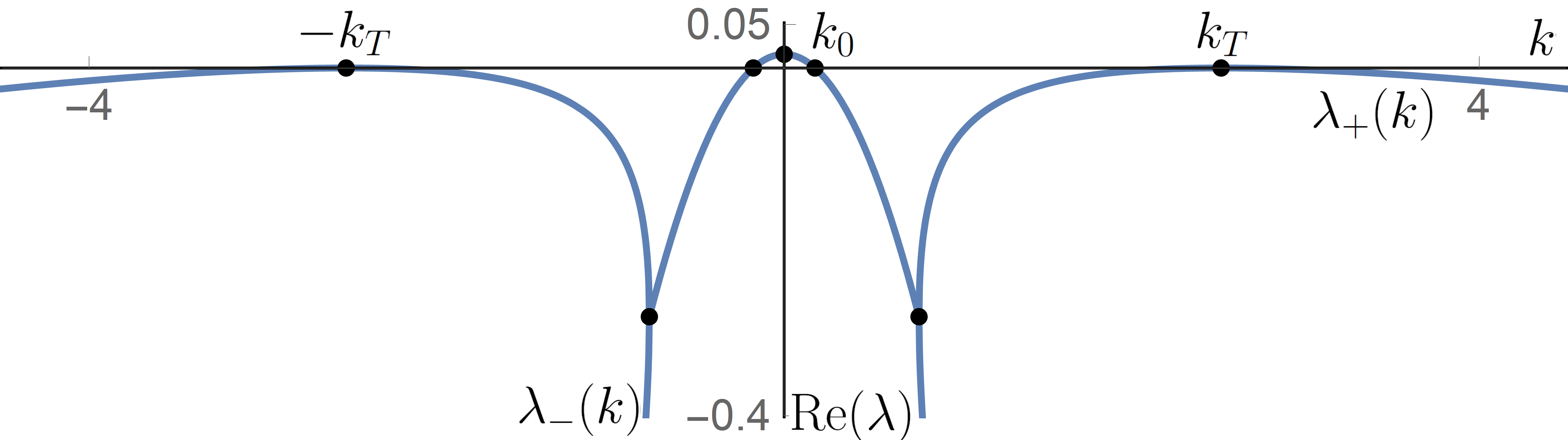}
%		\caption{}
	\end{subfigure}
\caption{Plots of the eigenvalues $\lambda_\pm(k)$ of the Jacobian matrix $DF$ in (\ref{eq:eigenvaluecurves}) as functions of the wavenumber $k$, at two different values of $a$. 
Left: at $a=1$, the (spatial) Hopf bifurcation. 
Right: at $a_T = \sqrt{1-2 \sqrt{\eps d}}$, the Turing bifurcation.  
At $a=1$, ${\rm Re}(\lambda_+(0))=0$, so that the $k=0$ mode is marginally stable. 
Also, at $a=1$, there are local maxima in $\lambda_+(k)$ at $k=\pm k_m\sim \pm (\eps/d)^{1/4}$,
and we note that $k_m \sim k_T$ for small $d$.  
In particular, $\lambda_+(k_m) = -2 \sqrt{\eps d}$ (to leading order), so that the modes with $k$ near $k_T$ are only very weakly stable at $a=1$.
Now, at $a=a_T$ the Turing bifurcation, 
the homogeneous state is marginally stable to perturbations with wavenumber $k_T$.
Also, it is
weakly unstable to perturbations with wavenumbers near $k=0$, since
$0 < \,$Re$(\lambda_\pm(k)) \leq \,$Re$(\lambda_\pm(0)) = \sqrt{\eps d}  \ll 1$ for $k \in (-k_0,k_0)$ with $k_0 = \sqrt{2}(\eps d)^{1/4} \ll 1$. 
%and at the Hopf bifurcation $\lambda_+(k) \leq \lambda_+(k_m) = -2 \sqrt{\eps} \delta$ for $k$ near $k_m$ (both at leading order).
Note that for small $k$, $\lambda_\pm(k) \notin \mathbb{R}$. 
Also, note that for reasons of presentation, we have set $\delta=\sqrt{d}= 0.05$, slightly larger than the value used throughout, and kept $\eps=0.1$, its standard value.} 
\label{fig:EVcurves}
\end{figure}

Now, for $a<1$, there are two key dynamical features. First, a narrow interval of unstable wavenumbers exists for $a<1$, centered around $k=0$, so that the homogeneous state becomes unstable to long wavelength perturbations that oscillate in time as $e^{i \omega_H t}$, with $\omega_H = $Im$(\lambda_+(0))$.  However, this instability is only a weak instability, since the rate constant (${\rm Re}(\lambda_+(0))=\sqrt{\eps d}$) is asymptotically small for $0<d \ll 1$ and $\eps=\mathcal{O}(1)$.

Second, a Turing instability occurs slightly below $a=1$.
It appears at the critical wavenumber and critical parameter determined by the conditions
\begin{equation}
\label{eq:Turingcond}
    \det DF = 0 \quad \text{ and } \quad \frac{\partial}{\partial k^2} \det DF = 0
\end{equation}
(for which $\lambda_\pm (k) \in \mathbb{R}$).
For the van der Pol PDE \eqref{eq:vdp}, we find 
\begin{equation}  \label{eq:kTaT}
k_T^2 = - \frac{f^\prime(a_T)}{2d} = \sqrt{\frac{\eps}{d}} \quad \text{ and } \quad a_T = \pm \sqrt{1 - 2\sqrt{\eps d} }.
\end{equation}
The focus in this article will be on the positive value of $a_T$; the results for the negative value may be obtained by using the symmetry $(u,v,a) \to (-u,-v,-a)$ of \eqref{eq:vdp}.

Precisely at $a=a_T$, the real part of the dominant eigenvalue of the Jacobian is zero for wavenumbers $k = \pm k_T = \pm {\left(\frac{\eps}{d}\right)}^{1/4}$. 
See the right panel in Fig. \ref{fig:EVcurves}. 
The Turing point $a_T$ marks the boundary between two different (linear) regimes.
On one side, for $a \lesssim a_T$, there are intervals of wavenumbers $k$, one about each of $\pm k_T$, over which the homogeneous steady state is linearly unstable to plane waves $e^{ikx}$ since the determinant is strictly negative there. 
On the other side, for $a \gtrsim a_T$, the homogeneous state is linearly stable to plane waves with $k$ near $k_T$ (though not to those with $k=0$  if $a<1$ by the above).  

Critically, the modes with $k$ near $k_T$ (which is $\sim k_m$ for small $d$) are only weakly stable for $a \gtrsim a_T$, as may be seen at $a=1$ in the left panel in Fig.~\ref{fig:EVcurves}.
There, $\lambda_+(k_m)=-2\sqrt{\eps d}$, so that the rate constant is asymptotically small, and nonlinear terms also play a central role.
In comparison, the magnitude of this weak stability of the Turing modes ($k\sim k_T$) is of the same size asymptotically as the magnitude of the weak instability of the Hopf modes ($k\sim 0$).  Therefore, it is important and useful first to study the two competing mechanisms individually and then to study their interactions. 

In this article, we focus primarily on the spatial dynamics of stationary solutions of (\ref{eq:vdp}), which turn out to capture important features of the overall dynamics of \eqref{eq:vdp} for $a$ near $a_T$. 
In addition, toward the end of the article, we identify some of the rich dynamics created by the interactions of the Turing and Hopf modes.

We study the Turing bifurcation by using the equivalent formulation obtained through the spatial ODE \eqref{eq:spatialODE-x-intro},  which we recall is
\begin{equation} \label{eq:spatialODE-x}
  \begin{split}
    \delta u_x &= p, \\ 
    \delta p_x &= f(u) - v, \\ 
    v_x &= q, \\
    q_x &= \eps(u-a).
  \end{split}
\end{equation}
and we recall the small parameter is $\delta = \sqrt{d}$. 
(We refer to \cite{IP1993,SU2017} for the general theory about the equivalence of this spatial ODE formulation.) 

The equilibrium at $(u,p,v,q) = (a,0,f(a),0)$ corresponds to the homogeneous steady state. Linear stability analysis shows that the Jacobian of \eqref{eq:spatialODE-x} has a quartet of eigenvalues
\begin{equation}
\label{eq:spatial-eivals}
\mu = \pm \frac{1}{\sqrt{2}\delta}
\left[
f'(a) \pm \sqrt{ (f'(a))^2 - 4 \eps \delta^2}
        \right]^{1/2},
\end{equation}
which are symmetric about the real- and imaginary-axes in the spectral plane. See Fig.~\ref{fig:spatialeigenvalues}.

\begin{figure}[h!]  
  \centering
  \includegraphics[width=5in]{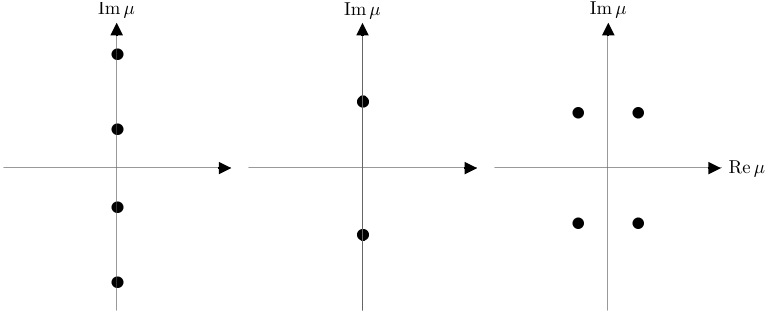}
  \put(-360,140){(a)}
  \put(-242,140){(b)}
  \put(-126,140){(c)}
  \caption{Eigenvalue quartets of the spatial problem \eqref{eq:spatialODE-x} around the reversible, 1:1~resonant Hopf bifurcation. (a) For each $a \in (-a_T,a_T)$, there is a quartet of pure imaginary eigenvalues.  (b) For $a=a_T$, there is a pair of repeated imaginary eigenvalues. (c) For each $|a| \in (a_T,\sqrt{1+2\delta \sqrt{\eps}})$, there is a pair of complex conjugate eigenvalues, one set with positive real parts and the other with negative real parts. The reversible, 1:1~resonant Hopf bifurcation  point $(a_T,0,f(a_T),0)$ is non-degenerate.
  }
  \label{fig:spatialeigenvalues}
\end{figure}

Exactly at the critical value $a_T=\sqrt{1 - 2 \delta \sqrt{\eps}}$ given by \eqref{eq:kTaT}, where $f'(a)<0$ and $(f'(a))^2 = 4 \eps \delta^2$, the quartet of eigenvalues consists of two coincident pairs of pure imaginary eigenvalues. 
Hence, at $a_T$, the equilibrium is a reversible, 1:1~resonant Hopf bifurcation point. 
Moreover, this point is non-degenerate, because the two pairs pass through this configuration transversely as $a$ passes through $a_T$. 
See Fig.~\ref{fig:spatialeigenvalues}. 
Furthermore, this Turing bifurcation is {\em singular} because one pair of the eigenvalues becomes singular in the limit $\delta \to 0$. 
For $-a_T < a < a_T$, the quartet consists of two pairs of pure imaginary eigenvalues with different imaginary parts, while for $a_T < a < \sqrt{1+2\delta\sqrt{\eps}}$ it consists of two pairs of complex conjugate eigenvalues, one with negative real parts and the other with positive real parts of equal magnitude, where the upper bound of this interval corresponds to the parameter value at which the eigenvalues become real.
(See \cite{CH1993,IMD1989,MDK2001,S2003,SU2017} for other examples of the relation between a Turing bifurcation in a PDE and a reversible 1:1~resonant Hopf bifurcation in the associated spatial ODE system.)

It is also useful to study the spatial ODE system \eqref{eq:spatialODE-x} in the stretched spatial variable $y = \frac{x}{\delta}$, 
\begin{equation}
\label{eq:spatialODE-y}
\begin{split}
u_y &= p \\
p_y &= f(u) - v \\
v_y &= \delta q \\
q_y &= \delta \eps (u-a).
\end{split}
\end{equation}
This system will be used to analyze  the steep gradients in the spatially periodic solutions that we study, while system \eqref{eq:spatialODE-x} will be used for the slowly-varying segments of the spatial periodic solutions.
For all $\delta>0$, the systems of ODEs \eqref{eq:spatialODE-x} and \eqref{eq:spatialODE-y} are equivalent.

\smallskip

\begin{remark}
For a spatially periodic solution of \eqref{eq:spatialODE-y} with period $T$, the wavenumber with respect to the stretched $y$ variable is $k_y = \tfrac{2\pi}{T}$. For comparison with the wavenumbers in \eqref{eq:kTaT}, we work with the wavenumber given by $k = \tfrac{1}{\delta} k_y = \tfrac{2\pi}{\delta T}$. Throughout the remainder of the article, we will use the wavenumber $k = \tfrac{2\pi}{\delta T}$. 
\end{remark}

System \eqref{eq:spatialODE-y} has a reversibility symmetry, which it inherits from the $x\to -x$ symmetry of the system of two coupled second-order ODEs that govern time-indepednent solutions of \eqref{eq:vdp}. Let 
\[
\mathcal{R}
     = \left[
    \begin{array}{cccc}
        1 & 0 & 0 & 0 \\
        0 & -1 & 0 & 0 \\
        0 & 0 & 1 & 0 \\
        0 & 0 & 0 & -1 
    \end{array}
    \right], \ \ 
    {\bf u} = \left[
    \begin{array}{c} u \\ p \\ v \\ q
    \end{array} \right], \quad {\rm and}
    \quad
    {\bf F} = \left[ \begin{array}{c}
    p \\ \frac{1}{3}u^3 - u - v \\
    \delta q \\
    \delta \eps (u-a)
    \end{array}
    \right].
\]
Then, $\mathcal{R}$ is a reversibility symmetry of \eqref{eq:spatialODE-y}, because it anti-commutes with the vector field:
\begin{equation}
\label{eq:RF}
    \mathcal{R} {\bf F} ({\bf u}) = - {\bf F}( \mathcal{R} {\bf u}).
\end{equation}
(See for example \cite{HI2011,IP1993} for general results about reversibility in the spatial ODEs governing stationary solutions of reaction-diffusion systems.)  The presence of this reversibility symmetry is also what guarantees the symmetries about the real- and imaginary-axes of the resolvent and spectrum of the operator $L$ obtained by linearizing the vector field ${\bf F}$ at a stationary state, since ${\bf u}$ is real-valued and since the reversibility $\mathcal{R}L = -L\mathcal{R}$ implies $\left( \mu {\mathbb I} + L \right)^{-1} \mathcal{R} = \mathcal{R} \left(\mu {\mathbb I} - L \right)^{-1}$. 

For all $\eps>0$, system~\eqref{eq:spatialODE-x} (equivalently \eqref{eq:spatialODE-y}) has the following conserved quantity/first integral: 
\begin{equation}
\label{eq:G}
    \mathcal{G}(u,p,v,q,a)=
    \frac{1}{2} \eps p^2 - \frac{1}{2} q^2 - \eps [\tilde{f}(u) + (a-u)v], 
\end{equation}
where $\tilde{f}(u) = \int f(u) \,du = \frac{u^4}{12} - \frac{u^2}{2}$.
(It may be derived for example by setting $u_t=0$ in the first equation of \eqref{eq:vdp} and multiplying it by $\eps u_x$, setting $v_t=0$ in the second equation in \eqref{eq:vdp} and multiplying it by $v_x$, then subtracting the second equation from the first so that all terms are perfect derivatives, and recalling the definitions of $p$ and $q$ given in \eqref{eq:spatialODE-x-intro}.)
%This integral is independent of $\delta$. 
At the stationary state,  $\mathcal{G}(a,0,f(a),0,a)=\frac{1}{12}\eps a^2 ( 6 - a^2 )$. 

The existence of small-amplitude, spatially periodic solutions of the ODE system near the critical Turing bifurcation and with wavenumber $k$ near $k_T$ is established by the following proposition. 
The proof is presented in Appendix~\ref{sec:app-proof}, where we use the standard procedure to derive the normal form for the reversible 1:1~resonant Hopf bifurcation in the spatial ODE and then directly apply Theorem 3.21 in Chapter 4.3.3 of \cite{HI2011} to derive the conclusions about the system dynamics. 

\smallskip

\begin{proposition}
\label{prop:Turingcanards}
 For $a-a_T$ sufficiently small, where $a_T=\sqrt{1-2\delta \sqrt{\eps}}$, we have:
 \begin{itemize}
\item[(i)]
For $\eps \delta^2 < \frac{64}{625}$, the spatial ODE system \eqref{eq:spatialODE-x} has a one-parameter family of spatially periodic solutions that surround the symmetric equilibrium $(a,0,f(a),0)$, and it has a two-parameter family of quasi-periodic solutions located on KAM tori. 
The same conclusion also holds for $\eps \delta^2 > \frac{64}{625}$ but then only when $a-a_T<0$.
\item[(ii)] 
For $a-a_T<0$ and $\eps \delta^2 > \frac{64}{625}$, there exists a one-parameter family of pairs of reversible homoclinics to periodic solutions.
\item[(iii)]
For $a-a_T>0$ and $\eps \delta^2 < \frac{64}{625}$, there is a pair of reversible homoclinic orbits to the symmetric equilibrium.
\item[(iv)]
For $a-a_T>0$ and $\eps \delta^2 > \frac{64}{625}$, there is just the symmetric equilibrium and no other bounded solutions.
\end{itemize}
\end{proposition}

We are most interested in the parameter values for which $\eps \delta^2 < \frac{64}{625}$, which corresponds to the main case in part (i) and to part (iii). 
These parts of the proposition  establish the existence of one-parameter families of small-amplitude, spatially periodic solutions of \eqref{eq:spatialODE-x} near $a_T$, and the limiting homoclinics.
The Turing bifurcation is sub-critical, and the coefficient on the cubic term in the Ginzburg-Landau PDE is positive \cite{IMD1989}.
(In contrast, for $\eps \delta^2 > \frac{64}{625}$, the Turing bifurcation is super-critical, and the coefficient on the cubic term in the Ginzburg-Landau equation is negative.) 

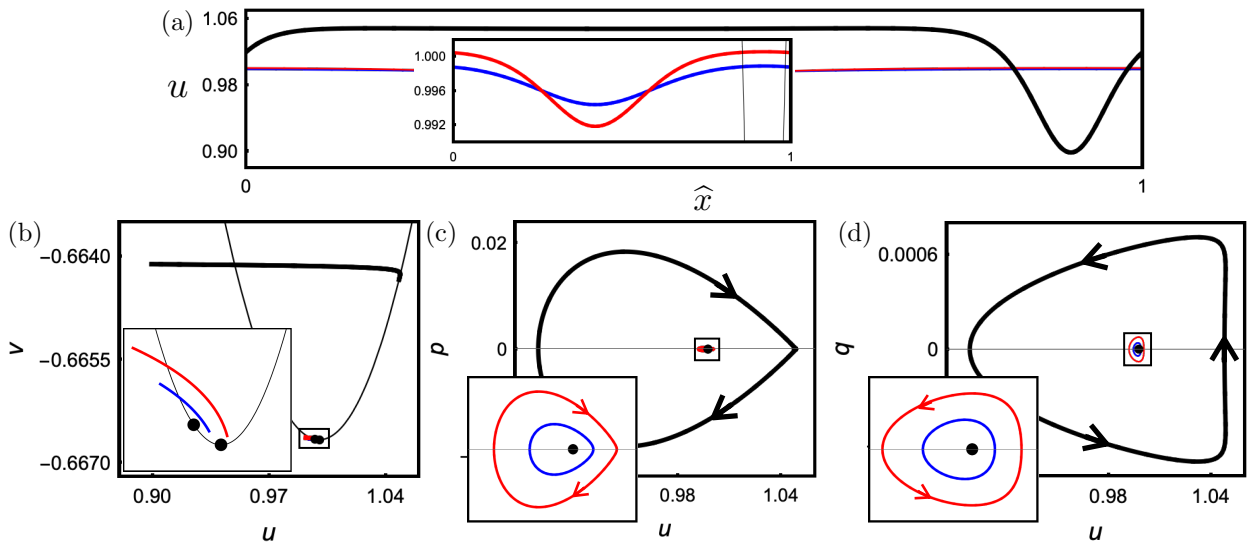
\begin{figure}[!htbp]  
  \centering
  \includegraphics[width=\textwidth]{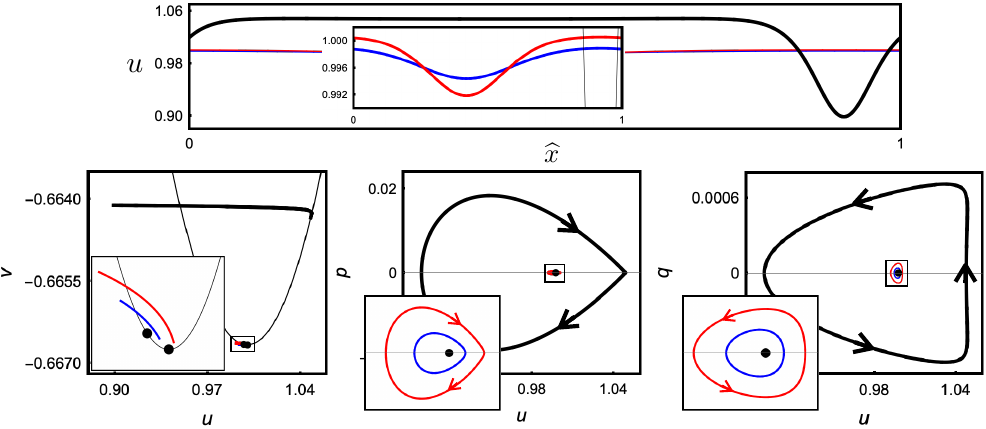}
  \put(-412,195){(a)}
  \put(-470,115){(b)}
  \put(-312,115){(c)}
  \put(-156,115){(d)}
  \caption{Solutions of \eqref{eq:spatialODE-y} near the critical Turing bifurcation at $a_T \approx 0.996832$ for $\eps = 0.1$, $\delta = 0.01$, and $a=1.03193$ (black), $a=0.997476$ (blue), and $a=0.997017$ (red).
  (a) Profiles of $u(\widehat x)$, where $\widehat x = \tfrac{x}{T}$ is the spatial variable normalized to the unit interval. The periods are $T=115$ so that the corresponding wavenumbers are $k\approx 5.464$ (cf. the critical wavenumber at the Turing bifurcation, $k_T \approx 5.623$).
  The projections of these solutions into the $(u,v)$, $(u,p)$, and $(u,q)$ planes are shown in panels (b), (c), and (d), respectively, with insets so the details are visible. 
 }  
  \label{fig:soln_knearcritical} 
  \end{figure}

Some representative spatially periodic solutions are shown (in blue and red) in Fig.~\ref{fig:soln_knearcritical} for parameter values $a$ close to $a_T$. 
These solutions exhibit small-amplitude oscillations about the homogeneous state, and their wavenumber $k$ is close to $k_T$. 
Moreover, the closer  $a$ is to $a_T$, the smaller the amplitude of the solution, the closer $k$ is to $k_T$, and the closer the spatial profile is to that of the plane wave $e^{ik_T x}$.  
Then, for $a$ slightly further away, nonlinear terms have a dominant effect (black orbit), and a corner begins to form in the orbit, which will develop into a cusp.

\smallskip
\begin{remark}
From the point of view of Ginzburg-Landau theory for the weakly nonlinear stability of spatially periodic patterns created in Turing bifurcations, the sub-critical case that we focus on is the least studied and understood. In fact, in the sub-critical case, the Ginzburg-Landau formalism typically predicts that the amplitude of the solutions under consideration will grow beyond that of their assumed near-onset magnitude, and thus beyond the region of validity of the Ginzburg-Landau set-up. 
As the analysis below will show, it turns out to be possible to go beyond this magnitude, and our results may thus be expected to shed a fundamental light on the behavior of patterns near a sub-critical Turing bifurcation (see also Section~\ref{sec:pdedyn}).
\end{remark}

%---------------------------------------------------------
\section{Spatially Periodic Canard Solutions}		\label{sec:nearTuring}
%---------------------------------------------------------

In this section, we present representative numerical simulations of the plethora of spatially periodic canard solutions of \eqref{eq:vdp} that are created asymptotically close to the Turing bifurcation. There are at least four different types of spatial canards, depending on amplitudes and wavenumbers.

%---------------------------------------------------------
\subsection{Small-amplitude spatial canards with \texorpdfstring{$\mathcal{O}(1)$}{Lg} wavenumbers}		\label{sec:numerics-1}  
%---------------------------------------------------------

Immediately beyond the parameter regime of essentially regular plane wave solutions, where the nonlinear terms become dominant, we find small-amplitude spatially periodic canard solutions with $\mathcal O(1)$ wavenumbers. 
A representative spatial canard is shown in Fig.~\ref{fig:kOrderOne_SAO}.
This solution has period $T \approx 353.366$ $(k \approx 1.778)$, and 
it lies on the level set $\mathcal G = \eps \left( \tfrac{2}{3}a-\tfrac{1}{4} \right)$.
The spatial profile (see Fig.~\ref{fig:kOrderOne_SAO}(a)) shows that it lies in a small neighborhood of the $u=1$ state over the entire period. 
Moreover, on a large portion of the period, the solution slowly increases or slowly decreases. In between these slow regimes, there is a spatial interval on which the solution makes three small-amplitude oscillations around the state $u=a$. 

\begin{figure}[h!]  
  \centering
  \includegraphics[width=\textwidth]{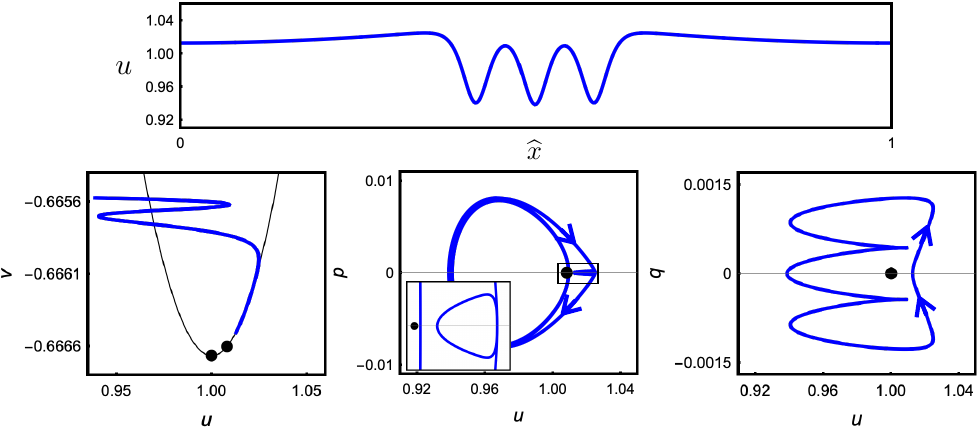}
  \put(-416,196){(a)}
  \put(-470,116){(b)}
  \put(-312,116){(c)}
  \put(-160,116){(d)}
  \caption{Small-amplitude spatially periodic canard solution of \eqref{eq:vdp} with $k\approx 1.778$ for $a=1.008$, $\eps=0.1$, and $\delta=0.01$.  For these parameters, $a_T \approx 0.996832$, so the value of $a$ is slightly above $a_T$. 
  (a) Profile of $u(\hat{x})$, where $\hat{x} = \tfrac{x}{T}$ is a scaled version of the spatial variable chosen to normalize the period ($T \approx 353.366$) to unit length. 
  (b) Projection onto the $(u,v)$ plane. 
  The cubic nullcline $v=f(u)$ is represented by the black curve, and the black dots mark the local minimum  $(u,v)=(1,-\frac{2}{3})$ of the nullcline and the equilibrium $(u,v)=(a,f(a))$.
  (c) Projection onto the $(u,p)$ plane. Inset: zoom on the neighborhood of the equilibrium (black dot).  
  (d) Projection onto the $(u,q)$ plane. The solution oscillates about the equilibrium at $(a,0,f(a),0)$ (black dot).
  This solution lies on the level set $\mathcal{G}=\eps\left( \frac{2}{3}a - \frac{1}{4}\right)$.
  } 
  \label{fig:kOrderOne_SAO}  
\end{figure}
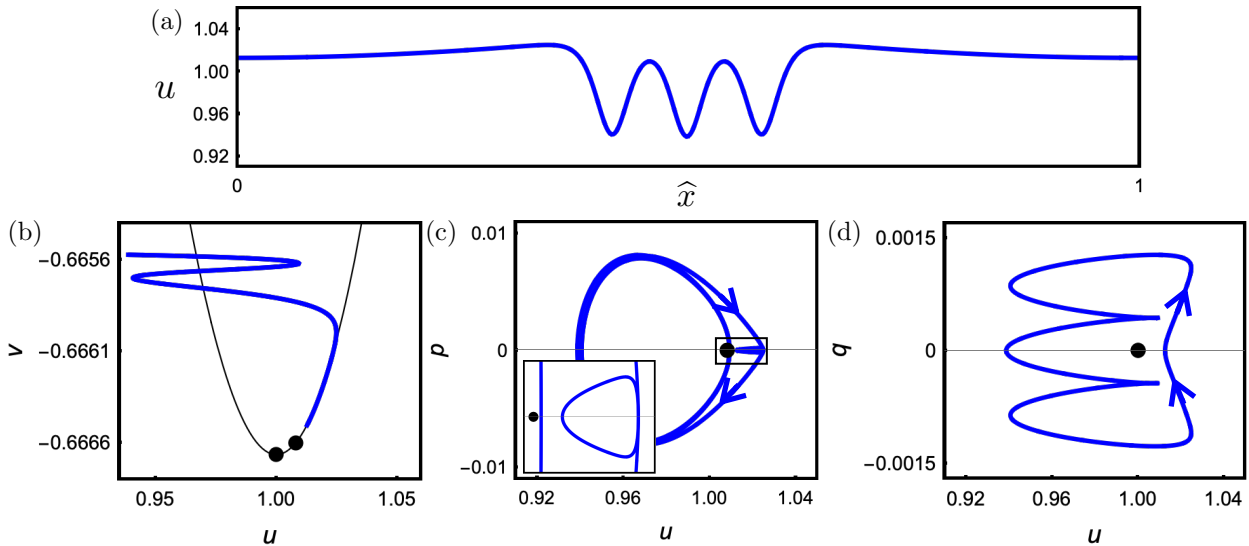

During the long portion of slow increase,  the orbit gradually moves up the right branch of the cubic nullcline $v=f(u)$ (black curve), with $u$ slowly increasing. 
See the projection onto the $(u,v)$ plane shown in Fig.~\ref{fig:kOrderOne_SAO}(b).
The orbit then transitions away and oscillates rapidly about the middle branch of the cubic nullcline. 
In the $(u,v)$ projection, the oscillations are manifested by the zigzag upward to the upper left extremum, and then back along the same path (due to the reversibility symmetry of the solution) to the right branch, completing the oscillatory segment on the central spatial interval. 
Along the remainder of the period, the solution slowly moves down along the right branch of the $v=f(u)$ curve, to its minimum value. Along this segment, $u$ slowly decreases. 
The projections onto the $(u,p)$ and $(u,q)$ planes (see panels (c) and (d), respectively) also show that the solution oscillates about the equilibrium state. 

We will show in Section~\ref{sec:slow} that system \eqref{eq:spatialODE-y} has a folded singularity --in particular a reversible folded saddle-node of type II-- which lies on the fold set $L^+$ and that portions of this  orbit (and of the other small-amplitude solutions of this type) lie along the canards of this singularity. Hence, these are spatial canard solutions. 

%---------------------------------------------------------
\subsection{Small-amplitude spatial canards with small wavenumbers}		\label{sec:numerics-2}
%---------------------------------------------------------

The spatial ODE system \eqref{eq:spatialODE-y} also possesses small-amplitude spatial canards with small ($\mathcal{O}(\delta)$)  wavenumbers. 
A representative solution is shown in Fig.~\ref{fig:smallamp-selfsimilar}. 
The solution stays close to the equilibrium state $u=a$ over almost the entire period, as may be seen from the spatial profile of $u$ shown in Fig.~\ref{fig:smallamp-selfsimilar}(a).
(Here, the value of $a$ is $\mathcal{O}(\delta^2)$ below $a=1$.)
Then, on an extremely short interval in space, it exhibits a 
small-amplitude excursion, with magnitude of at most $\mathcal O(\delta)$ (see also the inset in Fig.~\ref{fig:smallamp-selfsimilar}(a)). 

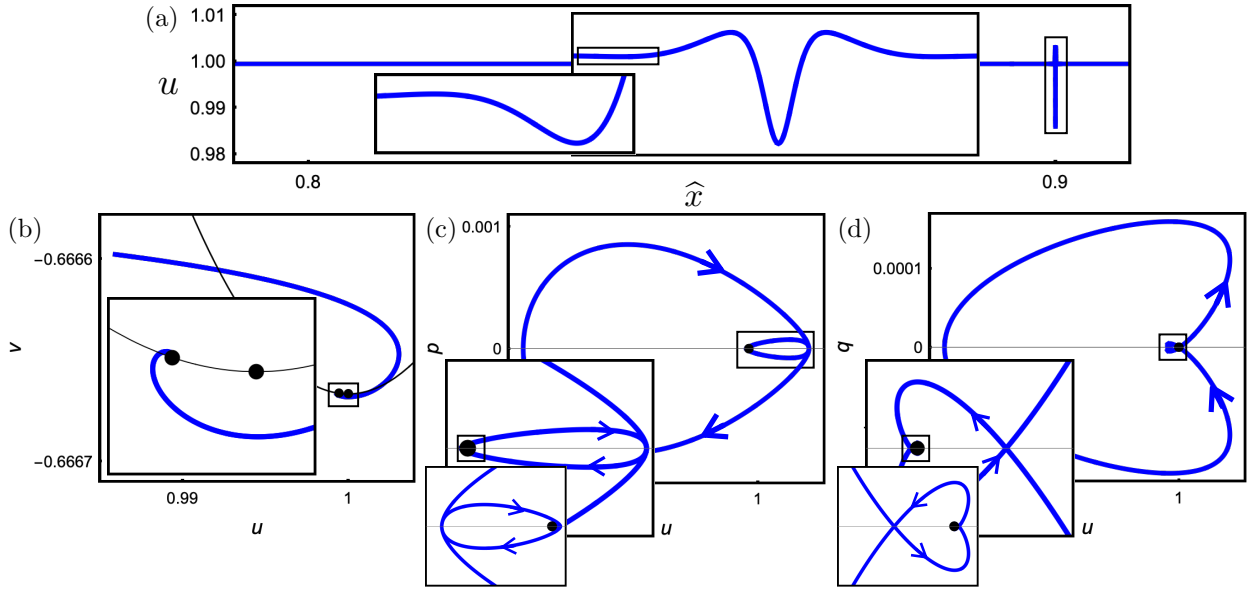
\begin{figure}[h!]  
  \centering
  \includegraphics[width=\textwidth]{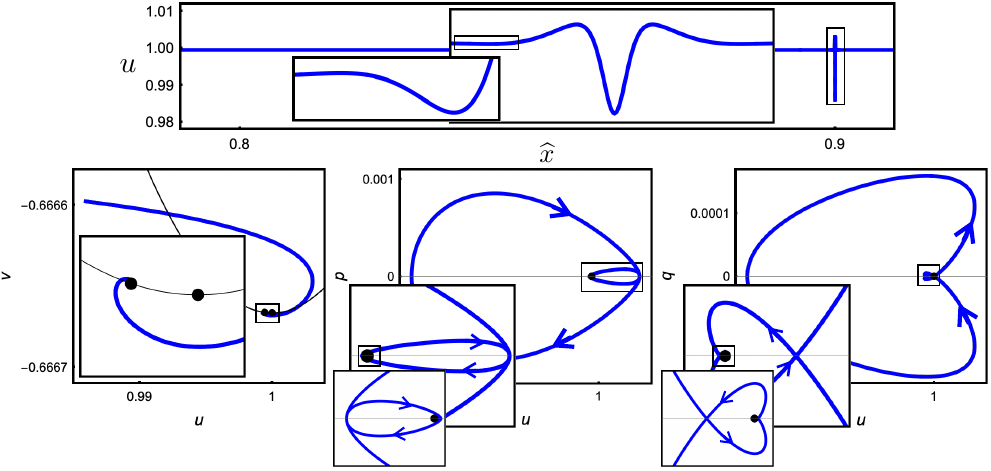}
  \put(-418,214){(a)}
  \put(-470,134){(b)}
  \put(-312,134){(c)}
  \put(-156,134){(d)}
  \caption{A small-amplitude, spatially periodic canard solution of \eqref{eq:vdp} with small wavenumber on the level set $\mathcal{G}= \eps\left( \frac{2}{3}a - \frac{1}{4} \right)$ for $a=0.999433, \eps = 0.1$, and $\delta = 0.01$.
  The period is $T = 1 \times 10^6$.
  (a) Profile of $u(\widehat x)$, where $\widehat x = \tfrac{x}{T}$ is the normalized spatial coordinate.  
  Insets: two successive zooms showing the oscillations in the short interval near $\widehat{x}=0.9$. 
  (b) Projection onto the $(u,v)$ plane. Inset: Zoom onto a small neighborhood of the equilibrium. The equilibrium is the left black dot, and the fold point corresponds to the right black dot.
  (c) Projection onto the $(u,p)$ plane. Insets: two successive zooms of the neighborhood of $E$ and three levels of the self-similarity. 
  (d) Projection onto the $(u,q)$ plane. Insets: two successive zooms of the neighborhood of $E$ and three levels of the self-similarity. (There are further, nested, successively smaller loops not shown.)
  The scale on the $u$ axis is the same in (b)-(d), and the widths of the insets in (d) are approximately $0.0015$ and $1 \times 10^{-4}$, respectively.
  (Note that, while the parameters are the same here as in Fig.~\ref{fig:introspatialperiodic}, the wavenumbers of the solutions differ by several orders of magnitude.) 
 } 
  \label{fig:smallamp-selfsimilar} 
  \end{figure}

During the long spatial interval over which $u$ is close to $a$, the solution slowly evolves in a neighborhood of the local minimum (see the projection of the solution onto the $(u,v)$ plane shown in Fig.~\ref{fig:smallamp-selfsimilar}(b)).
In particular, the solution slowly moves down  the right branch of the nullcline with $u$ decreasing, 
and then it passes underneath  the local minimum 
(right black dot in Fig.~\ref{fig:smallamp-selfsimilar}(b)), slowly spiraling in toward the equilibrium at $(u,v)=(a,f(a))$ (left black dot).
Subsequently, the solution slowly spirals away from the equilibrium, retracing its path (in the $(u,v)$ projection, due to the reversibility symmetry) back underneath the local minimum and toward the right branch of the nullcline, completing the long segment of slow evolution. 
On the short spatial interval that interrupts the long segment of slow evolution, the fast excursion corresponds to the jump away from the right branch to just beyond the middle branch, followed by the return along the same segment (in the projection), due to the reversibility symmetry.
We emphasize that these small-amplitude spatially periodic solutions are also canard solutions, because they also have long segments near the canards of a reversible folded saddle-node singularity, as we will show in Sections~\ref{sec:fast} and \ref{sec:desingFSNII}.
For completeness, we note that, in frames (c) and (d),  the fast portion of the solution on the short spatial interval corresponds to the segments located to the left of the box marking the first inset, {\it i.e.} to the left of approximately $u=0.998$. 

The dynamics of the solution in the neighborhood of the equilibrium also exhibit a degree of self-similarity. This may be seen in the projections onto the $(u,p)$ and $(u,q)$ planes; see frames (c) and (d), respectively.  
As $u$ and $v$ slowly evolve, their spatial derivatives $p$ and $q$ alternate in sign, and the solution exhibits nested spirals as it slowly approaches --and then slowly recedes from-- the equilibrium. 
Three levels of the nested, nearly self-similar spiraling are shown, and the solution exhibits additional, more-deeply nested levels (not shown). 
The nearly self-similar dynamics of this representative solution (and of other small-amplitude solutions with small wavenumbers) will be shown to follow from the self-similarity in a level set of the Hamiltonian that arises naturally in the desingularized reduced system (see Section~\ref{sec:selfsimilar}).

%---------------------------------------------------------
\subsection{Large-amplitude spatial canards with moderately small wavenumbers}		\label{sec:numerics-3}
%---------------------------------------------------------

In this section, we turn our attention to large-amplitude spatially periodic canard solutions, and we present a class of spatial canards  
that have moderately small ($\mathcal{O}(\sqrt{\delta})$) wavenumber. 
A representative spatial canard solution of this type is shown in Fig.~\ref{fig:kOrderOne_LAO}. 
Over a large portion of the spatial period, $u$ slowly decreases monotonically from $\sqrt{3}$ to just above $1$. 
Near its minimum, the $u$-component exhibits a small-amplitude spike (see the inset in Fig.~\ref{fig:kOrderOne_LAO}(a)).
Subsequently, $u$ slowly increases monotonically  near 1 back up to near $\sqrt{3}$. 
Finally, there is a fast downward jump
to $u=-\sqrt{3}$ on a narrow spatial interval, followed by a slow segment along which $u$ slowly increases to a local maximum just above $-\sqrt{3}$, and then a fast upward jump back to $\sqrt{3}$, to complete the period. 

\begin{figure}[h!]  
  \centering
  \includegraphics[width=\textwidth]{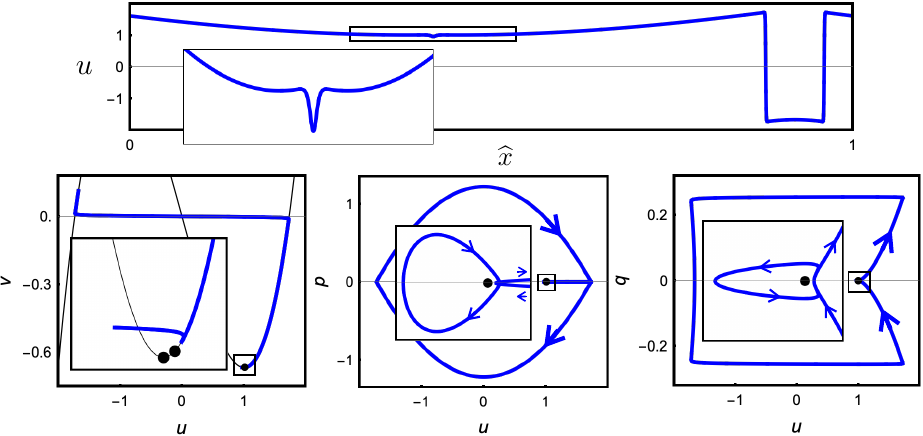}
  \put(-433,214){(a)}
  \put(-470,126){(b)}
  \put(-312,126){(c)}
  \put(-158,126){(d)}
  \caption{Large-amplitude spatially periodic canard solution of \eqref{eq:vdp} with moderately small wavenumber for $a=1.008$, $\eps=0.1$, and $\delta=0.01$. 
  The parameter values are the same as those in Fig.~\ref{fig:kOrderOne_SAO}.  a) Profile of $u(\hat{x})$, where $\hat{x} = \tfrac{x}{T}$ is a scaled version of the spatial variable chosen to normalize the period ($T \approx 2347.12$) to unit length.  (b) Projection onto the $(u,v)$ plane. The cubic nullcline $v=f(u)$ is represented by the black curve. Inset: zoom on the neighborhood of the fold point and equilibrium state (black markers).
  (c) Projection onto the $(u,p)$ plane. Inset: zoom on the neighborhood of the equilibrium.  
  (d) Projection onto the $(u,q)$ plane. 
  The rightmost portions correspond to the slow drift into and out of a neighborhood of the cusp equilibrium, which is the phenomenon that lies at the heart of the spatial canard dynamics.
  Inset: zoom on the neighborhood of the equilibrium (black marker) at $(a,0,f(a),0)$.
  This solution also lies on the level set $\mathcal{G}=\eps\left( \frac{2}{3}a - \frac{1}{4}\right)$.
  } 
  \label{fig:kOrderOne_LAO}  
\end{figure}
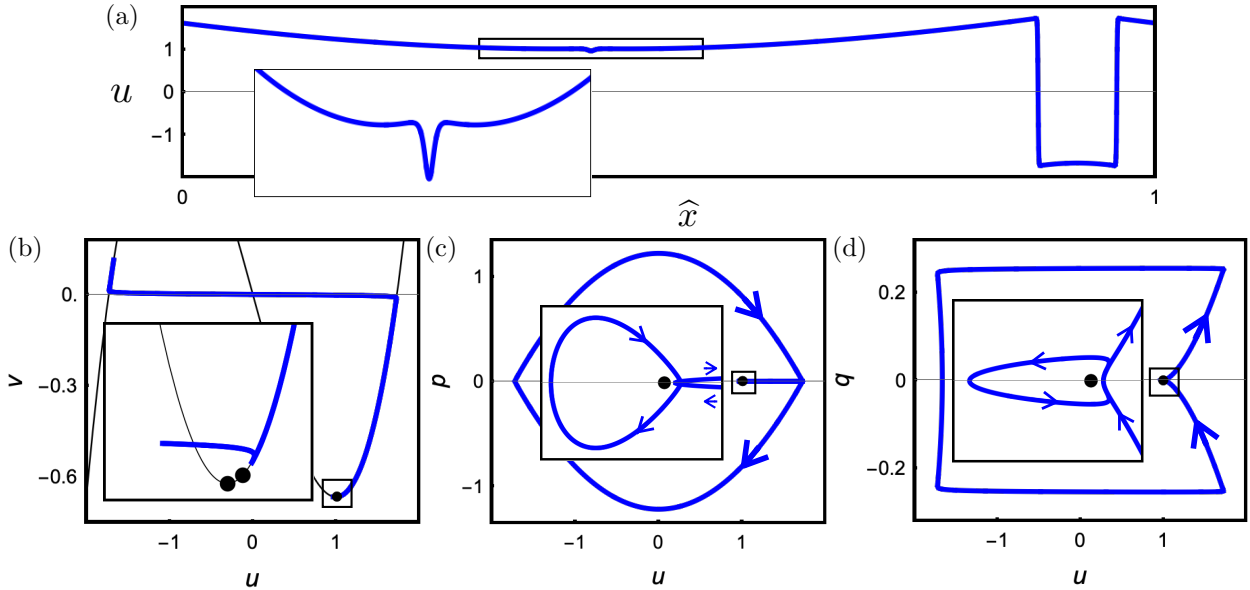

The dynamics along each of the slow and fast segments of the solution observed in the spatial profile of $u$ may be understood from the projections onto the $(u,v), (u,p)$, and $(u,q)$ planes, which are shown in Figs.~\ref{fig:kOrderOne_LAO}(b)--(d), respectively.
The long portion on which $u$ slowly decreases from $\sqrt{3}$ (near $v=0$) down toward 1 corresponds  to the slow evolution near the right branch of the cubic nullcline $v = f(u)$, see  Fig.~\ref{fig:kOrderOne_LAO}(b).
The solution comes close to the equilibrium (right black dot). 
Then, it exhibits a small, localized oscillation about the middle branch of the nullcline, before returning to the neighborhood of the right branch of the $v = f(u)$ nullcline, where it slowly moves upward until it reaches the neighborhood of the $v=0$ state. 

Subsequently, at the point where the solution reaches the $\{ v=0 \}$ state, a fast jump takes place, along which the solution transitions to the left branch of the $v=f(u)$ nullcline. 
It slowly flows up the left branch of the cubic nullcline until it reaches a maximum, turns around, and then heads back toward the $v=0$ state. 
Once it reaches the $v=0$ state, another fast transition is initiated, and the solution returns to the right branch of the cubic. 
The projections of the large-amplitude solution onto the $(u,p)$ and $(u,q)$ phase spaces are shown in Figs.~\ref{fig:kOrderOne_LAO}(c) and (d), respectively. 

These solutions are called spatial canards, 
because they have long segments near the canards of a reversible folded saddle-node singularity, as we will show in Section~\ref{subsec:singularLAO}. 
Moreover, there are also large-amplitude spatial canards with $\mathcal{O}(1)$ wavenumbers, as we will see in Section~\ref{sec:numerics-5}.
These have no loops and only relatively short segments near the true and faux canards of the RFSN-II point.

%---%-------------------------------------------------------------------------
\subsection{Large-amplitude spatial canards with small  wavenumbers}		\label{sec:numerics-4}
%-------------------------------------------------

System \eqref{eq:spatialODE-y} has a fourth main family of spatially periodic canards, which are large-amplitude, small wavenumber solutions.
In Fig.~\ref{fig:kOrderDelta_LAO}, we show a  representative solution with wavenumber $k \approx 0.000628$ ($T = 1 \times 10^6$). 
Over most of the spatial period, the solution slowly varies near $u=a$ 
(see Fig.~\ref{fig:kOrderDelta_LAO}(a)).
There is a narrow interval in which the solution has five, symmetrically-disposed, large-amplitude spikes, with peaks near $u=\sqrt{3}$ and troughs near $u=-\sqrt{3}$. 
This narrow interval is magnified in the main inset, so that the individual spikes are visible, including the  pairs of spikes at the edges and the central spike. 
Also, within the main inset, two further insets zoom onto the sharp, small spikes at the outer edges.

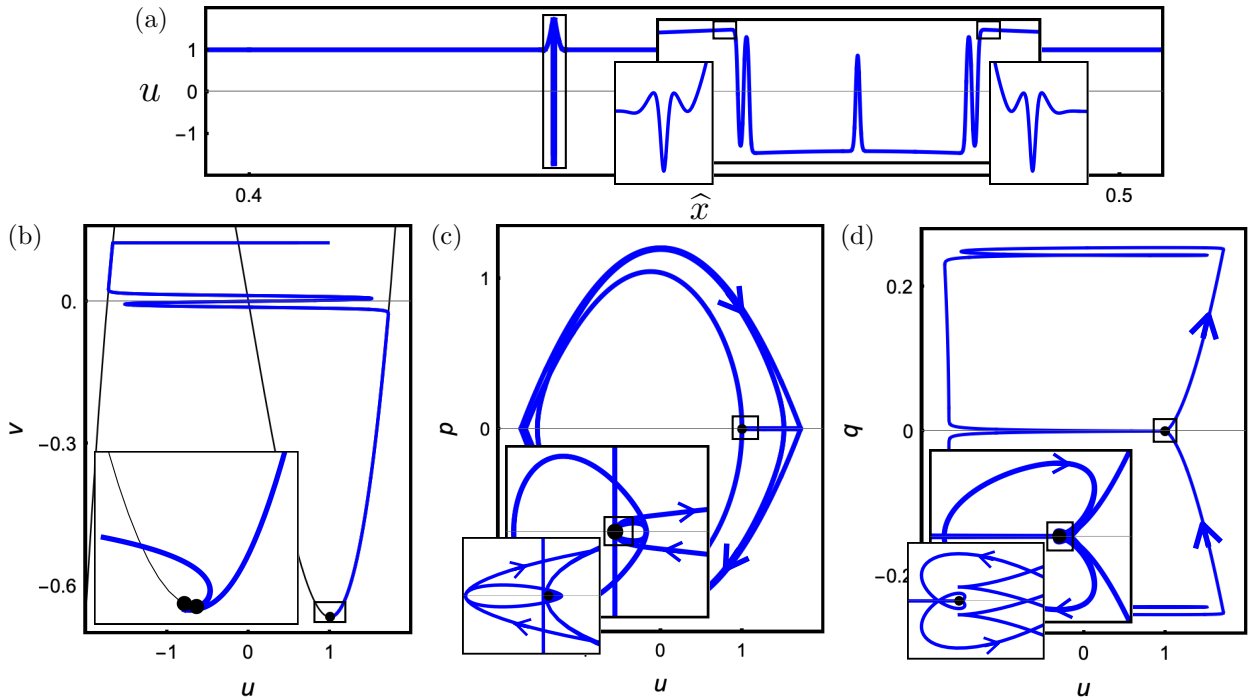
\begin{figure}[h!]  
  \centering
\includegraphics[width=\textwidth]{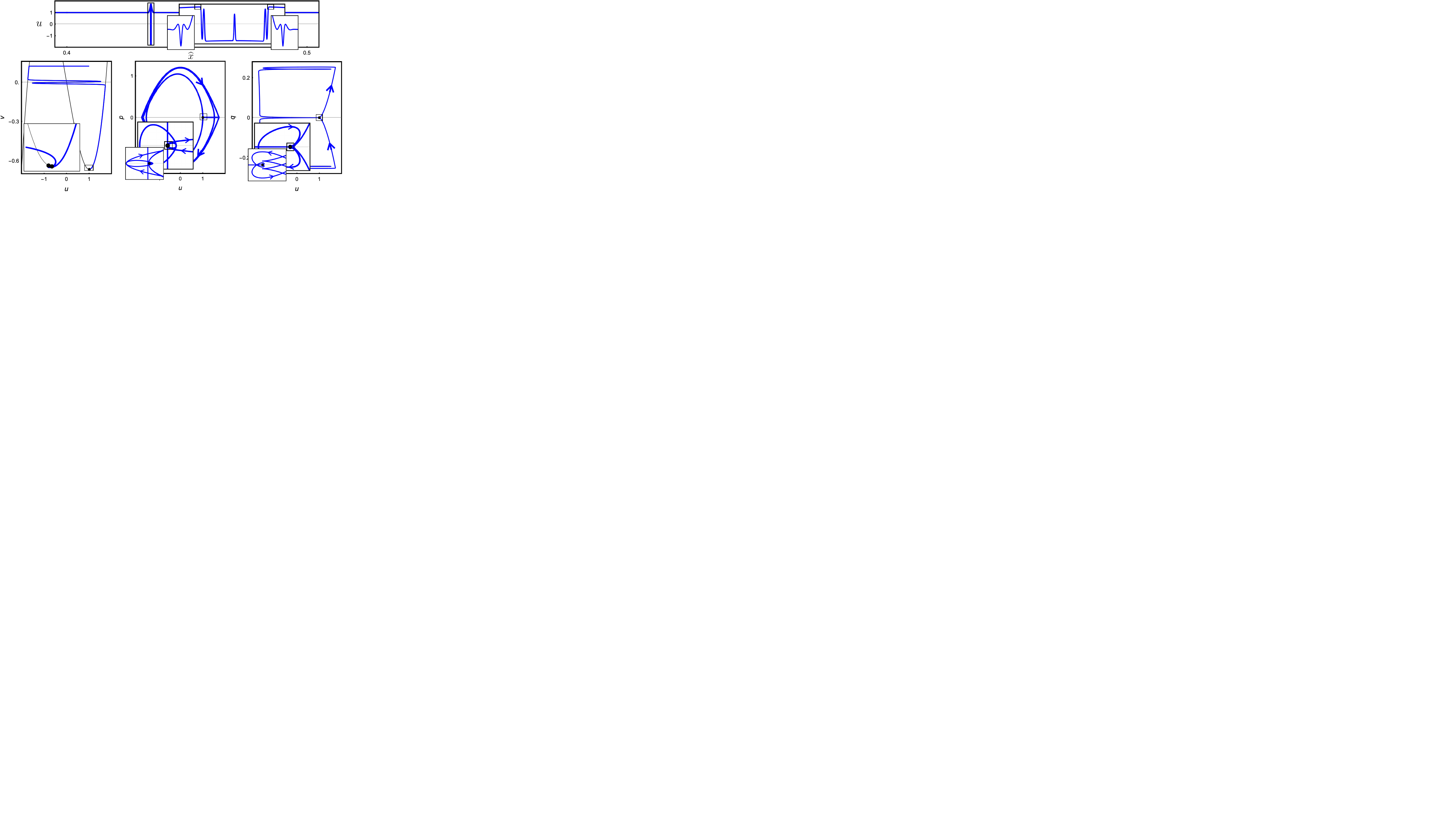}
  \put(-422,254){(a)}
  \put(-470,172){(b)}
  \put(-310,172){(c)}
  \put(-155,172){(d)}
  \caption{A large-amplitude spatially periodic canard solution of \eqref{eq:vdp} with small wavenumber for $a=0.998512$, $\eps=0.1$, and $\delta=0.01$. 
  The parameters are the same as in Fig.~\ref{fig:smallamp-selfsimilar}. 
  (a) Profile of $u(\hat{x})$, where $\hat{x} = \tfrac{x}{T}$ is the normalized spatial coordinate and $T = 1 \times 10^6$ is the period. 
  (b) Projection of the solution onto the $(u,v)$ plane. 
 Inset: zoom on the neighborhood of the equilibrium at $(a,0,f(a),0)$ (left black marker) showing the small-amplitude spikes.
  (c) Projection of the solution onto the $(u,p)$ plane, along with zooms of two successively smaller neighborhoods of the equilibrium.
  (d) Projection onto the $(u,q)$ plane. 
  The rightmost portions correspond to the `cusp'-like slow drift into and out of a neighborhood of the equilibrium.
  This is the essence of the spatial canard dynamics.
  The upper inset shows magnifications of the left pair of edge spikes.
  The two lower, successively smaller insets show a few levels of the infinite self-similarity in the neighborhood of the equilibrium.   
  The solution lies on the level set $\mathcal{G}=\eps\left( \frac{2}{3}a - \frac{1}{4}\right)$.
  } 
  \label{fig:kOrderDelta_LAO}  
\end{figure}
 
In the phase space, the solution is near the equilibrium during the large portion on which $u$ near $a$, see the left black dot in Fig.~\ref{fig:kOrderDelta_LAO}(b).  
Then, outside of this large portion of slow evolution, the solution has a narrow interval (see the main inset in Fig.~\ref{fig:kOrderDelta_LAO}(a)) in which it makes the two small and five large-scale oscillations. In particular, the small-amplitude, moderately fast oscillation about $u=1$ (secondary inset) corresponds in phase space to the small-amplitude excursion around the manifold of center points (near the local minimum of the nullcline, as shown in the inset oanel (b)).  
Then, the orbit travels fairly rapidly up the right branch of the cubic nullcline until it reaches a neighborhood of $v=0$, where the profile has the first of the five large-amplitude peaks. Subsequently, it makes a fast jump toward the left branch of the cubic, oscillating once before it reaches the left branch. 
This oscillation corresponds to the second spike at the left edge of the interval shown in the spatial profile (recall Fig.~\ref{fig:kOrderDelta_LAO}(a)). 

Further along, the orbit travels up the left branch of the cubic until it reaches some maximum, where it exhibits a fast jump corresponding to the middle large-amplitude spike of the five in the inset in Fig.~\ref{fig:kOrderDelta_LAO}(a). 
Finally, the solution travels back down near the left branch of the cubic nullcline until it reaches the neighborhood of $v=0$, where another fast jump is initiated, and it makes another large-amplitude oscillation on its way back to the right branch of the cubic. 
These final large-scale oscillations correspond  to the spikes at the right edge in Fig.~\ref{fig:kOrderDelta_LAO}(a), thereby completing the periodic solution. 

This large-amplitude, small wavenumber spatial canard exhibits dynamics that are close to being self-similar in the neighborhood of the $u=1$ state. 
This may be seen in the projections of this solution onto the $(u,p)$ and $(u,q)$ planes (see  Figs.~\ref{fig:kOrderDelta_LAO}(c) and (d), respectively). This near self-similarity will be discussed in Section~\ref{sec:selfsimilar}.

%---%-------------------------------------------------------------------------
\subsection{Numerical continuation of the families of spatial canards}		\label{sec:numerics-5}
%-------------------------------------------------

Having presented four main types of spatially periodic canard solutions, we now identify regions in the $(a,k)$ parameter space containing spatially periodic canards, as obtained using numerical continuation.  
(We refer to Appendix~\ref{app:numericalmethods} for details about the numerical continuation.) 
We began by finding a stable periodic orbit for $a=0$, continued this with respect to the period to find the orbits with minimal and maximal periods for $a=0$, and then performed two-parameter continuation on these in $a$ and the period.
The continuation results are shown in Fig.~\ref{fig:frommole2manatee}. 
For each of four values of $\delta$,  namely for $10^{-2}, 10^{-1}, 0.5,$ and $1$, spatially periodic canard solutions exist in the region between two curves.
The upper curve corresponds to minimal period solutions (with maximal $k$) and the lower curve to the maximal period solutions (with minimal $k$). 

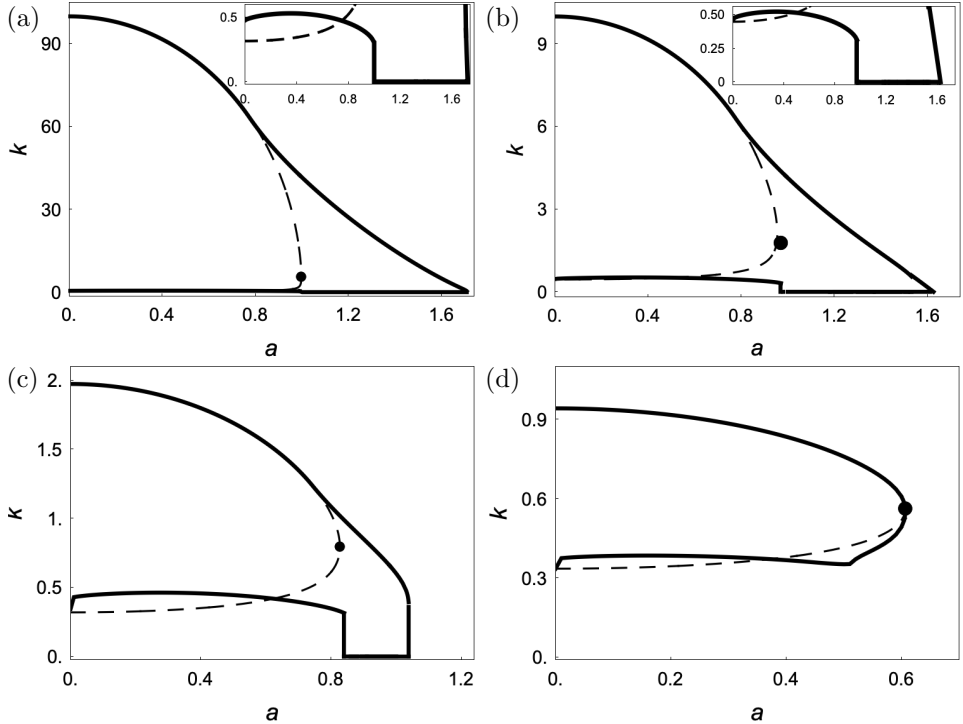
\begin{figure}[h!]    
  \centering
  \includegraphics[width=5in]{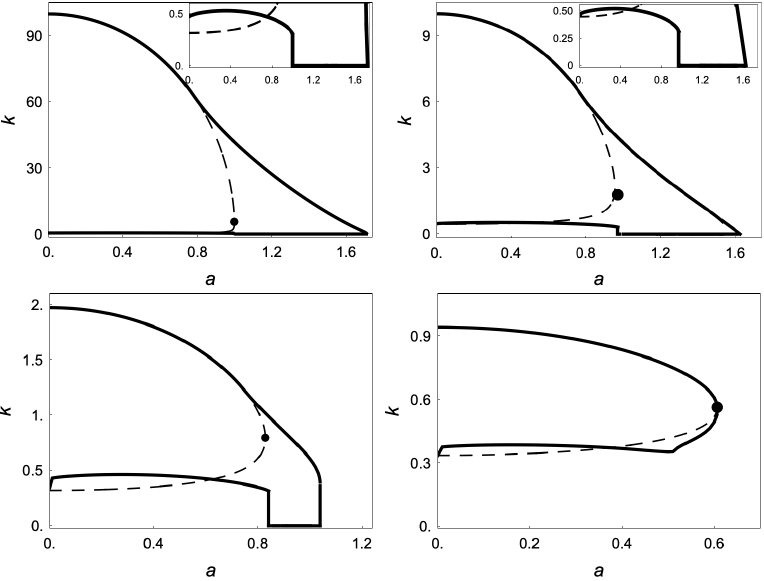}
  \put(-362,265){(a)}
  \put(-181,265){(b)}
  \put(-362,128){(c)}
  \put(-181,128){(d)}
\caption{Space of spatially periodic patterns of \eqref{eq:spatialODE-y} in the $(a,k)$ plane for $\eps = 0.1$ and (a) $\delta = 0.01$, (b) $\delta = 0.1$, (c) $\delta = 0.5$, and (d) $\delta = 1$. Spatially periodic solutions exist inside the region enclosed by the upper and lower solid curves. 
The upper solid curve denotes spatially periodic orbits with minimal spatial periods (maximal $k$), whilst the lower branch denotes those with maximal spatial periods (minimal $k$). Hence, for a fixed value of $a$, there are spatially periodic orbits for all $k$ between the minimal and maximal values. The sharp drop in $k$ occurs at $a=a_T$, and homoclinics exist for $a>a_T$ for all $\delta < \tfrac{8}{25\sqrt{\eps}}$ (see Proposition~\ref{prop:Turingcanards}). The dashed curves denote the loci of points at which ${\rm det}DF = 0$, recall \eqref{eq:Turingcond}, marking 
the linear stability boundary of the equilibrium. These curves have two $k$-intercepts. The upper intercept occurs at $k \sim \frac{1}{\delta}-\frac{1}{2}\eps\delta$ for $\delta \ll 1$ and the lower intercept is at $k \sim \sqrt{\eps} + \mathcal{O}(\eps^{3/2} \delta^2)$. Also, the Turing point $(a_T,k_T)=\left(\sqrt{1-2\delta \sqrt{\eps}},\left(\frac{\eps}{\delta^2}\right)^{1/4}\right)$ (solid dot) is located at the right extremum of the dashed curves where also $\frac{\partial}{\partial k^2} {\rm det} DF =0$, recall \eqref{eq:kTaT}.
}
\label{fig:frommole2manatee} 
\end{figure}

The solutions with small $k$ have long segments along which they are near the equilibrium,  making successively smaller nested loops, which are close to being self-similar.
The solutions with $\mathcal{O}(1)$ or higher $k$ have shorter or no segments near the equilibrium.
This trend is illustrated for values of $a$ on both sides of $a=1$ by the representative canards shown in Figs.~\ref{fig:kOrderOne_SAO}--\ref{fig:kOrderDelta_LAO}.
Consider first $a=0.999433$, so that the equilibrium lies to the left of the local minimum.
The small-amplitude canard with small $k$ shown in Fig.~\ref{fig:smallamp-selfsimilar} has a finite sequence of successively smaller, nearly self-similar loops, and hence it has a long segment near the equilibrium (and we note that similar canards exist also for $a>1$, {\it  e.g.} along the rightmost blue stalk in Fig.~\ref{fig:bifndetailed}(c)). 
By contrast, small-amplitude canards with $\mathcal{O}(1)$ $k$
(one of which is shown for example below in Section~\ref{subsec:Turingspatialcanards}) have few or no loops, and hence have only short or no segments near the equilibrium.
At this same value of $a$, the large-amplitude canards exhibit the same trend with respect to $k$, as may be seen for example by comparing the canard with small $k$ shown in Fig.~\ref{fig:kOrderDelta_LAO} with a canard with $\mathcal{O}(1)$  $k$.
Similar trends are seen for $a=1.008$, as well as for values of $a$ further from $a=1$ on both sides.

The lower curves in Figs.~\ref{fig:frommole2manatee} (a), (b), and (c) exhibit several sharp drops in $k$ near the critical Turing value $a_T = \sqrt{1-2\delta \sqrt{\eps}}$. 
These sharp drops occur because of the birth of infinite-period (i.e., zero wavenumber),  homoclinic solutions. 
The existence of these homoclinics follows from part (iii) of  Proposition~\ref{prop:Turingcanards}.
Finally, as illustrated in Fig.~\ref{fig:frommole2manatee} (d), the rightmost edge of the region of spatial periodics converges to the critical Turing bifurcation as $\delta \to \tfrac{8^-}{25\sqrt{\eps}}$, which is in accord with part (i) of Proposition~\ref{prop:Turingcanards}, since spatially periodic solutions only exist for $a<a_T$ for $\delta$ on the other side.
(For reference, we also show the linear stability boundary, $\delta^2 k^4 + (a^2-1) k^2+\eps = 0, $
of the equilibrium state, which is obtained from $\det DF = 0$ \eqref{eq:Turingcond}, as the black dashed curve.)

\begin{figure}[h!]    
  \centering
  \includegraphics[width=5in]{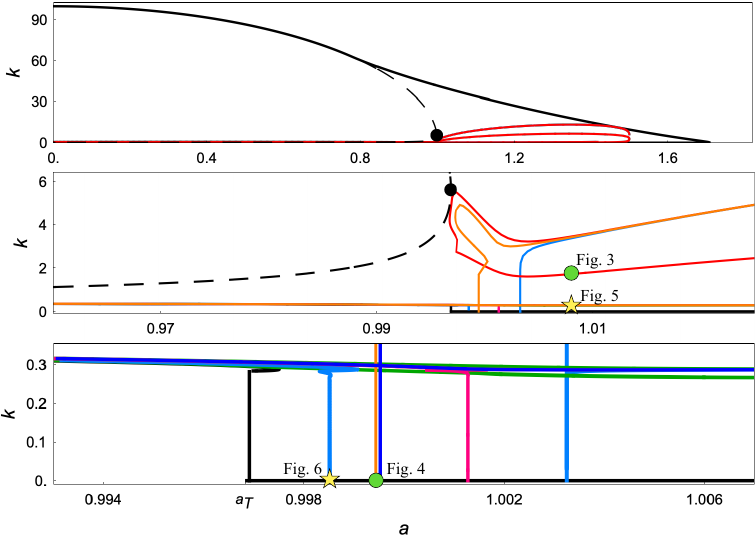}
  \put(-362,248){(a)}
  \put(-362,166){(b)}
  \put(-362,84){(c)}
\caption{Partial bifurcation diagram in the $(a,k)$ plane together with curves of constant $\mathcal{G}$ for (a) $k = \mathcal{O}(\delta^{-1})$, (b) $k = \mathcal{O}(1)$, and (c) $k = \mathcal{O}(\delta^{1/2})$ for $\eps=0.1$ and $\delta=0.01$.
The green circle markers correspond to the small-amplitude solutions shown in Figs.~\ref{fig:kOrderOne_SAO} and \ref{fig:smallamp-selfsimilar}, and the yellow star markers correspond to the large-amplitude solutions shown in Figs.~\ref{fig:kOrderOne_LAO} and \ref{fig:kOrderDelta_LAO}.
As in Fig.~\ref{fig:frommole2manatee}, 
the dashed black curves correspond to the linear stability boundary of the equilibrium, the Turing bifurcation occurs at $(a_T,k_T)=(0.996832...,5.6327...)$ here, and 
the solid black curves enclose the region where spatially periodic solutions exist. 
With the exception of the black branches, all branches lie on the level set $\mathcal{G}=\eps\left( \frac{2}{3}a - \frac{1}{4} \right)$.
The two red branches emerge from the Turing bifurcation (see panel (b)). 
Along these branches, the amplitude of the oscillations about the homogeneous state grows as $a$ increases beyond $a_T$, though not at the same rate along the two branches. 
Also, several of the $\mathcal{G}=\eps\left( \frac{2}{3}a - \frac{1}{4} \right)$ branches terminate at homoclinics (see panel (c)). 
} 
\label{fig:bifndetailed} 
\end{figure}

A more detailed view of the bifurcation structure of the spatially periodic solutions of \eqref{eq:spatialODE-y} is  presented in Fig.~\ref{fig:bifndetailed}, with frames (a)--(c) showing the wavenumbers $k$ of sizes $k \gg 1$, $k=\mathcal{O}(1)$,
and $k \ll 1$, respectively. 
First, the same diagram from Fig.~\ref{fig:frommole2manatee}(a) is shown in Fig.~\ref{fig:bifndetailed}(a), now also with  curves of constant $\mathcal{G}$ (red curves), fixed at the special value $\mathcal{G} = \eps \left( \tfrac{2}{3}a-\tfrac{1}{4} \right)$. These particular spatially periodic solutions only exist at small and $\mathcal{O}(1)$ wavenumbers. At the $\mathcal{O}(1)$ scale of wavenumbers (Fig.~\ref{fig:bifndetailed}(b)), the fine structure of the spatially periodic solutions becomes more apparent. Each of the different colored curves (red, orange and blue) correspond to different branches of spatially periodic solutions on the level set $\mathcal{G} = \eps \left( \tfrac{2}{3}a-\tfrac{1}{4} \right)$. The red branch emerges from the Turing bifurcation (black marker), continues into large $a$, decreases to small $k$, and remains at small wavenumbers (horizontal segments).
The detailed structure of the nearly horizontal segments for small wavenumbers is shown in Fig.~\ref{fig:bifndetailed}(c). It can be seen that there are even more branches of spatially periodic solutions with $\mathcal{G} = \eps \left( \tfrac{2}{3}a-\tfrac{1}{4} \right)$. Some of these branches exhibit sharp drops in $k$ and continue in toward the homoclinic solution along $k=0$ for $a>a_T$. 
Each of these branches in Fig.~\ref{fig:bifndetailed}(c) contains spatially periodic canard solutions.

Overall, the numerical continuation results show that system \eqref{eq:vdp} has many different types of spatially periodic canard solutions in the $(a,k,d)$ space for a wide range of values of $\eps>0$. 
We have found these types of spatially periodic canards for other values of $\eps$, including $\eps=0.8$.
We add that, in Section~\ref{sec:isolas}, 
the solutions and bifurcations along these branches will be studied in detail by using the results obtained in the analysis of the fast system, the desingularized reduced vector field, and the RFSN-II singularity (Sections~\ref{sec:fast}--\ref{sec:desingFSNII}).

To conclude this section, we note that the regions shown in Fig. \ref{fig:frommole2manatee} correspond to what are called `existence balloons' in \cite{DRS12}, {\it i.e.}, regions in (parameter,wave number)-space in which spatially periodic patterns exist. 
Naturally, the eventual goal would be to determine the Busse balloon, the sub-balloon of stable spatially periodic patterns \cite{B1978,D2019}. 
The Busse balloon plays a central role in pattern forming systems in fluid mechanics, reaction-diffusion systems, ecological models, and many other scientific and engineering problems \cite{AACS2024,B1978,BC1979,BW1971,CH1993,CN1984,D2019,GZK2018,MDK2001,S2003,VSC2023} and is also expected to play a similar role in understanding the impact of the canard-type solutions of this work on the dynamics exhibited by PDE (\ref{eq:vdp}).

%-----------------------------------------------------------
\section{Fast System (Layer Problem)}
\label{sec:fast}
%-----------------------------------------------------------

In this section, we study the fast system (layer problem) of the spatial ODE, which is obtained by taking the limit $\delta \to 0$ in \eqref{eq:spatialODE-y}: 
\begin{equation}
\label{eq:layer}
\begin{split}
  u_y &= p \\
  p_y &= f(u) - v \\
  v_y &= 0 \\
  q_y &= 0.
\end{split}
\end{equation}
Hence, $v$ and $q$ are constants, and the fast system is independent of $a$.
Moreover, for each fixed $v$, the fast system is Hamiltonian, 
\begin{equation}
\label{eq:H-fast}
  u_y = \frac{\partial H_{\rm fast}}{\partial p}, \quad 
  p_y = -\frac{\partial H_{\rm fast}}{\partial u},
  \quad {\rm with} \quad 
  H_{\rm fast}(u,p;v)= \frac{1}{2}p^2 - \tilde{f}(u) + uv.
\end{equation}
The functional form of $H_{\rm fast}$ is equivalent to that obtained by reducing $\mathcal{G}$ to the fast system and scaling by a factor of $\eps$. We recall from \eqref{eq:G} that $\tilde{f}= \frac{1}{12}u^4 - \frac{1}{2} u^2$.

%---------------------------------------------------------------------
\subsection{The critical manifold of the fast system}
\label{subsec:fastphaseplanes}
%---------------------------------------------------------------------

In the four-dimensional $(u,p,v,q)$ space, the fast system \eqref{eq:layer} has a two-dimensional critical manifold, 
\begin{equation}
\label{eq:S}
        S := \{ (u,p,v,q)
        : u\in \mathbb{R},\,\, p=0,\,\, v=f(u), \,\, q \in \mathbb{R} \},
\end{equation}
which corresponds to the union of all of the fixed points of \eqref{eq:layer}.
Moreover, given that the cubic function $f(u)=\frac{1}{3}u^3 - u$ has three segments separated by the local extrema at $(u,v)=\left(\pm 1, \mp \frac{2}{3}\right)$, the critical manifold $S$ has three branches:
\begin{equation}
    \label{eq:saddle+elliptic}
\begin{split}
     S_s^-
        &= \{ (u,p,v,q) \in S \, :  u=u_- < -1, \,\,  p = 0, \,\, v=f(u_-),\,\,  q \in \mathbb{R} \}, \\
     S_c 
         &= \{ (u,p,v,q) \in S :  u=u_0, \,\, -1 < u_0 < 1, \,\, p=0, \,\,  v=f(u_0), \,\,  q \in \mathbb{R} \}, \\
     S_s^+ 
         &= \{ (u,p,v,q) \in S : u=u_+>1, \,\, p=0, \,\, v=f(u_+), \,\,  q \in \mathbb{R} \}.
\end{split}
\end{equation}
$S_s^\pm$ are hyperbolic saddle invariant manifolds, because the fixed points $(u_\pm,0)$ are saddle fixed points of the $(u,p)$ subsystem for each $v>-\frac{2}{3}$ and each $v < \frac{2}{3}$, respectively. Similarly, $S_c$ is a non-hyperbolic, elliptic invariant manifold, because the fixed point $(u_0,0)$ of the $(u,p)$ subsystem is a nonlinear center for each $-\frac{2}{3} < v < \frac{2}{3}$.
The three invariant manifolds  are illustrated in Fig.~\ref{fig:S}.

\begin{figure}[h!]  
  \centering
  \includegraphics[width=5in]{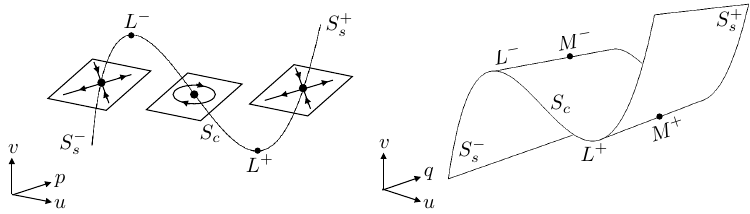}
  \put(-360,90){(a)}
  \put(-180,90){(b)}
  \caption{The critical manifold $S$ given by \eqref{eq:S} projected (a) in the $(u,p,v)$ space and (b) in the $(u,q,v)$ space. It consists of the two hyperbolic components $S_s^\pm$ and the elliptic component $S_c$ (recall \eqref{eq:saddle+elliptic}), which are separated by the fold curves $L^\pm$, see \eqref{eq:foldcurves}.
  For each fixed $a$, the 2D critical manifold consists of a 1-parameter family of 1D slow manifolds.
  }
  \label{fig:S}
\end{figure}

The fold curves are
\begin{equation}
\label{eq:foldcurves}
L^{\pm}=
\left\{
(u,p,v,q) | u = \pm 1, p=0,
v = \mp \frac{2}{3}, q \in \mathbb{R} \right\}.
\end{equation}
They form the boundaries of the invariant manifolds, where $S_s^{\pm}$ lose normal hyperbolicity.
Together with the three invariant manifolds, they provide the following natural decomposition of the critical manifold, 
\begin{equation}
\label{eq:S-decomp}
    S  = S_s^- \cup L^- \cup S_c \cup L^+ \cup S_s^+.
\end{equation}

%----------------------------------------------------------------------
\subsection{Homoclinics of the fast system}
\label{subsec:fasthomoclinics}
%----------------------------------------------------------------------
We prove the existence of homoclinic orbits in the fast system, which will be useful in the deconstruction of spatially periodic canard solutions.

\begin{proposition}  \label{prop:fasthomoclinics}
For each fixed $v \in (-\tfrac{2}{3},0)$, there exists a homoclinic solution of \eqref{eq:H-fast} corresponding to the intersection $W^s(u_+) \cap W^u(u_+)$, where $u_+ \in S_s^+$ is the root of $v = f(u)$ for which $u_+ >1$. 
Similarly, for each fixed $v \in (0,\tfrac{2}{3})$, there exists a homoclinic solution of \eqref{eq:H-fast} corresponding to the intersection $W^s(u_-) \cap W^u(u_-)$, where $u_- \in S_s^-$ is the root of $v = f(u)$ for which $u_- <-1$. 
\end{proposition}

\begin{proof}
We examine the fast system Hamiltonian along the cross-section $\{ p =0 \}$:
\[ H_{\rm fast}(u,0;\,v) = u v - \left( \tfrac{1}{12}u^4-\tfrac{1}{2}u^2 \right). \]
For a fixed $v$ in the interval $(0,\tfrac{2}{3})$, $H_{\rm fast}(u,0;\,v)$ is a quartic polynomial (Fig.~\ref{fig:fasthomoclinics}) in $u$ with local maxima at $u_{\pm}$ and a local minimum at $u_0$.
Recall that $u_{\pm}$ are saddle equilibria of \eqref{eq:layer} and $u_0$ is a center. 

\begin{figure}[h!]
  \centering
  \includegraphics[width=5in]{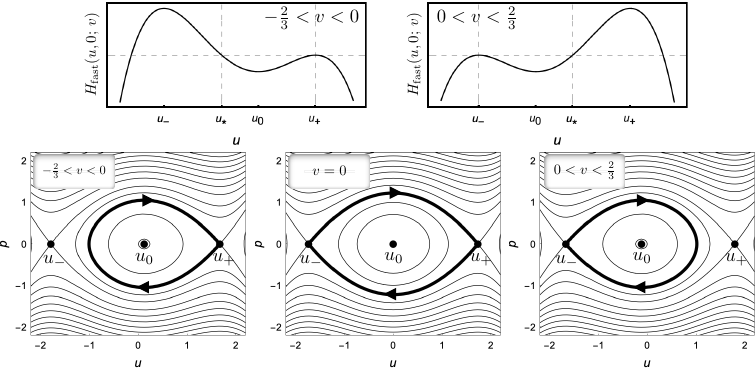}
  \caption{Top row: fast subsystem Hamiltonian in the cross-section $\{ p =0 \}$. Bottom row: phase planes of \eqref{eq:layer} for fixed $v$. The homoclinics (for $|v| \in (0,\tfrac{2}{3})$) and heteroclinic (for $v=0$) are highlighted as the thick black curves.}
  \label{fig:fasthomoclinics}
\end{figure}

The Hamiltonian $H_{\rm fast}(u,0;\,v)$ is monotonically increasing for all $u<u_-$, decreasing for $u_-<u<u_0$, increasing for $u_0<u<u_+$, and decreasing for $u>u_+$.
Moreover,
\[ H_{\rm fast}(u_+,0;\,v) > H_{\rm fast}(u_-,0;\,v) > H_{\rm fast}(u_0,0;\,v), \]
so that, by the Intermediate Value Theorem, there exists a $u_* \in (u_0,u_+)$ such that $H_{\rm fast}(u_-,0;\,v) = H_{\rm fast}(u_*,0;\,v)$.
Since the only equilibrium between $u_-$ and $u_*$ is the nonlinear center $u_0$, the orbits with $u \in [u_-,u_*]$ are closed. 
Hence, there exists a closed orbit of \eqref{eq:layer} that connects the fixed point $(u_-,0)$ to a point $(u_*,0)$ and then back to itself. 

The proof for $v \in (-\tfrac{2}{3},0)$ is similar.
\end{proof}

%----------------------------------------------------------------------
\subsection{Heteroclinics in the intersections \texorpdfstring{$W^u(S_s^+) \cap W^s(S_s^-)$ and $W^u(S_s^-) \cap W^s(S_s^+)$}{Lg}}
\label{subsec:fastheteroclinics}
%----------------------------------------------------------------------

In this section, we establish

\begin{proposition} \label{prop:jumpcondition}
In the $(u,p,v,q)$ space, $W^u(S_s^-) \cap W^s(S_s^+)$ and $W^u(S_s^+) \cap W^s(S_s^-)$. Moreover, these intersections lie in the plane $\{ v=0 \}$, and they are transverse.
\end{proposition}

The stable and unstable manifolds are shown in Fig.~\ref{fig:fastphasespace},
along with the heteroclinic orbits that lie in their transverse intersections.

\begin{figure}[!h]
  \centering
  \includegraphics[width=5in]{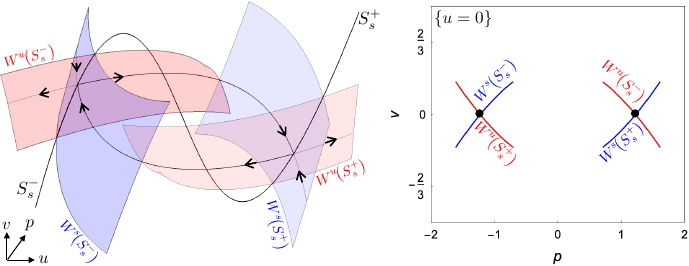}
  \put(-360,128){(a)}
  \put(-160,128){(b)}
  \caption{Dynamics of the fast system \eqref{eq:H-fast}. (a) The invariant stable manifolds $W^s(S_s^{\pm})$ (blue) and unstable manifolds $W^u(S_s^{\pm})$ (red) of the saddle manifolds $S_s^\pm$ in \eqref{eq:H-fast}, along with the heteroclinic orbits (thin black curves with arrows) that lie in their transverse intersections. The stable and unstable manifolds have been computed up to the cross-section $\left\{ u=0 \right\}$. The heteroclinics exist exactly in the $\{ v=0 \}$ hyperplane. (b) Intersections of the stable and unstable manifolds of the saddles in the cross-section $\left\{ u=0 \right\}$. The solid black dots represent the points at which the heteroclinic orbits intersect the cross-section. }
  \label{fig:fastphasespace}
\end{figure}

\begin{proof}
For $v=0$,
the saddles $(-\sqrt{3},0)$
and $(\sqrt{3},0)$
lie on the level set
$H_{\rm fast} (\pm\sqrt{3},0;0)=\frac{3}{4}$.
They are connected by heteroclinic orbits
in the $(u,p)$ plane,
as may be verified directly
from the functional form
of the Hamiltonian.

Next, for values of $v$ near --but not equal to-- zero,
the saddles $(u^{\pm},0)$ lie on different level sets of the Hamiltonian.
In particular, for each fixed $v>0$ and close to zero,
$H_{\rm fast} < \frac{3}{4}$ at the saddle $(u^-,0)$ on $S_s^-$,
and $H_{\rm fast} > \frac{3}{4}$ at the saddle $(u^+,0)$ on $S_s^+$.
Similarly, for each fixed $v\lesssim  0$ ,
$H_{\rm fast} > \frac{3}{4}$ at the saddle $(u^-,0)$ on $S_s^-$,
and $H_{\rm fast} < \frac{3}{4}$ at the saddle $(u^+,0)$ on $S_s^+$.
Therefore,
$W^u(S_s^-)$ and $W^s(S_s^+)$
intersect transversely
in the plane $\{ v=0 \}$,
as do the invariant manifolds
$W^u(S_s^+)$ and $W^s(S_s^-)$.
\end{proof} 

%----------------------------------------------------------------------
\subsection{The cusp point at \texorpdfstring{$(u,p)=(1,0)$ for $v=-\frac{2}{3}$}{Lg}}
\label{subsec:fascusp}
%----------------------------------------------------------------------

For $v=-\frac{2}{3}$,
the fast system \eqref{eq:H-fast} has a saddle-node point at $(u,p)=(1,0)$. The Jacobian at that point has a double zero eigenvalue, with
Jordan block given by 
\[
\left[  
\begin{array}{cc}  
0 & 1 \\
0 & 0
\end{array}
\right].
\]
In the phase plane, it is a cusp point with a single-branched stable manifold
and a single-branched unstable manifold. In parametrized form, the functions whose graphs give these manifolds have the form 
\[ p = \mp C (u-1)^{3/2}, \] 
for some $C>0$, respectively, and they lie on the level set $H_{\rm fast}(u,p)=- \frac{1}{4}$.

This cusp point is the point we will desingularize.  The stable manifold will correspond to one fixed point on the equator of the blown-up hemisphere, and the unstable manifold to another fixed point. The dynamics of the orbits near the cusp point will then be determined by studying the dynamics over the hemisphere.

%-----------------------------------------------------------------------------
\section{Slow System on \texorpdfstring{$S$}{Lg}, with the Folded and Ordinary Singularities} 
\label{sec:slow}
%-----------------------------------------------------------------------------------------------------

In this section, we study the slow system induced by the spatial ODE system on the two-dimensional critical manifold $S = \{ u \in \mathbb{R},\,\, p=0,\,\, v=f(u),\,\, q \in \mathbb{R} \}$, recall \eqref{eq:S}. We derive the locations and stability types of the folded and ordinary singularities. 

Recall that the slow manifold $S$ consists of the fixed points of the fast system in the limit $\delta \to 0$ ({\it i.e.,} of \eqref{eq:spatialODE-y} with $\delta=0$).
One obtains the slow dynamics on $S$ in the slow variable by differentiating the constraints $p=0$ and $v=f(u)$ that define $S$:
\begin{equation}
        \label{eq:initialreduced}
        \left[
        \begin{array}{cc}
        0 & 1 \\
        f'(u) & 0
        \end{array}
        \right]
        \left[
        \begin{array}{c}
        u_x \\ p_x
        \end{array}
        \right]
        =
        \left[
        \begin{array}{c}
    0  \\ v_x
        \end{array}
        \right].
\end{equation}
The adjoint operator is
\begin{equation}
{\rm adj}
\left[
\begin{array}{cc}
        0 & 1 \\ f'(u) & 0
\end{array}
\right]
=
\left[
\begin{array}{cc}
        0 & -1 \\ -f'(u) & 0
\end{array}
\right].
\end{equation}
Hence, left multiplying \eqref{eq:initialreduced}
by the adjoint and recalling the third and fourth components of
\eqref{eq:spatialODE-x}, we find
\begin{equation}
        \label{eq:reduced}
        \begin{split}
        -f'(u)
        \left[
        \begin{array}{c}
        u_x \\  p_x
        \end{array}
        \right]
        &= \left[
        \begin{array}{c}
        -q \\ 0
        \end{array}
        \right] \\
        \left[
        \begin{array}{c}
        v_x \\  q_x
        \end{array}
        \right]
        &= \left[
        \begin{array}{c}
        q \\ \eps(u-a)
        \end{array}
        \right].
        \end{split}
\end{equation}

Next, we desingularize the reduced system on $S$
by rescaling time with a factor of $f'(u)$.
Let $dx = f'(u)\,dx_d$.
In this new, dynamic time variable,
the reduced system \eqref{eq:reduced} is equivalent to
\begin{equation*}
        \left[
        \begin{array}{c}
                u_{x_d} \\  p_{x_d} \\v_{x_d} \\ q_{x_d}
        \end{array}
        \right]
        = \left[
        \begin{array}{c}
                q \\ 0 \\ f'(u) q \\ \eps f'(u) (u-a)
        \end{array}
        \right].
\end{equation*}
Finally, we observe that, on $S$, the constant solution for $p$ is $p=0$.
Also, the equation for $v$ decouples (and may be solved by quadrature).
Hence, we arrive at the desingularized reduced system
\begin{equation}
\label{eq:desing-reduced}
\begin{split}
        u_{x_d} &= q \\ 
        q_{x_d} &= \eps f'(u) (u-a).
\end{split}
\end{equation}

This is the main system studied in this section; it defines the folded and ordinary singularities. 
We note that on $S_s^\pm$ the direction of the flow of \eqref{eq:desing-reduced} is the same as that of \eqref{eq:reduced}, whereas on $S_c$ the flow direction is reversed since $f'(u)<0$ on $S_c$.
Also, \eqref{eq:desing-reduced} is Hamiltonian with
\begin{equation}
        \label{eq:hatH}
        H_{\rm desing}(u,q;a)= -\frac   {1}{2}q^2 + \eps \left( \frac{1}{4}u^4
           - \frac{1}{3}a u^3 - \frac{1}{2}u^2 + au \right),
\end{equation}
where $H_{\rm desing}$ is induced by \eqref{eq:G}, i.e., $H_{\rm desing} = \left.\mathcal G\right|_{S}$, and $u_{x_d}=-\tfrac{\partial H_{\rm desing}}{\partial q}$ and $q_{x_d}=\tfrac{\partial H_{\rm desing}}{\partial u}$.

The ordinary and folded singularities of the desingularized reduced system \eqref{eq:desing-reduced} are found as follows.
Off the fold set, both components of the vector field vanish at $(u,q)=(a,0)$. Hence, \eqref{eq:desing-reduced} has an ordinary singularity (equilibrium) at
\begin{equation*}
        E = \{ u=a, q=0 \}.
\end{equation*}
Then, on the fold set where $f'(u)=0$, there are two  folded singularities 
\begin{equation}
        \label{eq:Mpm}
        M^{\pm} = \{ u = \pm 1, q=0 \},
\end{equation}
where the vector field in \eqref{eq:desing-reduced} vanishes.

We classify the linear stability types of the ordinary and folded singularities using the eigenvalues of the Jacobian of \eqref{eq:desing-reduced},
\begin{equation*}
        \left[
        \begin{array}{cc}
          0 & 1 \\
        \eps f''(u) (u-a) + \eps f'(u) & 0
        \end{array}
        \right]
\end{equation*}
%At $E$, the Jacobian has eigenvalues $\pm \sqrt{\eps(a^2-1)}$. Hence, 
%
%\begin{itemize} 
%\item $E$ is a center for $| a | < 1$
%\item $E$ is a SN bifurcation point at $|a|=1$,
%and 
%\item $E$ is a saddle for $ | a | > 1$,
%\end{itemize}
%
%At $M^+$, 
%the eigenvalues of the Jacobian are $\pm\sqrt{2\eps (1 - a)}$.
%Hence,
%\begin{itemize}
%  \item $M^{+}$ is a folded saddle (FS) for $ a < 1 $
%  \item $M^+$ is a reversible folded saddle-node point of type II
%          (FSN-II) for $a = 1$ 
%  \item $M^{+}$ is a folded center (FC) for $a > 1$.
%\end{itemize}
%At $M^-$, the eigenvalues of the Jacobian are $\pm \sqrt{2\eps (1 + a)}$.
%Hence,
%\begin{itemize}
%  \item $M^{-}$ is an FC for $ a < -1 $
%  \item $M^-$ is a reversible folded saddle-node point of type II
%          (FSN-II) for $a = -1$ 
%  \item $M^{-}$ is an FS for $a > -1$.
%\end{itemize}
The classifications are listed in Table~\ref{table:class}.

\begin{table*}[htpb]
        \centering
            \begin{tabular}{| c | c | c | c | c | c |}
                \hline
                    \backslashbox{\footnotesize Singularity}{\footnotesize Interval} & $a<-1$ & $a=-1$ & $-1<a<1$ & $a=1$ & $a>1$ \\ \hline
                    $E$ & S & SN & C & SN & S \\
                    $M^+$ & FS & FS & FS & RFSN-II & FC \\
                    $M^-$ & FC & RFSN-II & FS & FS & FS \\ 
                    \hline
            \end{tabular}
            \caption{Classification of the ordinary singularity $E$ and the folded singularities $M^\pm$ as functions of the bifurcation parameter $a$. S = saddle; SN = saddle-node; C = center; FS = folded saddle; RFSN-II = reversible FSN-II; FC = folded center. At $a=-1$, the ordinary singularity $E$ coincides with the folded singularity $M^-$ in a RFSN-II. Similarly, at $a=1$, $E$ and $M^+$ coincide at the RFSN-II.}
    \label{table:class}
\end{table*}

\begin{figure}[h!]
    \centering
    \includegraphics[width=5in]{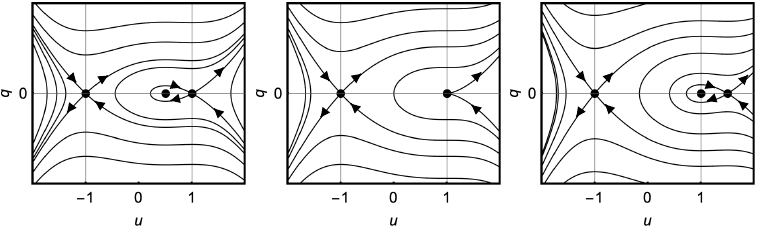}
    \put(-362,102){(a)}
    \put(-242,102){(b)}
    \put(-120,102){(c)}
    \caption{Phase plane of the desingularized reduced system \eqref{eq:desing-reduced} for different values of the bifurcation parameter $a$. 
    (a) $0<a<1$: 
    $M^\pm$ are FS, and $E$ is a center.
    (b) $a=1$: $M^-$ is a FS, and the point $(1,0)$ is a RFSN-II point, where 
    $E$ and $M^+$ coincide. The one-branched, local stable and unstable manifolds of $(1,0)$ are given by the graphs of the functions $q = \pm \frac{2 \sqrt{\eps}}{\sqrt{3}} (u-1)^{3/2}\left( 1 + \frac{3}{8} (u-1) \right)$ for $u \ge 1$. 
    (c) $1 < a < 3$: $M^-$ is still a FS, but $E$ and $M^+$ have exchanged stability types, with $M^+$ now a center and $E$ a saddle. 
    For all values of $a$, the orbits lie on the level sets of the Hamiltonian $H_{\rm desing}(u,q;a)$ given by \eqref{eq:hatH}.
    Moreover, these level sets are the projections onto the $(u,q)$ plane of the level sets of the conserved quantity $\mathcal{G}$ on $S$.
    The key folded singularity we analyze using geometric desingularization in the case $a=1$ is the RFSN-II point in (b) (see Section~\ref{sec:desingFSNII}), while for $a<1$ the FS point $M^+$ located at $(1,0)$ in (a) is the main folded singularity to be  studied 
    (current work).
    (For completeness, we note that for $a=3$ there are heteroclinic connections from $M^-$ to $E$ and a family of elliptic orbits that surround $M^+$; and, for $a>3$, $M^-$ is connected to itself by a homoclinic orbit that surrounds $M^+$; phase planes not shown.)  Here, $\eps=0.1$.
    }
\label{fig:phasespace-desingreduced}
\end{figure}

The phase plane of \eqref{eq:desing-reduced}
is sketched for different values of $a>0$
in Fig.~\ref{fig:phasespace-desingreduced}. (Those for $a<0$ may be obtained using the symmetry $(u,q,a) \to (-u,-q,-a)$.)
The orbits lie on the level sets
of the Hamiltonian $H_{\rm desing}$
given by \eqref{eq:hatH}. 
The equilibrium $E$ lies on the level set $H_{\rm desing}=\frac{\eps a^2}{12} ( 6 - a^2 )$, and the folded singularities $M^\pm$ lie on the level sets $H_{\rm desing}=\eps \left( \pm \frac{2}{3}a - \frac{1}{4} \right)$. Furthermore, we observe that $\frac{\partial H_{\rm desing}}{\partial u} (u,0) = \eps (u^2-1)(u-a)> 0 $
for $u \in (-1,a)$ in the case $0<a<1$ and for $u \in (-1,1)$
in the case $a\ge 1$. \medskip

\begin{remark}
By Fenichel theory \cite{Fenichel1979,Jones1995}, the normally hyperbolic saddle sheets, $S_s^{\pm}$, of the critical manifold perturb to nearby invariant slow manifolds, $S_{s,\delta}^{\pm}$, of \eqref{eq:spatialODE-y}. Moreover, the flow of \eqref{eq:spatialODE-y} restricted to $S_{s,\delta}^{\pm}$ is a small $\mathcal{O}(\delta)$ perturbation of the reduced flow on $S_s^{\pm}$. Canard theory can then be used to establish that the singular true and faux canards of the folded saddle persist as maximal canard solutions for $0<\delta \ll 1$ provided the eigenvalues of the folded saddle remain bounded away from zero \cite{Mitry2017,Szmolyan2001}. 

\begin{figure}[h!]
  \centering
  \includegraphics[width=0.95\textwidth]{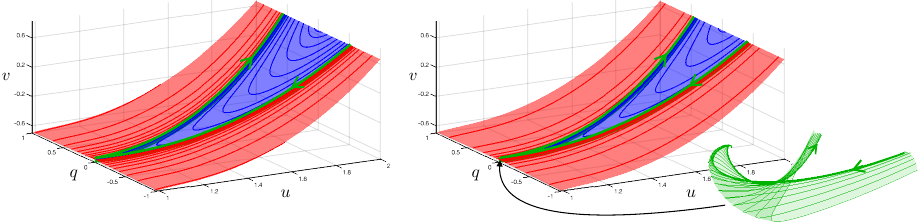}
  \put(-452,96){(a)}
  \put(-256,96){(b)}
  \caption{Saddle slow manifolds, $S_{s,\delta}^{+}$, for (a) $a=0.6$, i.e., far from the Turing value $a_T$ and (b) $a=1$, i.e., close to $a_T$. The red and blue surfaces, respectively, indicate the solutions that transition to fast dynamics and the solutions that turn away from $L^+$ and remain on $S_{s,\delta}^{+}$. The maximal true and faux canards of the folded saddle (green curves) are the separatrices that delimit the two classes of solutions on $S_{s,\delta}^+$. In (b), there is a family of canard solutions and these canards can execute almost self-similar loops (green surface).  
  Here, $\eps = 0.1$ and $\delta=0.01$.}
  \label{fig:slowmanifolds}
\end{figure}

The saddle slow manifolds, $S_{s,\delta}^+$, and associated maximal canards are shown in Fig.~\ref{fig:slowmanifolds}. The numerical method used to compute these is outlined in Appendix~\ref{app:numericalmethods_slowmanifolds}.
The saddle slow manifolds consist of two types of solutions. The first type are the solutions that approach the neighbourhood of the fold curve $L^+$ and subsequently escape via the fast dynamics. The other type are the solutions that turn away from the fold region and remain on $S_{s,\delta}^+$. The separatrices that divide between the two types of solutions are the maximal canards of the folded saddle. In the RFSN-II limit, the saddle slow manifolds and the canard dynamics become complicated (see Fig.~\ref{fig:slowmanifolds}(b).)

\end{remark}

%-----------------------------------------------------------------------------
\section{Desingularization of the Reversible FSN-II Point \texorpdfstring{$M^+$ at $a=1$}{Lg}}
\label{sec:desingFSNII}
%-----------------------------------------------------------------------------

In this section, we focus on the dynamics of the system for $a$ values near $a=1$ and desingularize the reversible FSN-II point $M^+$ located at $(u, p, v, q, a, \delta) = (1, 0, -\frac{2}{3}, 0, 1, 0)$.
First, we translate coordinates so that the RFSN-II point is shifted to the origin.
Let
\begin{equation}
\label{eq:translate}
(u,p,v,q,a,\delta)
=\left(1+\tilde{u},\tilde{p},-\frac{2}{3} + \tilde{v}, \tilde{q},
1+\tilde{a},\tilde{\delta}\right).
\end{equation}
In these new coordinates,
the spatial system \eqref{eq:spatialODE-y} becomes
\begin{equation}
\label{eq:tilde-spatialODE-y}
\begin{split}
\tilde{u}_y &= \tilde{p} \\
\tilde{p}_y &= \frac{\tilde{u}^3}{3} + \tilde{u}^2 - \tilde{v} \\
\tilde{v}_y &= \tilde{\delta} \tilde{q} \\
\tilde{q}_y &= \tilde{\delta} \eps (\tilde{u}-\tilde{a}) \\
\tilde{a}_y &= 0 \\
\tilde{\delta}_y &= 0.
\end{split}
\end{equation}
Also, the conserved quantity \eqref{eq:G} may be rewritten as
\begin{equation}
    \label{eq:Gtilde}
    \mathcal{G}(u,p,v,q,a) 
    =  \tilde{\mathcal{G}}(\tilde{u},\tilde{p},\tilde{v},\tilde{q},\tilde{a}) + \eps \left( \frac{2}{3}\tilde{a} + \frac{5}{12} \right)  
    =  \tilde{\mathcal{G}}(\tilde{u},\tilde{p},\tilde{v},\tilde{q},\tilde{a}) + \eps \left( \frac{2}{3}a - \frac{1}{4} \right),
\end{equation}
where $\tilde{\mathcal{G}}(\tilde{u},\tilde{p},\tilde{v},\tilde{q},\tilde{a})= \frac{1}{2} (\eps \tilde{p}^2 - \tilde{q}^2 ) 
- \eps \left( \frac{1}{12}\tilde{u}^4 + \frac{1}{3} \tilde{u}^3 + (\tilde{a}-\tilde{u})\tilde{v} \right)$.

The point 
$(\tilde{u},\tilde{p},\tilde{v},\tilde{q},\tilde{a},\tilde{\delta})
=(0,0,0,0,0,0)$, which corresponds to the RFSN-II point at $a=1$
%and acts as the organizing center 
(recall Table~\ref{table:class}), 
is a nilpotent point
of \eqref{eq:tilde-spatialODE-y}.
Hence, in order to unfold the dynamics, we desingularize this point 
by introducing the additional dependent variable $r$
and by dynamically rescaling the variables:
\begin{equation}
\label{eq:dynrescale}
        \tilde{u} = \sqrt{\eps} r^2 \bar{u}, \ \
        \tilde{p} = \eps r^3 \bar{p}, \ \
        \tilde{v} = \eps r^4 \bar{v}, \ \ 
        \tilde{q} = \eps^{3/2}  r^3 \bar{q}, \ \ 
        \tilde{a} = \sqrt{\eps} r^3 \bar{a}, \ \ 
        \tilde{\delta} = r^2 \bar{\delta}. 
\end{equation}
Here, $(\bar u,\bar p,\bar v,\bar q,\bar a, \bar \delta) \in \mathbb S^5$.
The factors of $\eps$ in \eqref{eq:ODE-K2} were chosen 
for symmetry in the equations for $\tilde u$ and $\tilde v$.

We will examine system
\eqref{eq:tilde-spatialODE-y}
in two coordinate charts:
the entry/exit chart
\begin{equation}
        \label{eq:defK1}
        K_1 = \{ \bar{u}=1 \}
\end{equation}
and the rescaling (or central) chart
\begin{equation}
        \label{eq:K2}
        K_2 = \{ \bar{\delta}=1 \}.
\end{equation}
As is customary in articles on geometric desingularization (a.k.a. blow up), in chart $K_i, i=1,2,$
the variable $\bar{\square}$
will be denoted by $\square_i$.
In chart $K_1$,
the rescaling \eqref{eq:dynrescale}
is given by
\begin{equation}
\label{eq:dynrescale-K1}
        \begin{split}
        \tilde{u} &= \sqrt{\eps} r_1^2, \ \
        \tilde{p} = \eps r_1^3 p_1, \ \
        \tilde{v} = \eps r_1^4 v_1, \ \
        \tilde{q} = \eps^{3/2}  r_1^3 q_1, \ \
        \tilde{a} = \sqrt{\eps}r_1^3 a_1, \ \
        \tilde{\delta} = r_1^2 \delta_1.
        \end{split}
\end{equation}
Then, in chart $K_2$,
the rescaling \eqref{eq:dynrescale}
is given by
\begin{equation}
\label{eq:dynrescale-K2}
        \begin{split}
        \tilde{u} &= \sqrt{\eps} r_2^2 u_2, \ \
        \tilde{p} = \eps r_2^3 p_2, \ \
        \tilde{v} = \eps r_2^4 v_2, \ \
        \tilde{q} = \eps^{3/2}  r_2^3 q_2, \ \
        \tilde{a} = \sqrt{\eps} r_2^3 a_2,\ \
        \tilde{\delta} = r_2^2.
        \end{split}
\end{equation}

This section is organized as follows.
In Subsection~\ref{sec:K2}, we present the analysis in the rescaling chart $K_2$, identifying the key algebraic solution in this chart and the self-similarity of the level set on which it lies. 
These results also justify the range of validity of the asymptotics, as stated in the Introduction, since  we have $\tilde{a} = \eps^{3/2}\tilde{\delta}^{3/2} q_2$ by the rescalings \eqref{eq:dynrescale-K2} and $\tilde{\delta}=\delta$ by \eqref{eq:translate}.
Subsection~\ref{sec:K1} contains the analysis in the entry/exit chart $K_1$.
Then, the transition maps between the two charts are presented in Subsection~\ref{sec:Gamm0-intoK1}, and these are used to track the key solution from chart $K_2$ into chart $K_1$. 
Next, in Subsection~\ref{sec:WcuWcs}, the transverse intersection of the main center-unstable and center-stable manifolds is shown, thereby establishing the formula \eqref{eq:ac} for the critical value $a_c(\delta)$.
Finally, Subsection~\ref{sec:FraserRoussel} presents sp,e analysis of the smoothness of key solutions.

%---------------------------------------------------------------------
\subsection{Rescaling (or central) chart \texorpdfstring{$K_2$ ($\delta_2=1$) }{Lg}}
\label{sec:K2}
%---------------------------------------------------------------------

To obtain the governing equations
in chart $K_2$,
we substitute \eqref{eq:dynrescale-K2}
into \eqref{eq:tilde-spatialODE-y}:
\begin{equation*}
\begin{split}
  \frac{du_2}{dy} &= r_2 \sqrt{\eps} p_2 \\
  \frac{dp_2}{dy} &= r_2 (u_2^2 - v_2) + \frac{1}{3} \sqrt{\eps} r_2^3 u_2^3 \\
  \frac{dv_2}{dy} &= r_2 \sqrt{\eps} q_2\\
  \frac{dq_2}{dy} &= r_2 (u_2 - r_2 a_2) \\ 
  \frac{da_2}{dy} &= 0.
\end{split}
\end{equation*}
Here, we used $\frac{dr_2}{dy}=0$ in $K_2$,
which follows since $\bar{\delta}\equiv 1$ implies
$\frac{d\bar{\delta}}{dy}=0$.
Then, introducing the rescaled independent variable
$y_2 = r_2 y$
and letting prime denote
$\frac{d}{dy_2}$,
we obtain the desingularized vector field in chart $K_2$,
\begin{equation}
\label{eq:ODE-K2}
\begin{split}
  {u_2}' &= \sqrt{\eps} p_2 \\
  {p_2}' &= u_2^2 - v_2 + \frac{1}{3} \sqrt{\eps} r_2^2 u_2^3 \\
  {v_2}' &= \sqrt{\eps} q_2\\
  {q_2}' &= u_2 - r_2 a_2 \\
  a_2' &= 0.
\end{split}
\end{equation}

For each $a_2$, this system is Hamiltonian
\begin{equation*}
\begin{split}
    \begin{bmatrix} u_2^\prime \\ q_2^\prime \\ p_2^\prime \\ v_2^\prime \end{bmatrix} &= \begin{bmatrix} \frac{\partial H_2}{\partial p_2} \\ \frac{\partial H_2}{\partial v_2} \\  -\frac{\partial H_2}{\partial u_2} \\ -\frac{\partial H_2}{\partial q_2} \end{bmatrix} = J \ \nabla_{(u_2,q_2,p_2,v_2)} H_2
        %{u_2}' &= \frac{\partial H_2}{\partial p_2} \\
        %{p_2}' &= -\frac{\partial H_2}{\partial u_2} \\
        %{v_2}' &= -\frac{\partial H_2}{\partial q_2} \\
        %{q_2}' &= \frac{\partial H_2}{\partial v_2}.
\end{split}
\end{equation*}
Here, $J$ is the $4\times 4$ anti-symmetric matrix $\begin{bmatrix} 0 & \mathbb{I} \\ -\mathbb{I} & 0 \end{bmatrix}$, $\mathbb I$ is the $2\times 2$ identity matrix, and $H_2$ is the Hamiltonian:
\begin{equation}
\label{eq:H2}
H_2(u_2,p_2,v_2,q_2;r_2)
   = \frac{1}{2}\sqrt{\eps}\left( p_2^2 - q_2^2 \right)
   + (u_2 - r_2 a_2) v_2
     - \left( \frac{1}{3} u_2^3 + \frac{1}{12} \sqrt{\eps} r_2^2 u_2^4 \right), 
\end{equation}
in which $r_2$ is a constant parameter.
Furthermore, this Hamiltonian $H_2$ is directly related through the transformation  \eqref{eq:dynrescale-K2} to the conserved quantity via:
$\tilde {\mathcal{G}}  (\tilde{u},\tilde{p},\tilde{v},\tilde{q},\tilde{a}) 
= \eps^{5/2} r_2^6 H_2 (u_2, p_2, v_2, q_2; r_2)$,  
recall \eqref{eq:Gtilde}.

The hyperplane $\{ r_2=0 \}$
is invariant in chart $K_2$.
On this hyperplane,
the system reduces to
\begin{equation}
        \label{eq:r2zero}
\begin{split}
  {u_2}^\prime &= \sqrt{\eps} p_2 \\
  {p_2}^\prime &= u_2^2 - v_2 \\
  {v_2}^\prime &= \sqrt{\eps} q_2\\
  {q_2}^\prime &= u_2.
\end{split}
\end{equation}
%
%----------------------------------------------------------------
\subsubsection{The key algebraic solution in chart \texorpdfstring{$K_2$}{Lg}}
\label{subsec:Gamma0}
%-----------------------------------------------------------------

In this section, we establish the existence of a key algebraic solution that acts as a separatrix in phase space.

\begin{lemma}
The system \eqref{eq:r2zero} possesses an explicit algebraic solution given by
\begin{equation}
\label{eq:Gamma0}
\Gamma_0 = (u_2(y_2), p_2(y_2), v_2(y_2), q_2(y_2))
         = \left(
           \frac{1}{12}\sqrt{\eps} y_2^2,
           \frac{1}{6} y_2,
           \frac{1}{144}\eps y_2^4 - \frac{1}{6},
           \frac{1}{36} \sqrt{\eps} y_2^3
         \right).
\end{equation}
Moreover, the orbit $\Gamma_0$ has the following two halves:
\[ \Gamma_0^- = \left. \Gamma_0 \right|_{y_2\leq 0} \quad \text{ and } \quad \Gamma_0^+ = \left. \Gamma_0 \right|_{y_2 \geq 0}. \]
The segment $\Gamma_0^-$ corresponds to the true canard of the RFSN-II singularity and the segment $\Gamma_0^+$ corresponds to the faux canard of the RFSN-II. 
\end{lemma}

As we will see, both halves of the orbit $\Gamma_0$  play central roles in the dynamics induced by the RFSN-II singularity and also in the singular limit of the spatially periodic canards. 
It lies in the zero level set of the Hamiltonian $H_2(u_2,p_2,v_2,q_2;0)$, as may be verified by direct substitution and observing that all $y_2$-dependent terms sum to zero.  
For reference, it may also be parametrized by $p_2$: $u_2 = 3 \sqrt{\eps} p_2^2,$ $v_2 = 9\eps p_2^4 - \frac{1}{6},$ and $q_2 = 6\sqrt{\eps} p_2^3$.

%-------------------------------------
\subsubsection{Geometric role of the algebraic solution \texorpdfstring{$\Gamma_0$}{Lg}}
%-------------------------------------
Here, we study the unperturbed problem \eqref{eq:r2zero} and provide a geometric interpretation of the algebraic solution $\Gamma_0$. 
Recall from \eqref{eq:H2} that the system \eqref{eq:r2zero} has the Hamiltonian 
\[ H_2(u_2,p_2,v_2,q_2;0) = \tfrac{1}{2}\sqrt{\eps} (p_2^2-q_2^2)+u_2 v_2 - \tfrac{1}{3} u_2^3, \] 
and the canard solutions, $\Gamma_0^{\pm}$, lie in the zero level contour $\{ H_2(u_2,p_2,v_2,q_2;0) = 0 \}$. 
By using this constraint to eliminate the $v_2$ variable, the dynamics restricted to the $\{ H_2(u_2,p_2,v_2,q_2;0) = 0 \}$ level set are given by 
\begin{equation} \label{eq:H2zero}
\begin{split}
u_2^\prime &= \sqrt{\eps} p_2 \\ 
p_2^\prime &= \tfrac{2}{3}u_2^2 + \tfrac{1}{2u_2}\sqrt{\eps} (p_2^2-q_2^2) \\ 
q_2^\prime &= u_2.
\end{split}
\end{equation}
To remove the singularity at $u_2 = 0$, we desingularize the vector field via the transformation $dy_2 = u_2 d\eta_2$, which gives
\begin{equation} \label{eq:H2zerodesing}
\begin{split}
\frac{du_2}{d\eta_2} &= \sqrt{\eps} u_2 p_2 \\ 
\frac{dp_2}{d\eta_2} &= \tfrac{2}{3}u_2^3 + \tfrac{1}{2}\sqrt{\eps} (p_2^2-q_2^2) \\ 
\frac{dq_2}{d\eta_2} &= u_2^2.
\end{split}
\end{equation}
The system \eqref{eq:H2zerodesing} is topologically equivalent to \eqref{eq:H2zero} for $u_2 >0$ and has opposite orientation for $u_2<0$. Moreover, the system \eqref{eq:H2zerodesing} has the symmetry 
\begin{equation} \label{eq:H2zerosymmetry} 
(u_2,p_2,q_2,\eta_2) \to (u_2,-p_2,-q_2,-\eta_2),
\end{equation}
which it inherits from the reversibility symmetry \eqref{eq:RF} of the original problem. 

The system \eqref{eq:H2zerodesing} possesses two lines of equilibria, 
\[ \mathscr{L}_{\pm} = \left\{ u_2 = 0, p_2 = \pm q_2: q_2 \in \mathbb R \right\}. \]
The line $\mathscr{L}_-$ has hyperbolic and center spectra, $\sigma_{\rm s/u} = \{ -\sqrt{\eps}q_2,-\sqrt{\eps}q_2\}$ (where $s/u$ depends on the sign of $q_2$) and $\sigma_c = \{ 0 \}$, respectively, with corresponding subspaces 
\[ \mathbb{E}^{\rm s/u}\left( \mathscr{L}_- \right) = \operatorname{span} \left\{ \begin{bmatrix} 1 \\ 0 \\ 0 \end{bmatrix}, \begin{bmatrix} 0\\1\\0 \end{bmatrix} \right\} \quad \text{ and } \quad \mathbb{E}^c\left( \mathscr{L}_- \right) = \operatorname{span} \begin{bmatrix} 0 \\ -1 \\ 1 \end{bmatrix}. \]
Thus, the half-lines 
\begin{equation} \label{eq:Lminus}
\mathscr{L}_-^s := \left\{ (u_2,p_2,q_2) \in \mathscr{L}_- : q_2 >0 \right\} \quad \text{ and } \quad \mathscr{L}_-^u := \left\{ (u_2,p_2,q_2) \in \mathscr{L}_- : q_2 < 0 \right\}
\end{equation} 
are center-stable and center-unstable, respectively.
Similarly, the line $\mathscr{L}_+$ has hyperbolic and center spectra, $\sigma_{\rm s/u} = \left\{ \sqrt{\eps} q_2, \sqrt{\eps} q_2 \right\}$ and $\sigma_c = \{ 0 \}$, respectively, with corresponding subspaces 
\[ \mathbb{E}^{\rm s/u} \left( \mathscr{L}_+ \right) = \operatorname{span} \left\{ \begin{bmatrix} 1 \\ 0 \\ 0 \end{bmatrix}, \begin{bmatrix} 0\\1\\0 \end{bmatrix} \right\} \quad \text{ and } \quad \mathbb{E}^c \left( \mathscr{L}_+ \right) = \operatorname{span} \begin{bmatrix} 0 \\ 1 \\ 1 \end{bmatrix}. \]
Thus, the half-lines 
\begin{equation} \label{eq:Lplus}
\mathscr{L}_+^s := \left\{ (u_2,p_2,q_2) \in \mathscr{L}_+ : q_2 <0 \right\} \quad \text{ and } \mathscr{L}_+^u := \left\{ (u_2,p_2,q_2) \in \mathscr{L}_+ : q_2 >0 \right\}
\end{equation}
are center-stable and center-unstable, respectively.

In Appendix~\ref{sec:app61},
we prove the following proposition that establishes the key geometric properties of the algebraic solution, $\Gamma_0$, and its stable and unstable manifolds in phase space:

\begin{proposition}
The restriction of the system \eqref{eq:H2zerodesing} to the half-space $\{ u_2 \geq 0\}$ possesses three classes of heteroclinic solutions. 
\begin{itemize}
\item A class 1 heteroclinic is a solution that connects a point $(0,-\tilde{q}_2,\tilde{q}_2) \in \mathcal{L}_-^u$ for some $\tilde{q}_2<0$ to a point $(0,-\hat{q}_2,\hat{q}_2) \in \mathcal{L}_-^s$ for some $\hat{q}_2>0$.
\item A class 2 heteroclinic is a solution that connects a point $(0,-\tilde{q}_2,\tilde{q}_2) \in \mathcal{L}_-^u$ for some $\tilde{q}_2<0$ to a point $(0,\hat{q}_2,\hat{q}_2) \in \mathcal{L}_+^s$ for some $\hat{q}_2<0$.
\item A class 3 heteroclinic is a solution that connects a point $(0,\tilde{q}_2,\tilde{q}_2) \in \mathcal{L}_+^u$ for some $\tilde{q}_2>0$ to a point $(0,-\hat{q}_2,\hat{q}_2) \in \mathcal{L}_-^s$ for some $\hat{q}_2>0$.
\end{itemize}
The stable manifold, $W^s(\Gamma_0^-)$, of the algebraic solution $\Gamma_0^-$ is the phase-space boundary between class 1 and class 2 heteroclinics. 
The unstable manifold, $W^u(\Gamma_0^+)$, of the algebraic solution $\Gamma_0^+$ is the phase-space boundary between class 1 and class 3 heteroclinics. 
\label{prop:61}
\end{proposition}

The results of this proposition are illustrated in Fig.~\ref{fig:secondblowup}. 

\begin{figure}[h!]
  \centering
  \includegraphics[width=3.75in]{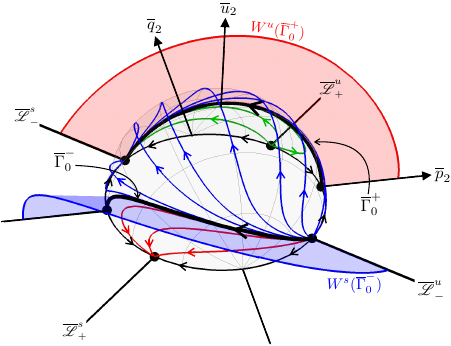}
  \caption{Dynamics on the blown-up hemisphere $\overline{u}_2^2+\overline{p}_2^2+\overline{q}_2^2=1$ in the half-space $\{ \overline{u}_2 \geq 0 \}$. The lines of equilibria, $\overline{\mathscr{L}}_{\pm}^u$, intersect the blown-up hemisphere on the equator at unstable nodes. The lines of equilibria, $\overline{\mathscr{L}}_{\pm}^s$, intersect the blown-up hemisphere on the equator at stable nodes. 
  The algebraic solution $\overline{\Gamma}_0^-$ intersects $\mathbb{S}^2$ at $(0,-1,0)$ and continues across the hemisphere as the heteroclinic solution that joins the points $(0,-1,0)$ and $\overline{\mathscr{L}}_{\pm}^u \cap \mathbb{S}^2$.
  The class 1, 2, and 3 heteroclinics are shown as blue, red, and green curves, respectively. The stable manifold (blue surface) of $\Gamma_0^-$ separates the class 1 and class 2 heteroclinics. The unstable manifold (red surface) of $\Gamma_0^+$ separates the class 1 and class 3 heteroclinics.}
  \label{fig:secondblowup}
\end{figure}

%---------------------------------------------------------------------
\subsection{Entry/exit chart \texorpdfstring{$K_1$}{Lg}}
\label{sec:K1}
%---------------------------------------------------------------------

We substitute \eqref{eq:dynrescale-K1} into \eqref{eq:tilde-spatialODE-y} and desingularize the vector field by transforming to a dynamic independent variable $y_1= r_1 y$. Hence, the system in chart $K_1$ is
\begin{equation}
        \label{eq:K1}
\begin{split}
\frac{dr_1}{dy_1} &= \frac{1}{2}\sqrt{\eps} p_1 r_1 \\
\frac{d \delta_1}{dy_1} &= -\sqrt{\eps}p_1 \delta_1 \\
\frac{dp_1}{dy_1} &= 1-v_1-\frac{3}{2}\sqrt{\eps}p_1^2+\frac{1}{3}\sqrt{\eps}r_1^2 \\
\frac{dv_1}{dy_1} &= \sqrt{\eps} \left(-2p_1 v_1 + \delta_1 q_1 \right) \\
\frac{dq_1}{dy_1} &= -\frac{3}{2} \sqrt{\eps} p_1 q_1+\delta_1 (1-r_1 a_1 ) \\
         \frac{da_1}{dy_1} &= -\frac{3}{2} \sqrt{\eps} p_1 a_1.
\end{split}
\end{equation}

This system has a series of invariant hyperplanes defined by $\{ r_1=0\}$, $\{ \delta_1=0\}$, $\{ a_1=0\}$, and their intersections. In this section, we present the analysis of system \eqref{eq:K1} in the hyperplane $\{r_1=0\}\cap\{\delta_1=0\} \cap \{ a_1=0\}$, where the dynamics are the simplest.  Then, we go directly to the analysis of system \eqref{eq:K1} in the full $(r_1,\delta_1,p_1,v_1,q_1,a_1)$ space. 
Intermediate results that are needed for this full analysis are presented in Appendix~\ref{sec:appK1}, where we analyze system \eqref{eq:K1} in the following sequence of invariant hyperplanes: 
$\{ \delta_1=0\} \cap \{ a_1=0\}$, 
$\{ r_1=0\} \cap \{ a_1 = 0 \}$, 
$\{ r_1 = 0 \} \cap \{ \delta_1 = 0 \}$, 
$\{ \delta_1=0 \}$, 
and $\{ r_1=0\}$,
%and $\{ a_1=0\}$, 
respectively.
For the remainder of this section and in Appendix~\ref{sec:appK1}, we will use an overdot to denote the derivative with respect to $y_1$: $\dot{}  = \frac{d}{dy_1}$.

In the invariant hyperplane
$\{r_1=0\}\cap\{\delta_1=0\} \cap \{ a_1 = 0 \}$,
the equations \eqref{eq:K1} reduce to
\begin{equation}
        \label{eq:r1delta1a1zero}
\begin{split}
        \dot{p}_1 &= 1 - v_1 - \frac{3}{2}\sqrt{\eps} p_1^2 \\
        \dot{v}_1 &= -2 \sqrt{\eps} p_1 v_1 \\
        \dot{q}_1 &= -\frac{3}{2}\sqrt{\eps} p_1 q_1.
\end{split}
\end{equation}
Hence, the $q_1$ equation decouples from the equations for $(p_1,v_1)$. 
This system \eqref{eq:r1delta1a1zero} has two invariant lines. The first is a line of equilibria
\begin{equation}
\label{eq:ell}
        \ell = \{ r_1=0,\,\, \delta_1=0,\,\, p_1=0,\,\, v_1=1,\,\, q_1 \in \mathbb{R},\,\, a_1=0 \}.
\end{equation}
It is a line of saddle points,
since the eigenvalues are
\begin{equation*}
        \lambda_s = -\sqrt{2} \eps^{1/4}, \quad
        \lambda_u = \sqrt{2}\eps^{1/4}, \quad \text{ and } \quad
        \lambda_c = 0.
\end{equation*}
The associated eigenvectors are
\begin{equation*}
{\boldsymbol w}_s = \left[ \begin{array}{c}
        2\sqrt{2}  \\ 4\eps^{1/4} \\ 3 \eps^{1/4} q_1
             \end{array} \right],
\quad
{\boldsymbol w}_u = \left[ \begin{array}{c}
        -2\sqrt{2}  \\ 4\eps^{1/4} \\ 3 \eps^{1/4} q_1
             \end{array} \right],
\quad \text{ and } \quad
{\boldsymbol w}_c = \left[ \begin{array}{c}
               0 \\ 0 \\ 1
             \end{array} \right].
\end{equation*}
See Fig.~\ref{fig:K1}. 
Therefore, by standard center manifold theory (see for example \cite{C1981}), system \eqref{eq:r1delta1a1zero} has a one-dimensional center manifold $W^c(\ell)$ in the invariant hyperplane
$\{ r_1 = 0 \} \cap \{ \delta_1 = 0 \} \cap \{ a_1=0 \}$, and it has one-dimensional fast stable and unstable fibers in the transverse directions.
This manifold, which we will show is embedded in a larger center manifold in the full system, and the fast stable and unstable fibers,
will play central roles in the dynamics of the full system \eqref{eq:K1}.

\begin{figure}[h!]
  \centering
  \includegraphics[width=5in]{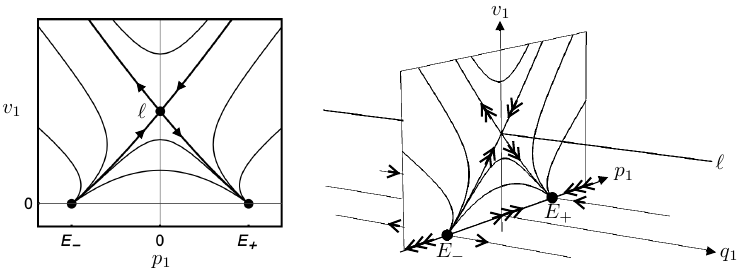}
  \put(-362,124){(a)}
  \put(-208,124){(b)}
  \caption{Dynamics in the invariant subspace $\left\{  r_1 = 0\right\} \cap \left\{ \delta_1 = 0 \right\} \cap \left\{ a_1 = 0 \right\}$. showing the line $\ell$ of saddle fixed points and the invariant line $I$ with the two equilibria $E_\pm$ on it.   
  (a) Projection onto the $(p_1,v_1)$ plane. (b) Projection onto the $(p_1,v_1,q_1)$ phase space. 
  Trajectories are given implicitly by $\sqrt{\eps} p_1^2 - \left( \frac{2}{3} - 2 v_1 \right) -c v_1^{3/2} = 0$, for $c \in \mathbb R$, with $q_1$ obtained by quadratures.
  }
  \label{fig:K1}
\end{figure}

The second invariant line in the phase space of equations \eqref{eq:r1delta1a1zero} is
\begin{equation}
        \label{eq:I}
        I = \left\{   r_1=0,\,\, \delta_1=0,\,\,
                      p_1 \in \mathbb{R},\,\,
                      v_1=0,\,\, q_1=0,\,\, a_1=0
            \right\}.
\end{equation}
On $I$, the vector field reduces to $\dot{p}_1 = 1 - \frac{3}{2}\sqrt{\eps}p_1^2$, and there are two equilibria:
\begin{equation}
        \label{eq:Epm}
        E_{\pm}= \left\{ r_1 =0,\,\, \delta_1=0, \,\,
                      p_1 = \pm \sqrt{\tfrac{2}{3}} \eps^{-1/4},\,\,
                      v_1=0,\,\, q_1=0,\,\, a_1=0
            \right\}.
\end{equation}
The equilibrium $E_+$ is stable with eigenvalues $-\sqrt{6}\eps^{1/4}, -2\sqrt{\frac{2}{3}}\eps^{1/4}, -\sqrt{\frac{3}{2}}\eps^{1/4}$,
and eigenvectors
\begin{equation*}
{\boldsymbol w}_1^s = \left[ \begin{array}{c}
               1 \\ 0 \\ 0
             \end{array} \right],
\quad
{\boldsymbol w}_2^s = \left[ \begin{array}{c}
        -\sqrt{3} \\ \sqrt{2}\eps^{1/4} \\ 0
             \end{array} \right],
\quad
{\boldsymbol w}_3^s = \left[ \begin{array}{c}
               0 \\ 0 \\ 1
             \end{array} \right].
\end{equation*}
The equilibrium $E_-$ is unstable with eigenvalues
$\sqrt{\frac{3}{2}}\eps^{1/4}$,
$2\sqrt{\frac{2}{3}}\eps^{1/4}$,
and $\sqrt{6}\eps^{1/4}$,
and eigenvectors
\begin{equation*}
{\boldsymbol w}_1^u = \left[ \begin{array}{c}
               0 \\ 0 \\ 1
             \end{array} \right],
\quad
{\boldsymbol w}_2^u = \left[ \begin{array}{c}
        \sqrt{3} \\ \sqrt{2}\eps^{1/4} \\ 0
             \end{array} \right],
\quad
{\boldsymbol w}_3^u = \left[ \begin{array}{c}
               1 \\ 0 \\ 0
             \end{array} \right].
\end{equation*}

The equilibria $E_\pm$ are shown in the $(p_1,v_1)$ subsystem in 
Fig.~\ref{fig:K1}(a) 
and in the $(p_1,v_1,q_1)$ system in Fig.~\ref{fig:K1}(b).
The number of arrows on the principal components of the stable and unstable manifolds indicates  the relative magnitudes of the eigenvalues. 
The orbits in this phase plane are defined implicitly by the one-parameter family  $\sqrt{\eps} p_1^2 - \left( \frac{2}{3} - 2 v_1 \right) -c v_1^{3/2} = 0$, where $c$ denotes the parameter.
Overall, the dynamics of \eqref{eq:r1delta1a1zero} are organized by the line of equilibria $\ell$ and the invariant line $I$. 

Now, we turn to the invariant sets and dynamics of the full system \eqref{eq:K1}, recalling that some of the intermediate results are in Appendix~\ref{sec:appK1}.
The line $I$ given by \eqref{eq:I} is invariant for the full system \eqref{eq:K1}, and $E_{\pm}$ are again fixed points on $I$, recall \eqref{eq:Epm}.

The equilibrium $E_+$ is a hyperbolic saddle with stable and unstable spectra 
\begin{equation}
    \label{eq:spectrum:E+}
\sigma^s_+ = \left\{ -\sqrt{6}\eps^{1/4},
-2\sqrt{\frac{2}{3}}\eps^{1/4},
-\sqrt{\frac{3}{2}}\eps^{1/4},
-\sqrt{\frac{3}{2}}\eps^{1/4},
-\sqrt{\frac{2}{3}}\eps^{1/4}\right\}
\quad{\rm and}
\quad\sigma_+^u = \left\{ \frac{1}{\sqrt{6}}\eps^{1/4}\right\}.
\end{equation}
At $E_+$, the stable and unstable subspaces are given by 
\begin{equation}
\label{eq:Esu4E+}
\mathbb{E}^s = {\rm span}  
\left\{ 
\left[ \begin{array}{c}
               0 \\ 0 \\ 1 \\ 0 \\ 0 \\ 0 
             \end{array} \right],
\left[ \begin{array}{c} 0 \\ 0 \\
        -\sqrt{3} \\ \sqrt{2}\eps^{1/4} \\ 0 \\ 0
             \end{array} \right],
\left[ \begin{array}{c}
               0 \\ 0 \\ 0 \\ 0 \\ 1 \\ 0 
             \end{array} \right],
\left[ \begin{array}{c}
               0 \\ 0 \\ 0 \\ 0 \\ 0 \\ 1 
             \end{array} \right],
\left[ \begin{array}{c}
               0 \\ \eps^{1/4} \\ 0 \\ 0 \\ \sqrt{6} \\ 0 
             \end{array} \right]
\right\}
\quad{\rm and}
\quad 
\mathbb{E}^u = {\rm span} 
\left\{ 
\left[ \begin{array}{c}
               1 \\ 0 \\ 0 \\ 0 \\ 0 \\ 0 
             \end{array} \right]
             \right\}.
 \end{equation}

The other equilibrium, $E_-$, is also a hyperbolic saddle with stable and unstable spectra given by
\begin{equation}
\label{eq:spectrum:E-}
    \sigma^s_- = 
    \left\{ 
    \frac{-1}{\sqrt{6}} \eps^{1/4} \right\}
    \quad{\rm and}
    \quad\sigma_-^u = \left\{ 
    \sqrt{\frac{2}{3}}\eps^{1/4},
    \sqrt{\frac{3}{2}}\eps^{1/4},
    \sqrt{\frac{3}{2}}\eps^{1/4},
2\sqrt{\frac{2}{3}}\eps^{1/4}, \sqrt{6}\eps^{1/4}
\right\}.
\end{equation}
At $E_-$, the stable and unstable subspaces are
\begin{equation}
\label{eq:Esu4E-}
\mathbb{E}^s = {\rm span} 
\left\{ 
\left[ \begin{array}{c}
               1 \\ 0 \\ 0 \\ 0 \\ 0 \\ 0 
             \end{array} \right]
             \right\}
             \quad{\rm and}
             \quad
\mathbb{E}^u = {\rm span} \left\{
\left[ \begin{array}{c}
               0 \\ -\eps^{1/4} \\ 0 \\ 0 \\ \sqrt{6} \\ 0 
             \end{array} \right],
\left[ \begin{array}{c}
               0 \\ 0 \\ 0 \\ 0 \\ 1 \\ 0 
             \end{array} \right],
\left[ \begin{array}{c}
               0 \\ 0 \\ 0 \\ 0 \\ 0 \\ 1 
             \end{array} \right],
\left[ \begin{array}{c} 0 \\ 0 \\
        \sqrt{3} \\ \sqrt{2}\eps^{1/4} \\ 0 \\ 0
             \end{array} \right],
\left[ \begin{array}{c} 0 \\ 0 \\
        1 \\ 0 \\ 0 \\ 0
             \end{array} \right]
\right\}.
\end{equation}
Thus, the blow-up transformation \eqref{eq:dynrescale-K1} splits the non-hyperbolic equilibrium state $(a,0,f(a),0)$ into the pair of hyperbolic fixed-points $E_{\pm}$. Moreover, by \eqref{eq:Esu4E+} $E_+$ has a five-dimensional stable manifold and a one-dimensional unstable manifold, with the latter being in the $r_1$ direction,
and by \eqref{eq:Esu4E-} $E_-$ has a one-dimensional stable manifold (in the $r_1$ direction) and a five-dimensional unstable manifold. 

Now, in addition to the equilibria $E_\pm$ and invariant line $I$, the full system \eqref{eq:K1} also has an important, three-dimensional manifold of equilibria:
\begin{equation}
\label{eq:mfldS1}
        \mathcal{S}_1 = \{ r_1 \in \mathbb{R},\,\, \delta_1=0, \,\, p_1=0,\,\, v_1=1 + \frac{1}{3}\sqrt{\eps} r_1^2,\,\, q_1 \in \mathbb{R},\,\, a_1 \in \mathbb{R} \}.
\end{equation}
It contains the line $\ell$, and it is a manifold of saddle fixed points, since the eigenvalues are
\begin{equation*}
        \lambda_s = -\sqrt{2 \sqrt{\eps} +  \eps r_1^2}, \quad
        \lambda_u = \sqrt{2 \sqrt{\eps} +  \eps r_1^2}, 
        \quad
        \lambda_c = 0, 0, 0, 0.
\end{equation*}
The associated stable and unstable eigenspaces are
\begin{equation}
\label{eq:EsEu}
\mathbb{E}^s = {\rm span} 
\left\{ 
\left[ \begin{array}{c}
               -\sqrt{\eps} r_1 \\ 0 \\ 2 \sqrt{2\sqrt{\eps} + \eps r_1^2} \\ \frac{4}{3}\sqrt{\eps}(3 + \sqrt{\eps}r_1^2) \\  3 \sqrt{\eps} q_1 \\  3\sqrt{\eps}a_1  
             \end{array} \right]
             \right\},
             \quad
\mathbb{E}^u = {\rm span} \left\{
\left[ \begin{array}{c}
-\sqrt{\eps} r_1 \\ 0 \\ - 2 \sqrt{2\sqrt{\eps} + \eps r_1^2} \\ \frac{4}{3}\sqrt{\eps}(3 + \sqrt{\eps}r_1^2) \\  3 \sqrt{\eps} q_1 \\  3\sqrt{\eps}a_1  
\end{array} \right] \right\},
\end{equation}
and the associated center subspace is 
\begin{equation}
\label{eq:Ec}
\mathbb{E}^c = {\rm span} \left\{ 
\left[ \begin{array}{c}
               0 \\ 0 \\ 0 \\ 0 \\ 1 \\ 0 
             \end{array} \right],
\left[ \begin{array}{c}
               0 \\ 0 \\ 0 \\ 0 \\ 0 \\ 1 
             \end{array} \right],
\left[ \begin{array}{c} 3 \\ 0 \\
        0 \\ 2\sqrt{\eps} r_1\\ 0 \\ 0
             \end{array} \right],
\left[ \begin{array}{c} 0 \\ 2 + \sqrt{\eps} r_1^2 \\
        q_1 \\ 0 \\ 0 \\ 0
             \end{array} \right]
\right\}.
\end{equation}
Therefore, application of  center manifold theory
(see for example \cite{C1981}) establishes that \eqref{eq:K1}
has a four-dimensional
center manifold
\begin{equation}
 \label{eq:M}
    M = W^c(\mathcal{\ell}).
\end{equation}
This manifold contains the surface $\mathcal{S}_1$ of equilibria.
It also contains the three-dimensional center manifold $N$ 
in the invariant hyperplane $ \{ r_1=0 \}$, see \eqref{eq:N} in Appendix~\ref{sec:appK1}.
Moreover, the branch of $N$ in $p_1<0$ and $\delta_1>0$ is unique.

The above analysis establishes the following lemma:
\begin{lemma}
The following properties hold for the system \eqref{eq:K1}.
  \begin{itemize}
  \item  There exists a stable invariant foliation, $\mathscr F^s$, with base $M$ and one-dimensional fibers. For any constant $c>-\sqrt 2$, the contraction along $\mathscr F^s$ during an interval $[0,X]$ is stronger than $e^{cX}$. 
  \item There exists an unstable invariant foliation, $\mathscr F^u$, with base $M$ and one-dimensional fibers. For any constant $c< \sqrt 2$, the expansion along $\mathscr F^u$ during an interval $[0,X]$ is stronger than $e^{cX}$. 
  \end{itemize}
\end{lemma}

%--------------------------------------------------------------------------
\subsection{Forward and backward asymptotics of \texorpdfstring{$\Gamma_0$}{Lg} in chart \texorpdfstring{$K_1$}{Lg}}
\label{sec:Gamm0-intoK1}
%----------------------------------------------------------

In this section, we study the forward and backward asymptotics (as $y_2 \to \pm \infty$) of the ``Vo connection" $\Gamma_0$, given by \eqref{eq:Gamma0}, along with its tangent vectors. 

We start with the coordinate transformation maps between the two charts. 
For $\delta_1>0$, the coordinate transformation from $K_1$ to $K_2$ is
\begin{equation}
        \begin{split}
\kappa_{12} (r_1,\delta_1,p_1,v_1,q_1,a_1)
          &= (r_2,u_2,p_2,v_2,q_2,a_2) 
           = \left(
               r_1 \sqrt{\delta_1}, \frac{1}{\delta_1},
               \frac{p_1}{\delta_1^{3/2}}, \frac{v_1}{\delta_1^2},
               \frac{q_1}{\delta_1^{3/2}},
               \frac{a_1}{\delta_1^{3/2}}
            \right).
        \end{split}
\end{equation}
For $u_2>0$,
the coordinate transformation from $K_2$ to $K_1$ is
\begin{equation}
        \begin{split}
\kappa_{21} (r_2,u_2,p_2,v_2,q_2,a_2)
        &= (r_1,\delta_1,p_1,v_1,q_1,a_1) 
         = \left(
             r_2 \sqrt{u_2}, \frac{1}{u_2}, \frac{p_2}{u_2^{3/2}},
             \frac{v_2}{u_2^2}, \frac{q_2}{u_2^{3/2}},
             \frac{a_2}{u_2^{3/2}}
           \right).
        \end{split}
\end{equation}

Now, for the forward asymptotics of $\Gamma_0$, we use $\kappa_{21}$ to find 
\begin{equation}
        \lim_{y_2 \to \infty} \kappa_{21}(\Gamma_0)
        = \lim_{y_2 \to \infty}
          \left(
          0, \frac{12}{\sqrt{\eps} y_2^2}, \frac{4\sqrt{3}}{\eps^{3/4} y_2^2},
          1-\frac{24}{\eps y_2^4}, \frac{2\sqrt{3}}{3\eps^{1/4}}, 0
          \right)
        = \left(
          0, 0, 0, 1,
          \frac{2\sqrt{3}}{3\eps^{1/4}}, 0
          \right).
\end{equation}
Hence,
on the invariant hyperplane $\{ r_1=0 \}$,
orbits on $\Gamma_0$
are forward asymptotic
to the fixed point
$(0,0,0,1,q^+,0)$
with $q^+=\frac{2\sqrt{3}}{3\eps^{1/4}}$,
on the line of equilibria $\ell$.

Similarly, the backward asymptotics of $\Gamma_0$
are given by:
\begin{equation}
        \lim_{y_2 \to -\infty} \kappa_{21}(\Gamma_0)
        = \lim_{y_2 \to -\infty}
          \left(
          0, \frac{12}{\sqrt{\eps} y_2^2}, \frac{-4\sqrt{3}}{\eps^{3/4} y_2^2},
          1-\frac{24}{\eps y_2^4}, \frac{-2\sqrt{3}}{3\eps^{1/4}}, 0 
          \right)
        = \left(
          0, 0, 0, 1,
          \frac{-2\sqrt{3}}{3\eps^{1/4}}, 0 
          \right).
\end{equation}
Hence, on the invariant hyperplane
$\{r_1=0\}$, solutions on $\Gamma_0$
are backward asymptotic to the fixed point
$(0,0,0,1,q^-,0)$ with
$q^-=\frac{-2\sqrt{3}}{3\eps^{1/4}}$,
on the line of equilibria $\ell$.

\medskip

The tangent vectors of $\Gamma_0$ in chart $K_1$ at the equilibria $(0,0,0,1,q^{\pm},0)$ are given by
\[ \lim_{y_2 \to \pm \infty} \frac{\frac{d}{dy_2}\kappa_{21}(\Gamma_0)}{\lVert \frac{d}{dy_2}\kappa_{21}(\Gamma_0) \rVert} 
= \lim_{y_2 \to \pm \infty} \frac{\left(0,\mp 24 \sqrt{\eps} y_2^2,-8\sqrt{3} \eps^{1/4} y_2^2,\pm 96,0 ,0\right)}{8\sqrt{3} \sqrt{48+\sqrt{\eps}y_2^4(1+3\sqrt{\eps})}} 
= \left( 0, \mp \frac{\sqrt{3}\eps^{1/4}}{\sqrt{1+3\sqrt{\eps}}},\frac{-1}{\sqrt{1+3\sqrt{\eps}}}, 0, 0, 0\right).
\]
Hence, the tangent vectors of $\Gamma_0$ in chart $K_1$ at the equilibria $(0,0,0,1,q^{\pm},0)$ are tangent to the center manifolds $M$. In fact, at these equilibrium points, the tangent vectors are parallel to the generalized eigenvector of the center subspace.

%--------------------------------------------------------------------------
\subsection{Intersection of center-unstable and center-stable manifolds at \texorpdfstring{$a_c(\delta)$}{Lg}}
\label{sec:WcuWcs}
%--------------------------------------------------------------------------

In this section, we present the derivation of \eqref{eq:ac}: 

\begin{lemma}
\label{lem:persistence}
   For $0 < \delta \ll 1$, the center-unstable manifold $W^{cu}(0,0,0,1,q^-,0)$ coincides with the center-stable manifold $W^{cs}(0,0,0,1,q^+,0)$ for 
   \begin{equation}
   \label{eq:lemma1-exp-a}
       a_c(\delta) = 1 - \frac{5 \eps}{48} \delta^2 + \mathcal{O}(\delta^3).
   \end{equation}
\end{lemma}

This lemma is proven in Appendix~\ref{sec:appGammadelta}. 
It involves a Melnikov type calculation.

An immediate consequence of the lemma is that, for all $\delta>0$ sufficiently small, there is a maximal canard orbit exactly at $a_c(\delta)=1 - \frac{5 \eps}{48}\delta^2 + \mathcal{O}(\delta^3)$.
Moreover, the orbit is unique up to translations since  the ODEs are autonomous. 
Of all the canard orbits in the spatial ODE system \eqref{eq:spatialODE-y} that pass through a neighborhood of the cusp point, the maximal canard has the longest segments near the stable and unstable manifolds of the cusp point, all the way out to where $u=+\sqrt{3}$.
This critical value is the analog in spatial dynamics of the critical parameter value at which the slow attracting and repelling manifolds coincide to all orders, and hence the explosion of temporal limit cycle canards occurs, in planar fast-slow ODEs.

%--------------------------------------------------------------------------
\subsection{Calculation of the center-unstable and center-stable manifolds for \texorpdfstring{ $0<\delta\ll 1$}{Lg}}
\label{sec:FraserRoussel}
%--------------------------------------------------------------------------
In this section, we establish the following proposition: 
\begin{proposition} \label{prop:548}
The branches of the critical manifold corresponding to the level set $\mathcal G = \eps \left(\tfrac{2}{3}a-\tfrac{1}{4} \right)$ perturb to invariant slow manifolds. These slow manifolds are smooth to at least $\mathcal O(\delta^2)$ in $p$ and smooth to at least $\mathcal O(\delta^3)$ in $q$, for 
\begin{equation}
    \label{eq:prop548}
a_c(\delta) = 1 + a_2 \delta^2 + \mathcal O(\delta^3).
\end{equation}
for any real $a_2$.
\end{proposition}

The proof uses the iterative scheme devised in \cite{Roussel1990} to calculate the invariant slow manifolds of \eqref{eq:spatialODE-x}, and determine the key parameter values for which the slow manifolds intersect. 
To do this, we follow the approach in \cite{B2013,ZS1984} in which the parameters are carefully chosen to remove poles in the expansions of the slow manifolds. 

\bigskip
{\bf Proof of Proposition~\ref{prop:548}.} 
We first translate the reversible FSN-II to the origin via the linear shift
\[ \tilde{u} = u-1, \quad \tilde{p} = p, \quad \tilde{v}=v+\tfrac{2}{3}, \quad \tilde{q} = q, \quad \tilde{a} = a-1. \]
Dropping the tildes, the dynamics are governed by 
\begin{equation} \label{eq:FR2fast2slow}
\begin{split}
    \dot u &= p \\ 
    \dot p &= f(u) - v \\ 
    \dot v &= \delta q \\ 
    \dot q &= \eps \delta (u-a)
\end{split}
\end{equation}
where $f(u) = u^2+\tfrac{1}{3}u^3$. In this formulation, the conserved quantity is 
\[ \widetilde{\mathcal{G}} = \eps(u-a)v - \eps \widetilde{f}(u) + \tfrac{1}{2}\eps p^2 - \tfrac{1}{2}q^2, \]
where $\widetilde{f}(u) = \tfrac{1}{3}u^3+\tfrac{1}{12}u^4$ is an antiderivative of $f$, and the level set $\mathcal{G} = \eps\left(\tfrac{2}{3}a-\tfrac{1}{4} \right)$ becomes $\widetilde{\mathcal{G}} = 0$. By eliminating the $q$-variable, we reduce the study of the 2-fast/2-slow system \eqref{eq:FR2fast2slow} with constraint $\widetilde{\mathcal{G}}=0$ to the study of the 2-fast/1-slow system 
\begin{equation} \label{eq:FR2fast1slow}
\begin{split}
    \dot u &= p \\ 
    \dot p &= f(u) - v \\ 
    \dot v &= \sqrt{2\eps}\, \delta\, \sqrt{(u-a)v-\widetilde{f}(u)+p^2}.
\end{split}
\end{equation}
Here, we focus on the positive root for $q$, and then obtain results for the negative root for $q$ by using the symmetry \eqref{eq:RF}. 

For the 2-fast/1-slow system \eqref{eq:FR2fast1slow}, the critical manifold is given by 
\[ S_0 = \left\{ p = 0 , v = f(u) \right\}  \]
with fold points at $(u,p,v)=(0,0,0)$ and $(u,p,v)=(-2,0,\tfrac{4}{3})$, and canard points at $(u,p,v,a)=(0,0,0,0)$ and $(u,p,v,a)=(-2,0,\tfrac{4}{3},-2)$. The canard point at the origin corresponds to the RFSN-II point of interest. 

We assume that the invariant slow manifold, $S_{\delta}$, has a graph representation in $u$, and that it can be expanded as a power series in $\delta$,
\[ S_{\delta} = \left\{ (u,p,v,q) : p = \sum_{k=0}^{\infty} p_k(u) \delta^k, \,\, v = \sum_{k=0}^{\infty} v_k(u) \delta^k, \,\, q = \sum_{k=0}^{\infty} q_k(u) \delta^k \right\} \]
We also expand the parameter $a$ as a power series
\[ a = \sum_{k=0}^{\infty} a_k \delta^k. \]
Substituting these expansions into \eqref{eq:FR2fast1slow} and equating coefficients of like powers of $\delta$, we find that the leading $\mathcal O(1)$ terms are given by 
\[ p_0 = 0, \quad v_0 = f(u), \quad  \text{ and } \quad q_0 = \sqrt{2\eps} \sqrt{(u-a_0)f(u)-\widetilde{f}(u)}, \]
corresponding to the critical manifold. Proceeding to the next terms in the expansion, we find that the $\mathcal{O}(\delta)$ terms are given by 
\[ p_1 = \frac{\sqrt{2\eps} \sqrt{(u-a_0)f(u)-\widetilde{f}(u)}}{f^\prime(u)}, \quad v_1 = 0, \quad \text{ and } \quad q_1 = 0. \]
To remove the singularity at $u=0$ in $p_1$, we set $a_0 = 0$ so that every term in the argument of the square root has a factor of $u^2$. In that case, the coefficients become
\[ p_1 = \frac{\sqrt{2\eps} \sqrt{u \left( \tfrac{1}{4}u + \tfrac{2}{3} \right)}}{2+u}, \quad v_1 = 0, \quad \text{ and } \quad q_1 = 0. \]

Continuing to higher order, we find that the $\mathcal{O}(\delta^2)$ terms are given by 
\[ p_2 = -\frac{a_1 \sqrt{\tfrac{2}{3}\eps} (3+u)}{\sqrt{u} (2+u) \sqrt{8+3u}}, \quad v_2 = -\frac{\eps(4+u)}{3(2+u)^3}, \quad \text{ and } \quad q_2 = -\sqrt{\eps u} \,\frac{4a_2(2+u)^3(3+u)-\eps(10+3u)}{2\sqrt{6} (2+u)^3 \sqrt{8+3u}}. \]
To eliminate the singularity at $u=0$ in $p_2$, we set $a_1 = 0$, so that 
\[ p_2 = 0, \quad v_2 = -\frac{\eps(4+u)}{3(2+u)^3}, \quad \text{ and } \quad q_2 = -\sqrt{\eps u} \,\frac{4a_2(2+u)^3(3+u)-\eps(10+3u)}{2\sqrt{6} (2+u)^3 \sqrt{8+3u}}. \]

The $\mathcal{O}(\delta^3)$ terms are given by 
\[ p_3 = \frac{\sqrt{\eps} \left( -12a_2(2+u)^5(3+u)+\eps(9u^3+54u^2+64u-40) \right)}{6\sqrt{6} (2+u)^6 \sqrt{u}\sqrt{8+3u}}, \quad v_3 = 0, \quad \text{ and } \quad q_3 = 0.  \]
Here, there is a singularity at $u=0$ in $p_3$, however, we choose not to enforce smoothness of the $p$-component of the slow manifold and leave $a_2$ free. 

Therefore, the proposition is proven, since $p$ is smooth up to and including  $\mathcal{O}(\delta^2)$ and $q$ up to and including $\mathcal{O}(\delta^3)$ for any choice of $a_2$. 

We add that the $\mathcal{O}(\delta^4)$ terms are given by 
\[ p_4 = -\sqrt{\frac{2\eps}{3}} \frac{a_3 (3+u)}{\sqrt{u}(2+u)\sqrt{8+3u}}, \quad v_4 = \eps \frac{P_6(u)}{3(2+u)^8}, \quad \text{ and } \quad q_4 = -\frac{\sqrt{\eps} P_{12}(u)}{24 \sqrt{6}
   (u+2)^8 u^{3/2} (8+3u)^{3/2}}.  \]
where $P_6(u)$ and $P_{12}(u)$ are, respectively, the $6^{\rm th}$ and $12^{\rm th}$ order polynomials
\begin{equation*}
\begin{split}
P_6(u) &= -a_2(2+u)^5(4+u)+\eps(3u^3+18u^2+14u-34) \\ 
P_{12}(u) &= 24 \eps a_2 \left(9 u^3+61 u^2+110 u+32\right) (u+2)^5 + 48 a_2^2 u (u+3)^2 (u+2)^8 \,\, + \\ 
& \qquad 48 a_4
   u^2 \left(3 u^2+17 u+24\right) (u+2)^8- \eps^2 \left(513 u^5+4608 u^4+12168 u^3+5120 u^2-12400 u-2560\right)
\end{split}
\end{equation*}
To cancel one factor of $u$ in the denominator of $q_4$, we must set $a_2 = -\frac{5\eps}{48}$. 
Thus, upon reverting to the original coordinates and parameters, Fraser-Roussel iteration shows that in order for the slow manifolds to be smooth to at least $\mathcal O(\delta^2)$ in $p$ and smooth to at least $\mathcal O(\delta^3)$ in $q$, the parameters must satisfy
\[ a = 1 - \frac{5\eps}{48}\delta^2 + \mathcal O(\delta^3).\]
consistent with the values obtained from the Melnikov analysis, recall \eqref{eq:ac}. \hfill $\square$
\medskip

\begin{remark}
On the level set $\mathcal G = \eps \left(\tfrac{2}{3}a-\tfrac{1}{4} \right)$, one can alternatively get smoothness to at least $\mathcal O(\delta^3)$ in $p$ and to at least $\mathcal O(\delta^3)$ in $q$ for 
\begin{equation}
\label{eq:ac2} 
a_{c2}(\delta) = 1 - \frac{5\eps}{144} \delta^2 + \mathcal O \left( \delta^3 \right). 
\end{equation}
This follows from the calculations in the proof of Proposition~\ref{prop:548}. At $\mathcal O(\delta^3)$, the pole at $u=0$ in $p_3$ is removed by seeking a value for $a_2$ such that the numerator in the expression for $p_3$ has $u$ as a factor, namely $a_2 = -\tfrac{5\eps}{144}$. The lack of smoothness in the singular limit is visible in the $(u,p)$ and $(u,q)$ projections for small $\delta)$, with corners and cusps (in the projections) developing. There appears to be a discrete set of critical values of $a$ marking the transitions between solutions with different numbers of loops, and different numbers of points of nonsmoothness in the limit of small $\delta$. \end{remark}

%-----------------------------------------------------------
\section{Analyzing the Geometry of the Spatial Canards} \label{sec:Turingspatialcanards}
%-----------------------------------------------------------

In this section, we use the analytical results from  Sections~\ref{sec:fast}--\ref{sec:desingFSNII},  for the analyses of the fast system, the desingularized reduced system, and the desingularization of the RFSN-II singularity, respectively, to deconstruct and understand the spatially periodic canard solutions.
This analysis reveals how the RFSN-II point at $a=a_T$  and its canards, as identified in Section~\ref{sec:slow}, are responsible for the creation of these spatial canard solutions.  
The geometry of small-amplitude spatial canards is analyzed in Section~\ref{subsec:Turingspatialcanards}, and that of  large-amplitude spatial canards in Section~\ref{subsec:singularLAO}.

%-----------------------------------------------------------
\subsection{Geometry of the Small-Amplitude Spatial Canards} \label{subsec:Turingspatialcanards}
%-----------------------------------------------------------
We present the geometry of small-amplitude spatially periodic canards.
For illustration, we use the solution shown in Fig.~\ref{fig:deconstructSAO}, and we recall also the representative small-amplitude canards in Figs.~\ref{fig:kOrderOne_SAO} and \ref{fig:smallamp-selfsimilar}. 

\begin{figure}[h!]
  \centering
  \includegraphics[width=\textwidth]{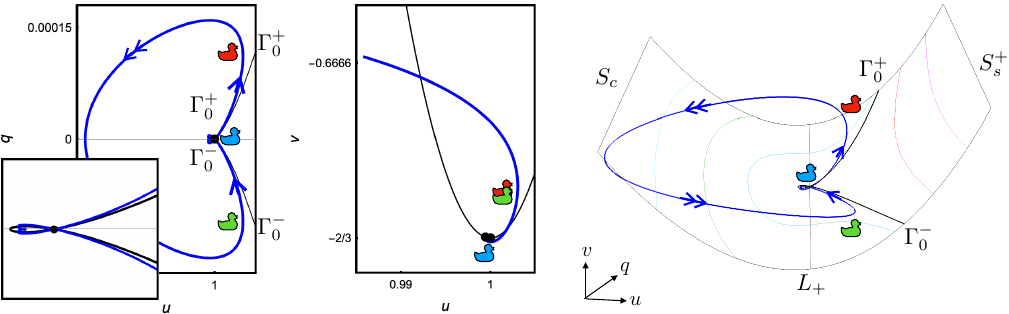}
  \put(-470,136){(a)}
  \put(-336,136){(b)}
  \put(-204,136){(c)}
  \caption{Geometric deconstruction of a small-amplitude, canard solution with spatial period $T=500$ and wavenumber $k=\frac{2\pi}{\delta T}=\frac{2\pi}{5} \approx 1.257$, for $a=0.999433, \eps = 0.1$, and $\delta = 0.01$. This is the same solution shown in Fig.~\ref{fig:introspatialperiodic}. 
  Three complementary views: (a) projection onto the $(u,q)$ plane, showing the true canard $\Gamma_0^+$ and faux canard $\Gamma_0^-$ of the folded saddle; (b) projection onto the $(u,v)$ plane; and (c) projection into the $(u,v,q)$ space. 
  The colored contours on $S$ correspond to different values of the conserved quantity $\mathcal{G}$. The black contours correspond to the level set $\{ \mathcal{G}=\eps \left(\tfrac{2}{3}a-\tfrac{1}{4} \right) \}$ and they contain the true and faux canards, $\Gamma_0^{\pm}$. 
  Note that, while the parameters are the same here as in Fig.~\ref{fig:smallamp-selfsimilar}, the wavenumbers are vastly different. Here, the solution only makes two nested, small loops about the folded saddle (as shown in the inset in (a)), and hence the wavenumber is still $\mathcal{O}(1)$. In contrast, the solution in Fig.~\ref{fig:smallamp-selfsimilar} makes many nested, successively smaller loops and has an asymptotically small wavenumber.
  }
  \label{fig:deconstructSAO}
\end{figure}

In the singular limit, the small-amplitude orbit consists of the following three segments, all confined to the level set $\left\{ \mathcal{G} = \eps \left( \tfrac{2}{3}a-\tfrac{1}{4} \right) \right\}$: 
\begin{itemize}
\item A short slow segment from the folded saddle at $(u,p,v,q) = (1,0,-\tfrac{2}{3},0)$  (blue ducky) that closely follows the faux canard, $\Gamma_0^+$, of the folded saddle up to some nearby $v$-value (near the red ducky).

\item A fast segment following a homoclinic solution (see Proposition~\ref{prop:fasthomoclinics}) of the layer problem until it returns to a neighborhood of $S_s^+$ (near the green ducky). 

\item A short slow segment that closely follows the true canard, $\Gamma_0^-$, of the folded saddle until it returns to the folded saddle. 
\end{itemize}
In this manner, the singular solutions  of small-amplitude, spatially periodic canards with $\mathcal{O}(1)$ wavenumbers are completely understood through the dynamics of the RFSN-II singularity, its true and faux canards, and some invariant manifold theory, following the results of the analysis in Sections~\ref{sec:fast}--\ref{sec:desingFSNII} .
The geometry of other small-amplitude spatially periodic canard solutions with $\mathcal{O}(1)$ wavenumbers, such as that shown in Fig.~\ref{fig:kOrderOne_SAO}, is similar. 

The geometry of small-amplitude canards with asymptotically small wavenumbers may also be understood in a similar manner. 
These are very close in profile to the small-amplitude canards with $\mathcal{O}(1)$ wavenumbers over most of the period, but they also exhibit nested, successively smaller, nearly  self-similar twists (loops in certain projections) that are confined to the neighborhood of the equilibrium and the folded saddle
(recall for example the solution shown in Fig.~\ref{fig:smallamp-selfsimilar}).
These nearly self-similar loops arise for small $\delta$ due to scale invariance of the zero level of the Hamiltonian $H_2$ in the rescaling  chart $K_2$.

%-----------------------------------------------------------
\subsection{Geometry of the Large-Amplitude Spatial Canards}
\label{subsec:singularLAO}
%-----------------------------------------------------------

In this section, we use the analytical results 
derived in Sections~\ref{sec:fast}--\ref{sec:desingFSNII}, 
about the fast system, the desingularized reduced system, and the blowup analysis, to fully deconstruct the large-amplitude, spatially periodic canard solutions of \eqref{eq:spatialODE-x}. 
We choose the solution from Fig.~\ref{fig:kOrderDelta_LAO} for the deconstruction, and we review some of the main properties of that solution again in  Fig.~\ref{fig:deconstructLAO}.

\begin{figure}[h!]
  \centering
  \includegraphics[width=\textwidth]{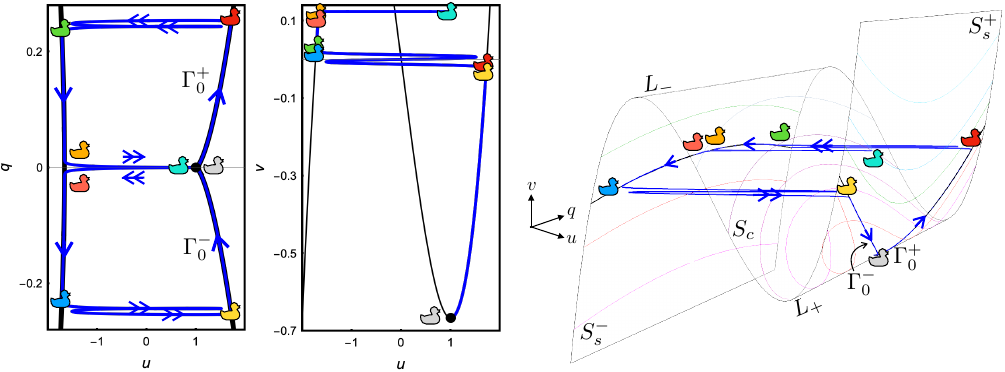}
  \put(-468,162){(a)}
  \put(-354,162){(b)}
  \put(-232,162){(c)}
  \caption{Geometric deconstruction of the large-amplitude, small-wavenumber solution from Fig.~\ref{fig:kOrderDelta_LAO}. Three complementary views: (a) projection onto the $(u,q)$ plane, showing the true canard $\Gamma_0^+$ and faux canard $\Gamma_0^-$ of the folded saddle; (b) projection onto the $(u,v)$ plane; and (c) projection onto the $(u,v,q)$ space. The colored contours on $S$ correspond to different values of the conserved quantity $\mathcal{G}$. The black contours correspond to the level set $\{ \mathcal{G}=\eps \left(\tfrac{2}{3}a-\tfrac{1}{4} \right) \}$ and they contain the true and faux canards, $\Gamma_0^{\pm}$.}
  \label{fig:deconstructLAO}
\end{figure}

The singular limit solution is obtained by concatenating orbit segments from the fast and slow subsystems, \eqref{eq:H-fast} and \eqref{eq:desing-reduced}. 
The singular orbit consists of seven segments, all of which are constrained to the level set $\left\{ \mathcal{G} = \eps \left( \tfrac{2}{3}a-\tfrac{1}{4} \right) \right\}$. 
For convenience, we denote the level set $\left\{ \mathcal{G} = \eps \left( \tfrac{2}{3}a-\tfrac{1}{4} \right) \right\}$ restricted to orbits of the layer problem by $\mathcal{G}_{\rm fast}$ and we denote the level set $\left\{ \mathcal{G} = \eps \left( \tfrac{2}{3}a-\tfrac{1}{4} \right) \right\}$ restricted to the slow flow on $S$ by $\mathcal{G}_{\rm slow}$.

\begin{itemize}
\item Slow segment on $S_s^+$ along the faux canard $\Gamma_0^+$ from the folded saddle at $(u,p,v,q)=(1,0,-\tfrac{2}{3},0)$ (grey ducky) to the upper right corner point (red ducky) at
\[ (u,p,v,q)=\left( \sqrt{3},0,0,\sqrt{\tfrac{2}{3}\eps(3-2a)} \right). \]
(For $a = 0.998512$, the upper right corner point is located at $(u,p,v,q) \approx (1.7321,0,0,0.2586)$.) 

\item Fast jump along the heteroclinic orbit in the transverse intersection $W^u(S_s^+) \cap W^s(S_s^-)$ in $\left\{ v = 0\right\}$ (see Proposition~\ref{prop:jumpcondition}) from the take off point at the upp[er right corner (red ducky) to the touch down point at the upper left corner (green ducky) at 
\[ (u,p,v,q)=\left( -\sqrt{3},0,0,\sqrt{\tfrac{2}{3}\eps(3-2a)} \right). \] This heteroclinic consists of three segments that occur in rapid succession, one from a neighborhood of $S_s^+$ to near $S_s^-$, followed by another back to the neighborhoof of $S_s^+$, and then the third is forward asymptotic to $S_s^-$; it lies in a secondary intersection of the two invariant manifolds.  
(For $a = 0.998512$, this corner point is located at $(u,p,v,q) \approx (-1.7321,0,0,0.2586)$.)
\item Slow segment on $S_s^-$ from the touch down point (green ducky) to the local maximum (in $u$) of the $\mathcal{G}_{\rm slow}$ contour (gold ducky) at 
\[ (u,p,v,q) = \left( u_-,0,f(u_-),0 \right), \] 
where $u_- = \tfrac{1}{3} \left( 2a-3-2\sqrt{a^2+3a} \right)$. (For $a = 0.998512$, this gold ducky is located at $(u,p,v,q) \approx (-1.6664,0,0.1239,0)$.)  
\item Fast jump consisting of the homoclinic orbit of the layer problem (see Proposition~\ref{prop:fasthomoclinics}) from the take off point (gold ducky) to the local maximum of the homoclinic (cyan ducky) 
and then to the touch down point on $S_s^-$ at $(u,p,v,q) = \left( u_-,0,f(u_-),0 \right)$ (magenta ducky).  
\item Slow segment on $S_s^-$ from the touch down point (magenta ducky) until the solution reaches the $\left\{ v = 0 \right\}$ hyperplane at 
$ (u,p,v,q)=\left(-\sqrt{3},0,0,-\sqrt{\tfrac{2}{3}\eps(3-2a)} \right) $
(blue ducky).
\item Fast jump in the transverse intersection $W^u(S_s^-) \cap W^s(S_s^+)$ in $\left\{ v = 0 \right\}$ (see Proposition~\ref{prop:jumpcondition}) from the take off point (blue ducky) to the touch down point at the lower right corner (yellow ducky) on $S_s^+$ at 
\[ (u,p,v,q) = \left(\sqrt{3},0,0,\sqrt{\tfrac{2}{3}\eps(3-2a)} \right) \]
This heteroclinic consists of three segments in rapid succession and lies in a secondary intersection of the manifolds.
(For $a = 0.998512$, the lower right corner point is located at $(u,p,v,q) \approx (1.7321,0,0,-0.2586)$.)
\item Slow segment on $S_s^+$ along the true canard, $\Gamma_0^-$, of the RFSN-II point
from the touch down point (yellow ducky) up to the cusp  (grey ducky), thus completing the singular cycle. 
\end{itemize}
%By the theory of canard-induced mixed-mode oscillations \cite{Brons2006}, this singular periodic canard solution persists for $0< \delta \ll 1$ as a nearby spatially periodic solution (blue orbit in Fig.~\ref{fig:deconstructLAO}). 

\bigskip

\begin{remark}
For the solution shown in Fig.~\ref{fig:deconstructLAO}, the connection from $S_s^+$ to $S_s^-$ (and vice versa) is a secondary heteroclinic and features additional large-amplitude spikes not present in  the singular limit. 
We observed other large-amplitude, small wavenumber solutions without additional spikes that  correspond to primary heteroclinics. 
The deconstruction of such solutions is similar to that presented above. 
In addition, we observed solutions that have more of these spikes along the outer edges of the orbit.   \end{remark}

%----------------------------------------------------------
\section{Isolas of Canard-Induced Spatial Periodics}
 \label{sec:isolas}
%-----------------------------------------------

In this section, we employ many of the results derived in Sections~\ref{sec:fast}-\ref{sec:desingFSNII}, about the fast/layer problem, the desingularized reduced vector field, and the geometry of the RFSN-II singularity, to study the bifurcations of the spatially periodic canard solutions along a representative isola in the $(a,k)$ parameter plane (see Fig.~\ref{fig:isola}). 
The branches of the isola have been color coded (green, orange, blue, and red) according to the nature of the spatially periodic solutions that lie on each segment.

\begin{figure}[h!]
  \centering
  \includegraphics[width=5in]{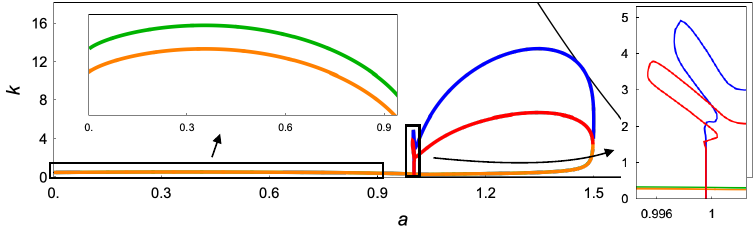}
  \caption{Representative isola of spatially periodic solutions shown in the $(a,k)$ parameter plane, along the $\left\{ \mathcal{G} = \eps \left(\tfrac{2}{3}a-\tfrac{1}{4}\right)  \right\}$ level set. Left inset: zoom on the segments near $k = 0$. Right inset: zoom on the region where the isola emanates/terminates.}
  \label{fig:isola}
\end{figure}

%----------------------------------------------------------
\subsection{Growth of canard segments} \label{subsec:growslow}
%----------------------------------------------------------
We analyze the growth of the canard segments for solutions on the green branch of the isola (see Fig.~\ref{fig:transition_growslow}(a)). 
These solutions are four-stroke, spatial relaxation cycles. 

\begin{figure}[h!]
  \centering
  \includegraphics[width=\textwidth]{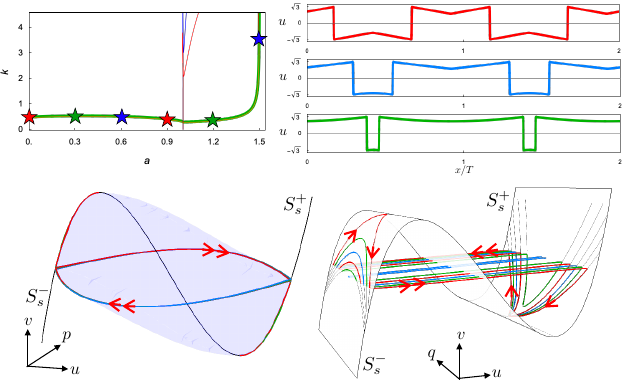}
  \put(-468,284){(a)}
  \put(-264,284){(b)}
  \put(-468,145){(c)}
  \put(-220,145){(d)}
  \caption{Four-stroke, spatial relaxation cycles along the green branch  of the isola.
  The closer to $a=0$ along this branch, the closer the slow segment on $S_s^-$ with $q>0$ is to the true canard of the folded saddle at $(u,p,v,q)=(-1,0,\tfrac{2}{3},0)$, and similarly the closer the slow segment on $S_s^-$ with $q<0$ is to the faux canard of the folded saddle.
  (a) Branch of the isola along which the length of the canard segments grow. 
  (b) Spatial profiles of $u(\hat{x})$ with $a=0, 0.6, 1.2$.
  (c) Projection into the $(u,p,v)$ space showing that the fast jumps are mediated by the fast subsystem heteroclinic. 
  (d) Projection into the $(u,q,v)$ space showing that the slow segments on $S_s^{\pm}$ are constrained to the $\left\{ \mathcal{G} = \eps \left( \tfrac{2}{3}a-\tfrac{1}{4} \right) \right\}$ level sets (thin, black curves). 
  }
  \label{fig:transition_growslow}
\end{figure}
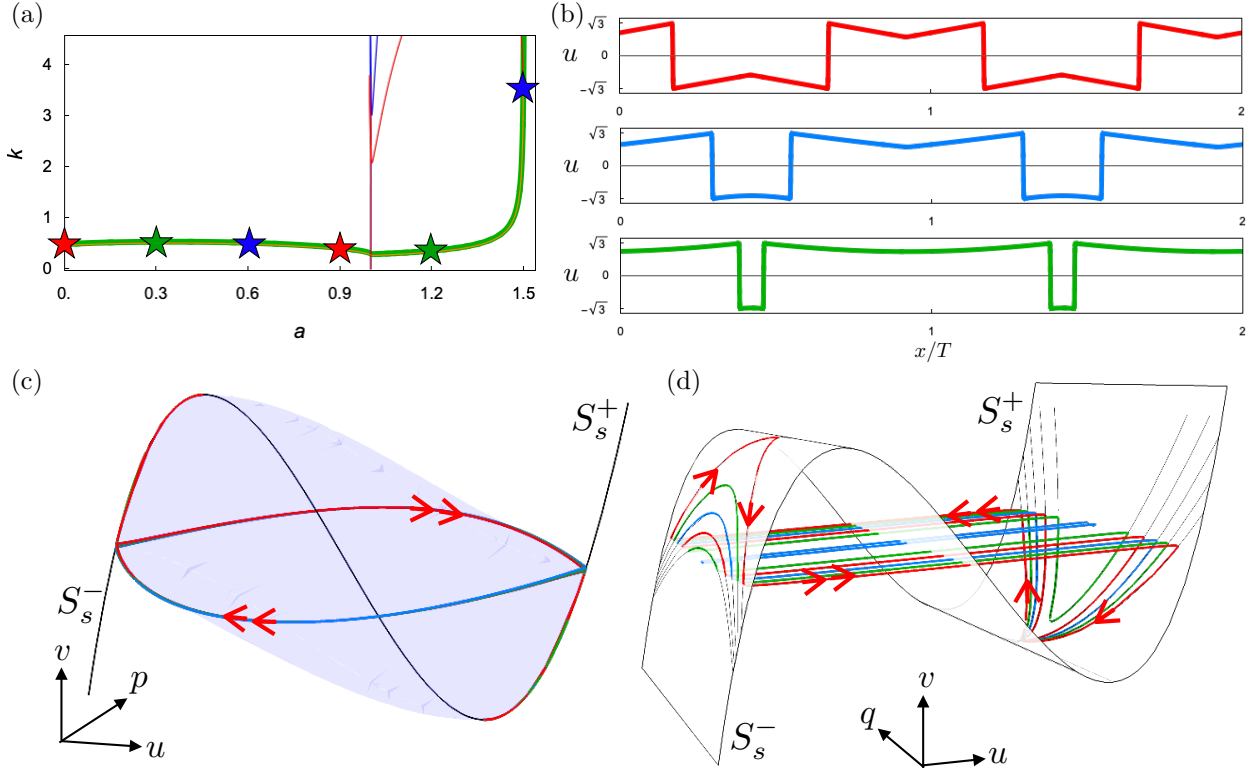

As illustrated in Fig.~\ref{fig:transition_growslow}(c) and (d), in the singular limit, the relaxation cycles consist of:  
\begin{itemize}
\item A slow segment on $S_s^+$ that starts at the touchdown point $(u,p,v,q)=(\sqrt{3},0,0,-\sqrt{2\eps(1-\tfrac{2}{3}a)})$, flows down the critical manifold toward the fold set $L^+$ to the turning point (local minimum) at $(u,p,v,q) = (u_{\rm turn},0,\tfrac{1}{3}u_{\rm turn}^3-u_{\rm turn},0)$ where $u_{\rm turn} = \tfrac{2}{3}a-1+\tfrac{2}{3}\sqrt{a^2+3a}$, and then flows back up the critical manifold to the takeoff point at $(u,p,v,q)=(\sqrt{3},0,0,\sqrt{2\eps(1-\tfrac{2}{3}a)})$.
\item A fast jump at $v = 0$ corresponding to the heteroclinic of the layer problem that takes the orbit from the takeoff point at $(u,p,v,q)=(\sqrt{3},0,0,\sqrt{2\eps(1-\tfrac{2}{3}a)})$ to the touchdown point at $(u,p,v,q)=(-\sqrt{3},0,0,\sqrt{2\eps(1-\tfrac{2}{3}a)})$. 
\item A slow segment that flows along $S_s^-$ from the touchdown point at $(u,p,v,q)=(-\sqrt{3},0,0,\sqrt{2\eps(1-\tfrac{2}{3}a)})$ up to the turning point (local maximum) at $(u,p,v,q) = (u_{\rm turn},0,\tfrac{1}{3}u_{\rm turn}^3-u_{\rm turn},0)$ where $u_{\rm turn} = \tfrac{2}{3}a-1-\tfrac{2}{3}\sqrt{a^2+3a}$, and then down to the takeoff point at $(u,p,v,q)=(-\sqrt{3},0,0,-\sqrt{2\eps(1-\tfrac{2}{3}a)})$. 
\item Another fast jump at $v = 0$ corresponding to the heteroclinic of the layer problem that takes the orbit from the takeoff point at $(u,p,v,q)=(-\sqrt{3},0,0,-\sqrt{2\eps(1-\tfrac{2}{3}a)})$ to the touchdown point at $(u,p,v,q)=(\sqrt{3},0,0,-\sqrt{2\eps(1-\tfrac{2}{3}a)})$.
\end{itemize}

At the right edge of this branch, there is a    
saddle-node bifurcation near $a \approx 1.5$ (Fig.~\ref{fig:transition_growslow}(a)), which marks the transition to the blue branch.
As the parameter $a$ is decreased toward $a=1$, the slow segment on $S_s^+$ with $q<0$ converges to the true canard of the folded saddle-node, and the slow segment on $S_s^+$ with $q>0$ converges to the faux canard of the folded saddle-node. 
Then, as $a$ is further decreased along this branch to $a=0$, the slow segment on $S_s^-$ with $q>0$ grows toward the true canard of the folded saddle at $(u,p,v,q)=(-1,0,\tfrac{2}{3},0)$, and the slow segment on $S_s^-$ with $q<0$ approaches the faux canard of the folded saddle.

%----------------------------------------------------------
\subsection{Spike formation} \label{subsec:spikeformation}
%----------------------------------------------------------
In this section, we analyze the sequence of canards that exhibit spike formation along the orange branch of the isola (see  Fig.~\ref{fig:transitionspikeformation} (a)). 
The solutions here consist of large-amplitude cycles with spikes initiated on $S_s^-$. The geometry of these solutions is similar to that of the solution presented in Section~\ref{subsec:singularLAO}.  

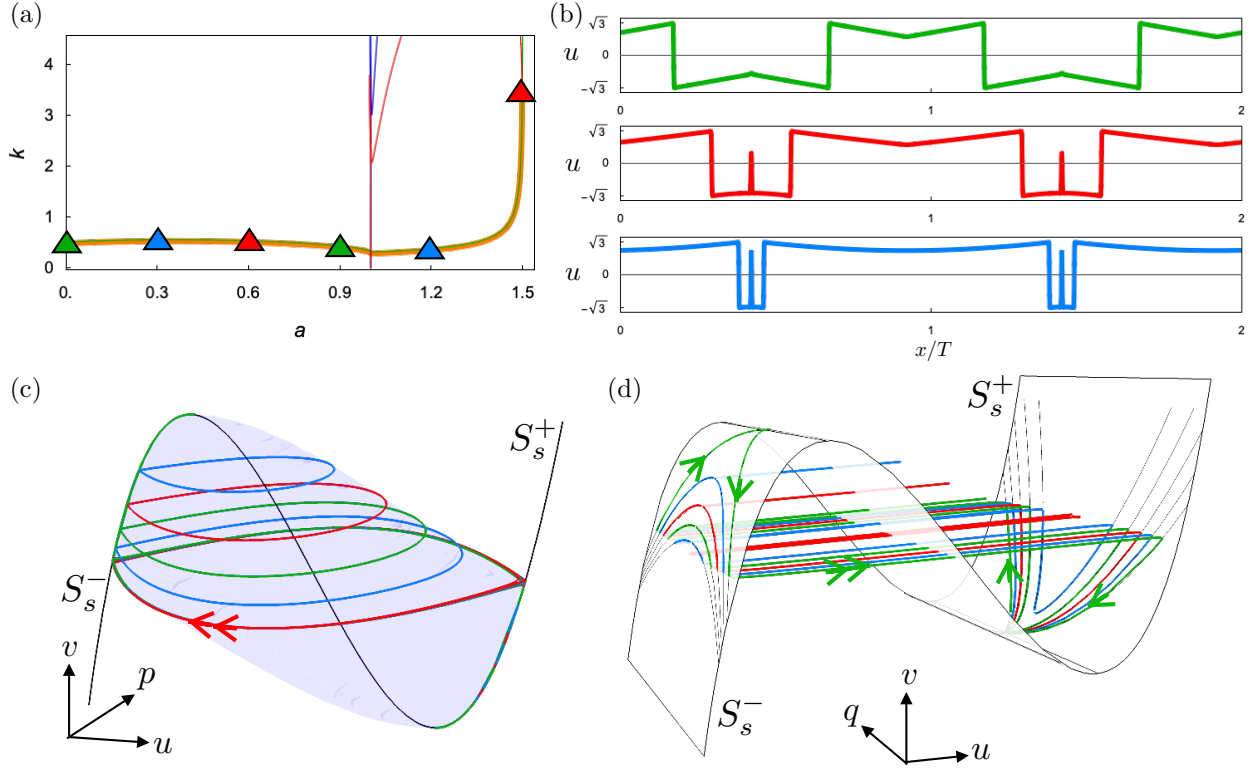
\begin{figure}[h!]
  \centering
  \includegraphics[width=\textwidth]{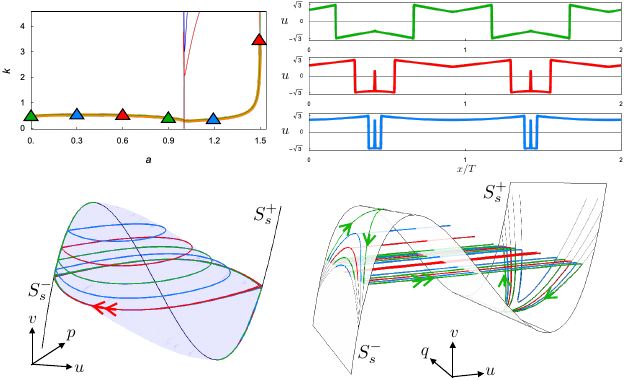}
  \put(-468,282){(a)}
  \put(-264,282){(b)}
  \put(-468,140){(c)}
  \put(-242,140){(d)}
  \caption{Spike formation in four-stroke, spatial  relaxation cycles.
   (a) Branch (orange) ranch of the isola along which the spikes grow. 
  (b) Spatial profiles of $u(\hat{x})$ with $a=0, 0.6, 1.2$.
  (c) Projection into the $(u,p,v)$ space showing that the outer fast jumps are mediated by the fast subsystem heteroclinics and that the internal spikes are created by the fast system homoclinics to $S_s^-$, starting up at the left RFSN-II point on $L^-$. 
  (d) Projection into the $(u,q,v)$ space showing that the slow segments on $S_s^{\pm}$ are constrained to the $\left\{ \mathcal{G} = \eps \left( \tfrac{2}{3}a-\tfrac{1}{4} \right) \right\}$ level sets (thin, black curves). 
  }
  \label{fig:transitionspikeformation}
\end{figure}

At $a=0$ (left green triangle), the solution is a four-stroke, spatial relaxation cycle, with slow segments given by the true and faux canards of the FS points on $L^\pm$, and fast jumps in $\{ v=0 \}$.
Then, moving left to right along the orange branch, as the value of $a$ is increased from $a=0$, an extra spike forms at the local maximum of the slow segment on $S_s^-$. 
It is a given to leading order by a homoclinic of the fast subsystem. 
The size (in the $(u,p)$ projection) of the extra spike increases with $a$, whilst the $v$-height at which it occurs decreases away from $v = \tfrac{2}{3}$. 
At the saddle-node bifurcation at $a \approx 1.5$, the $v$-level of the fast jump is at $v = 0$, the fast jump corresponds to the heteroclinic of the layer problem  
(recall Section~\ref{subsec:fastheteroclinics}), and the slow segment on $S_s^-$ has vanished. 
We also note that at the other end, there is a saddle-node bifurcation near $a \approx -0.0002$, where the orange branch transitions to the green branch

%----------------------------------------------------------
\subsection{Two-stroke relaxation cycles} \label{subsec:2to4stroke}
%----------------------------------------------------------
In this section, we analyze the solutions along the blue branch of the isola (Fig.~\ref{fig:transition_2to4stroke}(a)).
This branch emanates from one of the homoclinic orbits ($k=0$) near $a=1$. 
The homoclinic solution has small amplitude and is  centered on the steady state. Continuation of the solution away from the homoclinic limit causes the wavenumber to increase to a value close to the critical Turing value $k_T$. As this occurs, the corresponding spatially periodic solution remains small amplitude and confined to a small neighborhood of the fold $L^+$.

\begin{figure}[h!]
  \centering
  \includegraphics[width=\textwidth]{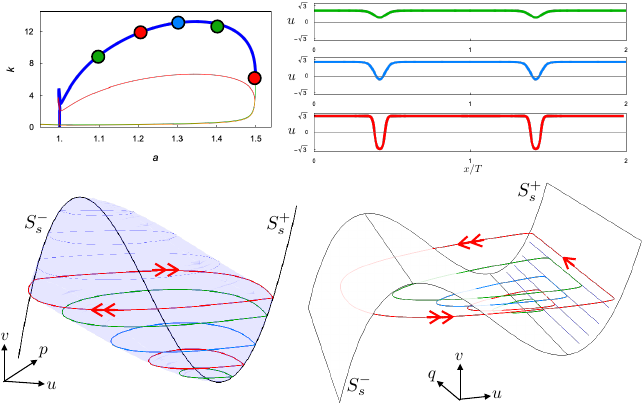}
  \put(-468,292){(a)}
  \put(-268,292){(b)}
  \put(-468,140){(c)}
  \put(-242,140){(d)}
  \caption{(a) Branch (blue) of the isola corresponding to two-stroke relaxation oscillations. 
  (b) Spatial profiles $u({\hat x})$ of the solutions for $a=1.1, 1.3, 1.5$. 
  (c) Projection of the solutions in the $(u,p,v)$ space and (d) in the $(u,q,v)$ space. 
  The $v$-level at which the fast jump occurs increases from $v = -\tfrac{2}{3}$ at $a\approx 1$ up to $v = 0$ at $a \approx 1.5$.}
  \label{fig:transition_2to4stroke}
\end{figure}
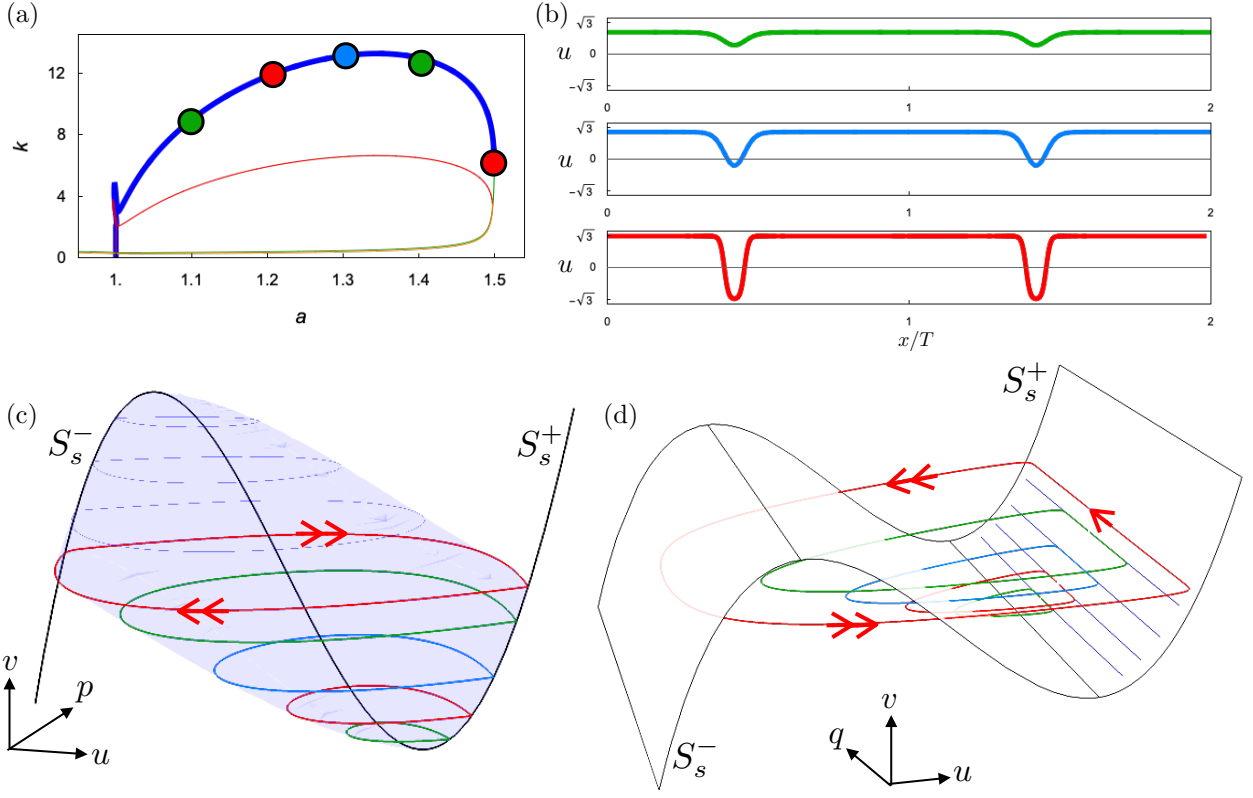

As the value of $a$ is increased along this branch, the solutions develop into two-stroke relaxation oscillations that consist of a slow segment on $S_s^+$ and a fast (near homoclinic) jump. 
The slow segment on $S_s^+$ starts at a `touchdown' point at $(u,p,v,q) = (u_{\rm touchdown}(v),0,v,q_{\rm touchdown}(v))$, where 
\begin{equation*} 
  \begin{split}
    u_{\rm touchdown}(v) &= \frac{2^{1/3}}{(3v+\sqrt{9v^2-4})^{1/3}} + \frac{(3v+\sqrt{9v^2-4})^{1/3}}{2^{1/3}}, \\ 
    q_{\rm touchdown}(v) &= -\sqrt{\eps} \sqrt{\tfrac{2}{3}a-\tfrac{1}{4}-au_{\rm touchdown}+\tfrac{1}{2}u_{\rm touchdown}^2+\tfrac{1}{3}au_{\rm touchdown}^3-\tfrac{1}{4}u_{\rm touchdown}^4},
  \end{split}
\end{equation*}
for $v \in (-\tfrac{2}{3},0)$. It flows down the $\mathcal{G}=\eps (\tfrac{2}{3}a-\tfrac{1}{4})$ contour to the turning point $(u,p,v,q) = (u_{\rm turn},0,v_{\rm turn},0)$, where 
\[ \tfrac{1}{3} u_{\rm turn}^3 - u_{\rm turn} = v_{\rm turn} = \frac{2}{81}(2a-3)(2a-3+2\sqrt{a^2+3a})(2a+3+2\sqrt{a^2+3a}). \]
Then, it flows back up the $\mathcal{G}=\eps (\tfrac{2}{3}a-\tfrac{1}{4})$ contour to the `takeoff' point at $(u,p,v,q) = (u_{\rm takeoff}(v),0,v,q_{\rm takeoff}(v))$, where $u_{\rm takeoff} = u_{\rm touchdown}$ and $q_{\rm takeoff} = - q_{\rm touchdown}$.
At the takeoff point, a fast jump is initiated which connects the takeoff point to the touchdown point via the homoclinic of the layer problem that exists at that value of $v$ (see Proposition~\ref{prop:fasthomoclinics}). 

When the blue branch of the isola reaches the saddle-node bifurcation at $a \approx 1.49881$, the $v$-level at which the fast jump occurs is $v = 0$,  and the homoclinic jump mechanism switches to a heteroclinic jump mechanism.

%----------------------------------------------------------
\subsection{Two-stroke double-loop cycles} \label{subsec:doubleloop}
%----------------------------------------------------------
Finally, we analyze solutions along the red branch of the isola, which emerges from the orange branch  at the saddle-node bifurcation near $a \approx 1.5$. 
The solutions on the red branch of the isola are two-stroke relaxation oscillations, like those in Section~\ref{subsec:2to4stroke}. However, the fast segments of these orbits consist of two near-homoclinic cycles. 

\begin{figure}[h!]
  \centering
  \includegraphics[width=\textwidth]{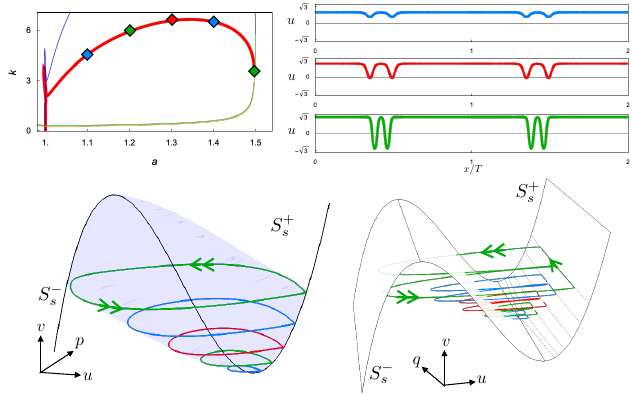}
  \put(-468,292){(a)}
  \put(-264,292){(b)}
  \put(-468,140){(c)}
  \put(-214,140){(d)}
  \caption{
  (a) Branch (red) of the isola corresponding to two-stroke double-loop relaxation cycles.  These solutions have $\mathcal{O}(1)$ values of $k$, in an interval about $k_T$. 
  (b) Spatial profiles $u({\hat x})$ of the solutions for $a=1.1, 1.3, 1.5$.
  Solutions corresponding to the colored markers are shown (c) in the $(u,p,v)$ space and (d) in the $(u,q,v)$ space. 
  }
  \label{fig:transition_doubleloop}
\end{figure}
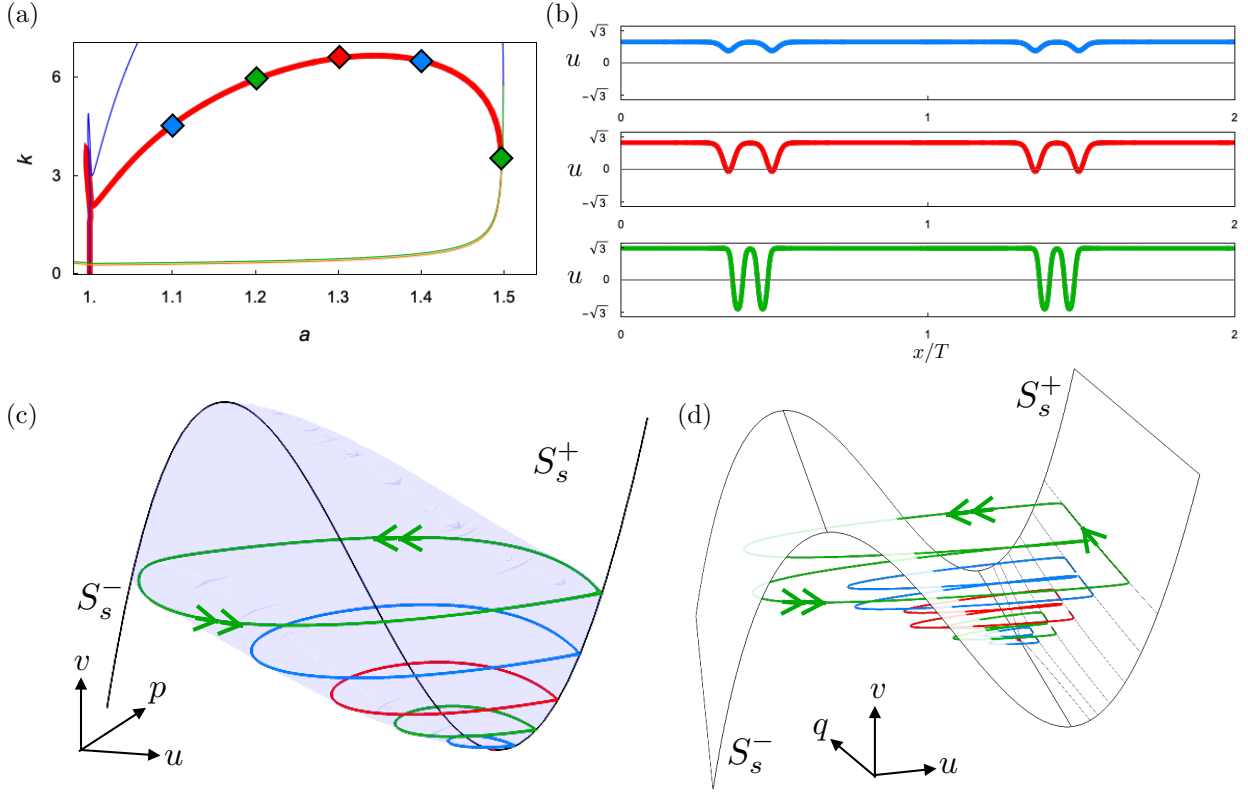

As the parameter $a$ is decreased along this branch  away from the saddle-node bifurcation near $a\approx 1.5$, the $v$-level at which the fast jump occurs decreases,  and the amplitudes of the double-loop homoclinics shrink. 
Then, near $a=1$, the red branch of the isola becomes nearly vertical. Along this nearly vertical  part of the isola, the amplitudes of the solutions are very small, and the solution becomes homoclinic to the equilibrium near the fold $L^+$ (still with small amplitude). 

%-----------------------------------------------------------
\section{The nearly self-similar dynamics of some spatial canards}
\label{sec:selfsimilar}
%-----------------------------------------------------------

In this section, we discuss some of the spatial canard solutions that exhibit nearly self-similar dynamics. Recall for example the canards shown in Figs.~\ref{fig:smallamp-selfsimilar} and \ref{fig:kOrderDelta_LAO}. First, we study self-similarity in the equations in the rescaling chart, \eqref{eq:ODE-K2}, and then we study it in the full fourth order system \eqref{eq:spatialODE-y}.

In the coordinate chart $K_2$, the nearly self-similar dynamics for $0<\delta \ll 1$ may be understood as follows. 
For $\delta=0$, the equations are Hamiltonian with
\begin{equation}
\label{eq:H2-r2zero}
H_2(u_2,p_2,v_2,q_2;0)
   = \frac{1}{2}\sqrt{\eps}\left( p_2^2 - q_2^2 \right)
   + u_2 v_2
     - \frac{1}{3} u_2^3, 
\end{equation}
where we recall \eqref{eq:H2}.
Examination of the level set $\{ H_2 = 0 \}$ reveals that it is scale-invariant. 
In particular, on the hyperplane $\{r_2=0\}$, which corresponds to the singular limit of $\delta=0$, the zero level set of the Hamiltonian $H_2(u_2,p_2,v_2,q_2;0)$ is invariant under the scaling
$(u_2,p_2,v_2,q_2) \to (\zeta \tilde{u}_2, \zeta^{3/2} \tilde{p}_2, \zeta^2 \tilde{v}_2, \zeta^{3/2} \tilde{q}_2)$ for any real number $\zeta$.
In fact, $H_2(u_2,p_2,v_2,q_2;0)=\zeta^3 H_2(\tilde{u}_2,\tilde{p}_2,\tilde{v}_2,\tilde{q}_2;0)$.
The projection of this level set onto the $(u,q)$ plane has infinitely many self-crossings and nested loops. It is infinitely self-similar.

Then, for $0<\delta \ll 1$, this scale invariance is broken. Nevertheless, for sufficiently small values of $\delta$, the solutions exhibit nearly self-similar dynamics.
This nearly self-similar dynamics manifests itself in the maximal canards of the full equations with $0 < r_2 \ll 1$ in chart $K_2$. 
A boundary value problem was set up using these equations, together with the boundary conditions
that enforce the symmetry $(u_2,p_2,v_2,q_2) \to (u_2,-p_2,v_2,-q_2)$ and an integral constraint to preserve the Hamiltonian.
(See Appendix~\ref{sec:appA2} for details.) 
The continuation method starts with the algebraic solution $\Gamma_0$ at $r_2=0$ and $a_2=0$ given by \eqref{eq:Gamma0}. 
That solution is continued into $r_2>0$ until $r_2=\sqrt{\delta}$. 
Then, it is continued in $a_2$ to yield the maximal canards for the system in chart $K_2$.
Throughout, the constraint is imposed to conserve the value of the Hamiltonian.
The key observation is that these maximal canards  approach the infinitely self-similar structure in the limit.
The continuation of the maximal canards is shown in the plane of $a_2$ and the $L^2$ norm in Fig.~\ref{fig:selfsimilar} (a), and orbits of the maximal canards are shown in Figs.~\ref{fig:selfsimilar} (b)--(g) for a sequence of $a_2$ values. 

\begin{figure}[h!]
  \centering
  \includegraphics[width=\textwidth]{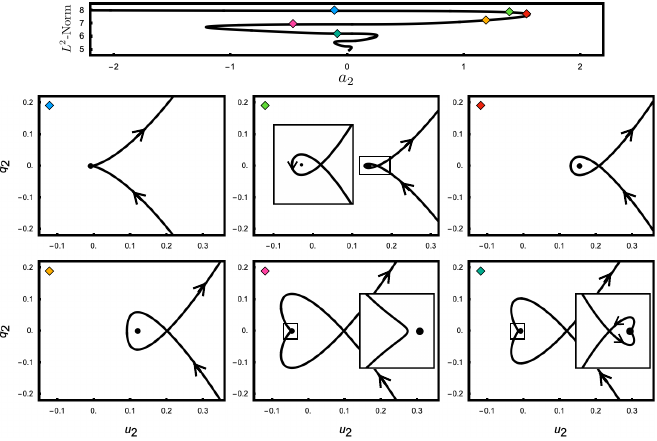}
  \put(-446,305){(a)}
  \put(-468,246){(b)}
  \put(-306,246){(c)}
  \put(-152,246){(d)}
  \put(-468,127){(e)}
  \put(-306,127){(f)}
  \put(-152,127){(g)}
  \caption{Nearly self-similar dynamics of the spatial canards in the central chart $K_2$ of the blow-up for $\eps = 0.1$ and $r_2 = \sqrt{\delta} = 0.1$. The solutions shown here correspond to the zero level set of the full Hamiltonian $H_2(u_2,p_2,v_2,q_2;r_2)$. 
  (a) The bifurcation diagram of solutions on the $H_2(u_2,p_2,v_2,q_2;r_2) = 0$ contour has an almost self-similar structure.
  For panels (b)--(g), the projections of the solutions into the $(u_2,q_2)$ phase space are shown. The insets show the dynamics near the equilibrium point $(u_2,p_2,v_2,q_2) = (r_2 a_2,0,u_2^2+\tfrac{1}{3} \sqrt{\eps} r_2^2 u_2^3, 0)$, indicated by the black marker.. 
  (b) The solution for $a_2 = -0.0928$ approaches the equilibrium and then turns away in a cusp-like manner. 
  (c) The solution for $a_2 = 1.4104$ exhibits a single loop around the equilibrium. 
  (d) For $a_2 = 1.5466$, the amplitude of the loop has grown.
  (e) For $a_2 = 1.2002$, the leftmost part of the loop develops a long, steep  segment.
  (f) The solution for $a_2 = -0.4594$ turns in towards the equilibrium and develops a second loop.
  (g) The solution for $a_2 = -0.0744$ exhibits two loops around the equilibrium state. 
  }
  \label{fig:selfsimilar}
\end{figure}
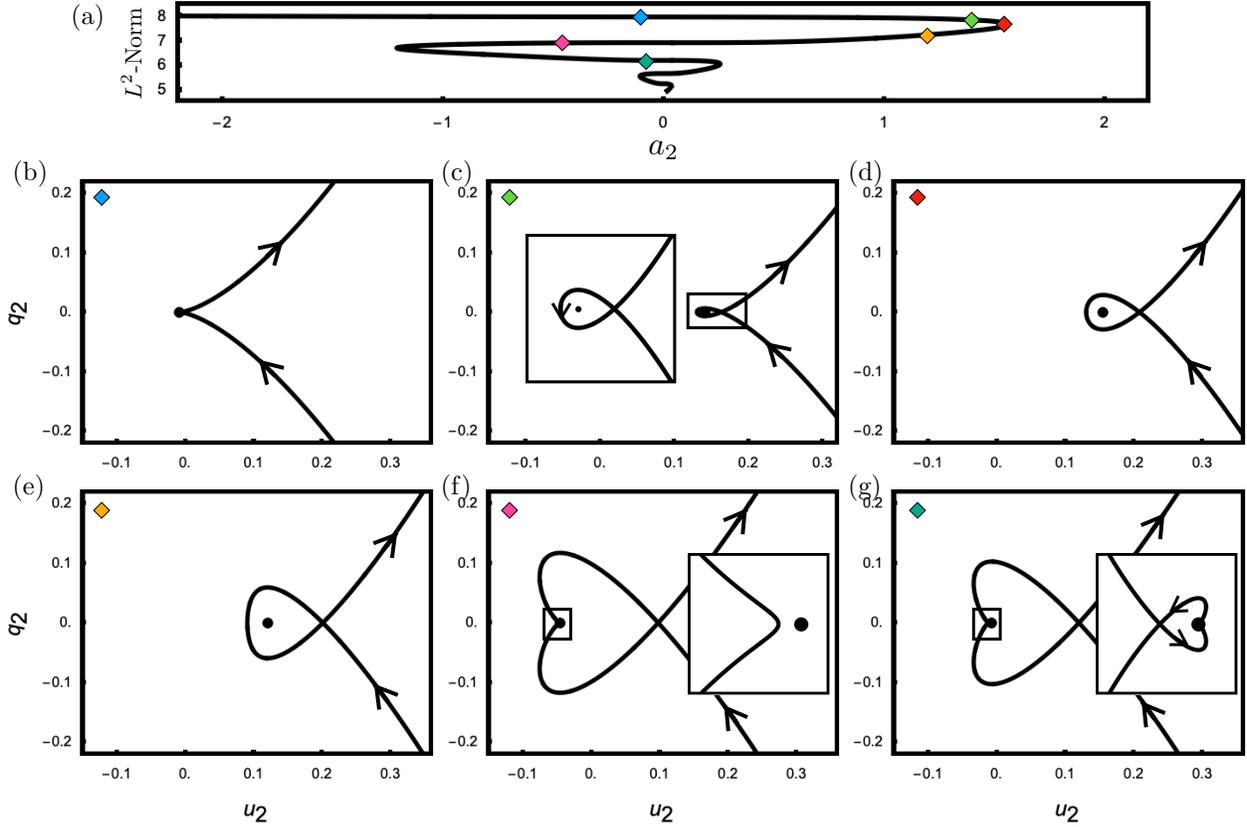

The bifurcation curve undulates about the line $a_2=0$, as shown in Fig.~\ref{fig:selfsimilar}(a).
The magnitude of the undulations decreases as the $L^2$ norm decreases. 
Also, there appears to be a self-similarity in the curve. 
For example, if one takes the lower part of the bifurcation curve shown in Fig.~\ref{fig:selfsimilar}(a) below the second local extremum on the left, stretches that vertically and horizontally so that it has the same height as the curve shown, and overlays that stretched version onto the curve shown, then they look almost identical.

The nearly self-similar structure of the bifurcation diagram persists into the full fourth-order spatial system \eqref{eq:spatialODE-y}, see Fig.~\ref{fig:???}. 
The undulations decay and the branch eventually converges to a single value of $a$. For $\eps = 1$ and $\delta = 0.05$, we find that the $a$ value to which the branch converges is $a_{c,\rm{num}} \approx 0.99972881$. Comparing this with the critical value, $a_c(\delta)$, from the Melnikov analysis (see formula \eqref{eq:ac}), we have that $a_c(0.05) \approx 0.99973958$. Hence, 
\[ a_c(0.05)-a_{c,\rm{num}} = \mathcal{O}(\delta^4). \]

\begin{figure}[h!]
  \centering
  \includegraphics[width=0.75\textwidth]{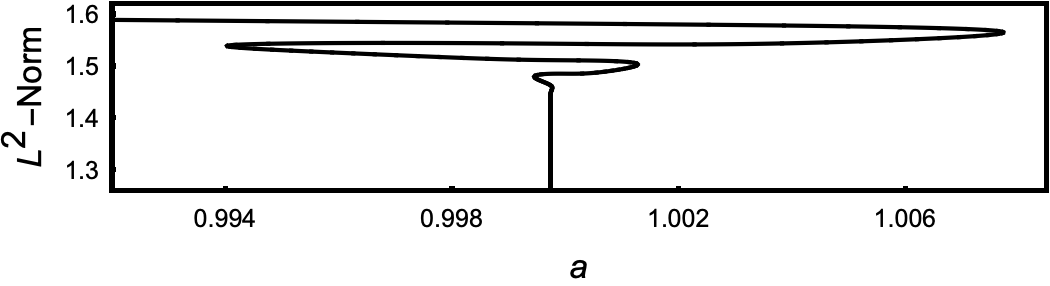}. 
  \caption{Bifurcation diagram for a family of orbit segments corresponding to canard solutions in the full fourth-order spatial system \eqref{eq:spatialODE-y} for $\eps = 1$ and $\delta = 0.05$. This bifurcation diagram was computed by finding maximal canards using the methods outlined in Appendix~\ref{app:numericalmethods_slowmanifolds} and then numerically continuing these canards with respect to $a$. The bifurcation curve undulates about the value $a = a_{c,\rm{num}} \approx 0.99972881$, which is $\mathcal{O}(\delta^4)$ close to the analytic prediction $a_c(\delta)$. 
  The canard solutions at the top of the bifurcation curve have no loops and simply enter and exit the neighbourhood of the folded singularity in a cusp-like fashion (in the $(u,q)$ plane). As $a$ is varied down the bifurcation curve, these canard solutions develop small, nearly self-similar loops in the neighbourhood of the folded singularity. The further down the bifurcation curve, the more self-similar loops in the orbit segments.}
\end{figure}

%-----------------------------------------------------------
\section{The spatial canards of \texorpdfstring{\eqref{eq:spatialODE-y}}{Lg} are  analogs in spatial dynamics of temporal limit cycle canards in fast-slow ODEs}
\label{sec:spatialdynamicsanalog}
%-----------------------------------------------------------

The spatially periodic canards in the PDE system \eqref{eq:vdp} are spatial analogs of the time-periodic canards known to exist in fast-slow systems of ODEs \cite{BCDD1981,BBE1991,Brons2006,D1984,DR1996,E1983,EP1998,KS2001,MKKR1984,Moehlis,RKZE2003}. 
(For reference, we recall that for example in the van der Pol ODE $\dot u=v-f(u)$, $\dot v=\eps(a-u)$, the Hopf bifurcation occurs at $a=1$ and the canard explosion is centered on $a_{\rm ODE} \sim 1-\frac{1}{8}\eps- \frac{3}{32}\eps^2$, asymptotically as $\eps \to 0$, recall \cite{BCDD1981,D1984,E1983,KS2001}.)

First, the Turing points $a_T$ are reversible 1:1~resonant Hopf bifurcations in the spatial ODE system, as shown in Section~\ref{sec:turingbifn}. 
The reversible Hopf points are the analogs in spatial dynamics of the Hopf bifurcation points that occur in (temporal) fast-slow ODEs.
Second, families of small-amplitude, spatially periodic solutions are created in the Turing bifurcation.
These solutions have wavenumbers close to $k_T$ and profiles close to the plane wave $e^{ik_Tx}$ (recall the red and blue orbits in Fig.~\ref{fig:soln_knearcritical}).
These spatially periodic solutions are the analogs of the small-amplitude, temporally oscillating solutions that are created in (singular) Hopf bifurcations in fast-slow systems of ODEs. 
Third, the critical value $a_c(\delta)$, recall \eqref{eq:ac}, asymptotically close to $a_T$ is the analog in spatial dynamics of the canard explosion value ($a_{\rm ODE}$ for the van der Pol ODE, for example) in fast-slow oscillators, which is asymptotically close to the Hopf point.
Fourth, on each side of the maximal spatial canards, there are families of (non-maximal) spatial canards, just as there are (non-maximal) temporal canards (headless ducks and ducks with heads) on each side of maximal limit cycle canards.
In this manner, the maximal spatial canards act as separatrices that locally partition the phase space into regions of distinct behaviour, and the bifurcations of maximal canards delimit distinct modes of activity in the parameter space, just as is the case for maximal temporal canards. 
Fifth, the singularity responsible for the spatial canards is an RFSN-II point. 
This is the analog of the canard point in the temporal fast-slow systems, which may be viewed as an FSN-II point in the extended $(u,v,a)$ phase space of the kinetics ODE.
Sixth, in carrying out the geometric desingularization analysis of the RFSN-II point in the spatial ODE system \eqref{eq:spatialODE-y}, we identified a key algebraic separatrix solution $\Gamma_0$, recall \eqref{eq:Gamma0}, in the rescaling chart that is the key segment of the singular limit of the spatial canards. 
It is the spatial analog of the parabolic separatrix solution in the rescaling chart (see \cite{KS2001}) in the analysis of the explosion of temporal limit cycle 
canards in fast-slow ODEs.

There are also interesting differences between the spatial canards and the temporal limit cycle canards, which arise due to differences between the spatial ODEs and the temporal (kinetics) ODEs, as well as due to the increased dimension of the phase space. 
In the classical explosion of temporal canards, the left and right branches are one dimensional attracting slow manifolds and the middle branch a one dimensional repelling slow manifold. 
Moreover, the maximal canards exist for the parameter values when these two slow manifolds coincide to all orders. 
In contrast, in the spatial ODE, the left and right branches are two dimensional saddle slow manifolds, and the middle branch a two dimensional elliptic slow manifold, recall Fig.~\ref{fig:S}. 
Further, the maximal canards exist when the single-branched stable and unstable manifolds of the cusp point, which are the true and faux canards of the RFSN-II point, 
continue into each other to all orders.

%-----------------------------------------------------------
\section{PDE Dynamics}   \label{sec:pdedyn}
%-----------------------------------------------------------
In Sections~\ref{sec:nearTuring}--\ref{sec:spatialdynamicsanalog}, we focused primarily on studying the rich family of stationary, spatially periodic canard patterns created by Turing bifurcations in the singularly perturbed van der Pol PDE \eqref{eq:vdp} for values of $a$ near the  Turing bifurcation $a_T = \sqrt{1-2\delta \sqrt{\varepsilon}}$, as well as numerically for $a\in [0,a_T)$. 
A natural --and necessary-- next question is: which of these patterns are observable, {\it i.e.}, which are  stable as solutions of the PDE?  
This is in general a hard and deep question.
In this section, we will scratch the surface of this challenge, presenting a number of PDE simulations that give some hints about what kind of (stability) results may be expected.

We fix $(\eps,\delta)$ at $(0.1,0.01)$, which implies that $a_T = 0.996832...$.
Moreover, we consider two primary values of $a$: $a=1 > a_T$, so that the background state $(a,f(a))$ is stable as solution of \eqref{eq:vdp} and $a=0.99 < a_T$ for which $(a,f(a))$ has undergone the Turing bifurcation. 
We consider the dynamics generated by \eqref{eq:vdp} from the point of view of the onset of pattern formation, {\it i.e.}, we consider initial conditions $(u_0(x),v_0(x))$ that are close (in $L^1$ norm) to $(u,v) \equiv (a,f(a))$, and we look for the {\it a priori} small ($\sim$ close to $(a,f(a))$)  patterns generated by \eqref{eq:vdp}. 
For $a=1$, $|a-a_T|= 0.00316...$
This implies, from the point of view of weakly nonlinear stability theory/the Ginzburg-Landau approach (see for example \cite{CH1993,D2019,E1965,IMD1989,S2003,SU2017,W1997}), that the magnitudes of the patterns expected near the Turing bifurcation are of the order of  $\sqrt{|a-a_T|} = 0.0562...$. 
Therefore, we only work with initial conditions $(u_0(x),v_0(x))$ that satisfy $|u_0(x)-a|, |v_0(x)-f(a)| \leq 0.10 < 2\sqrt{|a-a_T|}$ uniformly on the $x$-interval under consideration. 

Now, since $\eps \delta^2 < 64/625$ for $\mathcal{O}(1)$ values of $\eps>0$ and $0<\delta \ll 1$, the Ginzburg-Landau equation associated to the Turing bifurcation has a positive Landau coefficient {\it i.e.,} a positive coefficient on the cubic term  (Section \ref{sec:turingbifn}).  
Thus, one does not expect small stable patterns for $a < a_T$.
In fact, one expects the amplitude of the patterns generated by the Turing bifurcation to grow beyond the validity of the Ginzburg-Landau approach \cite{E1993,H1991,SU2017}. 
Then, for $a > a_T$, only the trivial state (corresponding to $(a,f(a))$) will be stable. 
Moreover, its domain of attraction will be small, in the sense that solutions with initial data that is  $\mathcal{O}(\sqrt{|a-a_T|})$ close to $(a,f(a))$, but not sufficiently close, are also expected to outgrow the Ginzburg-Landau domain of validity.   

\begin{figure}[h!]
\centering
        \begin{subfigure}[t]{0.45\textwidth}
		\centering
		\includegraphics[width=\linewidth]{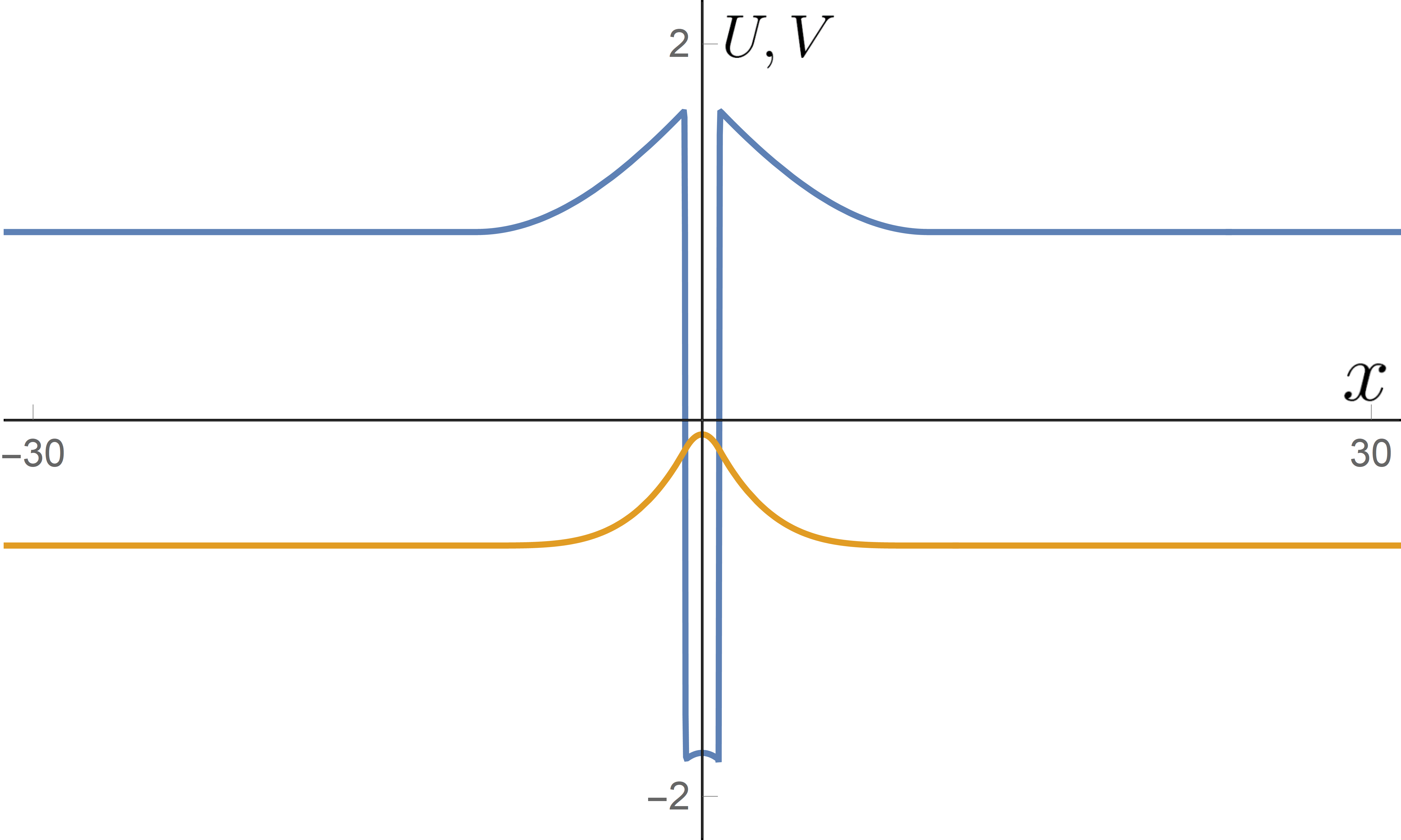}
   %		\caption{}
	\end{subfigure}
        \hspace{1cm}
	\begin{subfigure}[t]{0.45\textwidth}
		\centering
		\includegraphics[width=\linewidth]{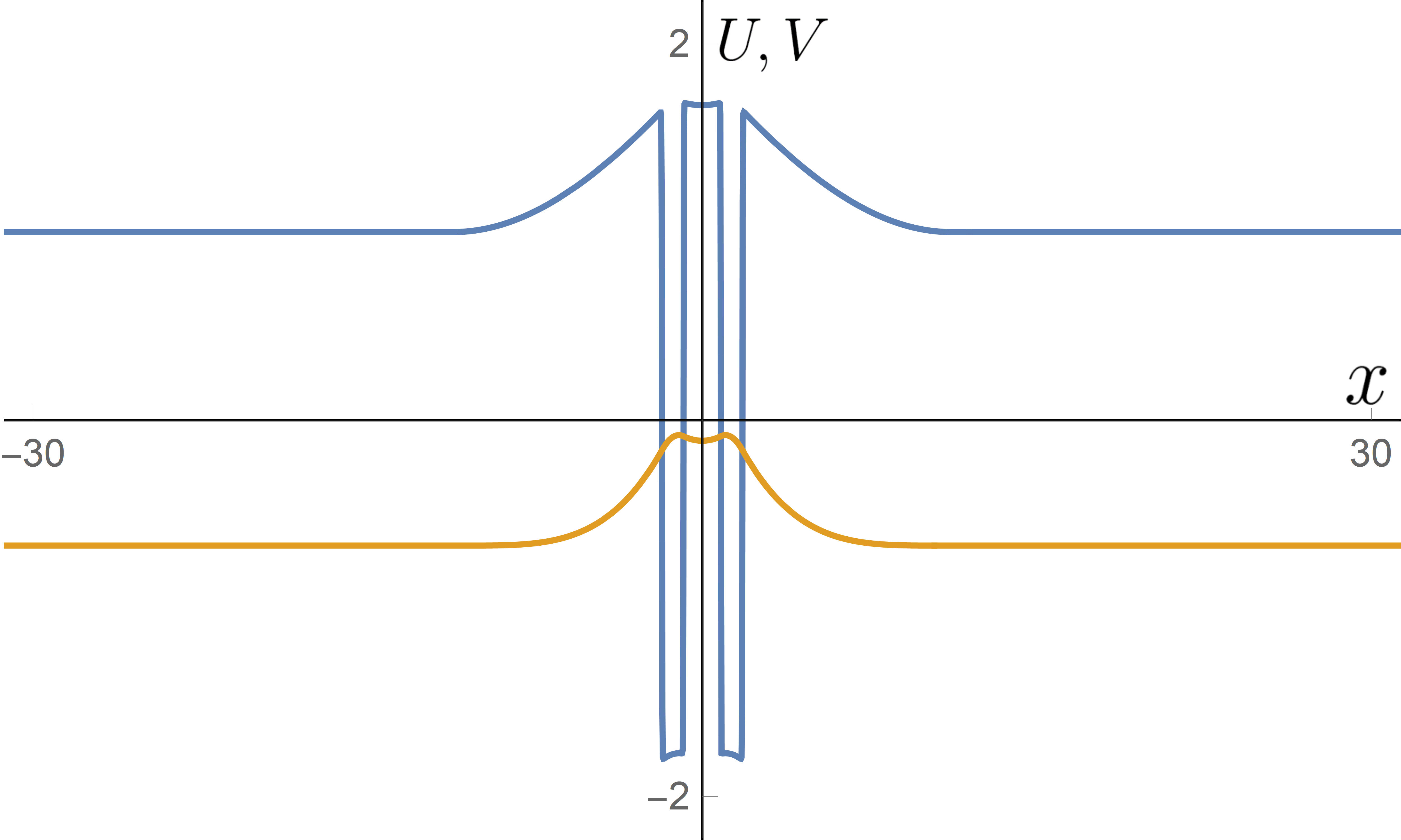}
%		\caption{}
	\end{subfigure}
\\
\vspace{.5cm}
        \begin{subfigure}[t]{0.45\textwidth}
		\centering
		\includegraphics[width=\linewidth]{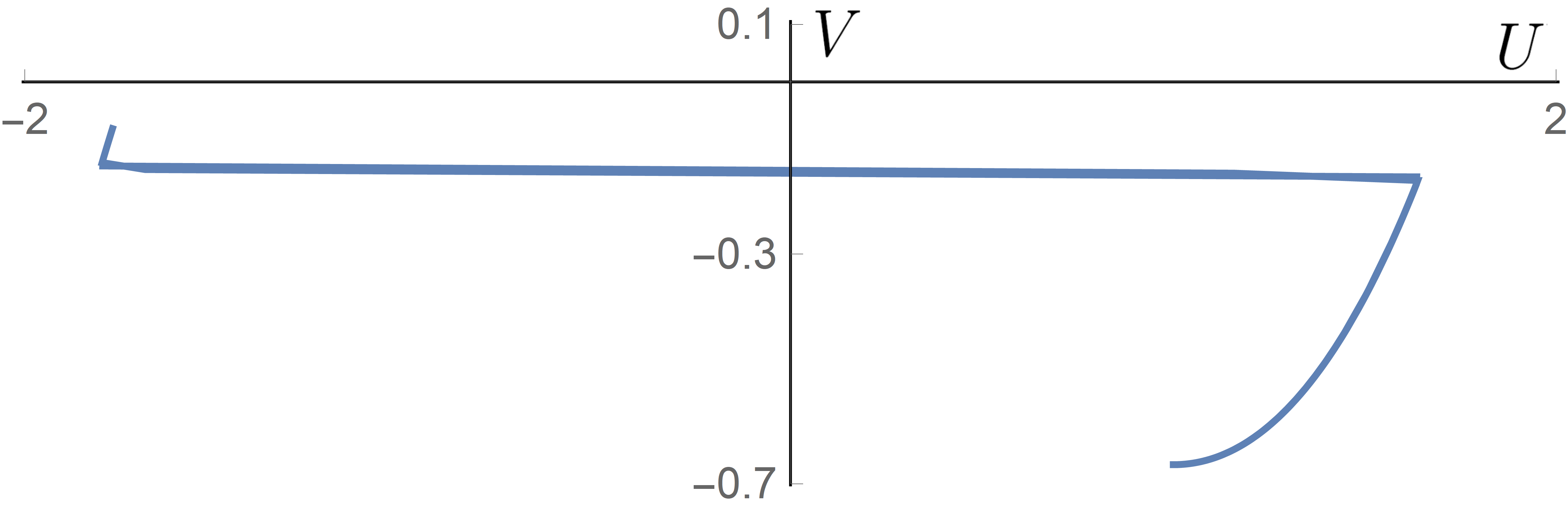}
%		\caption{}
	\end{subfigure}
        \hspace{1cm}
        \begin{subfigure}[t]{0.45\textwidth}
		\centering
		\includegraphics[width=\linewidth]{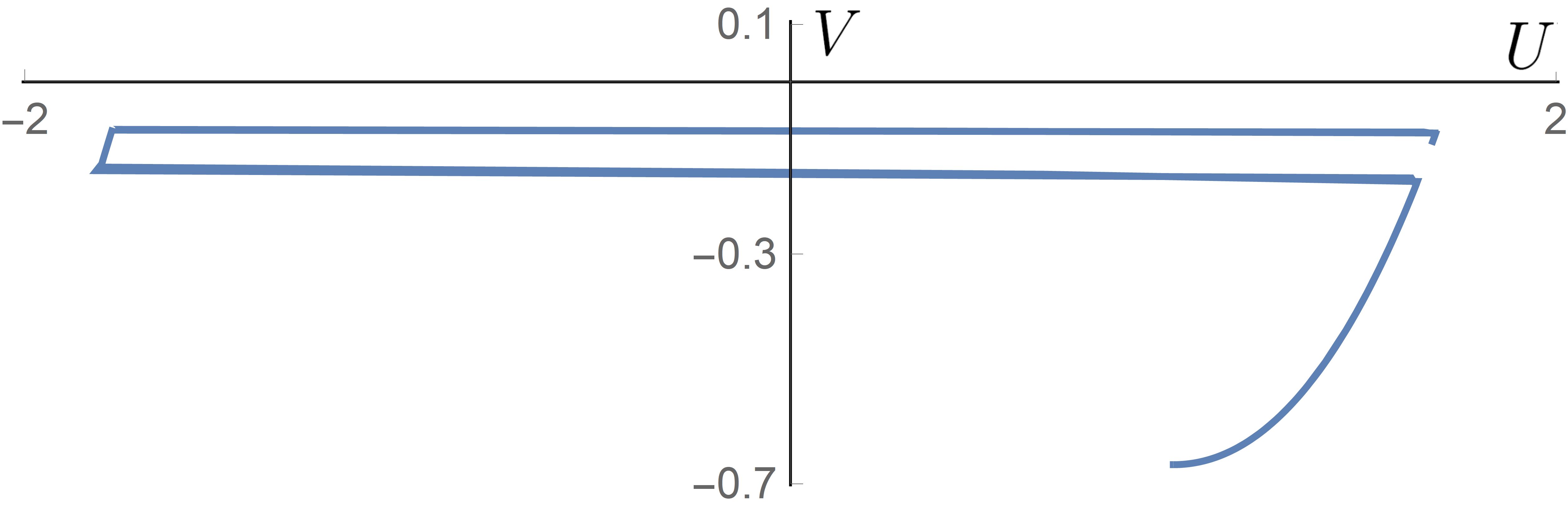}
%		\caption{}
	\end{subfigure}

\caption{Stable attractors of \eqref{eq:vdp} for two different initial data.
Top row: spatial profiles of the final attracting (stable) large amplitude patterns, with $u(x)$ in blue and $v(x)$ in orange. Bottom row: projections of the same final patterns as in the first row into the $(u,v)$-plane. 
The simulations were conducted on the domain $x \in (-100,100)$ with homogeneous Neumann boundary conditions, but the attractors are only shown for  $x \in (-30,30)$.
For the simulation shown in the first column, the initial data $(u_0(x),v_0(x)) = 1  + 0.05\,[\tanh(10(x - 0.5)) - \tanh(10(x + 0.5))], -2/3 - 0.025\,[\tanh(0.5(x - 1)) - \tanh(0.5(x + 1))]$ was used. For the simulation shown in the second column, $v_0(x)$ is the same, and $u_0(x)$ is changed, with  $\tanh(10(x \pm 0.5))$ replaced by $\tanh(10(x \pm 1.5))$.
The parameters are $(a,\varepsilon, \delta) = (1,0.1,0.01)$.}
\label{fig:Sims1}
\end{figure}

Fig.~\ref{fig:Sims1} shows the final attracting patterns of two numerical simulations with $a=1.00$ that both start with localized initial conditions at most $0.10$ away from the $(1,-2/3)$ background state, {\i.e.}, the initial perturbations of $(a,f(a)) = (1,-2/3)$ can be considered to be sufficiently small so that the dynamics can be expected to be within the range of validity of the Ginzburg-Landau approach. 
In fact, the only difference between the two  simulations is the widths (in $x$) of the `plateaus' where $u_0(x)$ differs from $a=1$. 
In both cases, the initial conditions are outside the domain of attraction of $(u(x,t),v(x,t)) \equiv (1,-2/3)$ (by construction: further reducing the amplitude of the initial conditions brings $(u_0(x),v_0(x))$ into the domain of attraction of $(1,-2/3)$). 
Indeed, the dynamics generated by \eqref{eq:vdp} depart from the Ginzburg-Landau region near $(1,-2/3)$, and the magnitude of the pattern grows to $\mathcal{O}(1)$. 
In both cases, a (stationary) localized homoclinic multi-front pattern appears as attractor.
Depending on the width of the initial plateau, it is either a 1-pulse/2 front or a 2-pulse/4-front pattern, see Fig. \ref{fig:Sims1}. 
In the second row of Fig. \ref{fig:Sims1}, the orbits  of the final attracting states are plotted in the $(u,v)$ phase space (parameterized by $x$). 
This further confirms that the attractors are of the large-amplitude spatial canard types studied in Sections~\ref{sec:numerics-3}, \ref{sec:numerics-4}, and \ref{subsec:singularLAO}. 

\begin{figure}[t]
\centering
        \begin{subfigure}[t]{0.44\textwidth}
		\centering
		\includegraphics[width=\linewidth]{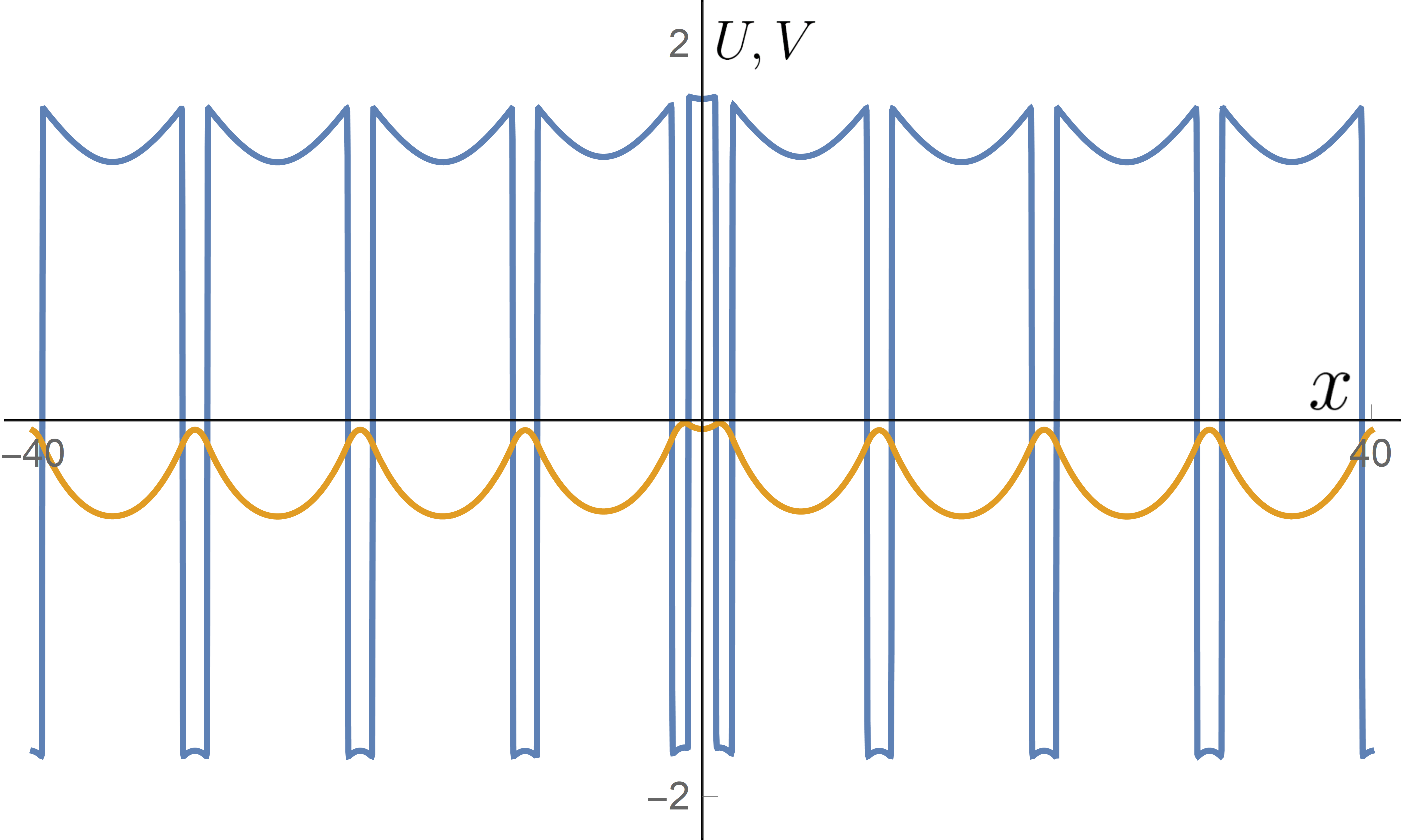}
%		\caption{}
	\end{subfigure}
%        \hspace{1cm}
%	\begin{subfigure}[t]{0.245\textwidth}
%		\centering
%		\includegraphics[width=\linewidth]{TransientPattern-1-a099}
%		\caption{}
%	\end{subfigure}
        \begin{subfigure}[t]{0.53\textwidth}
		\centering
		\includegraphics[width=\linewidth]{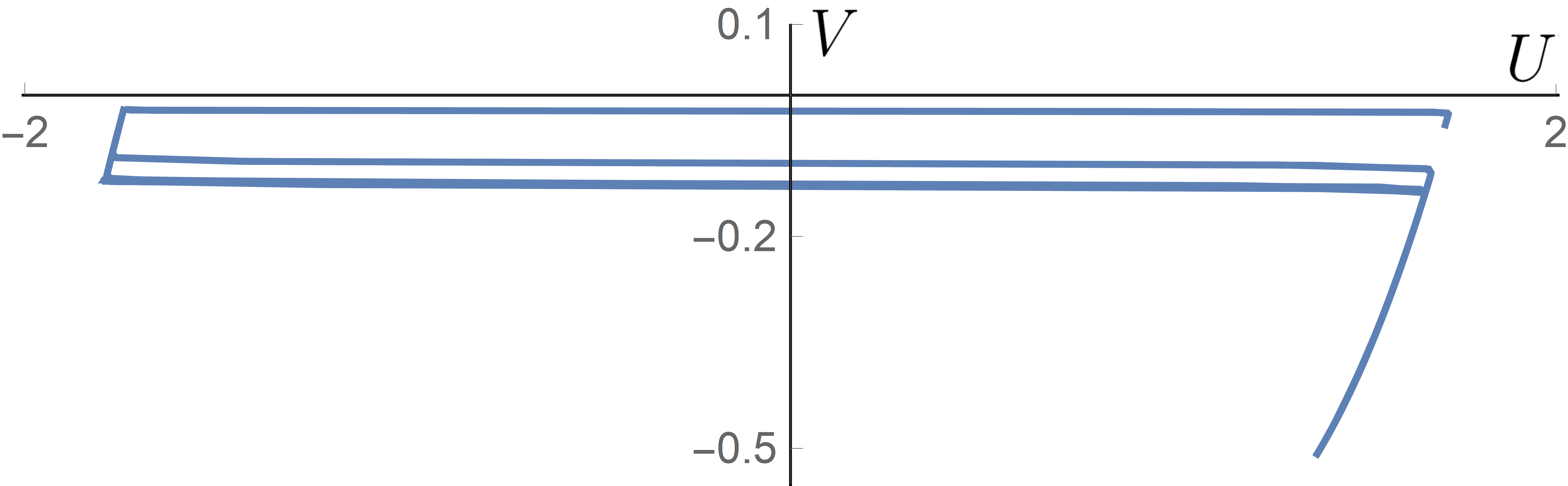}
%		\caption{}
	\end{subfigure}
\caption{Large-amplitude spatially periodic attractors of \eqref{eq:vdp} obtained from simulations for $x \in (-100,100)$ with homogeneous Neumann boundary conditions.
Left: the final attracting state, plotted for $x \in (-40,40)$.
Right: the final pattern plotted in the $(u,v)$-plane (parameterized by $x$). 
The initial data is $(u_0(x),v_0(x)) = 0.99  + 0.05\,[\tanh(10(x - 0.5)) - \tanh(10(x + 0.5))], f(0.99) - 0.025\,[\tanh(0.5(x - 1)) - \tanh(0.5(x + 1))]$. 
The parameters are $(a,\varepsilon, \delta) = (0.99,0.1,0.01)$.}
\label{fig:Sims2}
\end{figure}

Next, we consider exactly the same setting as in the simulations of Fig. \ref{fig:Sims1} but decrease $a$ to $a=0.99 < a_T$. 
As a consequence, the pattern dynamics of \eqref{eq:vdp} must be significantly different, since the background state $(u(x,t),v(x,t)) \equiv (a,f(a))$ has become unstable. 
In Fig. \ref{fig:Sims2}, the attracting pattern is plotted that emerges from the simulations with $a=0.99$ from initial conditions that are similar to those in the second column of Fig.~\ref{fig:Sims1} (in the sense that the deviations from $(a,f(a))$ are identical in both cases). 
Here also, the attractor is a stationary large-amplitude homoclinic pattern, but now not homoclinic to the homogeneous state $(a,f(a))$, but to a spatially periodic pattern -- with a well-selected wavenumber. 

This periodic pattern appears step-by-step, on a long time scale. 
Up to $t \approx 380$, the pattern is  `transitionally stable', in the sense that it remains very similar to the attractor observed for $a=1.00$ in Fig. \ref{fig:Sims1}. 
Then, a fast growing spike appears that evolves into a new localized 2-front pulse -- or better: two of these, one each for $x > 0$ and for $x <0$ -- between $t = 390$ and $ t=400$ (see the first row of Fig. \ref{fig:Sims3}). 
In the second row of Fig. \ref{fig:Sims3}, the same patterns are plotted, now again in $(u,v)$-space. 
The (spatial) canard dynamics near the point $(a,f(a)) \approx (1,-2/3)$ seem to play a dominant role in the formation and shape of the spike. 
Each `new period' of the periodic pattern to which the final state is asymptotic is formed in the same manner -- with long periods of `transient stability' in which the pattern only evolves marginally and slowly. 
Here, we leave it as a subject for future research to go deeper into the relation between the spatial dynamics near $(1,-2/3)$ and the formation of the periodic pattern and especially the mechanism by which the periodic pattern --that must be an element in a one-parameter family of spatially periodic patterns-- is selected.    
\begin{figure}[t]
\centering
        \begin{subfigure}[t]{0.32\textwidth}
		\centering
		\includegraphics[width=\linewidth]{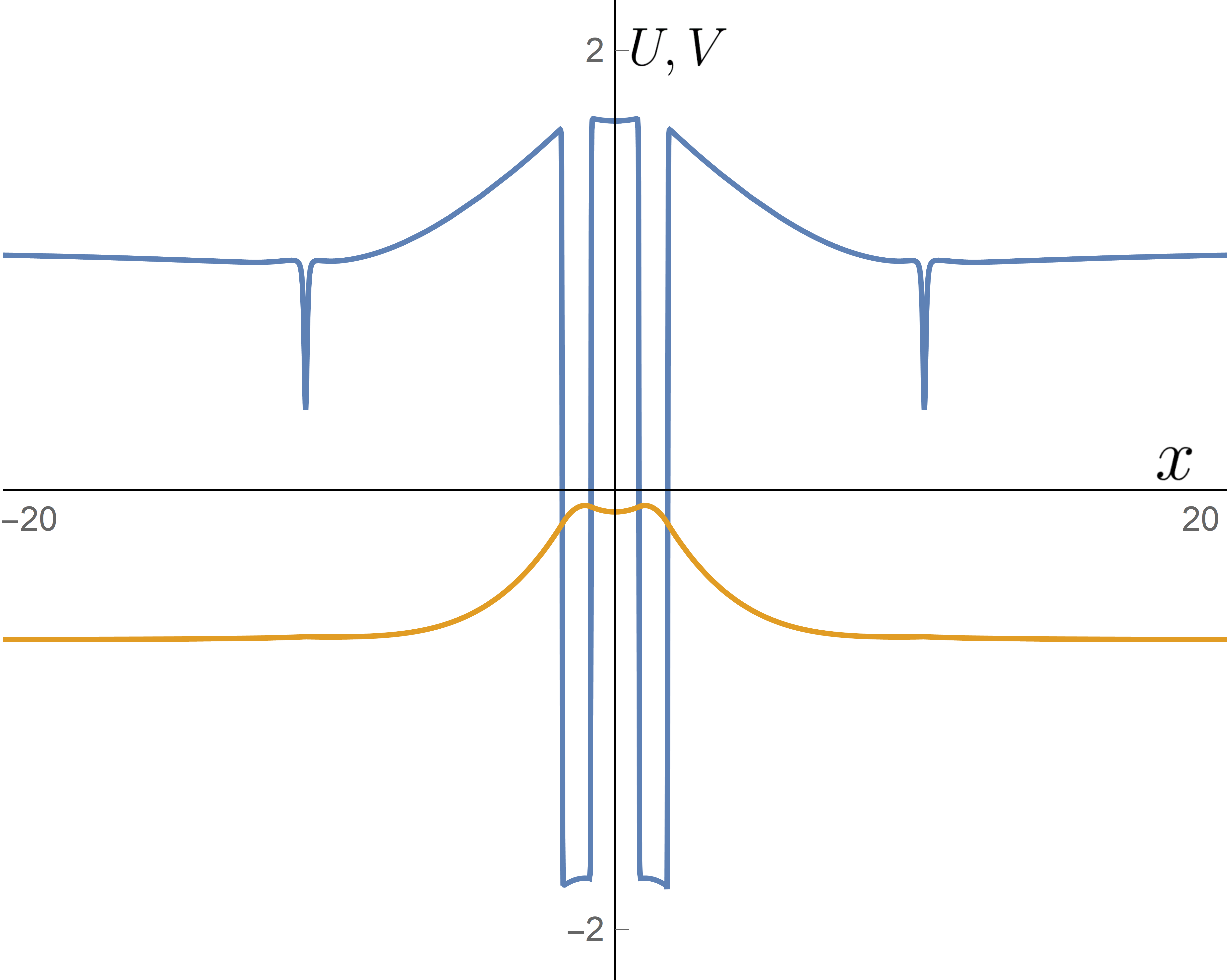}
%		\caption{}
	\end{subfigure}
        \begin{subfigure}[t]{0.32\textwidth}
		\centering
		\includegraphics[width=\linewidth]{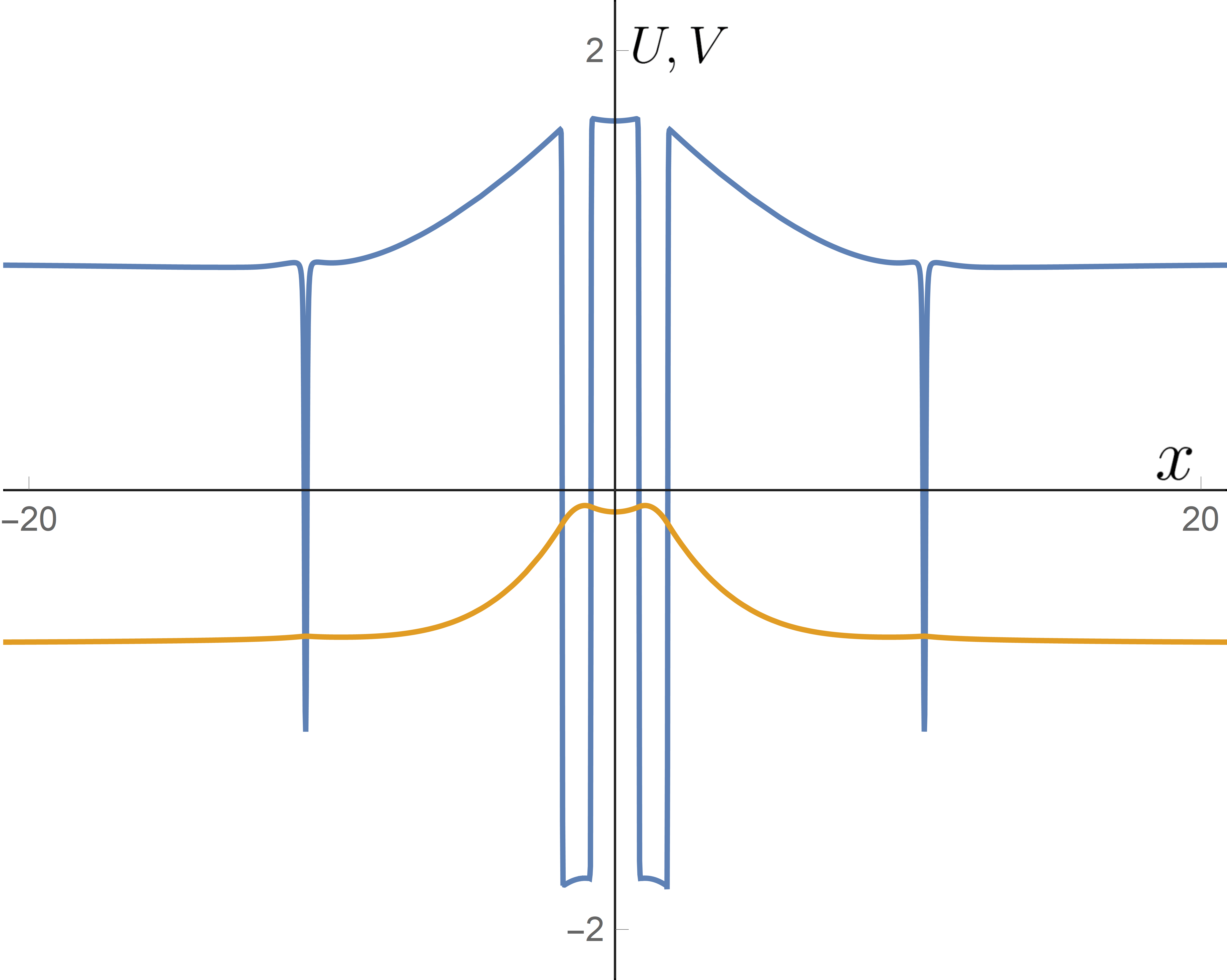}
%		\caption{}
	\end{subfigure}
        \begin{subfigure}[t]{0.32\textwidth}
		\centering
		\includegraphics[width=\linewidth]{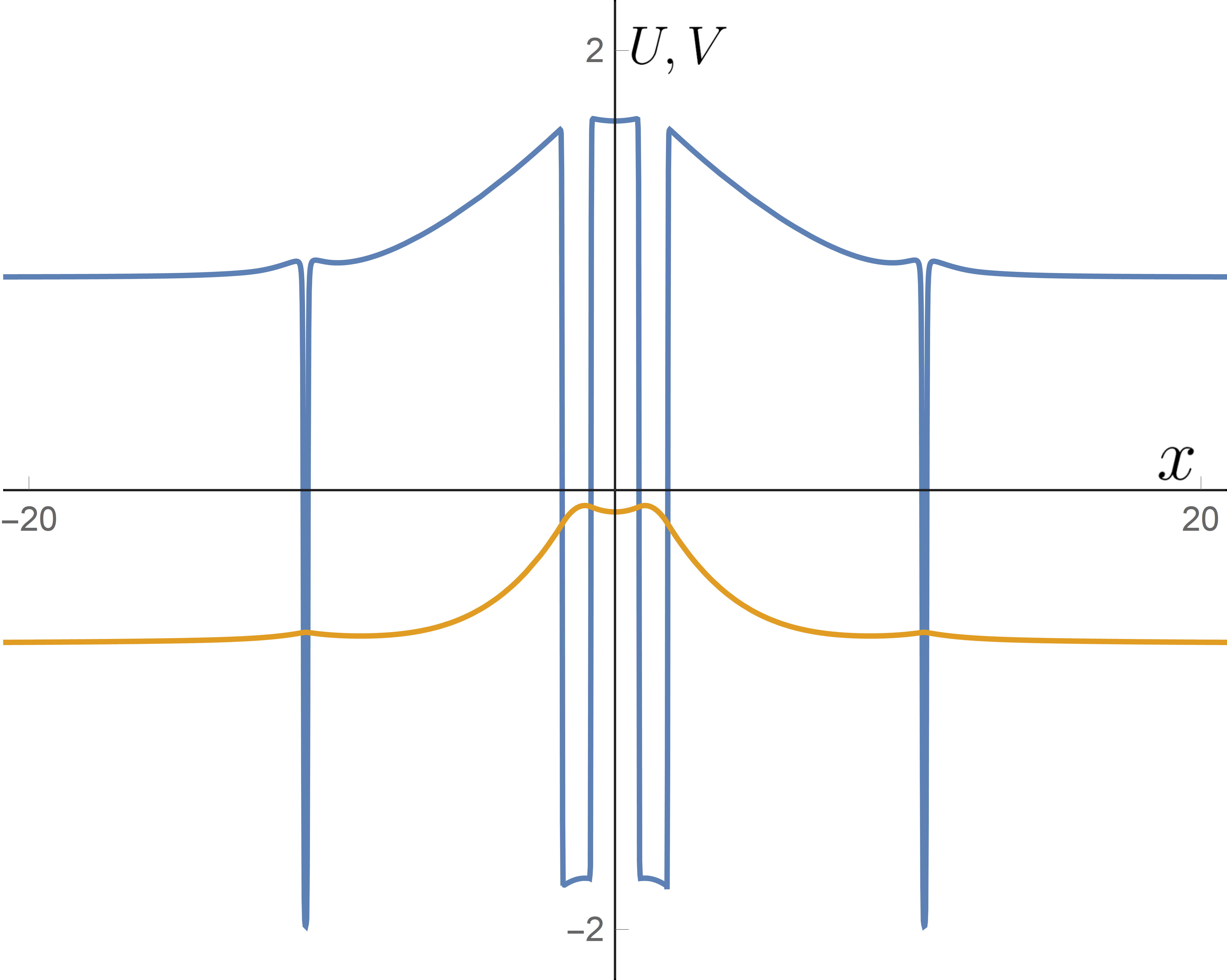}
%		\caption{}
	\end{subfigure}
%       \begin{subfigure}[t]{0.24\textwidth}
%		\centering
%		\includegraphics[width=\linewidth]{GrowingSpike-400-a099}
%		\caption{}
%   	\end{subfigure}
	\\
\vspace{.5cm}
        \begin{subfigure}[t]{0.32\textwidth}
		\centering
		\includegraphics[width=\linewidth]{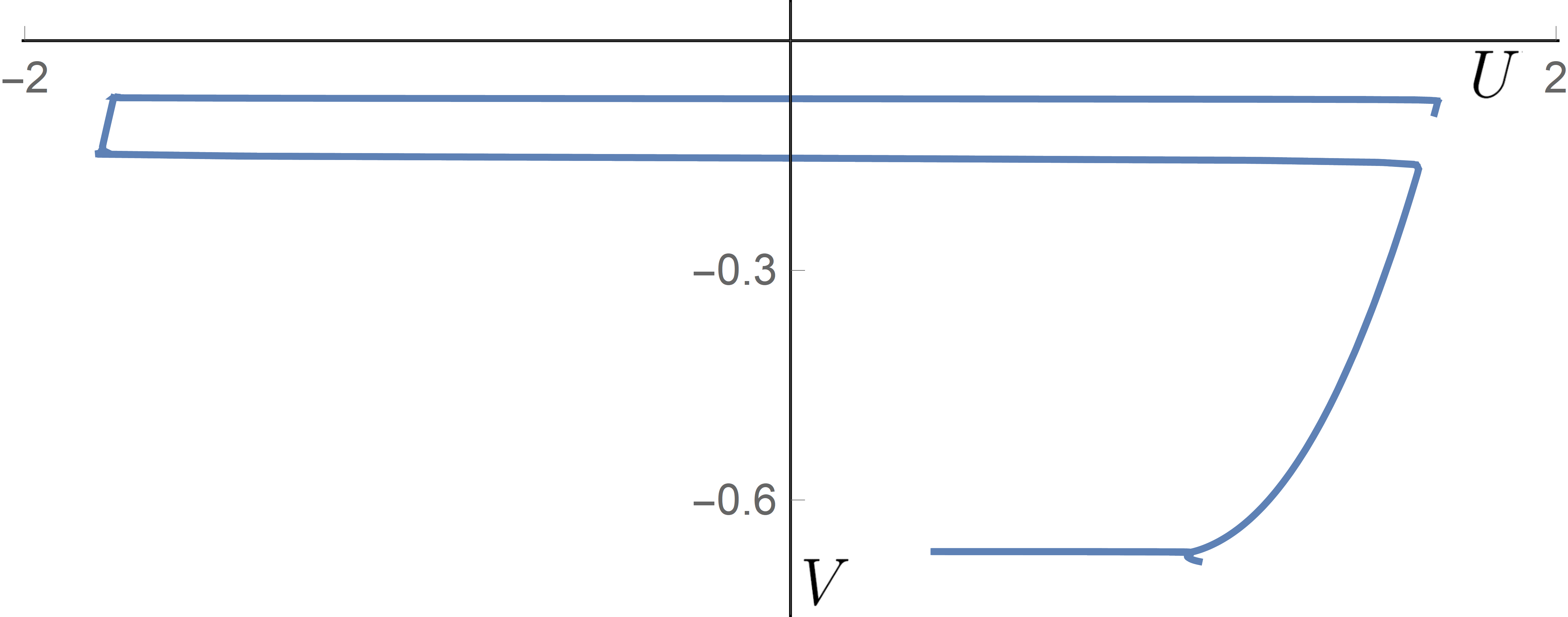}
%		\caption{}
	\end{subfigure}
        \begin{subfigure}[t]{0.32\textwidth}
		\centering
		\includegraphics[width=\linewidth]{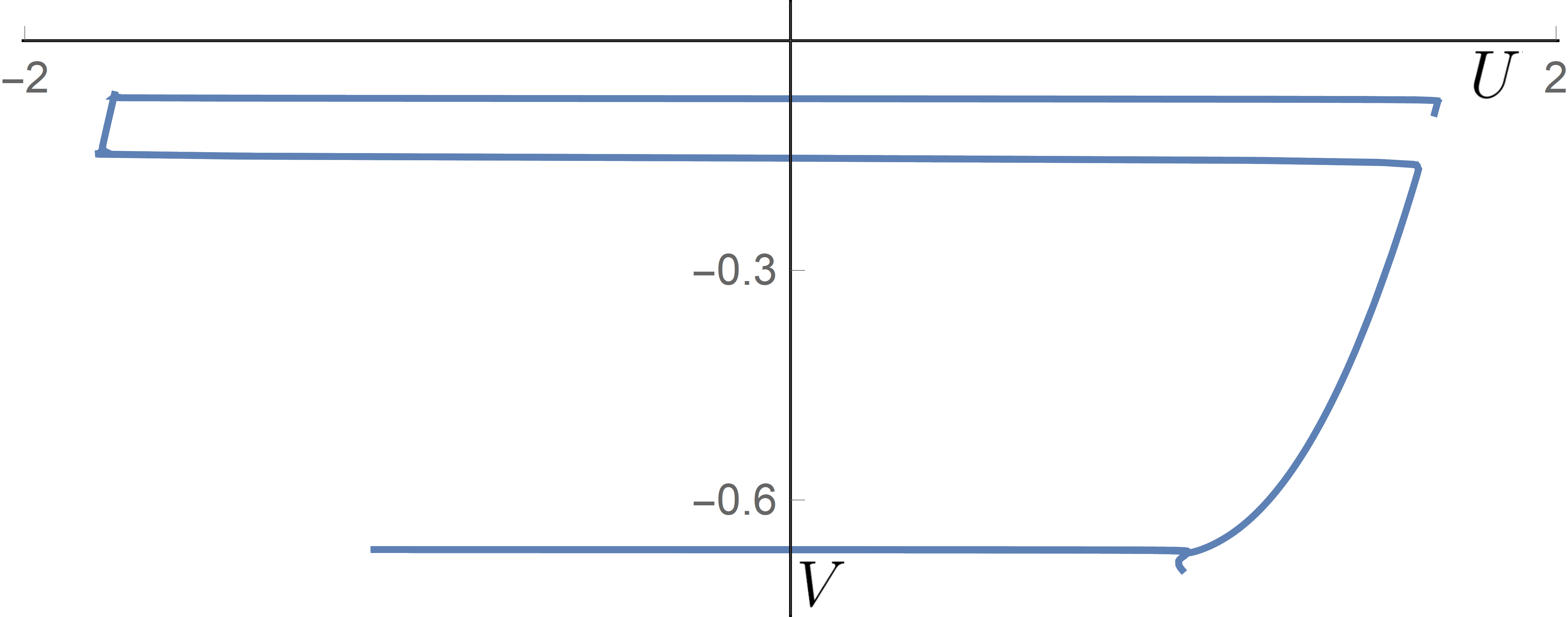}
%		\caption{}
	\end{subfigure}
        \begin{subfigure}[t]{0.32\textwidth}
		\centering
		\includegraphics[width=\linewidth]{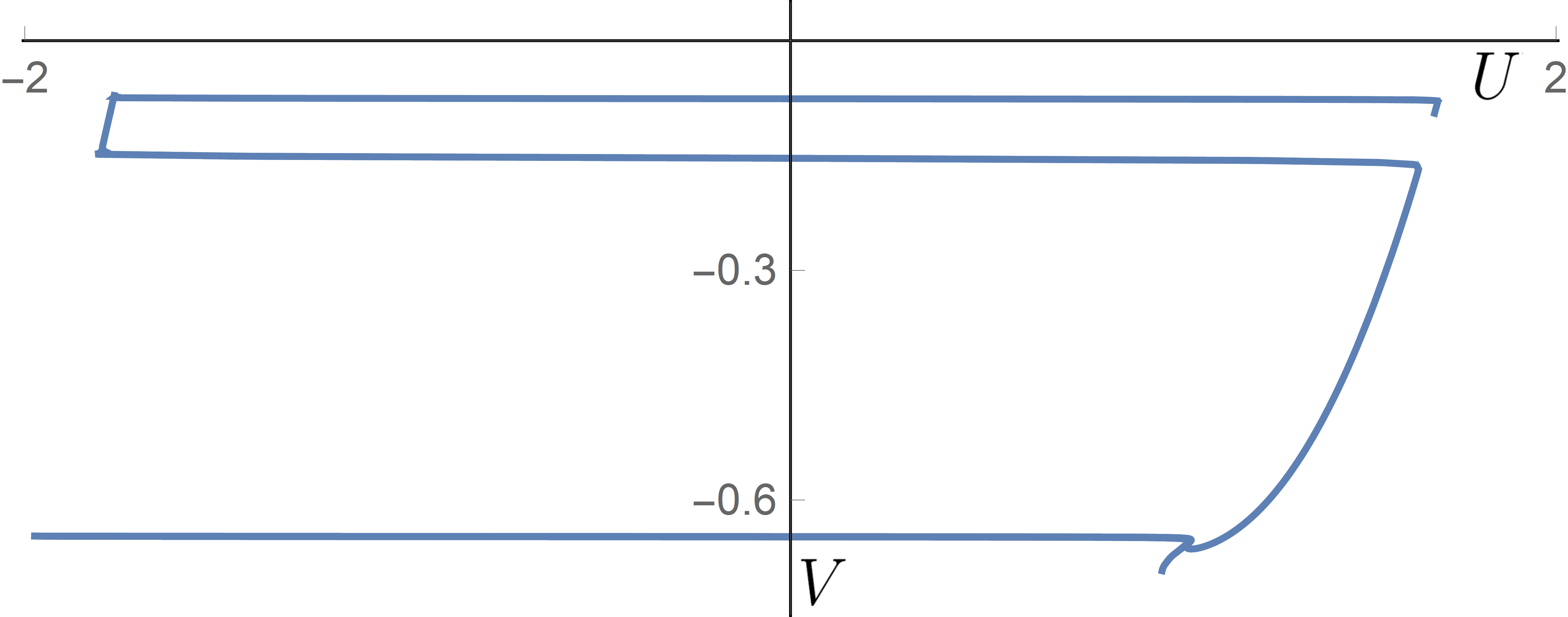}
%		\caption{}
	\end{subfigure}
%        \begin{subfigure}[t]{0.24\textwidth}
%		\centering
%		\includegraphics[width=\linewidth]{GrowingSpikeUV-400-a099}
%		\caption{}
%	\end{subfigure}
\caption{Top row: Spatial profiles of $u(x,t)$ and $v(x,t)$ at times $t= 394, 396, 398$ of the solution obtained from simulations of \eqref{eq:vdp} for $x \in (-100,100)$ with homogeneous Neumann boundary conditions
The fast growing spike is shown for $x$ restricted to $(-20,20)$.
Bottom row: the same patterns now plotted in $(u,v)$-space.
The initial conditions are as in Fig. \ref{fig:Sims2}. 
The parameters are 
$(a,\varepsilon, \delta) = (0.99,0.1,0.01)$.} 
\label{fig:Sims3}
\end{figure}

Overall, the simulations of the van der Pol PDE \eqref{eq:vdp} that we have conducted so far indicate that there are substantial classes of initial conditions that are in the basins of attraction of (large-amplitude) stationary patterns of the types  considered in the analysis here. 
Moreover, this behavior appears to be related to the width of plateaus in both components $u_0(x)$ and $v_0(x)$ of the initial data.
If one keeps all other aspects as in the simulations shown, then there appears to be a critical threshold for the width of the plateaus.
In particular, data for which the plateau width is below the threshold evolve to these attractors. Whereas, for data in which the plateau width exceeds the threshold, the system no longer exhibits stationary large-amplitude fronts and pulses. Instead, the PDE exhibits attracting patterns that vary periodically in time with a small amplitude, while their spatial variation may either be uniformly small, or consist of distinct small- and large-amplitude parts (see also Fig.~\ref{fig:Sims4} below).

To gain an intuitive understanding of this threshold, we recall from Section~\ref{sec:turingbifn} that, near the (spatial) Turing bifurcation at $a=a_T$, the trivial state $(a,f(a))$ is already unstable with respect to the long (spatial) wavelength perturbations associated to the (temporal) Hopf bifurcation at $a=1$, since $a_T<1$ by \eqref{eq:kTaT} --see Fig.~\ref{fig:EVcurves}. 
Hence, if the initial data consist of perturbations of $(a,f(a))$ for which the support is sufficiently large, then one expects that the unstable Hopf modes will be triggered.
On the other hand, if the initial data consists of sufficiently localized perturbations, {\it i.e.,} has  plateaus that are too narrow, then the dynamics of the PDE are predominantly driven by the Turing mode, as seen above. 
Moreover, these observations also suggest that to understand the dynamics of patterns in (\ref{eq:vdp}) that do not originate from such localized perturbations, one needs to consider both instability mechanisms, {\it i.e.} one needs to study the co-dimension 2 bifurcation in which the spatial Turing mode and the temporal Hopf mode interact. 

In that case, the natural Ginzburg-Landau set-up involves two (complex) amplitudes: $A(\xi,\tau)$ that modulates the (linear) Turing mode $e^{i k_T x}$ and $B(\xi,\tau)$ that modulates the Hopf mode $e^{i \omega_T t}$ (cf. Section~\ref{sec:turingbifn}). In the generic setting, the dynamics of small-amplitude solutions near this bifurcation has been studied in \cite{SW2022} (in the setting of reaction-diffusion systems).
In fact, the validity of a coupled (cubic) system of equations for $A$ and $B$ has been established in \cite{SW2022} for the generic situation in which
the shapes of the eigenvalue curves $\lambda_\pm(k)$ are independent of the small parameter $\delta$.  
Here, significant additional analysis appears to be required to go beyond the generic situation, because in (\ref{eq:vdp}) the distance between the Hopf and Turing bifurcations is $2 \sqrt{\eps d} = 2 \sqrt{\eps} \delta$. Thus, it is natural to take  $\delta > 0$ as the main small parameter in the Ginzburg-Landau analysis (recall also that $\delta = 0.01$ (and $\eps = 0.1$) in most  simulations here). 
Furthermore, in order to derive the $(A,B)$-system that governs the dynamics of small patterns $\mathcal{O}(\delta)$ close to $a_T$, one needs to adapt the generic scalings of \cite{SW2022}. We also leave this analysis for future work. 

Finally, regarding the PDE dynamics, there is also the question whether (some of) the stationary small-amplitude canard patterns presented in Sections~\ref{sec:numerics-1}, \ref{sec:numerics-2}, and \ref{subsec:Turingspatialcanards} may also be stable as solutions of the PDE \eqref{eq:vdp}. 
In our (limited) numerical simulations to date, we have not yet encountered such patterns. As observed above, by choosing the extent of the support of the initial conditions sufficiently large, {\it i.e.,} above threshold, we see that system (\ref{eq:vdp}) may exhibit small-amplitude patterns. 
In the simulation shown in the top row of Fig. \ref{fig:Sims4}, we consider localized initial data for which the spatial extent is somewhat larger than in Figs. \ref{fig:Sims1}, \ref{fig:Sims2}, and \ref{fig:Sims3} and take (again)  $a=1$. 
We observe the formation of a small-amplitude pattern that varies periodically in time and in space. Its spatial variation has a long wavelength character, which indicates that the amplitude $A(\xi,\tau)$ of the Turing mode has vanished. 
Thus, the PDE dynamics seem to be dominated by the Hopf bifurcation associated to the van der Pol equation and the associated (complex, uncoupled) Ginzburg-Landau equation for $B(\xi, \tau)$. 

This changes drastically if we decrease $a$ to $0.999$ (see the bottom row of Fig. \ref{fig:Sims4}). 
Here, as also in the simulations shown in Figs.~\ref{fig:Sims1} (second column), \ref{fig:Sims2}, and \ref{fig:Sims3}, a localized, four-front, large-amplitude spatial structure appears near $x=0$ (see Fig.~\ref{fig:Sims4}). 
However, unlike in Figs. \ref{fig:Sims1} and \ref{fig:Sims2}, the background state is now unstable with respect to the Hopf bifurcation, and the tails of the large-amplitude pattern vary periodically in time -- and on the long spatial scale. (In fact, the localized structure itself also oscillates in time, but with an amplitude that is much smaller than that of the oscillations beyond the large-amplitude part of the pattern.) 
Note that the attractor of Fig. \ref{fig:Sims4} has the (expected) nature of a stationary localized structure as in Fig. \ref{fig:Sims2} that is destabilized by a Hopf bifurcation stemming from its essential spectrum (cf. \cite{SS01}).

Of course, this evidence is far from sufficient to conclude that the stationary small-amplitude canard patterns cannot be stable.
So far, our findings only indicate that in the study of potentially stable small-amplitude patterns, one necessarily must also take the spatial aspects of the Hopf bifurcation into account. 
In other words, one needs to merge the (non-generic) coupled Ginzburg-Landau dynamics of the Turing mode $A(\xi,\tau) e^{i k_T x}$ and the Hopf mode $B(\xi,\tau) e^{i \omega_H t}$ with the canard analysis in the present work. Moreover, a much more detailed and extensive (numerical) investigation of the PDE dynamics is crucial: our limited experiments show that the nature of the final attractor depends in a subtle way on small variations in the initial conditions and (consequentially) on the accuracy of the numerical procedure. 
These questions would be interesting and relevant subjects for future work. 

\begin{figure}[t]
\centering
        \begin{subfigure}[t]{0.32\textwidth}
		\centering
		\includegraphics[width=\linewidth]{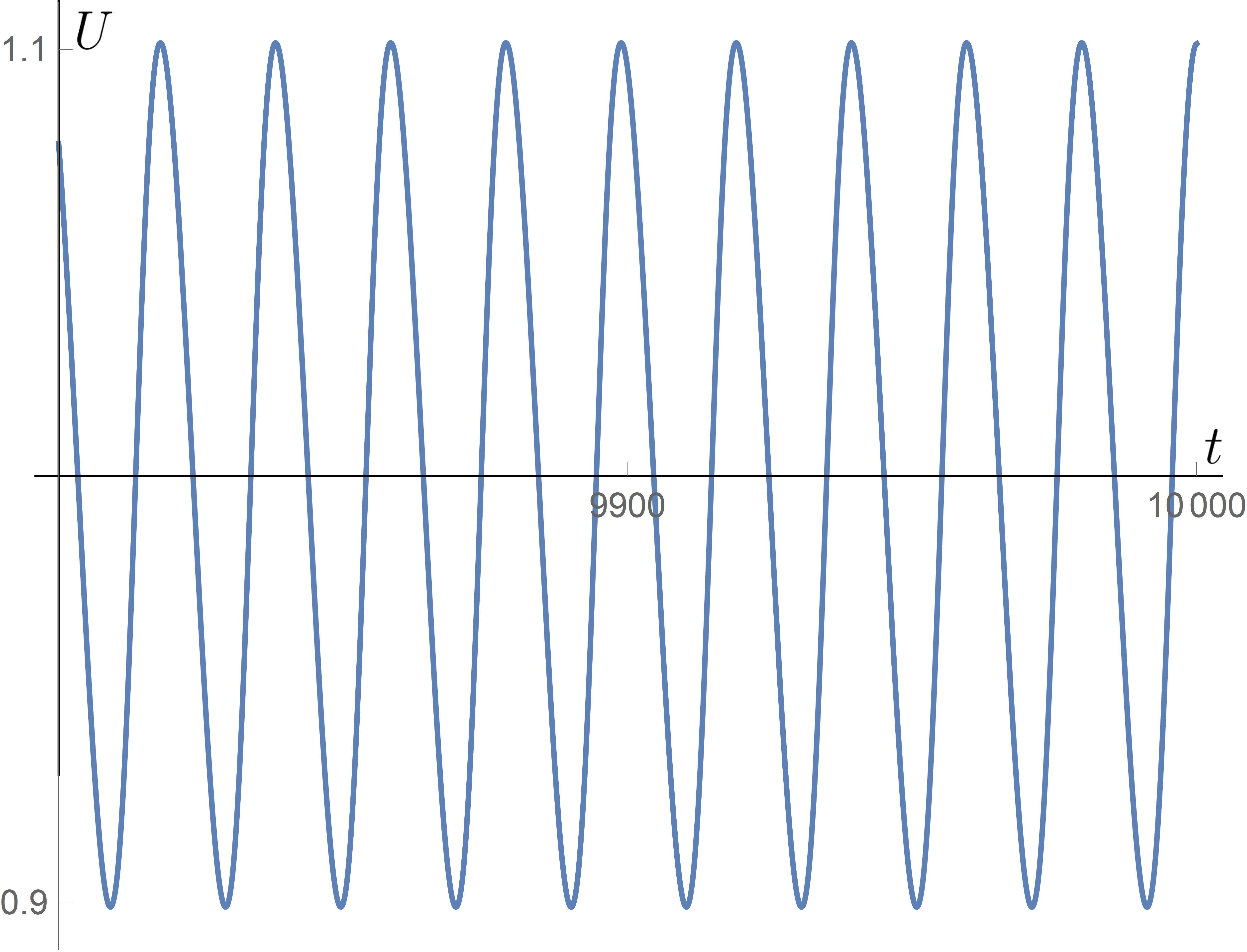}
%		\caption{}
	\end{subfigure}
        \begin{subfigure}[t]{0.32\textwidth}
		\centering
		\includegraphics[width=\linewidth]{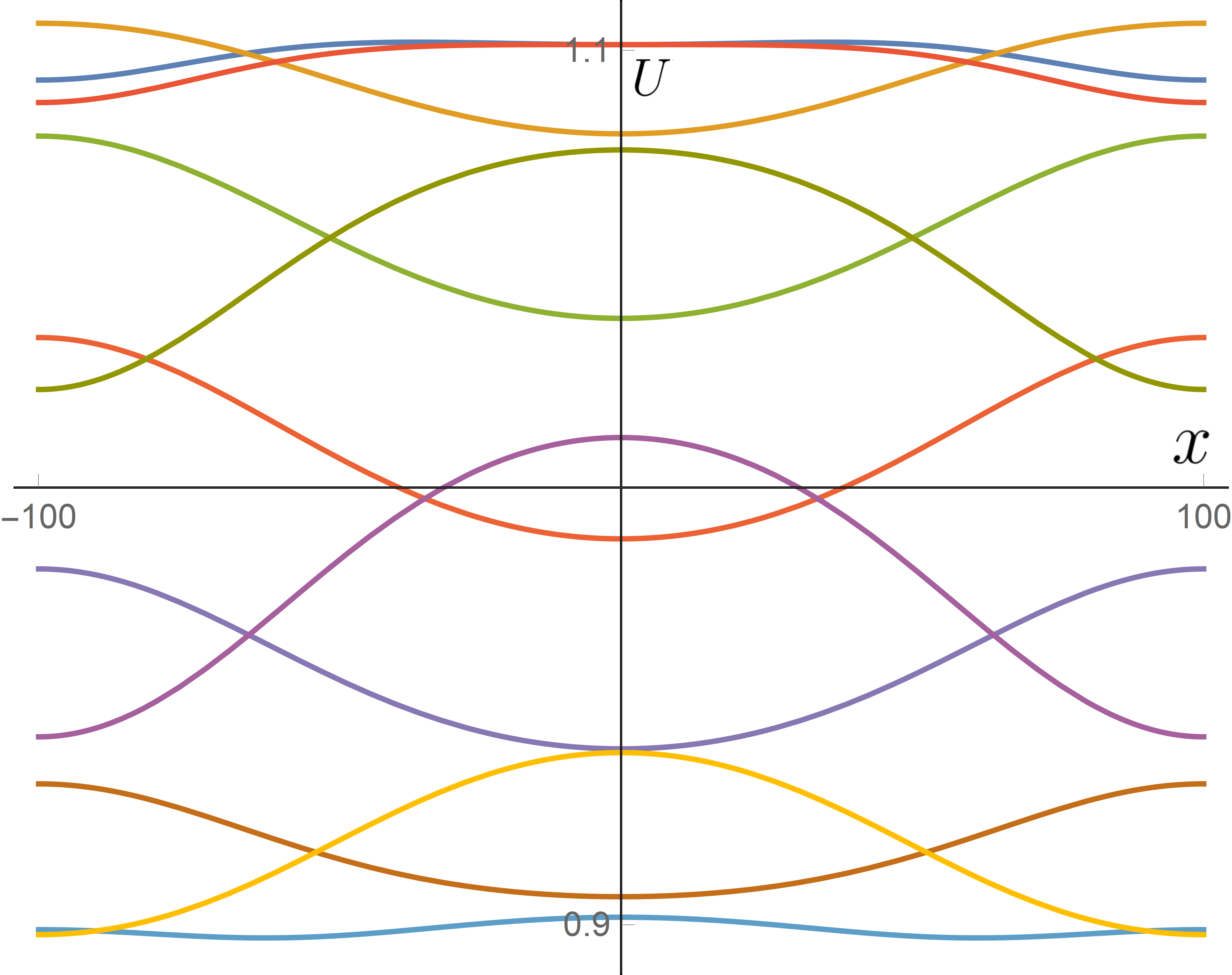}
%		\caption{}
	\end{subfigure}
        \begin{subfigure}[t]{0.32\textwidth}
		\centering
		\includegraphics[width=\linewidth]{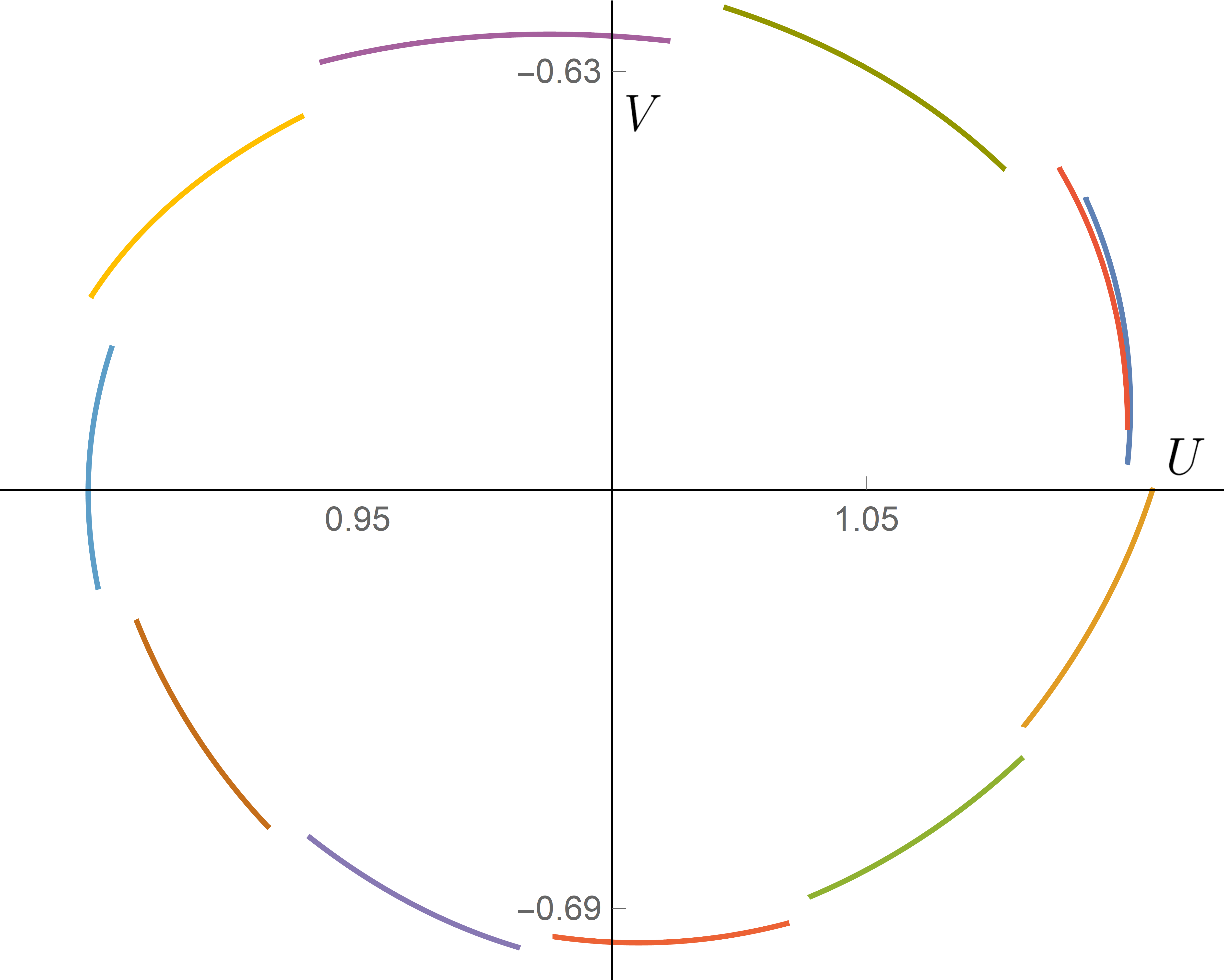}
%		\caption{}
	\end{subfigure}
	\\
\vspace{.75cm}
%        \begin{subfigure}[t]{0.32\textwidth}
%		\centering
%		\includegraphics[width=\linewidth]{FinalPattern-int-x0-a0999}
%		\caption{}
%	\end{subfigure}
        \begin{subfigure}[t]{0.49\textwidth}
		\centering
		\includegraphics[width=\linewidth]{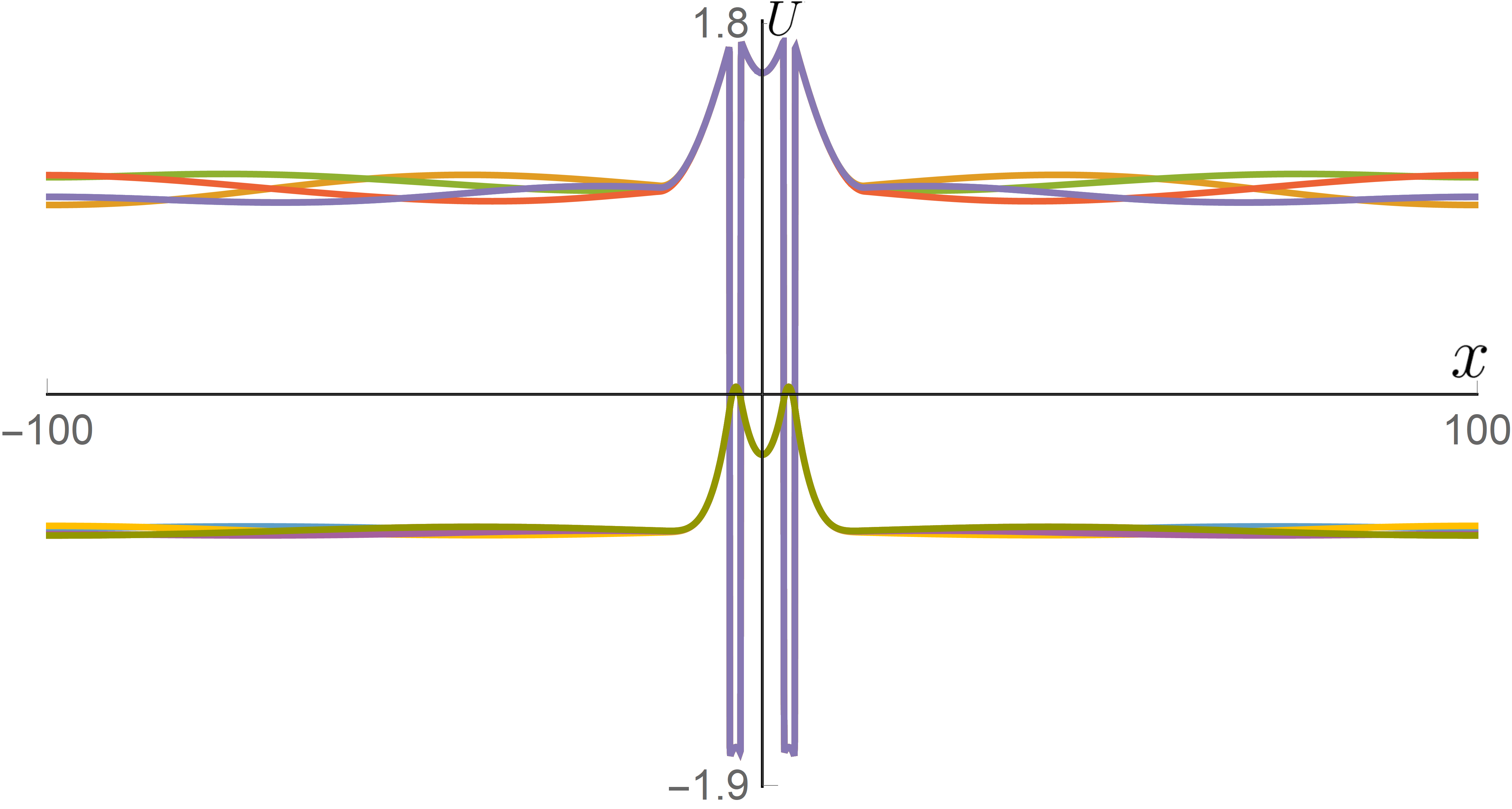}
%		\caption{}
	\end{subfigure}
        \begin{subfigure}[t]{0.49\textwidth}
		\centering
		\includegraphics[width=\linewidth]{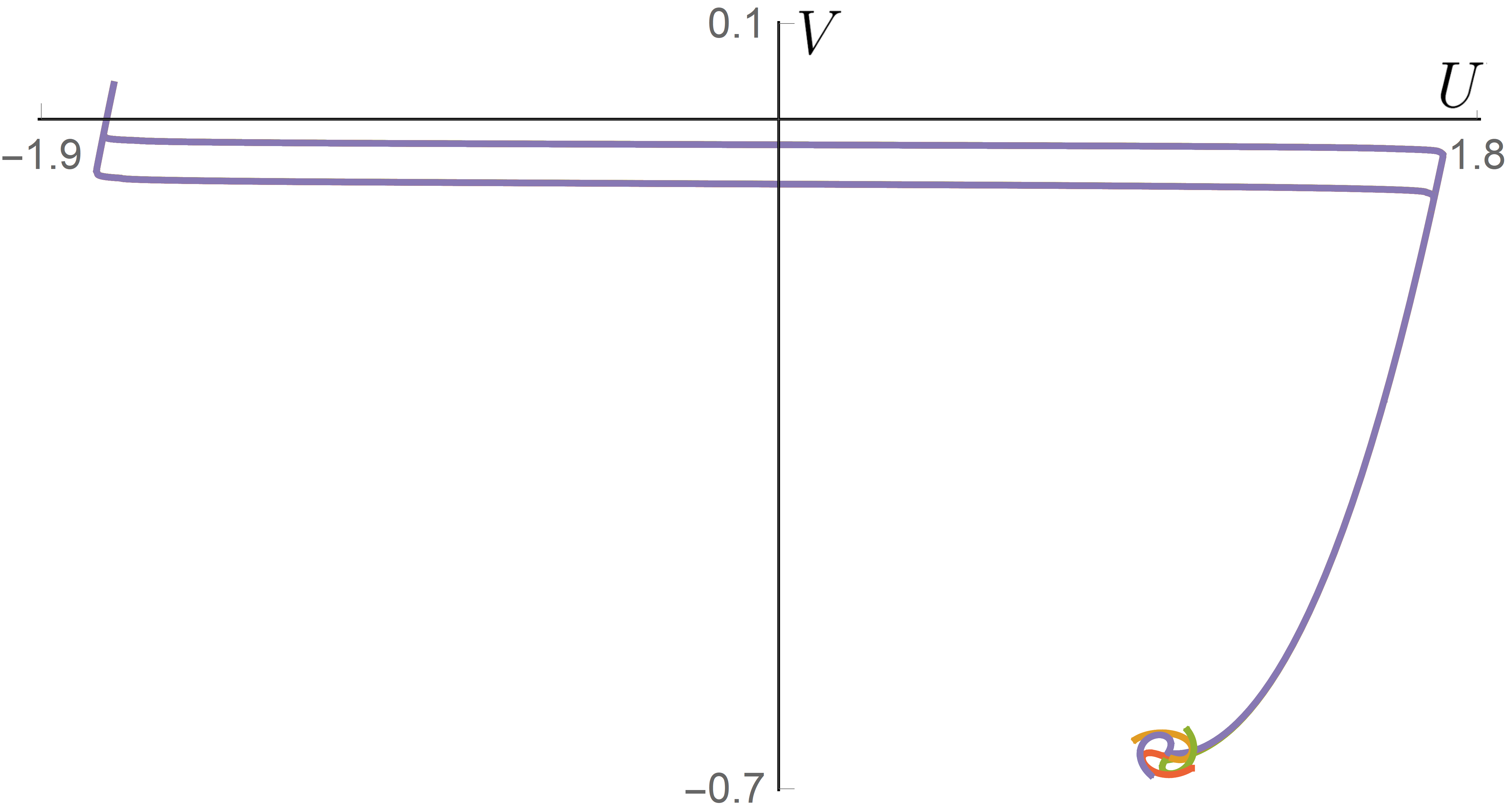}
%		\caption{}
	\end{subfigure}
\caption{Non-stationary attractors of \eqref{eq:vdp} obtained by simulations starting from localized initial data with a spatial extent that is larger than that used  in Figs.~\ref{fig:Sims1}, \ref{fig:Sims2}, and \ref{fig:Sims3}, $(u_0(x),v_0(x)) = (1.00  + 0.05\,[\tanh(10(x - 4)) - \tanh(10(x + 4))], -2/3 - 0.025\,[\tanh(0.5(x - 8)) - \tanh(0.5(x + 8))])$ for $x \in (-100,100)$, with homogeneous Neumann boundary conditions and $(\varepsilon, \delta) = (0.1,0.01)$. 
Top row: $a=1$, a time-periodic, (uniformly) small-amplitude attractor with a long spatial wavelength is formed. Left: $U(0,t)$ for $t \in (9800,1000)$; middle: $U(x,t_j)$ for $x \in (-100,100)$ with $t_j = 9980 + 2j$, $j=0, ..., 10$; right: $(U(x,t_j), V(x,t_j))$ for $x \in (-100,100)$ in the $(U,V)$-plane with again $t_j = 9980 + 2j$, $j=0, ..., 10$. Note that the $t=t_j$-patterns circle closely around the critical point $(a,f(a)) = (1.0, -2/3)$. Bottom row: $a=0.999$, {\it i.e.}, just beyond the Hopf bifurcation, but still before the Turing bifurcation. Left: $U(x,t_j)$ and $V(x,t_j)$ for $x \in (-100,100)$ with $t_j = 9980 + 5j$, $j=0, ..., 4$; right: $(U(x,t_j), V(x,t_j))$ for $x \in (-100,100)$ in the $(U,V)$-plane with again $t_j = 9980 + 5j$, $j=0, ..., 4$. Note that the `circling spiral' around $(a,f(a))$ has the same magnitude as the (non-spiralling) circle in the $(U,V)$ plane of the $a=1$ case.} 
\label{fig:Sims4}
\end{figure}
\smallskip

\begin{remark}
The types of stationary, multi-front patterns shown in Fig.~\ref{fig:Sims1} have been constructed for various problems in the literature on singularly perturbed reaction-diffusion equations (see for example \cite{BDHL2023} and references therein). 
However, in the known constructions, the asymptotic homogeneous states of the homoclinic patterns correspond to critical points on normally hyperbolic manifolds, which is of course not the case here.
\end{remark}

%------------------------------------------------------------------
\section{Conclusions and Outlook \label{sec:conclusions}}
%-----------------------------------------------------------------%------------------------------------------------------------------
\subsection{Summary \label{subsec:summary}}
%------------------------------------------------------------------

In this article, we reported on the discovery of classes of spatially periodic canard solutions that emerge from Turing bifurcations in the van der Pol PDE \eqref{eq:vdp} in one dimension, a phenomenon that we have dubbed ``Turing's canards".
The canards that we studied analytically and numerically include classes of small-amplitude and large-amplitude spatially periodic canard solutions, with large, $\mathcal{O}(1)$, and small wavenumbers. (See Figs.~\ref{fig:kOrderOne_SAO}--\ref{fig:kOrderDelta_LAO}, as well as  Figs. \ref{fig:deconstructSAO} and \ref{fig:deconstructLAO} for representative canards, and see  Figs.~\ref{fig:frommole2manatee},  \ref{fig:bifndetailed}, and \ref{fig:isola}--\ref{fig:transition_doubleloop} for bifurcation diagrams.) 
Furthermore, we observed numerically that several of these classes of spatially-periodic canards are attractors in the PDE.

The spatial ODE system, recall \eqref{eq:spatialODE-y}, governs time-independent solutions of the PDE.
It has reversible, 1:1 resonant Hopf bifurcation points exactly at the parameter values where the PDE undergoes Turing bifurcations.
Quartets of eigenvalues merge there into two identical purely imaginary pairs, and hence hyperbolicity of the equilibrium/homogeneous state is lost,
recall Fig.~\ref{fig:spatialeigenvalues} and Proposition~\ref{prop:Turingcanards} in Sec.~\ref{sec:turingbifn}. 
We performed a complete analysis of the Turing/reversible 1:1 Hopf point at $a_T=\sqrt{1-2\delta\sqrt{\eps}}$, recall \eqref{eq:kTaT}, in the full four-dimensional phase space (and the results for the other RFSN-II points follow by symmetry).
We studied the two-dimensional fast/layer system (Sec.~\ref{sec:fast}) and the two-dimensional slow system (Sec.~\ref{sec:slow}).
Both are one-degree-of-freedom Hamiltonian systems, due to the reversibility symmetry of the full spatial ODE system.  
Our analysis showed that the critical manifolds, which govern the slow dynamics to leading order, are two-dimensional, cubic-shaped manifolds consisting of saddle points of the fast system on the left and right branches and of  center points on the middle branch. 

We identified the key folded singularities, namely the reversible folded saddle-node points of type II (RFSN-II points), that lie on the fold sets between the saddle and center branches of the cubic crirical manifold. 
These RFSN-II exist asymptotically close to the Turing bifurcations. 
Using the method of geometric desingularization (see Section~\ref{sec:desingFSNII}), we showed that the true and faux canards of the RFSN-II points are responsible for creating the spatially periodic canard patterns that we discovered, with the spatial canards having long segments near the true and faux canards. 

The analysis of the dynamics in the coordinate charts led to the discovery that there is a special algebraic solution $\Gamma_0$ (see Section~\ref{sec:K2}) in the rescaling chart that constitutes the core component of the maximal spatial canards in the full system. 
The orbit $\Gamma_0$ consists of two branches, one corresponding to each of the true and faux canards of the RFSN-II point, that approach the cusp point from above and below. 
It is the geometrically unique orbit in the rescaling chart that asymptotes to the two critical points in the extry-exit chart on the equator of the hemisphere, thereby serving as a separatrix (or ``river" type solution) over the hemisphere (in analogy to the parabola in the rescaling chart of the canard explosion in the fast-slow van der Pol ODE, recall \cite{KS2001}). 
Perturbation analysis of this algebraic solution then led to the calculation of a critical value  $a_c(\delta)$ 
(recall \eqref{eq:ac}) at which this orbit persists to leading order in $\delta$ (where the asympotic expansion was calculated using the dynamic, small-amplitude, perturbation parameter $r_2$ in the rescaling chart). 
This is the value at which the true and faux canards continue into each other.
Further analysis of key solutions focused on their smoothness in $\delta$, and critical values of $a$ were identified, the next one of which is $a_{c2}(\delta)$  (recall \eqref{eq:ac2}). 
These calculations were performed on a fixed level set of the conserved quantity $\tilde{\mathcal{G}}$, and can be generalized.  
All of these canards are maximal canards that have the longest segments near the true and faux canards of the RFSN-II points, and they serve as boundaries in phase and parameter space separating spatially periodic canards of different profiles.

The dynamics of the canards change as one moves along the isolas of periodic solutions in the $(a,k)$ parameter plane. 
One elementary change along branches of isolas occurs when the length of the canard segments grows as $a$ changes (recall the green branch in Fig.~\ref{fig:transition_growslow}). 
Next, for branches of isolas of spatially periodic solutions with canard segments along the true and faux canards of the RFSN-II point on the right slow manifold $S_s^+$, we studied the nucleation of interior spikes at the RFSN-II point on the left critical manifold $S_s^-$ (recall the orange branch in Fig.~\ref{fig:transitionspikeformation}). 
In addition, we followed these new spikes through the transition in parameter space from small to large spikes (recall  Section~\ref{subsec:spikeformation}).
A number of further bifurcations of spatially periodic canards were also observed, including a bifurcation to canards with double loops (recall Fig.~\ref{fig:transition_doubleloop}).

Self-similarity plays a central role in the spatially periodic canards. 
We showed that, in the singular limit $\delta=0$, the zero level set of the Hamiltonian of the slow system is scale invariant, and hence that it has an infinite self-similarity. 
This self-similarity manifests through a sequence of crossing points and small, nested loops. 
Then, we observed that for $0<\delta\ll 1$, the self-similarity is broken. 
The spatial patterns exhibit nearly self-similar dynamics, recall Fig.~\ref{fig:selfsimilar}.
Solutions with different numbers of nested, successively-smaller, nearly self-similar loops are shown in Figs.~\ref{fig:smallamp-selfsimilar} and \ref{fig:kOrderDelta_LAO}. 
Moreover, the wavenumber decreases as the number of successive loops increases.

The spatial canards are analogs in spatial dynamics of the situation in fast-slow ODEs with time-periodic limit cycle canards, where the explosion of temporal canards occurs near --and asymptotically close to-- the singular Hopf bifurcation. 
The new spatially periodic canards were found  (recall Sec.~\ref{sec:spatialdynamicsanalog}) to have many features in common with the classical temporal limit cycle canards, as well as several  important new features. The most interesting of these new features is that the critical manifold in the spatial ODE system is two-dimensional and consists of branches of saddle equilibria and center equilibria of the fast system, which contrasts with the one-dimensional critical manifolds of attracting and repelling equilibria that give rise to canard explosions in fast-slow ODEs.

To complement the above results for the spatial ODE system, we also studied some basic dynamical properties of the PDE \eqref{eq:vdp}, recall Section~\ref{sec:pdedyn}.  
We showed that the Turing bifurcation from which the spatially periodic canards is sub-critical for the PDE.
The standard Ginzburg-Landau theory shows that small-amplitude perturbations should grow because the coefficients of both the linear and cubic terms are positive, and it cannot determine what the nonlinear saturation mechanisms might be. 
Hence, this study also sheds new light on what attractors can exist in the sub-critical case.

The PDE simulations that we have performed so far showed that the large-amplitude, spatially periodic canards are attractors in the full PDE \eqref{eq:vdp}. With localized initial data, which are at the homogeneous state over most of the interval and have localized $\tanh$-shaped perturbations (`plateaus'), we observed that, when the plateaus are not too wide, the attractors are stationary, large-amplitude canard patterns that are homoclinic in space to the (stable) homogeneous state $(a,f(a))$ for $a \gtrsim a_T$.
Then, for $a \lesssim a_T$, with the same initial data (plateaus not too wide), the attractors are large-amplitude canard patterns homoclinic to spatially periodic patterns, since the homogeneous state is linearly unstable here. Recall Figs.~\ref{fig:Sims1}-\ref{fig:Sims3}. 
By contrast, when the plateaus in the initial data are too wide, then the attractors exhibit time periodic dynamics for $a<1$, due to the Hopf bifurcation.
In addition, a class of small-amplitude attractors is observed in the PDE. 
For $1 > a \gtrsim a_T$, these are periodic in both time and space, and they are observed for initial data in which the support is sufficiently wide, greater than a threshold. 
By contrast, for $a \lesssim  a_T$, the same iniital data develops into a large-amplitude attractor whose spatial tails vary periodically in time.   
Recall Fig.~\ref{fig:Sims4}. 
Furthermore, the simuations suggest that there is a rich interplay between the Turing and Hopf modes, especially since the real parts of the dominant eigenvalue are both of size $\mathcal{O}(\delta)$ at the Hopf value $a=1$ and the Turing bifurcation $a_T=\sqrt{1-2\delta\sqrt{\eps}}$ (recall Fig.~\ref{fig:EVcurves}), and it is expected that the PDE dynamics are governed by coupled Ginzburg-Landau equations (one for each mode), as described in Section~\ref{sec:pdedyn}.

%------------------------------------------------------------------
\subsection{RFSN-II points and spatial canards in general reaction-diffusion systems \label{subsec:gen}}
%------------------------------------------------------------------

In this brief section, we generalize some of the main results about the classes of small-amplitude, spatially periodic canards asymptotically close to the Turing bifurcation value $a_T$ established here for the van der Pol PDE \eqref{eq:vdp} to other reaction-diffusion systems with separated diffusivities, which undergo Turing bifurcations and which have reaction kinetics consisting of two or more branches separated by non-degenerate fold points.
In particular, we consider 
systems of reaction-diffusion PDEs
\begin{equation} \label{eq:genRDsys}
  \begin{split}
    u_t &= \hat{f}(u,v) + d u_{xx}, \\ 
    v_t &= \hat{g}(u,v) + v_{xx}, 
  \end{split}
\end{equation}
where $0<d\ll 1$.
To generalize the results about the RFSN-II points and their canards, system \eqref{eq:genRDsys} needs to have a sufficiently large open set in the $(u,v)$ plane on which the critical set $\hat{f}(u,v)=0$ has a locally unique (non-degenerate) quadratic solution $v=h_0(u)$  {\it i.e.}, a point $(u_0,v_0)$ in the open set such that $\hat{f}(u_0,h_0(u_0))=0$, $\frac{\partial \hat{f}}{\partial u} (u_0,h_0(u_0))=0$, and  $\frac{\partial^2 \hat{f}}{\partial u^2}(u_0,h_0(u_0)) \ne 0$.   
With this assumption, there will be a critical manifold consisting of a saddle branch and a center branch that meet along a fold curve in the $(u,p,v,q)$ space
(recall \eqref{eq:foldcurves} and Fig.~\ref{fig:S} for the van der Pol system).

Like all reaction-diffusion systems in one dimension, these multi-scale reaction-diffusion systems are invariant under the transformation $x \to -x$. 
Hence, the systems of spatial ODEs governing the shapes of time-independent patterns have a reversibility symmetry of the form $\mathcal{R}$, recall \eqref{eq:RF}, and the Turing bifurcation points in these systems correspond to reversible 1:1~resonant Hopf bifurcations in the spatial ODE systems. 
In addition, the fast system (or layer problem) 
\begin{equation} \nonumber
\begin{split}
    u_y &= p \\
    p_y &= -\hat{f}(u,v)
\end{split}
\end{equation}
will be Hamiltonian, with $v$ as a parameter, recall \eqref{eq:H-fast}.
Furthermore, the desingularized reduced vector fields on the critical manifold $S=\left\{  p = 0, v = h_0(u) \right\}$ will be
\begin{equation} \label{eq:desing-genRDsys}
  \begin{split}
    u_{x_d}&= q, \\ 
    q_{x_d} &= \frac{\hat{f}_u}{\hat{f}_v} \hat{g}(u,h_0(u)).
  \end{split}
\end{equation}
These systems are also Hamiltonian, just as the desingularized reduced system \eqref{eq:desing-reduced} is for the van der Pol system.  
Due to the reversibility symmetry, the ordinary singularities (fixed points) of the desingularized vector fields must be saddles, centers, or saddle-nodes, and the folded singularities must be reversible folded saddle (RFS), reversible folded center, or reversible folded saddle-node (RFSN) points, since the eigenvalues of the Jacobians must be symmetric with respect to the real and imaginary axes.
Finally, the true and faux canards of RFSN-II and RFS points will give rise to the canard segments of the spatially periodic solutions.

%------------------------------------------------------------------
\subsection{Discussion \label{subsec:disc}}
%------------------------------------------------------------------

Open questions about the full PDE \eqref{eq:vdp} have been listed at the end of Section~\ref{sec:pdedyn}. 
Here, we list some open questions about the spatial ODE system.

We are presently performing the geometric desingularization of the reversible folded saddle (RFS) points on the fold sets $L^\pm$ in \eqref{eq:vdp}, which were shown to exist for parameter values further from the Turing bifurcation (and hence from the RFSN-II singularity) in Section~\ref{sec:slow}.  
Simulations (see for example  Figs.~\ref{fig:transition_growslow}--\ref{fig:transition_doubleloop}) show that these RFS singularities are responsible for the creation of the spatially periodic canard solutions observed for $a$ values sufficiently far away from $a_T$. The further $a$ is from $a_T$, the true and faux canards of the RFS play a similar role in the geometric deconstructions of the spatially periodic solutions as those of the RFSN-II for $a$ near $a_T$. Also, the ordinary singularity $E$, which is a center for $a<1$, is located further from the fold set, and hence it would be of interest to determine how the number and structure of the loops changes as $a$ decreases.
Furthermore, in the opposite limit as $a \to 1$, preliminary calculations indicate  that the canards of the RFS points converge to the canards created by the RFSN-II.

The results here motivate a more general rigorous analysis of folded singularities in singularly perturbed systems of ODEs with reversibility symmetry, especially into the geometry of the invariant manifolds associated to RFSN-II and RFS points. 
We refer to \cite{KW2010} for the theory of general (not necessarily reversible) folded saddle nodes, and to \cite{Mitry2017} for the theory of general folded saddles.  
 
The three-component reaction-diffusion model studied in \cite{BDHL2023} consists of one activator and two inhibitors. 
The kinetics of the activator and first inhibitor are essentially those of the FitzHugh-Nagumo ODE. 
The kinetics of the second inhibitor are also linear.
All three species diffuse, with the diffusivity of the activator being asymptotically smaller than those of the inhibitors.
Overall, the reaction-diffusion subsystem for the activator and first inhibitor is of the general form \eqref{eq:genRDsys}.
Hence, it would be of interest to study the roles of the folded singularities and their canards in the onset of the spikes in the two-component FHN system in which the inhibitor also diffuses, as well as in the full three-component reaction-diffusion system.
The formation of the spike on the left branch, from $S_s^-$ due to the nearby RFSN-II point and its canards observed here for \ref{eq:vdp} (recall Sec.~\ref{subsec:spikeformation}) may explain the nucleation of some spikes in the three-component model of \cite{BDHL2023}.
In the other direction, analysis similar to that in \cite{BDHL2023} of how a small spike grows into a full-fledged spike should apply here for system \ref{eq:vdp}. 

Another question concerns the double asymptotic limit in which $\eps$ and $\delta$ are small. It is well-known that the asymptotic limit of small $\eps$ gives rise to temporal canards in the kinetics, and we have established here that the asymptotic limit of small $\delta$ gives rise to canards in the spatial dynamics. 
How do the temporal limit cycle canards of the kinetics problem for small $\eps$ interact with spatial canards?

\bigskip\noindent
{\bf Acknowledgments.}
The authors gratefully acknowledge Irv Epstein, Guido Schneider, and Gene Wayne,
for useful comments. We also thank Irv Epstein for bringing reference \cite{KLSBE2021} to our attention, for the possible evidence of spatial canards. 
The results in this article were presented by T.V. at the conference Multiscale Systems: Theory and Applications, held July 8-12, 2024, at the Lorentz Center, Leiden University, Leiden, NL.
NSF-DMS 1616064 and NSF-DMS 1853342 provided partial support to T.K. and T.V., respectively. 
The research of A.D. is supported by the ERC-Synergy project RESILIENCE (101071417).

\appendix 

%-----------------------------------------------------------
\section{Numerical methods}
\label{app:numericalmethods}
%-----------------------------------------------------------

%--------------------------
\subsection{Continuation of periodic solutions} 
%--------------------------

Periodic solutions of the system \eqref{eq:spatialODE-y} with non-trivial first integral $\mathcal G$ given by \eqref{eq:G} were computed and numerically continued using the method developed in \cite{GV2007}. More specifically, families of periodic solutions were computed by calculating solutions of the auxiliary system 
\begin{equation} \label{eq:continuation}
  \begin{split}
      \dot{\boldsymbol u} &= T \left(\boldsymbol{F}(\boldsymbol u) + \eta \nabla \mathcal G \right),
  \end{split}
\end{equation}
subject to periodic boundary conditions $\boldsymbol u(0) = \boldsymbol u(1)$. Here,
$\boldsymbol u = (u,p,v,q)$, $\boldsymbol F$ is the vector field in \eqref{eq:spatialODE-y}, $T$ is the spatial period, $\eta$ is a new auxiliary parameter (fixed at zero), and the overdot denotes the derivative with respect to the scaled spatial variable $\widehat y = \tfrac{y}{T}$.

The branches of $\mathcal G(\boldsymbol u) = g$, where $g$ is a constant, were computed by appending the integral constraint
\[ \int_0^1 \mathcal G(\boldsymbol u) \, d\widehat y = g, \]
to the system \eqref{eq:continuation} subject to the periodic boundary conditions $\boldsymbol u(0) = \boldsymbol u(1)$.

The numerical continuation was implemented using the continuation software AUTO \cite{AUTO}. 
The bifurcation diagrams for the families of spatially periodic canard solutions of \eqref{eq:spatialODE-y} shown in Figs.~\ref{fig:frommole2manatee},   \ref{fig:bifndetailed}, and later figures were also computed using AUTO.  We note that  additional branches of periodic solutions may exist.

%--------------------------
\subsection{Computation of saddle slow manifolds and maximal canards}  \label{app:numericalmethods_slowmanifolds}
%--------------------------
The saddle slow manifolds shown in Fig.~\ref{fig:slowmanifolds} were computed following the method developed in \cite{Hasan2018}. More specifically, the saddle slow manifolds were computed in two parts: one for the solutions enclosed by the true and faux canards of the folded saddle and one for the solutions outside that lie outside the region enclosed by the true and faux canards of the folded saddle. 

The subsets of $S_{s,\delta}^+$ outside the region enclosed by the true and faux canards of the folded saddle were computed by solving the system \eqref{eq:continuation}, where $\boldsymbol{u} = (u,p,v,q)$ and $\boldsymbol F$ is the vector field in \eqref{eq:spatialODE-y}, and subject to the boundary conditions
\begin{equation*}
\begin{split}
\boldsymbol u(0) \in \left\{ v = \tfrac{1}{3}u^3-u, \,\, q = q_0 : q_0 < 0 \right\} 
\quad \text{ and } \quad 
\boldsymbol u(1) \in \left\{ p = 0, \,\, u = 1 \right\}.
\end{split}
\end{equation*}
The left-end condition enforces the constraint that solutions enter the saddle slow manifold along the $p$-nullcline at a fixed $q$-distance from the folded singularity. The right-end condition is a statement that solutions leave the neighbourhood of the slow manifold along the $u$-nullcline and terminate at the fold. 

The subsets of $S_{s,\delta}^+$ enclosed by the true and faux canards of the folded saddle were computed by solving the system \eqref{eq:continuation}, where $\boldsymbol{u} = (u,p,v,q)$ and $\boldsymbol F$ is the vector field in \eqref{eq:spatialODE-y}, and subject to the boundary conditions
\begin{equation*}
\begin{split}
\boldsymbol u(0) \in \left\{ v = \tfrac{1}{3}u^3-u, \,\, u = u_0 : u_0 > 1 \right\},
\quad
\boldsymbol u(1) \in \left\{ v = \tfrac{1}{3}u^3-u \right\}, 
\quad \text{ and } \quad 
\left\{ q(0) + q(1) = 0 \right\}.
\end{split}
\end{equation*}
The left-end condition specifies that solutions enter the saddle slow manifold along the $p$-nullcline at a fixed $u$-distance from the fold. The remaining boundary conditions ensure that the solutions turn away from the fold and stay on the slow manifold. 

The maximal canards were then computed as saddle-node bifurcations of the above two-point boundary value problems.

%--------------------------
\subsection{Continuation of canard orbits in chart \texorpdfstring{$K_2$}{Lg}}  \label{sec:appA2}
%--------------------------
For the results shown in Section~\ref{sec:selfsimilar}, solutions of the blown-up system \eqref{eq:ODE-K2} on the zero level set $H_2(u_2,p_2,v_2,q_2,r_2) = 0$ were computed and numerically continued by solving the auxiliary problem \eqref{eq:continuation}, where $\boldsymbol{u} = (u_2,p_2,v_2,q_2)$, $\boldsymbol F$ is the vector field in \eqref{eq:ODE-K2}, and the conserved quantity is $\mathcal{G} = H_2$. These equations were solved subject to the boundary conditions 
\[ p_2(0) + p_2(1) = 0, \quad \text{ and } \quad q_2(0) + q_2(1) = 0, \]
which enforces the symmetry $(u_2,p_2,v_2,q_2) \to (u_2,-p_2,v_2,-q_2)$, together with the integral constraint
\[ \int_0^1 H_2(u_2,p_2,v_2,q_2,r_2)\,d\widehat{y} = 0 \]
which constrains the solution to the zero level contour of the Hamiltonian. 

%--------------------------
\subsection{ODE and PDE Simulations} 
%--------------------------
Finally, direct numerical simulations of solutions of the spatial dynamics of \eqref{eq:spatialODE-y} were carried out using Mathematica's in-built ODE solvers.
Direct numerical solutions of the full PDE \eqref{eq:vdp} were performed using
Mathematica's NDSolveValue-package with the ``MethodOfLines'' and ``SpatialDiscretization'' prescibed by $\{$``TensorProductGrid'', ``MinPoints'' $\to$ 10000$\}$.

%------------------------------------------------------------------
\section{The proof of Proposition  \texorpdfstring{\ref{prop:Turingcanards}}{Lg} \label{sec:app-proof}}
%------------------------------------------------------------------

In this appendix, we prove Proposition \ref{prop:Turingcanards}. by applying Theorem 3.21 from Chapter 4.3.3  of \cite{HI2011}.  
We work with the spatial ODE system in a general form, 
\begin{equation}
\label{eq:appNF-spatialODE}
    \begin{split}
        0 &= v - f(u) + d_u u_{xx} \\
        0 &= \eps(a-u) + d_v v_{xx},
    \end{split}
\end{equation}
where the diffusivities, $d_u$ and $d_v$ are positive. 

To derive the normal form, we rewrite these spatial ODEs as the following fourth-order system:
\begin{equation*}
    \begin{split}
        u_x &= p \\
        p_x &= \frac{1}{d_u} 
        (f(u)-v) \\
        v_x &= q \\
        q_x &= \frac{\eps}{d_v} (u-a).
    \end{split}
\end{equation*}
We translate the variables so that the equilibrium $(a,0,f(a),0)$ is at the origin and use the notation ${\bf u}=(u_1,u_2,u_3,u_4)$ of Chapter 4.3.3  for the dependent variables: 
$u=a+u_1, p = u_2, v = f(a) + u_3,$ and $q=u_4.$ 
Hence, the system is
\begin{equation}
\label{eq:appA-start}
    \begin{split}
        {u_1}_x &= u_2 \\
        {u_2}_x &= \frac{1}{d_u} 
        \left[ (a^2-1) u_1 + a u_1^2 + \frac{1}{3} u_1^3 - u_3  \right]\\
        {u_3}_x &= u_4 \\
        {u_4}_x &= \frac{\eps}{d_v} u_1.
    \end{split}
\end{equation}
The quartet of eigenvalues is given by
\begin{equation*}
    \pm\frac{1}{\sqrt{2}d_u d_v} \sqrt{ (a^2-1)d_u d_v^2 \pm i d_v^{3/2} \sqrt{ 4 \eps d_u - (a^2-1)^2 d_v} }.
\end{equation*}
At $a=a_T=\sqrt{1 - 2 \sqrt{\frac{\eps d_u}{d_v}}}$, the quartet degenerates into two coincident pairs of pure imaginary eigenvalues 
\begin{equation}
    \label{eq:omega}
\pm i \omega = \pm i \left(\frac{\eps}{d_u d_v}\right)^{1/4}.
\end{equation}
The eigenvalues $\pm i \omega$ are algebraically double and geometrically simple.
This is the non-degenerate, reversible, 1:1~resonant Hopf bifurcation identified in Section~\ref{sec:turingbifn}, recall also Fig.~\ref{fig:spatialeigenvalues} (where we note that $d_u=d=\delta^2$ and $d_v=1$ in \eqref{eq:vdp}).

In order to unfold this point, we set 
\begin{equation}
    \label{eq:HI-mu}
    a = a_T + \mu,
\end{equation}
where the parameter $\mu$ here is different from the spatial eigenvalue $\mu$ used in Section~\ref{sec:turingbifn},
and we write the system \eqref{eq:appA-start} as
\begin{equation}
    \label{eq:HI-form}
    {\bf u}_x = {\bf \mathcal{F}} ( {\bf u}, \mu) = {\bf L} {\bf u} 
      +{\bf R}_{20}({\bf u},{\bf u}) 
    + {\bf R}_{30}({\bf u},{\bf u},{\bf u}) 
    + \mu {\bf R}_{11} ({\bf u})
    + \mu {\bf R}_{21}({\bf u},{\bf u})
    + \mu^2 {\bf R}_{12}({\bf u}), 
\end{equation}
where the operators are defined as
\begin{equation*}
    {\bf L} = \left[
    \begin{array}{cccc}
        0 & 1 & 0 & 0 \\
        -2 \omega^2 & 0 & \frac{-1}{d_u} & 0 \\
        0 & 0 & 0 & 1 \\
        \frac{\eps}{d_v}& 0 & 0 & 0 
    \end{array}
    \right], \ \ 
    {\bf R}_{20}({\bf u},{\bf v}) = \left[
    \begin{array}{c}
        0 \\
        \frac{a_T}{d_u} u_1 v_1 \\
        0 \\
        0 
    \end{array}
    \right], \ \ 
    {\bf R}_{30}({\bf u},{\bf v}, {\bf w}) = \left[
    \begin{array}{c}
        0 \\
        \frac{1}{3d_u} u_1 v_1 w_1\\
        0 \\
        0 
    \end{array}
    \right], \ \ 
\end{equation*}
\begin{equation*}
    {\bf R}_{11}({\bf u}) = \left[
    \begin{array}{c}
        0 \\
        \frac{2a_T}{d_u} u_1\\
        0 \\
        0 
    \end{array}
    \right], \ \ 
    {\bf R}_{21}({\bf u},{\bf v}) = \left[
    \begin{array}{c}
        0 \\
        \frac{1}{d_u} u_1 v_1 \\
        0 \\
        0 
    \end{array}
    \right], \ \ 
    {\bf R}_{12}({\bf u}) = \left[
    \begin{array}{c}
        0 \\
        \frac{1}{d_u} u_1\\
        0 \\
        0 
    \end{array}
    \right]. \ \ 
\end{equation*}
The terms represent, respectively  ${\bf L}$: the Jacobian at the origin at $a_T$; 
${\bf R}_{20}$: 
the quadratic terms in ${\bf u}$;
${\bf R}_{30}$: 
the cubic terms in ${\bf u}$;  ${\bf R}_{11}$ and ${\bf R}_{21}$: the terms  that are linearly proportional to the unfolding parameter $\mu$, and ${\bf R}_{12}$: the term proportional to $\mu^2$. (We  follow the general notation in \cite{HI2011} for the unfolding of this bifurcation, labeled as $(i\omega)^2$ in Chapter 4.3.3. The subscripts on ${\bf R}$ indicate the powers of ${\bf u}$ and $\mu$, respectively, in each term in \eqref{eq:HI-form}; for example, the term involving ${\bf R}_{21}$ is quadratic in ${\bf u}$ and linear in $\mu$.)

For ${\bf L}$, the Jacobian at $a_T$, we use the following eigenvector and generalized eigenvector associated to $i\omega$: 
\begin{equation}
    \label{eq:zeta0}
    \zeta_0 = \left[
    \begin{array}{c}
        1 \\
        i \omega  \\
    \frac{-\eps}{\omega^2 d_v} \\
    \frac{-i \eps}{\omega d_v} 
    \end{array}
    \right], 
    \quad
    \zeta_1 = \left[
    \begin{array}{c}
        \frac{-i}{\omega}\\
        2 \\
    \frac{-i\eps}{\omega^3 d_v} \\
    0 
    \end{array}
    \right].
\end{equation}
These satisfy $({\bf L}-i\omega)\zeta_0=0$ and $({\bf L}-i\omega)\zeta_1 = \zeta_0$. 
Moreover, since ${\bf L}$ is real, one also has 
$({\bf L} + i \omega) \bar{\zeta}_0 = 0$
and $({\bf L} + i \omega) \bar{\zeta}_1 = \bar{\zeta}_0$,
where the overbar denotes the complex conjugate.
The set $ \{ \zeta_0, \zeta_1, \bar{\zeta}_0,\bar{\zeta}_1 \}$ is used as a basis for $\mathbb{R}^4$, and we represent the vector ${\bf u} = A \zeta_0 + B \zeta_1 + \bar{A} \bar{\zeta}_0 + \bar{B} \bar{\zeta}_1$ by ${\bf u}=(A,B,\bar{A},\bar{B})$, where $A, B \in \mathbb{C}$.

Now, we derive the normal form of \eqref{eq:HI-form}. 
It may be obtained directly by applying Lemma 3.17 in Chapter 4.3.3 of \cite{HI2011}, as follows. 
The linear part, ${\bf L}$, is conjugate to the block Jordan matrix,
\begin{equation}
\label{eq:J-appHI}
    {\bf J} = \left[
    \begin{array}{cccc}
        i \omega & 1 & 0 & 0 \\
        0 & i \omega  & 0 & 0 \\
        0 & 0 & -i \omega  & 1 \\
        0 & 0 & 0 & -i \omega 
    \end{array}
    \right]. 
\end{equation}
For the nonlinear part, which also includes the $\mu$-dependent terms, we observe that Hypotheses 3.1, 3.2, and 3.14 from \cite{HI2011} are satisfied
for every integer $k \ge 3$, {\it i.e.,} the vector is $C^k$ for every $K\ge 3$ since it is polynomial.  Hence, Lemma 3.17 establishes that, for any integer $p$ with $2 \le p \le k$,  there exist neighborhoods $\mathcal{V}_1$ of the origin in $\mathbb{R}^4$ and $\mathcal{V}_2$ of the origin in $\mathbb{R}$ and a real-valued polynomial $\Phi$ of degree $p$ such that the near-identity coordinate change,
${\bf u} = A \zeta_0 + B \zeta_1 + \bar{A} \bar{\zeta}_0 + \bar{B} \bar{\zeta}_1 + \Phi(A,B,\bar{A},\bar{B},\mu)$ defined on $\mathcal{V}_1 \times \mathcal{V}_2$, transforms \eqref{eq:HI-form} into the following normal form:
\begin{equation}
    \label{eq:AB-generalNFform}
    \begin{split}
        A_x &= i \omega A + B + i A P\left(\mu, A \bar{A}, \frac{i}{2}(A \bar{B} - \bar{A} B) \right) + \rho_A \\
        B_x &= i \omega B 
        + i B P\left(\mu,A\bar{A}, \frac{I}{2}(A\bar{B}-\bar{A}B)\right) + A Q\left(\mu, A \bar{A}, \frac{i}{2}(A \bar{B}-\bar{A} B)\right) + \rho_B.
    \end{split}
\end{equation}
Moreover, $\Phi(0,0,0,0,0)=0$, $\partial_{(A,B,\bar{A},\bar{B})} \Phi(0,0,0,0,0)=0$, the coefficients of the monomials of degree $q$ in $\Phi(\cdot,\mu)$ are $C^{k-q}$ in $\mu$, and 
$\Phi$ satisfies the reversibility symmetry
$\mathcal{R} \Phi(A,B,\bar{A},\bar{B},\mu) = \Phi(\bar{A},-\bar{B}, A, -B, \mu)$.
Here, $P$ and $Q$ are real-valued polynomials of degree $p-1$.
The remainder terms $\rho_A(A,B,\bar{A},\bar{B},\mu)$ and $\rho_B(A,B,\bar{A},\bar{B},\mu)$ are $C^k$ functions, satisfy the estimate 
$|\rho_A| + |\rho_B | = o( (|A| + |B|)^p)$, and have the following symmetries:
$\mathcal{R} \rho_A(A,B,\bar{A},\bar{B},\mu) = -\bar{\rho}_A(\bar{A},\bar{B}, A, B, \mu)$ and
$\mathcal{R} \rho_B(A,B,\bar{A},\bar{B},\mu) = \bar{\rho}_B(\bar{A},\bar{B}, A, B, \mu)$ for each $\mu$.
(See the general normal form for reversible 1:1~resonant Hopf bifurcations given by equations (3.25) in Chapter 4.3.3 of  \cite{HI2011}.)

As shown in \cite{HI2011}, the existence of equilibria, spatially periodic orbits, quasi-periodic orbits, and homoclinic orbits of \eqref{eq:AB-generalNFform} may be determined by working with the principal parts of the polynomials $P$ and $Q$. That is, it suffices to work with $p=2$. 
Hence, throughout the rest of this appendix, we set
\begin{equation}
    P= \hat{\alpha} \mu + \hat{\beta} A \bar{A} + \frac{i}{2}\hat{\gamma} (A \bar{B} - \bar{A} B) \qquad {\rm and}\qquad Q = \hat{a} \mu + \hat{b} A \bar{A} + \frac{i}{2}\hat{c}(A\bar{B}-\bar{A}B).
\end{equation}

After lengthy calculations using the invariance equations, one finds:
\begin{equation*}
\label{eq:ahat-etal}
\begin{split} 
    \hat{a}&= \frac{a_T}{2 d_u}, \ \
    \hat{b}=\frac{25-8\sqrt{\frac{d_v}{\eps d_u}}}{36d_u}, \ \ 
    \hat{c}= \frac{109 d_u \left(\frac{d_v}{\eps}\right)^{1/4} - 32(\frac{d_v}{\eps})^{3/4}}{72 d_u^{5/4}}, \\ 
    \hat{\alpha}&= \frac{-a_T}{4 \omega d_u}, \ \ 
    \hat{\beta}= \frac{-1}{16}\left( \frac{d_v}{\eps d_u^3} \right)^{1/4},\ \ 
    \hat{\gamma}=  \frac{-32 d_v +37\sqrt{\eps d_u d_v} }{216 \eps d_u}.
    \end{split}
\end{equation*}
(Note that these six parameters are the same as those in Hypothesis 3.18 in \cite{HI2011}, except that we have added the hats.)

Next, we turn to the signs of these six parameters in the principal parts of $P$ and $Q$. Since $a_T>0$, we know $\hat{a}>0$ and $\hat{\alpha}<0$.
Also, $\hat{\beta}<0$. 
The sign of $\hat{b}$ depends on the ratios of the diffusivities:
\begin{equation}
\label{eq:signbhat}
\hat{b} < 0 \ \ {\rm for} \ \ \frac{\eps d_u}{d_v} < \frac{64}{625},
        \quad {\rm and} \quad 
    \hat{b} > 0 \ \ {\rm for} \ \ \frac{\eps d_u}{d_v} > \frac{64}{625}.
\end{equation}
This sign analysis of $\hat{b}$ determines the parameter conditions in Proposition~\ref{prop:Turingcanards}. The former is referred to as the focusing/sub-critical case, and the latter as the defocusing/super-critical case.
The sign of $\hat{c}$ is given by:
\begin{equation*}
    \hat{c} > 0 \ \ {\rm for} \ \ d_u >  \frac{32}{109}\sqrt{\frac{d_v}{\eps}}, \quad {\rm and} \quad 
    \hat{c} < 0 \ \ {\rm for} \ \ d_u < \frac{32}{109}\sqrt{\frac{d_v}{\eps}}.
\end{equation*}
Finally, the sign of $\hat{\gamma}$ is given by
\begin{equation*}
    \hat{\gamma} > 0 \ \ {\rm for} \ \ d_u >  \left(\frac{32}{37}\right)^2\frac{d_v}{\eps}, \quad {\rm and} \quad 
    \hat{\gamma} < 0 \ \ {\rm for} \ \ d_u < \left(\frac{32}{37}\right)^2 \frac{d_v}{\eps}.
\end{equation*}

%------------------------------------------------------------------
\section{The proof of Proposition~\ref{prop:61}}  {\label{sec:app61}}
%------------------------------------------------------------------

In this appendix, we prove Proposition~\ref{prop:61}.
\begin{proof}
We analyze the dynamics around the nilpotent equilibrium $(0,0,0)$ by performing the blow-up transformation
\[ u_2 = \overline{r}\,\overline{u}_2, \quad p_2 = \overline{r}\, \overline{p}_2, \quad \text{ and } \quad q_2 = \overline{r}\,\overline{q}_2, \]
with $(\overline{u}_2,\overline{p}_2,\overline{q}_2)\in \mathbb{S}^2$ and $\overline{r} \geq 0$, which inflates the nilpotent equilibrium to the unit sphere. For our analysis, we restrict attention to the half-space $\{ \overline{u}_2 \geq 0 \}$ which completely contains the algebraic solutions $\Gamma_0^{\pm}$. We will examine the dynamics in three coordinate charts: 
\[ K_{21}: \{ \overline{u}_2=1 \}, \quad K_{22}: \{ \overline{p}_2=1 \}, \quad \text{ and } \quad K_{23}: \{ \overline{q}_2 = 1\}, \]
and then appeal to the symmetry \eqref{eq:H2zerosymmetry} to obtain the dynamics on the remainder of the blown-up hemisphere. We denote an object $\Phi$ of \eqref{eq:H2zerodesing} in the blown-up coordinates by $\overline{\Phi}$ and in the charts $K_{2j}$ by $\Phi_{2j}$ for $j=1,2,3$. \medskip

{\em Dynamics in chart $K_{21}: \{ \overline{u}_2=1 \}$}: 
The blow-up transformation in chart $K_{21}$ is 
\[ u_2 = r_{21}, \quad p_2 = r_{21}\, p_{21}, \quad \text{ and } \quad q_2 = r_{21}\,q_{21}. \]
Transformation and desingularization ($d\zeta_{21} = r_{21} d\eta_2$) gives
\begin{equation}
\begin{split}
    \dot{r}_{21} &= \sqrt{\eps} r_{21} p_{21} \\ 
    \dot{p}_{21} &= \tfrac{2}{3}r_{21}-\tfrac{1}{2} \sqrt{\eps} \left( p_{21}^2+q_{21}^2 \right) \\ 
    \dot{q}_{21} &= 1-\sqrt{\eps} p_{21}q_{21}
\end{split}
\end{equation}
where the overdot denotes the derivative with respect to the rescaled spatial coordinate $\zeta_{21}$. 

The set $\{ r_{21} = 0 \}$ is invariant. The dynamics in the $\{ r_{21} = 0 \}$ subspace are given by 
\begin{equation} \label{eq:K21rzero}
\begin{split}
    \dot{p}_{21} &= -\tfrac{1}{2} \sqrt{\eps} \left( p_{21}^2+q_{21}^2 \right) \\ 
    \dot{q}_{21} &= 1-\sqrt{\eps} p_{21}q_{21}.
\end{split}
\end{equation}
The coordinate change $(z_{21},w_{21}) = \tfrac{1}{2}\left( p_{21}-q_{21}, p_{21}+q_{21} \right)$ transforms the system \eqref{eq:K21rzero} to the decoupled system of Riccati equations
\begin{equation} \label{eq:K21rzeroRiccati}
\begin{split}
    \dot{z}_{21} &= -\sqrt{\eps} z_{21}^2-\tfrac{1}{2} \\ 
    \dot{w}_{21} &= -\sqrt{\eps} w_{21}^2+\tfrac{1}{2}.
\end{split}
\end{equation}
From this, we find that solutions of \eqref{eq:K21rzero} are given by 
\begin{equation*}
\begin{split}
    p_{21}(\zeta_{21}) &= \tfrac{1}{\sqrt{2}\eps^{1/4}} \left[ \tanh \left( \tfrac{\eps^{1/4}}{\sqrt{2}}\zeta_{21} + \tanh^{-1}\left( \sqrt{2}\eps^{1/4} \tilde{w}_{21} \right) \right) - \tan \left( \tfrac{\eps^{1/4}}{\sqrt{2}}\zeta_{21} - \tan^{-1}\left( \sqrt{2}\eps^{1/4} \tilde{z}_{21} \right) \right) \right] \\ 
    q_{21}(\zeta_{21}) &= \tfrac{1}{\sqrt{2}\eps^{1/4}} \left[ \tanh \left( \tfrac{\eps^{1/4}}{\sqrt{2}}\zeta_{21} + \tanh^{-1}\left( \sqrt{2}\eps^{1/4} \tilde{w}_{21} \right) \right) + \tan \left( \tfrac{\eps^{1/4}}{\sqrt{2}}\zeta_{21} - \tan^{-1}\left( \sqrt{2}\eps^{1/4} \tilde{z}_{21} \right) \right) \right]
\end{split}
\end{equation*}
where $\tilde{z}_{21}$ and $\tilde{w}_{21}$ are the values of $z_{21}$ and $w_{21}$ at $\zeta_{21}=0$. Among these solutions, we distinguish two particular solutions:
\begin{equation} \label{eq:ell21pm}
\ell_{21}^{\pm} := \left\{ p_{21}+q_{21} = \pm \sqrt{2}\eps^{-1/4} \right\}.  
\end{equation}
The solutions with initial conditions above $\ell_{21}^-$ are forward asymptotic (i.e., $\zeta_{21} \to +\infty$) to the line $\ell_{21}^+$
%
%\begin{equation} \label{eq:ell21p}
%\ell_{21}^+ := \left\{ p_{21}+q_{21} = \sqrt{2}\eps^{-1/4} \right\}.,
%\end{equation}
%
and the solutions with initial conditions below $\ell_{21}^+$ are backward asymptotic (i.e., $\zeta_{21} \to -\infty$) to the line $\ell_{21}^-$. 
%
%\begin{equation} \label{eq:ell21m}
%\ell_{21}^- := \left\{ p_{21}+q_{21} = -\sqrt{2}\eps^{-1/4} \right\}.
%\end{equation}
The dynamics in the invariant subspace $\{r_{21}=0\}$ are shown in Fig.~\ref{fig:secondblowupcharts}(a).

\begin{figure}[h!]
  \centering
  \includegraphics[width=\textwidth]{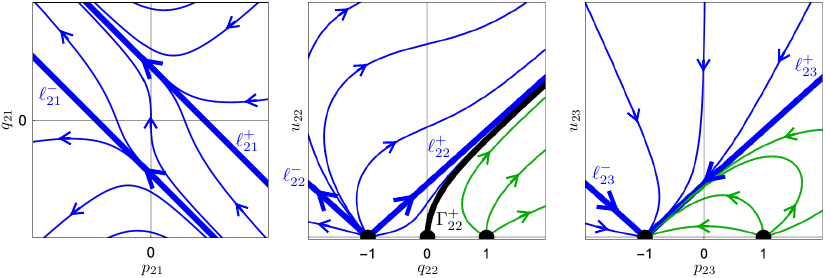}
  \put(-468,150){(a)}
  \put(-310,150){(b)}
  \put(-152,150){(c)}
  \caption{Dynamics of the unperturbed problem in the three main coordinate charts. (a) In chart $K_{21}$, the straight line solution $\ell_{21}^+$ is attracting and the straight line solution $\ell_{21}^-$ is repelling. (b) In chart $K_{22}$, the unstable manifold, $\Gamma_{22}^+$ (black curve), divides between (blue) solutions that emanate from $(u_{22},q_{22})=(0,-1)$ and (green) solutions that emanate from $(u_{22},q_{22})=(0,1)$. (c) In chart $K_{23}$, the straight line solution $\ell_{23}^+$ divides between (blue) solutions that are backward asymptotic to infinity and (green) solutions that are backward asymptotic to the equilibrium at $(u_{23},p_{23})=(0,1)$.}
  \label{fig:secondblowupcharts}
\end{figure}

{\em Dynamics in chart $K_{22}: \left\{ \overline{p}_2 = 1 \right\}$}: The blow-up transformation in chart $K_{22}$ is
\[ u_2 = r_{22} \, u_{22}, \quad p_2 = r_{22}, \quad \text{ and } \quad q_{2} = r_{22}\, q_{22}. \]
Transformation and desingularization ($d\zeta_{22} = r_{22} d\eta_2$) gives 
\begin{equation}
\begin{split}
    \dot{r}_{22} &= r_{22} \left( \tfrac{1}{2}\sqrt{\eps}(1-q_{22}^2)+\tfrac{2}{3} r_{22} u_{22}^3 \right) \\ 
    \dot{u}_{22} &= \tfrac{1}{2}\sqrt{\eps}u_{22}(1+q_{22}^2)-\tfrac{2}{3}r_{22} u_{22}^4 \\ 
    \dot{q}_{22} &= -\tfrac{1}{2}\sqrt{\eps}q_{22}(1-q_{22}^2)+u_{22}^2-\tfrac{2}{3}r_{22}u_{22}^3 q_{22},
\end{split}
\end{equation}
where the overdot has been recycled to denote the derivative with respect to $\zeta_{22}$. 

The line $\left\{ r_{22} = 0, u_{22} = 0 \right\}$ is invariant with dynamics governed by 
\[ \dot{q}_{22} = -\tfrac{1}{2}\sqrt{\eps} q_{22}(1-q_{22}^2). \]
The equilibrium at $q_{22}=0$ is stable, and the equilibria at $q_{22} = \pm 1$ are unstable. 

The plane $\{ u_{22} = 0 \}$ is also invariant. The dynamics restricted to this plane are given by 
\begin{equation}
\begin{split}
    \dot{r}_{22} &= \tfrac{1}{2}\sqrt{\eps}r_{22}(1-q_{22}^2) \\ 
    \dot{q}_{22} &= -\tfrac{1}{2}\sqrt{\eps}q_{22}(1-q_{22}^2).
\end{split}
\end{equation}
The system possesses a pair of lines of equilibria, $\mathscr{L}_{22,+}^u = \left\{ q_{22} = 1 \right\}$ and $\mathscr{L}_{22,-}^u = \left\{ q_{22} = -1 \right\}$, which are both center-unstable with eigenvalues $\lambda_u = \sqrt{\eps}$ and $\lambda_c = 0$. The associated eigenspaces are given by 
\[ \mathbb{E}^u\left( \mathscr{L}_{22,\pm}^u \right) = \begin{bmatrix} \mp r_{22} \\ 1 \end{bmatrix} \quad \text{ and } \quad \mathbb{E}^c\left( \mathscr{L}_{22,\pm}^u \right) = \begin{bmatrix} 1 \\ 0 \end{bmatrix}. \]
The lines $\mathscr{L}_{22,\pm}^u$ correspond to the center-unstable lines of equilibria $\mathscr{L}_{\pm}^u$. 

The plane $\{ r_{22} = 0 \}$ is invariant. The dynamics restricted to this subspace are given by 
\begin{equation} 
\begin{split}
    \dot{u}_{22} &= \tfrac{1}{2}\sqrt{\eps}u_{22}(1+q_{22}^2) \\ 
    \dot{q}_{22} &= -\tfrac{1}{2}\sqrt{\eps}q_{22}(1-q_{22}^2)+u_{22}^2.
\end{split}
\end{equation}
As shown in Fig.~\ref{fig:secondblowupcharts}(b), there is a saddle equilibrium at the origin, with stable eigenvalue $\lambda_s = -\tfrac{1}{2}\sqrt{\eps}$ and stable eigendirection aligned with the $q_{22}$-axis, and unstable eigenvalue $\lambda_u = \tfrac{1}{2}\sqrt{\eps}$ and unstable eigendirection aligned with the $u_{22}$-axis. 
Moreover, there is a pair of unstable degenerate nodes at $(u_{22},q_{22})=(0,\pm 1)$ with spectra given by $\sigma_u = \{ \sqrt{\eps}, \sqrt{\eps} \}$ and corresponding eigenspaces
\[ \mathbb{E}^u\left( 0,\pm 1 \right) = \operatorname{span} \left\{ \begin{bmatrix} 1 \\ 0 \end{bmatrix}, \begin{bmatrix} 0 \\ 1 \end{bmatrix} \right\}. \]

Thus, the equilibria $(r_{22},u_{22},q_{22})=(0,0,\pm 1)$ correspond to the intersections of the center-unstable lines, $\overline{\mathcal{L}}_{\pm}^u$, of equilibria (see \eqref{eq:Lminus} and \eqref{eq:Lplus}) with the blown-up hemisphere. 

We concl;ude with the following two key observations. First, there are two straight line solutions, 
\begin{equation} \label{eq:ell22pm}
\begin{split}
    \ell_{22}^{\pm} := \left\{ r_{22}=0, u_{22} = \pm \tfrac{1}{\sqrt{2}} \eps^{1/4} \left( q_{22}+1 \right) \right\}.
\end{split}
\end{equation}
Second, the unstable manifold, $W^u(0,0)=:\Gamma_{22}^+$, of the saddle equilibrium at the origin splits the $(u_{22},q_{22})$ phase space into two regions. Solutions with initial conditions to the left of $\Gamma_{22}^+$ are backward asymptotic to the equilibrium at $(r_{22},u_{22},q_{22})=(0,0,-1)$, and solutions with initial conditions to the right of $\Gamma_{22}^+$ are backward asymptotic to the equilibrium at $(r_{22},u_{22},q_{22})=(0,0,1)$. The dynamics in the $\{ r_{22}=0 \}$ subspace are shown in Fig.~\ref{fig:secondblowupcharts}(b).
\medskip

{\em Dynamics in chart $K_{23}: \{ \overline{q}_2 = 1 \}$}: The blow-up transformation in chart $K_{23}$ is given by 
\[ u_2 = r_{23}\,u_{23}, \quad p_2 = r_{23}\,p_{23}, \quad \text{ and } \quad q_2 = r_{23}. \]
Transformation and desingularization ($d\zeta_{23}=r_{23}d\eta_2$) gives
\begin{equation}
\begin{split}
    \dot{r}_{23} &= r_{23}u_{23}^2 \\ 
    \dot{u}_{23} &= u_{23}\left( \sqrt{\eps}p_{23}-u_{23}^2 \right) \\ 
    \dot{p}_{23} &= \tfrac{1}{2}\sqrt{\eps}\left( p_{23}^2-1 \right)-u_{23}^2p_{23}+\tfrac{2}{3} r_{23} u_{23}^3,
\end{split}
\end{equation}
where the overdot now denotes derivatives with respect to $\zeta_{23}$. 

The line $\{ r_{23}=0, u_{23}=0 \}$ is invariant. The dynamics on this line are governed by 
\[ \dot{p}_{23} = \tfrac{1}{2}\sqrt{\eps} \left( p_{23}^2-1 \right). \]
There is a stable equilibrium at $p_{23}=-1$ and an unstable equilibrium at $p_{23}=1$. 

The plane $\{ u_{23}=0 \}$ is invariant. The dynamics in the $(r_{23},p_{23})$ plane are given by 
\begin{equation}
\begin{split}
    \dot{r}_{23} &= 0 \\ 
    \dot{p}_{23} &= \tfrac{1}{2}\sqrt{\eps}\left( p_{23}^2-1 \right).
\end{split}
\end{equation}
The line $\mathscr{L}_{23,-}^s = \{ p_{23}=-1 \}$ of equilibria is center-stable and corresponds to the line $\mathscr{L}_-^s$. 
The line $\mathscr{L}_{23,+}^u = \{ p_{23}=1 \}$ of equilibria is center-unstable and corresponds to the line $\mathscr{L}_+^u$. 

The plane $\{ r_{23}=0 \}$ is invariant. The dynamics on this subspace are governed by 
\begin{equation}
\begin{split}
    \dot{u}_{23} &= u_{23}\left( \sqrt{\eps}p_{23}-u_{23}^2 \right) \\ 
    \dot{p}_{23} &= \tfrac{1}{2}\sqrt{\eps}\left( p_{23}^2-1 \right)-u_{23}^2p_{23}.
\end{split}
\end{equation}
As shown in Fig.~\ref{fig:secondblowupcharts}(c), there is a stable degenerate node at $(u_{23},p_{23})=(0,-1)$ with spectrum $\sigma_s=\{ -\sqrt{\eps},-\sqrt{\eps} \}$ and stable subspace 
\[ \mathbb{E}^s\left( 0,-1 \right) = \operatorname{span} \left\{ \begin{bmatrix} 1 \\ 0 \end{bmatrix}, \begin{bmatrix} 0 \\ 1 \end{bmatrix} \right\}.\]
There is also an unstable degenerate node at $(u_{23},p_{23})=(0,1)$ with spectrum $\sigma_u=\{ \sqrt{\eps},\sqrt{\eps} \}$ and unstable subspace 
\[ \mathbb{E}^u\left( 0,1 \right) = \operatorname{span} \left\{ \begin{bmatrix} 1 \\ 0 \end{bmatrix}, \begin{bmatrix} 0 \\ 1 \end{bmatrix} \right\}.\]
Solutions converge to the equilibium at $(u_{23},p_{23})=(0,-1)$ along the lines 
\begin{equation} \label{eq:ell23pm} 
\ell_{23}^{\pm} := \{ u_{23}=\pm \tfrac{1}{\sqrt{2}} \eps^{1/4}(p_{23}+1) \}.
\end{equation}
The dynamics in the $\{r_{23}=0\}$ subspace are shown in Fig.~\ref{fig:secondblowupcharts}(c). 

The equilibrium $(r_{23},u_{23},p_{23})=(0,0,1)$ corresponds to the intersection of the center-unstable line, $\overline{\mathcal{L}}_+^u$, of equilibria with the blown-up hemisphere. Similarly, the equilibrium $(r_{23},u_{23},p_{23})=(0,0,-1)$ corresponds to the intersection of the center-stable line, $\overline{\mathcal{L}}_-^s$, of equilibria with the blown-up hemisphere.
\medskip 

{\em Transition between charts $K_{21}:\{ \overline{u}_2=1 \}$ and $K_{22}:\{ \overline{p}_2=1 \}$}: The coordinate change, $\varphi_{12}(r_{21},p_{21},q_{21})$, from chart $K_{21}$ to chart $K_{22}$ is given by 
\[ \varphi_{12}(r_{21},p_{21},q_{21}) = (r_{22}, u_{22}, q_{22}) = \left( r_{21}p_{21}, \tfrac{1}{p_{21}}, \tfrac{q_{21}}{p_{21}} \right), \quad \text{ for } \,p_{21} >0. \]
The inverse map, $\varphi_{21}(r_{22},u_{22},q_{22})$, which transports coordinates from chart $K_{22}$ to $K_{21}$, is given by 
\[ \varphi_{21}(r_{22},u_{22},q_{22}) = (r_{21},p_{21},q_{21}) = \left( r_{22}u_{22}, \tfrac{1}{u_{22}},\tfrac{q_{22}}{u_{22}} \right), \quad \text{ for }\, u_{22}>0.  \]
Under this coordinate change, we find that 
\[ \lim_{\zeta_{21} \to -\infty} \varphi_{12} \left( \ell_{21}^{\pm} \right) = \lim_{\zeta_{22} \to -\infty} \ell_{22}^{\pm} = \{ (r_{22},u_{22},q_{22}) = (0,0,-1) \}. \]
Thus, solutions emanate from the point $\overline{\mathscr{L}}_-^u \cap \mathbb{S}^2$ on the equator of the blown-up hemisphere. 
See Fig.~\ref{fig:secondblowup}.
\medskip

{\em Transition between $K_{21}:\{ \overline{u}_2=1 \}$ and $K_{23}:\{ \overline{q}_2=1 \}$}: The coordinate change, $\varphi_{13}(r_{21},p_{21},q_{21})$, from chart $K_{21}$ to chart $K_{23}$ is given by 
\[ \varphi_{13}(r_{21},p_{21},q_{21}) = (r_{23},u_{23},p_{23}) = \left( r_{21}q_{21}, \tfrac{1}{q_{21}}, \tfrac{p_{21}}{q_{21}} \right), \quad  \text{ for } \, q_{21} >0.  \]
The inverse map, $\varphi_{31}(r_{23},u_{23},p_{23})$, which transports coordinates from chart $K_{23}$ to $K_{21}$, is given by
\[ \varphi_{31}(r_{23},u_{23},p_{23}) = (r_{21},p_{21},q_{21}) = \left( r_{23}u_{23}, \tfrac{p_{23}}{u_{23}}, \tfrac{1}{u_{23}} \right), \quad  \text{ for } \, u_{23} >0.  \]
From these transition maps, we find that the image of the attracting line $\ell_{21}^+$ is the line $\ell_{23}^+$ and is forward asymptotic to the equilibrium $(r_{23},u_{23},p_{23})=(0,0,-1)$. 
That is, 
\[ \lim_{\zeta_{21} \to \infty} \varphi_{13} \left( \ell_{21}^+ \right) = \lim_{\zeta_{23}\to \infty} \ell_{23}^+ = \{ (r_{23},u_{23},p_{23})=(0,0,-1) \}.  \]
Thus, solutions terminate at the point $\overline{\mathscr{L}}_-^a \cap \mathbb{S}^2$ on the blown-up hemisphere. 
See Fig.~\ref{fig:secondblowup}.
\medskip

{\em Transition between $K_{22}:\{ \overline{p}_2=1 \}$ and $K_{23}:\{ \overline{q}_2=1 \}$}:
The coordinate change, $\varphi_{23}(r_{22},u_{22},q_{22})$, from chart $K_{22}$ to chart $K_{23}$ is given by 
\[ \varphi_{23}(r_{22},u_{22},q_{22}) = (r_{23},u_{23},p_{23}) = \left( r_{22}q_{22}, \tfrac{u_{22}}{q_{22}}, \tfrac{1}{q_{22}} \right), \quad \text{ for } \, q_{22} > 0.  \]
The inverse map, $\varphi_{32}(r_{23},u_{23},p_{23})$, which transports coordinates from chart $K_{23}$ to $K_{22}$, is given by
\[ \varphi_{32}(r_{23},u_{23},p_{23}) = (r_{22},u_{22},q_{22}) = \left( r_{23}p_{23}, \tfrac{u_{23}}{p_{23}}, \tfrac{1}{p_{23}} \right), \quad \text{ for } \, p_{23} >0.  \]
The images of the lines $\ell_{22}^{\pm}$ are the lines $\ell_{23}^{\pm}$ and they are forward asymptotic to the equilibrium at $(r_{23},u_{23},p_{23})=(0,0,-1)$. That is,
\[ \lim_{\zeta_{22} \to \infty} \varphi_{23} \left( \ell_{22}^{\pm} \right) = \lim_{\zeta_{23}\to \infty} \ell_{23}^{\pm} = \{ (r_{23},u_{23},p_{23})=(0,0,-1) \}. \]
Thus, solutions terminate at the point $\overline{\mathscr{L}}_-^a \cap \mathbb{S}^2$ on the equator of the blown-up hemisphere. 
See Fig.~\ref{fig:secondblowup}.
\medskip

{\em Symmetry:} With the above analysis in hand, we can obtain the dynamics in the coordinate charts $K_{24}:=\left\{ \overline{p}_2 = -1 \right\}$ and $K_{25}:=\left\{ \overline{q}_2 = -1 \right\}$ via the symmetry \eqref{eq:H2zerosymmetry}. In particular, the image of the unstable manifold, $\Gamma_{22}^+$, of the saddle equilibrium at the origin in chart $K_{22}: \left\{ \overline{p}_2=1\right\}$ under the symmetry transformation \eqref{eq:H2zerosymmetry} is the stable manifold, $W^s(0,0,0) =: \Gamma_{24}^-$, of the saddle equilibrium at the origin in chart $K_{24}: \left\{ \overline{p}_2=-1\right\}$.
\medskip

{\em Dynamics on the hemisphere:} From our analysis of the dynamics in the charts $K_{21}, K_{22}$, and $K_{23}$, together with the transition maps $\varphi_{ij}$ between them, we conclude the following. 
\begin{itemize}
\item Solutions emanate from the point $\overline{\mathscr{L}}_-^u \cap \mathbb{S}^2$ on the equator of the blown-up hemisphere, travel over the top of the hemisphere in the region enclosed by $\overline{\Gamma}_0^-$ and $\overline{\Gamma}_0^+$, and terminate at the point $\overline{\mathscr{L}}_-^s \cap \mathbb{S}^2$ on the equator of the blown-up hemisphere. These are the class 1 heteroclinics.
\item Solutions emanate from the point $\overline{\mathscr{L}}_+^u \cap \mathbb{S}^2$ on the equator of the blown-up hemisphere, travel over the top of the hemisphere in the region enclosed by $\overline{\Gamma}_0^+$ and the equator, and terminate at the point $\overline{\mathscr{L}}_-^s \cap \mathbb{S}^2$ on the equator of the blown-up hemisphere. These are the class 2 heteroclinics. 
\item The unstable manifold, $\overline{\Gamma}_0^+$, of the equilibrium point corresponding to the intersection of the positive $\overline{p}_2$-axis with $\mathbb{S}^2$ is the separatrix that divides between class 1 and class 2 heteroclinics. 
\item Solutions emanate from the point $\overline{\mathscr{L}}_-^u \cap \mathbb{S}^2$ on the equator of the blown-up hemisphere, travel over the top of the hemisphere in the region enclosed by $\overline{\Gamma}_0^-$ and the equator, and terminate at the point $\overline{\mathscr{L}}_+^s$. These are the class 3 heteroclinics. 
\item The stable manifold, $\overline{\Gamma}_0^-$, of the equilibrium point corresponding to the intersection of the negative $\overline{p}_2$-axis with $\mathbb{S}^2$ is the separatrix that divides between class 1 and class 3 heteroclinics.
\end{itemize}

\end{proof}

%------------------------------------------------------------------
\section{Analysis of the governing equations in the entry/exit chart \texorpdfstring{$K_1$}{Lg}} {\label{sec:appK1}}
%------------------------------------------------------------------

In this appendix, we present the analysis of \eqref{eq:K1} in the sequence of invariant hyperplanes:
$\{ \delta_1=0\} \cap \{ a_1=0\}$, 
$\{ r_1=0\} \cap \{ a_1 = 0 \}$, 
$\{ r_1 = 0 \} \cap \{ \delta_1 = 0 \}$, 
$\{ \delta_1=0 \}$, 
and $\{ r_1=0\}$,
%and $\{ a_1=0\}$, 
respectively.
These are the intermediate results used in Section~\ref{sec:K1} to obtain the invariant sets and dynamics of the full system \eqref{eq:K1}.

In the invariant hyperplane
$\{\delta_1=0\} \cap \{ a_1=0 \}$, system \eqref{eq:K1} reduces to
\begin{equation}
        \label{eq:delta1zero}
\begin{split}
        \dot{r}_1 &= \frac{1}{2} \sqrt{\eps} p_1 r_1 \\
        \dot{p}_1 &= 1 - v_1 - \frac{3}{2}\sqrt{\eps} p_1^2
                              + \frac{1}{3}\sqrt{\eps}r_1^2 \\
        \dot{v}_1 &= -2 \sqrt{\eps} p_1 v_1 \\
        \dot{q}_1 &= -\frac{3}{2} \sqrt{\eps} p_1 q_1.
\end{split}
\end{equation}
The first three components of this vector field are independent of $q_1$, and hence the fourth equation decouples from the others.
The line $I$ given by \eqref{eq:I}
is also invariant for this larger system,
and $E_{\pm}$ are again fixed points on $I$,
recall \eqref{eq:Epm}.
The equilibrium $E_+$ is a saddle with three stable eigenvalues
$-\sqrt{6}\eps^{1/4}$,
$-2\sqrt{\frac{2}{3}}\eps^{1/4}$,
and $-\sqrt{\frac{3}{2}}\eps^{1/4}$,
and one unstable eigenvalue
$\frac{1}{\sqrt{6}} \eps^{1/4}$.
The stable subspace is given by
\begin{equation*}
        {\rm span}
        \left\{
        \left[ \begin{array}{c}
               0 \\ 1 \\ 0 \\ 0
        \end{array} \right],
\quad
\left[ \begin{array}{c}
        0 \\ -\sqrt{3} \\ \sqrt{2} \eps^{1/4} \\ 0
       \end{array} \right],
\quad
\left[ \begin{array}{c}
               0 \\ 0 \\ 0 \\ 1
       \end{array} \right]
\right\},
\end{equation*}
and the unstable subspace is in the $r_1$-direction,
\begin{equation*}
        {\rm span}
        \left\{
        \left[ \begin{array}{c}
               1 \\ 0 \\ 0 \\ 0
        \end{array} \right]
        \right\}.
\end{equation*}
The equilibrium $E_-$ is a saddle with three unstable eigenvalues
$\sqrt{\frac{3}{2}}\eps^{1/4}$,
$2\sqrt{\frac{2}{3}}\eps^{1/4}$,
and $\sqrt{6}\eps^{1/4}$,
and one stable eigenvalue
$-\frac{1}{\sqrt{6}} \eps^{1/4}$.
The unstable subspace is given by
\begin{equation*}
        {\rm span}
        \left \{
\left[ \begin{array}{c}
          0 \\ 0 \\ 0 \\ 1
       \end{array} \right],
\quad
\left[ \begin{array}{c}
           0 \\ \sqrt{3}\\ \sqrt{2} \eps^{1/4} \\ 0
        \end{array} \right],
\quad
        \left[ \begin{array}{c}
           0 \\ 1 \\ 0 \\ 0
        \end{array} \right]
\right\},
\end{equation*}
and the stable subspace is in the $r_1$-direction,
\begin{equation*}
        {\rm span}
        \left\{
        \left[ \begin{array}{c}
               1 \\ 0 \\ 0 \\ 0
             \end{array} \right]
        \right\}.
\end{equation*}

There is also a two-dimensional surface of equilibria
in the hyperplane $\{ \delta_1=0 \} \cap \{a_1=0\}$.
It is given by
\begin{equation}
\label{eq:surface}
        \mathcal{S}=\{ r_1 \in \mathbb{R},\,\,
                       \delta_1=0,\,\, p_1=0,\,\,
                       v_1=1 + \frac{1}{3}\sqrt{\eps}r_1^2,\,\,
                       q_1 \in \mathbb{R},\,\, a_1=0\}.
\end{equation}
The eigenvalues
of the Jacobian at points on $\mathcal{S}$
are
\begin{equation*}
        \lambda_s = -\sqrt{2\sqrt{\eps} + r_1^2 \eps}, \quad
        \lambda_u = \sqrt{2\sqrt{\eps} + r_1^2 \eps}, \quad
        \lambda_c = 0,0.
\end{equation*}
Hence, $\mathcal{S}$
is a surface of saddle points,
and it corresponds to the saddle branches
$S_s^\pm$ of the critical manifold $S$.
The associated stable, unstable, and center subspaces are
\begin{equation*}
\mathbb E_s = \operatorname{span} \left\{ \left[ \begin{array}{c}
        -\sqrt{\eps} r_1 \\
         2\sqrt{2\sqrt{\eps} + r_1^2 \eps}  \\
         \frac{4}{3}\sqrt{\eps}(3 + r_1^2 \sqrt{\eps})  \\
         3 \sqrt{\eps} q_1
       \end{array} \right] \right\},
\quad
\mathbb E_u = \operatorname{span} \left\{ \left[ \begin{array}{c}
            -\sqrt{\eps} r_1 \\
            -2\sqrt{2\sqrt{\eps} + r_1^2 \eps}  \\
            \frac{4}{3}\sqrt{\eps}(3 + r_1^2 \sqrt{\eps})  \\
            3 \sqrt{\eps} q_1
       \end{array} \right] \right\},
\quad
\mathbb E_c 
             = \left\{ \left[ \begin{array}{c}
                 0 \\  0 \\ 0 \\ 1
       \end{array} \right],
             \left[ \begin{array}{c}
                     3 \\  0 \\ 2 \sqrt{\eps} r_1 \\ 0
       \end{array} \right] \right\}.
\end{equation*}
The center manifold $W^c(\mathcal{S})$ is two-dimensional
in this hyperplane, and it emanates from $\ell$.

\medskip

Next, we examine the invariant hyperplane
$\{r_1=0\} \cap \{a_1=0\}$.
Here, system \eqref{eq:K1} reduces to
\begin{equation}
        \label{eq:r1zero}
\begin{split}
        \dot{\delta}_1 &= -\sqrt{\eps} p_1 \delta_1 \\
        \dot{p}_1 &= 1 - v_1 - \frac{3}{2}\sqrt{\eps} p_1^2 \\
        \dot{v}_1 &= \sqrt{\eps} (-2 p_1 v_1 + \delta_1 q_1)\\
        \dot{q}_1 &= -\frac{3}{2}\sqrt{\eps} p_1 q_1 + \delta_1.
\end{split}
\end{equation}
This system is fully coupled,
in contrast to the above system. The line $I$ given by \eqref{eq:I} is also invariant for this larger system, and $E_{\pm}$ are again fixed points on $I$, recall \eqref{eq:Epm}.
The equilibrium $E_+$ is stable with eigenvalues
$-\sqrt{6}\eps^{1/4}$,
$-2\sqrt{\frac{2}{3}}\eps^{1/4}$,
$-\sqrt{\frac{3}{2}}\eps^{1/4}$,
and $-\sqrt{\frac{2}{3}}\eps^{1/4}$.
The associated eigenvectors are
\begin{equation*}
        \left[ \begin{array}{c}
               0 \\ 1 \\ 0 \\ 0
        \end{array} \right],
\quad
\left[ \begin{array}{c}
        0 \\ -\sqrt{3} \\ \sqrt{2} \eps^{1/4} \\ 0
       \end{array} \right],
\quad
\left[ \begin{array}{c}
               0 \\ 0 \\ 0 \\ 1
       \end{array} \right],
\quad
\left[ \begin{array}{c}
        \eps^{1/4} \\ 0 \\ 0 \\ \sqrt{6}
       \end{array} \right].
\end{equation*}
The equilibrium $E_-$ is unstable with eigenvalues
$\sqrt{\frac{2}{3}}\eps^{1/4}$,
$\sqrt{\frac{3}{2}}\eps^{1/4}$,
$2\sqrt{\frac{2}{3}}\eps^{1/4}$,
and $\sqrt{6}\eps^{1/4}$, and eigenvectors 
\begin{equation*}
\left[ \begin{array}{c}
        -\eps^{1/4} \\ 0 \\ 0 \\ \sqrt{6}
       \end{array} \right],
\quad
\left[ \begin{array}{c}
          0 \\ 0 \\ 0 \\ 1
       \end{array} \right],
\quad
\left[ \begin{array}{c}
           0 \\ \sqrt{3}  \\ \sqrt{2} \eps^{1/4}\\ 0
        \end{array} \right],
\quad
        \left[ \begin{array}{c}
           0 \\ 1 \\ 0 \\ 0
        \end{array} \right].
\end{equation*}
Also, the line $\ell$ (recall \eqref{eq:ell})
is a line of saddle equilibria
of \eqref{eq:r1zero},
since the eigenvalues are
\begin{equation*}
        \lambda_s = -\sqrt{2} \eps^{1/4}, \quad
        \lambda_u = \sqrt{2} \eps^{1/4}, \quad
        \lambda_c = 0,0.
\end{equation*}
The associated stable, unstable, and center subspaces are
\begin{equation*}
\mathbb E_s = \operatorname{span} \left\{\left[ \begin{array}{c}
         0 \\
         2\sqrt{2}\eps^{-1/4} \\
         4 \\
         3 q_1
       \end{array} \right] \right\},
\quad
\mathbb E_u = \operatorname{span} \left\{ \left[ \begin{array}{c}
             0 \\
            -2\sqrt{2}\eps^{-1/4} \\
            4 \\
            3 q_1
       \end{array} \right] \right\},
\quad
\mathbb E_c = \operatorname{span}\left\{  
             \left[ \begin{array}{c}
                 0 \\  0 \\ 0 \\ 1
       \end{array} \right],\,\,
                \left[ \begin{array}{c}
        \frac{4}{4-3\sqrt{\eps}q_1^2}\\ \frac{2q_1}{4-3\sqrt{\eps}q_1^2} \\ 0 \\ 0
       \end{array} \right] \right\}.
\end{equation*}
Hence, \eqref{eq:r1zero} has a two-dimensional center manifold
\begin{equation}
N \equiv W^c(\mathcal{\ell}) \cap \{ r_1=0 \}.
\end{equation}
Moreover, for $\delta_1>0$, $N$ is unique in the half-space $\{ p_1 < 0 \}$.
Also, $\delta_1$ increases on $N$
in this half-space.

\medskip
Proceeding, we present the main results for the dynamics in the $\{ r_1=0\} \cap \{ \delta_1=0\}$ hyperplane. The equations here are given by system \eqref{eq:r1delta1a1zero} with the equation $\dot{a}_1= -\frac{3}{2}\sqrt{\eps} p_1 a_1$ appended. Since the $(p_1,v_1,q_1)$ subsystem is independent of $a_1$ on this hyperplane, the dynamics in these three variables are the same as the dynamics for \eqref{eq:r1delta1a1zero}, and then one may solve for $a_1$ using quadrature.
The line $I$ given by \eqref{eq:I} is invariant, and $E_{\pm}$ are fixed points on it.
The equilibrium $E_+$ is stable with eigenvalues
$-\sqrt{6}\eps^{1/4}$,
$-2\sqrt{\frac{2}{3}}\eps^{1/4}$,
$-\sqrt{\frac{3}{2}}\eps^{1/4}$, and $-\sqrt{\frac{3}{2}}\eps^{1/4}$, 
and eigenvectors
\begin{equation*}
{\boldsymbol w}_1^+ = \left[ \begin{array}{c}
               1 \\ 0 \\ 0 \\ 0
             \end{array} \right],
\quad
{\boldsymbol w}_2^+ = \left[ \begin{array}{c}
        -\sqrt{3} \\ \sqrt{2}\eps^{1/4} \\ 0 \\ 0
             \end{array} \right],
\quad
{\boldsymbol w}_3^+ = \left[ \begin{array}{c}
               0 \\ 0 \\ 1 \\ 0
             \end{array} \right],
             \quad
{\boldsymbol w}_4^+ = \left[ \begin{array}{c}
               0 \\ 0 \\ 0 \\ 1
             \end{array} \right].
\end{equation*}
The equilibrium $E_-$ is unstable with eigenvalues
$\sqrt{\frac{3}{2}}\eps^{1/4}$,
$\sqrt{\frac{3}{2}}\eps^{1/4}$, 
$2\sqrt{\frac{2}{3}}\eps^{1/4}$, $\sqrt{6}\eps^{1/4}$,
and eigenvectors
\begin{equation*}
{\boldsymbol w}_1^- = \left[ \begin{array}{c}
               0 \\ 0 \\ 1 \\ 0
             \end{array} \right],
\quad
{\boldsymbol w}_2^- = \left[ \begin{array}{c}
               0 \\ 0 \\ 0 \\ 1
             \end{array} \right],
\quad
{\boldsymbol w}_3^- = \left[ \begin{array}{c}
        \sqrt{3} \\ \sqrt{2}\eps^{1/4} \\ 0 \\ 0
             \end{array} \right],
\quad
{\boldsymbol w}_4^- = \left[ \begin{array}{c}
               1 \\ 0 \\ 0 \\ 0
             \end{array} \right].
\end{equation*}

There is also a two-dimensional surface of equilibria
\begin{equation}
\label{eq:surface-hat}
        \hat{\mathcal{S}}=\{ r_1=0,\,\,
                       \delta_1=0,\,\, p_1=0,\,\,
                       v_1=1,\,\,
                       q_1 \in \mathbb{R},\,\, a_1 \in \mathbb{R}\}.   
\end{equation}
The eigenvalues of the Jacobian at points on $\hat{\mathcal{S}}$ are
\begin{equation*}
        \lambda_s = -\sqrt{2}\eps^{1/4}, \quad
        \lambda_u = \sqrt{2}\eps^{1/4}, \quad
        \lambda_c = 0,0.
\end{equation*}
Hence, $\hat{\mathcal{S}}$
is a surface of saddle points.
The associated stable, unstable, and center subspaces are
\begin{equation*}
\mathbb E_s = \operatorname{span} \left\{ \left[ \begin{array}{c}
        2\sqrt{2} \\
         4\eps^{1/4} \\
         3 \eps^{1/4}q_1 \\
         3 \eps^{1/4} a_1
       \end{array} \right] \right\},
\quad
\mathbb E_u = \operatorname{span} \left\{ \left[ \begin{array}{c}
          -2\sqrt{2} \\
         4\eps^{1/4} \\
         3 \eps^{1/4}q_1 \\
         3 \eps^{1/4} a_1          
       \end{array} \right] \right\},
\quad
\mathbb E_c 
             = {\rm span} \left\{ \left[ \begin{array}{c}
                 0 \\  0 \\ 1 \\ 0
       \end{array} \right],
             \left[ \begin{array}{c}
                     0 \\  0 \\ 0 \\ 1
       \end{array} \right] \right\}.
\end{equation*}
Therefore, this subsystem has a two-dimensional center manifold $W^c(\hat{\mathcal{S}}) \cap \{ r_1=0 \} \cap \{ \delta_1=0 \}$. 
It emanates from $\ell$, and $W^c(\hat{\mathcal{S}}) \cap \{ r_1=0 \} \cap \{ \delta_1 =0 \}$ contains
$W^c(\ell) \cap \{ r_1=0 \} \cap \{ \delta_1 =0 \}\cap \{ a_1=0 \}$.

\medskip
Up next is the invariant hyperplane $\{ \delta_1 = 0 \}$.
The full system \eqref{eq:K1} is fifth-order here:
\begin{equation}
        \label{eq:delta1}
\begin{split}
        \dot{r}_1 &= \frac{1}{2}\sqrt{\eps} p_1 r_1 \\
        \dot{p}_1 &= 1 - v_1 - \frac{3}{2} \sqrt{\eps} p_1^2
                              + \frac{1}{3} \sqrt{\eps} r_1^2 \\
        \dot{v}_1 &= -2 \sqrt{\eps} p_1 v_1 \\
        \dot{q}_1 &= -\frac{3}{2} \sqrt{\eps} p_1 q_1 \\
        \dot{a}_1 &= \frac{3}{2} \sqrt{\eps} p_1 a_1.
\end{split}
\end{equation}
The line $I$ given by \eqref{eq:I} is invariant, and $E_{\pm}$ are fixed points on it.
The equilibrium $E_+$ is a saddle with stable eigenvalues
$-\sqrt{6}\eps^{1/4}$,
$-2\sqrt{\frac{2}{3}}\eps^{1/4}$,
$-\sqrt{\frac{3}{2}}\eps^{1/4}$, and $-\sqrt{\frac{3}{2}}\eps^{1/4}$, and unstable eigenvalue $\frac{1}{\sqrt{6}}\eps^{1/4}$.
The associated stable and unstable eigenspaces are
\begin{equation*}
\mathbb E_s = \operatorname{span} \left\{
\left[ \begin{array}{c}
              0 \\ 1 \\ 0 \\ 0 \\ 0
             \end{array} \right],
\,\,
\left[ \begin{array}{c}
        0 \\-\sqrt{3} \\ \sqrt{2}\eps^{1/4} \\ 0 \\ 0
             \end{array} \right],
\,\,
\left[ \begin{array}{c}
          0 \\ 0 \\ 0 \\ 1 \\ 0
             \end{array} \right],
\,\,
\left[ \begin{array}{c}
           0 \\ 0 \\ 0 \\ 0 \\ 1
             \end{array} \right]
\right\}, \quad 
\mathbb E_u = \operatorname{span} \left\{ \left[ \begin{array}{c}
            1 \\ 0 \\ 0 \\ 0 \\ 0
             \end{array} \right] \right\}. 
\end{equation*}
The equilibrium $E_-$ is a saddle with stable eigenvalue
$-\frac{1}{\sqrt{6}} \eps^{1/4}$
and unstable eigenvalues $\sqrt{\frac{3}{2}}\eps^{1/4}$,
$\sqrt{\frac{3}{2}}\eps^{1/4}$, 
$2\sqrt{\frac{2}{3}}\eps^{1/4}$, $\sqrt{6}\eps^{1/4}$,
The associated stable and unstable eigenspaces are 
\begin{equation*}
\mathbb E_s =\operatorname{span} \left\{ \left[ \begin{array}{c}
           1 \\ 0 \\ 0 \\ 0 \\ 0
             \end{array} \right] \right\},
\quad
\mathbb E_u  = \operatorname{span} \left\{ \left[ \begin{array}{c}
           0 \\ 0 \\ 0 \\ 1 \\ 0
             \end{array} \right],
\,\,
\left[ \begin{array}{c}
         0 \\ 0 \\ 0 \\ 0 \\ 1
             \end{array} \right],
\,\,
\left[ \begin{array}{c}
        0 \\ \sqrt{3} \\ \sqrt{2}\eps^{1/4} \\ 0 \\ 0
             \end{array} \right],
\,\,
\left[ \begin{array}{c}
               0 \\ 1 \\ 0 \\ 0 \\ 0
             \end{array} \right]
\right\}.
\end{equation*}

There is also a two-dimensional surface of equilibria
\begin{equation}
\label{eq:surface-tilde}
        \tilde{\mathcal{S}}=\{ r_1=0,\,\,
                       \delta_1=0,\,\, p_1=0,\,\,
                       v_1=1 + \frac{1}{3} \sqrt{\eps} r_1^2, \,\,
                       q_1 \in \mathbb{R},\,\, a_1 \in \mathbb{R}\}
\end{equation}
The eigenvalues of the Jacobian at points on $\tilde{\mathcal{S}}$ are
\begin{equation*}
        \lambda_s = -\sqrt{2\sqrt{\eps} + \eps r_1^2}, \quad
        \lambda_u = \sqrt{2\sqrt{\eps} + \eps r_1^2}, \quad
        \lambda_c = 0, 0, 0.
\end{equation*}
Hence, $\tilde{\mathcal{S}}$
is a surface of saddle points. 
The associated stable and unstable subspaces are
\begin{equation*}
\mathbb E_s = \operatorname{span} \left\{ \left[ \begin{array}{c}
        -\sqrt{\eps} r_1 \\ 2\sqrt{2\sqrt{\eps} + \eps r_1^2} \\
         \frac{4}{3}\sqrt{\eps}(3+\sqrt{\eps}r_1^2) \\
         3 \sqrt{\eps}q_1 \\
         3 \sqrt{\eps} a_1
       \end{array} \right] \right\},
\quad
\mathbb E_u = \operatorname{span} \left\{ \left[ \begin{array}{c}
          -\sqrt{\eps} r_1 \\ -2\sqrt{2\sqrt{\eps} + \eps r_1^2} \\
         \frac{4}{3}\sqrt{\eps}(3 + \sqrt{\eps}r_1^2) \\
         3 \sqrt{\eps} q_1 \\
         3 \sqrt{\eps} a_1          
       \end{array} \right] \right\}.
\end{equation*}
The associated center subspace is 
\begin{equation*} 
\mathbb E_c 
             = {\rm span} \left\{ \left[ \begin{array}{c}
                 0 \\ 0 \\ 0 \\ 1 \\ 0
       \end{array} \right],
             \left[ \begin{array}{c}
                     0 \\ 0 \\ 0 \\ 0 \\ 1
       \end{array} \right], 
        \left[ \begin{array}{c}
                     3 \\ 0 \\ 2\sqrt{\eps}r_1  \\ 0 \\ 0
       \end{array} \right] 
       \right\}.
\end{equation*}
Therefore, this subsystem has a three-dimensional center manifold $W^c(\tilde{\mathcal{S}}) \cap \{ r_1=0 \} \cap \{ \delta_1=0 \}$. 
It emanates from $\ell$, coincides with the critical manifold $S$, and $W^c(\tilde{\mathcal{S}}) \cap \{ \delta_1 =0 \}$ contains
$W^c(\ell) \cap \{ \delta_1 =0 \}\cap \{ a_1=0 \}$.

\medskip
In the invariant hyperplane $\{ r_1=0 \}$, system \eqref{eq:K1} is the following fifth-order set of equations:
\begin{equation}
        \label{eq:r1=0}
\begin{split}
\frac{d \delta_1}{dy_1} &= -\sqrt{\eps}p_1 \delta_1 \\
\frac{dp_1}{dy_1} &= 1-v_1-\frac{3}{2}\sqrt{\eps}p_1^2 \\
\frac{dv_1}{dy_1} &= \sqrt{\eps} \left(-2p_1 v_1 + \delta_1 q_1 \right) \\
\frac{dq_1}{dy_1} &= -\frac{3}{2}\sqrt{\eps} p_1 q_1 +\delta_1 \\
         \frac{da_1}{dy_1} &= -\frac{3}{2}\sqrt{\eps} p_1 a_1.
\end{split}
\end{equation}
The line $I$ given by \eqref{eq:I} is still invariant for this system, and $E_{\pm}$ remain as fixed points on $I$, recall \eqref{eq:Epm}.
The equilibrium $E_+$ is stable with spectrum 
\begin{equation}
    \label{eq:r1=0-E+}
\sigma^s_+ = \left\{ -\sqrt{6}\eps^{1/4},
-2\sqrt{\frac{2}{3}}\eps^{1/4},
-\sqrt{\frac{3}{2}}\eps^{1/4},
-\sqrt{\frac{3}{2}}\eps^{1/4},
-\sqrt{\frac{2}{3}}\eps^{1/4}\right\}
\end{equation}
At $E_+$, the stable subspace is 
\begin{equation}
\label{eq:Esu4E+r1=0}
\mathbb{E}^s = {\rm span} \ \ 
\left\{ 
\left[ \begin{array}{c}
               0 \\ 1 \\ 0 \\ 0 \\ 0 
             \end{array} \right],
\left[ \begin{array}{c} 0 \\
        -\sqrt{3} \\ \sqrt{2}\eps^{1/4} \\ 0 \\ 0
             \end{array} \right],
\left[ \begin{array}{c}
             0 \\ 0 \\ 0 \\ 1 \\ 0 
             \end{array} \right],
\left[ \begin{array}{c}
             0 \\ 0 \\ 0 \\ 0 \\ 1 
             \end{array} \right],
\left[ \begin{array}{c}
             \eps^{1/4} \\ 0 \\ 0 \\ \sqrt{6} \\ 0 
             \end{array} \right]
\right\}
 \end{equation}

The other equilibrium, $E_-$, is unstable with spectrum given by
\begin{equation}
\label{eq:r1=0:E-}
    \sigma_-^u = \left\{ 
    \sqrt{\frac{2}{3}}\eps^{1/4},
    \sqrt{\frac{3}{2}}\eps^{1/4},
    \sqrt{\frac{3}{2}}\eps^{1/4},
2\sqrt{\frac{2}{3}}\eps^{1/4}, \sqrt{6}\eps^{1/4}
\right\}.
\end{equation}
At $E_-$, the unstable eigenspace is 
\begin{equation}
\label{eq:Esu4E-r1=0}
\mathbb{E}^u = {\rm span} \ \left\{
\left[ \begin{array}{c}
             -\eps^{1/4} \\ 0 \\ 0 \\ \sqrt{6} \\ 0 
             \end{array} \right],
\left[ \begin{array}{c}
             0 \\ 0 \\ 0 \\ 1 \\ 0 
             \end{array} \right],
\left[ \begin{array}{c}
             0 \\ 0 \\ 0 \\ 0 \\ 1 
             \end{array} \right],
\left[ \begin{array}{c} 
        0 \\ \sqrt{3} \\ \sqrt{2}\eps^{1/4} \\ 0 \\ 0
             \end{array} \right],
\left[ \begin{array}{c} 0 \\
        1 \\ 0 \\ 0 \\ 0
             \end{array} \right]
\right\}.
\end{equation}

There is also a manifold of equilibria
\begin{equation}
\label{eq:mfldS0}
\mathcal{S}_0 = \{ r_1=0,\,\,
    \delta_1=0,\,\, p_1=0,\,\, v_1=1,\,\, q_1 \in \mathbb{R},\,\, a_1 \in \mathbb{R} \}.
\end{equation}
It is a manifold of saddle fixed points, since the eigenvalues are
\begin{equation*}
        \lambda_s = -\sqrt{2} \eps^{1/4}, \quad
        \lambda_u = \sqrt{2} \eps^{1/4}, 
        \quad
        \lambda_c = 0, 0, 0.
\end{equation*}
The associated eigenspaces are
\begin{equation*}
%\label{eq:EsEur1=0}
\mathbb{E}^s = {\rm span} \ 
\left\{ 
\left[ \begin{array}{c}
             0 \\ 2 \sqrt{2}\\ 4\eps^{1/4} \\  3 \eps^{1/4} q_1 \\  3\eps^{1/4}a_1  
             \end{array} \right]
             \right\},
             \quad
\mathbb{E}^u = {\rm span} \ \left\{
\left[ \begin{array}{c}
0 \\ - 2 \sqrt{2} \\ 4\eps^{1/4} \\  3\eps^{1/4} q_1 \\  3\eps^{1/4}a_1  
\end{array} \right] \right\},
%\end{equation}
%\begin{equation}
%\label{eq:Ecr1=0}
\quad
\mathbb{E}^c = {\rm span} \ \left\{ 
\left[ \begin{array}{c}
               0 \\ 0 \\ 0 \\ 1 \\ 0 
             \end{array} \right],
\left[ \begin{array}{c}
               0 \\ 0 \\ 0 \\ 0 \\ 1 
             \end{array} \right],
\left[ \begin{array}{c} \frac{4}{4-3\sqrt{\eps}q_1^2} \\
        \frac{2q_1}{4-3\sqrt{\eps}q_1^2} \\ 0 \\ 0 \\ 0
             \end{array} \right]
\right\}.
\end{equation*}
Therefore, in the hyperplane $\{ r_1=0 \}$, there is 
a three-dimensional
center manifold
\begin{equation}
 \label{eq:N}
    N = W^c(\mathcal{\ell}).
\end{equation}
This manifold contains the surface $\mathcal{S}_0$ of equilibria.
Moreover, as shown in Section~\ref{sec:K1}, these key manifolds persist in the full system \eqref{eq:K1} for $0< r_1 \ll 1$.

\medskip

%------------------------------------------------------------------
\section{The proof of Lemma \texorpdfstring{\ref{lem:persistence}}{Lg}} {\label{sec:appGammadelta}}
%------------------------------------------------------------------

\begin{proof}
   The proof is split into two steps. First, we calculate a regular perturbation expansion in powers of $r_2$ of the solution $\Gamma_0$ in the $H_2=0$ level set. This expansion is valid in $K_2$ on arbitrary finite intervals of the independent variable $y_2$. Second, we identify the fixed points on the equator of the hemisphere to which this perturbed solution limits as $y_2 \to \pm \infty$, thereby establishing the persistence of the heteroclinic connection, which corresponds to the intersection of the two invariant manifolds. 
   %Third, and finally, we show that this persistent heteroclinic lies in the transverse intersection of the two invariant manifolds. 
   All calculations in these steps were performed both by hand and symbolically.

{\bf Step 1.}
Let 
\begin{equation}
\label{eq:Gammadelta}
    \Gamma_\delta =  \Gamma_0 + r_2 \Gamma_1 + r_2^2 \Gamma_2 + r_2^3 \Gamma_3 + \mathcal{O}(r_2^4), \quad {\rm with} \quad
    \Gamma_{k}=(u_{2k},p_{2k},v_{2k},q_{2k}), \ \ k=1,2,\ldots.
\end{equation}
We substitute this expansion into \eqref{eq:ODE-K2} and solve for $\Gamma_k$, order by order in powers of $r_2$. At $\mathcal{O}(r_2^0)$, we recover $\Gamma_0$, as given by \eqref{eq:Gamma0}. 
At $\mathcal{O}(r_2)$, 
the equations are
\begin{equation*}
    \begin{split}
        {u_{21}}' &= \sqrt{\eps} p_{21} \\
  {p_{21}}' &= 2 u_{20}u_{21} - v_{21} \\
  {v_{21}}' &= \sqrt{\eps} q_{21}\\
  {q_{21}}' &= u_{21} - a_2.
    \end{split}
\end{equation*}
In the algebraic solution, there is a free parameter $\chi_{21}\in \mathbb{R}$: 
\begin{equation*}
\begin{split}
u_{21}(y_2) &= \chi_{21} y_2 + \frac{3}{2} a_2, \quad
p_{21}(y_2)=\frac{1}{\sqrt{\eps}} \chi_{21}, \\
v_{21}(y_2)&=\frac{1}{6}\sqrt{\eps}\chi_{21}y_2^3 + \frac{1}{4}\sqrt{\eps} a_2 y_2^2, \quad
q_{21}(y_2) =\frac{1}{2} \chi_{21} y_2^2 +\frac{1}{2} a_2 y_2.
\end{split}
\end{equation*}

At $\mathcal{O}(r_2^2)$,  the equations are
\begin{equation*}
    \begin{split}
        {u_{22}}' &= \sqrt{\eps} p_{22} \\
  {p_{22}}' &= 2 u_{20}u_{22} + u_{21}^2 - v_{22} + \frac{1}{3}\sqrt{\eps} u_{20}^3 \\
  {v_{22}}' &= \sqrt{\eps} q_{22}\\
  {q_{22}}' &= u_{22}.
    \end{split}
\end{equation*}
Here, there is also a free parameter $\chi_{22}\in \mathbb{R}$
in the algebraic solution. We find 
\begin{equation*}
\begin{split}
u_{22}(y_2) &= \frac{3}{\sqrt{\eps}} \chi_{21}^2  + \frac{5}{96}\sqrt{\eps} + \chi_{22} y_2 - \frac{5\eps^{3/2}}{3,456} y_2^4, \\
p_{22}(y_2)&=\frac{1}{\sqrt{\eps}} \chi_{22} - \frac{5}{864}\eps y_2^3, \\
v_{22}(y_2)&=\frac{9}{4} a_2^2 + 3 a_2 \chi_{21} y_2 + \left( \frac{3}{2}\chi_{21}^2 + \frac{5\eps}{192}\right) y_2^2 + \frac{1}{6}\sqrt{\eps}\chi_{22}y_2^3 - \frac{\eps^2}{20,736} y_2^6, \\
q_{22}(y_2) &=\frac{3}{\sqrt{\eps}} a_2 \chi_{21} 
+\left(  
\frac{3}{\sqrt{\eps}} \chi_{21}^2 + \frac{5\sqrt{\eps}}{96}
\right) y_2 + \frac{1}{2} \chi_{22} y_2^2 - \frac{\eps^{3/2}}{3,456} y_2^5.
\end{split}
\end{equation*}

At $\mathcal{O}(r_2^3)$,  the equations are
\begin{equation*}
    \begin{split}
        {u_{23}}' &= \sqrt{\eps} p_{23} \\
  {p_{23}}' &= 2 u_{20}u_{23} + 2u_{21}u_{22} - v_{23} + \sqrt{\eps} u_{20}^2 u_{21} \\
  {v_{23}}' &= \sqrt{\eps} q_{23}\\
  {q_{23}}' &= u_{23}.
    \end{split}
\end{equation*}
The algebraic solution at this order has the free parameter $\chi_{23}\in \mathbb{R}$:
\begin{equation*}
\begin{split}
u_{23}(y_2) &= \frac{6}{\sqrt{\eps}} \chi_{21} \chi_{22} + \chi_{23} y_2 - \frac{7\eps}{96} a_2 y_2^2 - \frac{5\eps}{144}\chi_{21} y_2^3,\\
p_{23}(y_2)&=\frac{1}{\sqrt{\eps}} \chi_{23} - \frac{7\sqrt{\eps}}{48} a_2 y_2 - \frac{5\sqrt{\eps}}{48} \chi_{21} y_2^2, \\
v_{23}(y_2)&= a_2 \left( \frac{9}{\sqrt{\eps}} \chi_{21}^2 + \frac{29 \sqrt{\eps}}{96} \right) 
+\left( 3 a_2 \chi_{22} + \frac{6}{\sqrt{\eps}} \chi_{21}^3 + \frac{5\sqrt{\eps}}{16} \chi_{21} 
\right) y_2 \\
&+ 3 \chi_{21} \chi_{22} y_2^2 + \frac{\sqrt{\eps}}{6} \chi_{23} y_2^3 - \frac{7 \eps^{3/2}}{1,152}a_2 y_2^4
-\frac{\eps^{3/2}}{576} \chi_{21} y_2^5, \\
q_{23}(y_2) &=\frac{1}{\sqrt{\eps}} \left( 3 a_2 \chi_{22} + \frac{6}{\sqrt{\eps}} \chi_{21}^3 + \frac{5\sqrt{\eps}}{16} \chi_{21} \right)
+\frac{6}{\sqrt{\eps}} \chi_{21} \chi_{22} y_2 \\
&+ \frac{1}{2} \chi_{23} y_2^2 - \frac{7\eps}{288} a_2 y_2^3 - \frac{5\eps}{576} \chi_{21} y_2^4.
\end{split}
\end{equation*}

\noindent
{\bf Step 2.}
We substitute the expansions for $(u_2,p_2,v_2,q_2)$ from Step 1 into the Hamiltonian $H_2$ given by \eqref{eq:H2}. After lengthy calculations,  we find that all of the terms that depend on powers of $y_2$ vanish, and one is left with:
\begin{equation*}
    H_2 \vert_{\Gamma_\delta} =
    -\frac{1}{12} a_2 r_2 - \frac{5\sqrt{\eps}}{576} r_2^2
    -\frac{1}{6}\sqrt{\eps}a_2^2 r_2^4
    +\mathcal{O}(r_2^5).
\end{equation*}
Next, we multiply $H_2$ by $12\sqrt{\eps} r_2^2$, recall that
$\tilde{a}=\sqrt{\eps}r_2^3 a_2$ by \eqref{eq:dynrescale-K2}, 
and impose $H_2 \vert_{\Gamma_\delta} = 0$. This yields
\[ 
\tilde{a} = - \frac{5\eps}{48} r_2^4 + \mathcal{O}(r_2^6).
\]
Finally, we recall that $\delta=r_2^2$ in chart $K_2$ and $a=1+\tilde{a}$ by  \eqref{eq:translate}.
Therefore, we have derived the formula $a_c(\delta) = 1 - \frac{5\eps}{48}\delta^2 + \mathcal{O}(\delta^3)$, given in \eqref{eq:lemma1-exp-a}.

\end{proof}

%----------------------------------------------------------------------------------------------------

%----------------------------------------------------------------------------------------------------

%=========================================================================================
\end{document}